\documentclass[a4paper]{amsart}

\usepackage[margin=3cm]{geometry}
\usepackage{amsmath, amssymb, amsthm}
\usepackage{mathtools}
\usepackage{color}
\usepackage{comment}
\usepackage{amsfonts, latexsym}
\usepackage{extpfeil}
\usepackage{tikz-cd}
\usepackage{mathrsfs} 
\usepackage{array, booktabs, float}
\usepackage{accessibility}
\usepackage{hyperref}

\usepackage{imakeidx}
\makeindex

\theoremstyle{plain}
\newtheorem{theorem}{Theorem}[section]
\newtheorem{proposition}[theorem]{Proposition}
\newtheorem{lemma}[theorem]{Lemma}
\newtheorem{corollary}[theorem]{Corollary}

\theoremstyle{definition} 
\newtheorem{definition}[theorem]{Definintion}

\theoremstyle{remark} 
\newtheorem{remark}[theorem]{Remark}
\newtheorem{example}[theorem]{Example}

\newcommand{\Q}{\mathbb{Q}}
\newcommand{\Z}{\mathbb{Z}}
\newcommand{\g}{\mathfrak{g}}
\newcommand{\Blow}{\mathbf{B}^{\text{low}}}
\newcommand{\Bup}{\mathbf{B}^{\text{up}}}
\newcommand{\Glow}{G^{\textup{low}}}
\newcommand{\Gup}{G^{\textup{up}}}

\newcommand{\e}{\tilde{e}}
\newcommand{\f}{\tilde{f}}
\newcommand{\bfk}{\mathbf{k}}
\newcommand{\Sym}{\mathfrak{S}}
\newcommand{\gMod}[1]{{#1}\textup{-gMod}}
\newcommand{\gmod}[1]{{#1}\textup{-gmod}}
\newcommand{\gproj}[1]{{#1}\textup{-gproj}}
\newcommand{\quantum}[1]{U_q^{#1}(\mathfrak{g})}
\newcommand{\universalR}[2]{{\operatorname{R}}_{{#1},{#2}}^{\mathrm{uni}}}
\newcommand{\renormalizedR}[2]{{\operatorname{R}}_{{#1}, {#2}}^{\mathrm{ren}}}
\newcommand{\braider}[2]{\operatorname{R}_{#1}({#2})}
\newcommand{\fpdgMod}[1]{{#1}\textup{-gMod}_{\mathrm{f.p.d.}}}
\newcommand{\Extform}[2]{\langle {#1}, {#2} \rangle_{\mathrm{Ext}}}
\newcommand{\Lusform}[2]{({#1}, {#2} )_{\mathrm{Lus}}}
\newcommand{\prroot}{\Phi_+^{\mathrm{re}}}
\newcommand{\minroot}{\Phi_+^{\mathrm{min}}}
\newcommand{\Mod}[1]{{#1}\textup{-Mod}}
\newcommand{\qcoordinate}{A_q(\mathfrak{n})}
\newcommand{\Rmatrix}[2]{\operatorname{R}_{{#1},{#2}}}

\DeclareMathOperator{\Hom}{Hom}
\DeclareMathOperator{\wt}{wt}
\DeclareMathOperator{\Res}{Res}
\DeclareMathOperator{\height}{ht}
\DeclareMathOperator{\Ind}{Ind}
\DeclareMathOperator{\Coind}{Coind}
\DeclareMathOperator{\hd}{hd}
\DeclareMathOperator{\soc}{soc}
\DeclareMathOperator{\rad}{rad}
\DeclareMathOperator{\qdim}{qdim}
\DeclareMathOperator{\id}{id}
\DeclareMathOperator{\Ext}{Ext}
\DeclareMathOperator{\End}{End}
\DeclareMathOperator{\Cok}{Cok}
\DeclareMathOperator{\Ker}{Ker}
\DeclareMathOperator{\Image}{Im}
\DeclareMathOperator{\gldim}{gl{.}dim}
\DeclareMathOperator{\prdim}{pr{.}dim}
\DeclareMathOperator{\flatdim}{fl{.}dim}
\DeclareMathOperator{\ext}{ext}
\DeclareMathOperator{\grend}{end}
\DeclareMathOperator{\W}{W}
\DeclareMathOperator{\spn}{span}
\DeclareMathOperator{\Pro}{Pro}
\DeclareMathOperator{\Tor}{Tor}

\title[Affine highest weight structures over quiver Hecke algebras]{Affine highest weight structures on module categories over quiver Hecke algebras}
\author{HARUTO MURATA}

\address{Graduate School of Mathematical Sciences, the University of Tokyo, 3-8-1 Komaba, Meguro-ku, Tokyo 153-8914, Japan.}
\email{muraharu@ms.u-tokyo.ac.jp}


\makeindex

\begin{document}
\begin{abstract}
We prove that the category of finitely generated graded modules over the quiver Hecke algebra of arbitrary type admits numerous stratifications in the sense of Kleshchev. 
A direct consequence is that the full subcategory corresponding to the quantum unipotent subgroup associated with any Weyl group element is an affine highest weight category. 
Our results significantly generalize earlier works by Kato, Brundan, Kleshchev, McNamara and Muth.  
The key ingredient is a realization of standard modules via determinantial modules. 
We utilize the technique of R-matrices to study these standard modules. 
\end{abstract}

\maketitle

\tableofcontents

\section{Introduction}
Quiver Hecke algebra, discovered independently by Khovanov-Lauda \cite{MR2525917,MR2763732} and Rouquier \cite{rouquier20082kacmoodyalgebras}, 
is an infinite dimensional graded algebra associated with a root datum.
It is important because it controls 2-morphisms of the 2-Kac-Moody algebra. 
Furthermore, representation theory of the quiver Hecke algebra is related to various other algebras including symmetric groups \cite{MR2551762}, affine Hecke algebras \cite{rouquier20082kacmoodyalgebras}, and quantum affine algebras \cite{MR3748315}. 
It is also interesting in its own right. 
One reason is the monoidal functor known as the convolution product. 
To analyze structures of convolution products, a powerful methodology utilizing R-matrices - distinguished homomorphisms between convolution products - has been developed \cite{MR3748315, MR3758148, MR3790066, MR4717658}. 

Another significant aspect of the representation theory of the quiver Hecke algebra is a categorification of the quantum group, 
which has been a major driving force in this area of research. 
The starting point is an isomorphism between the Grothendieck ring of a category of graded modules and the quantum group $\quantum{-}$ \cite{MR2525917, MR2763732}. 
On one hand, the quantum group provides numerical information about modules over the quiver Hecke algebra, such as Jordan-H\"older multiplicities and branching rules. 
This observation traces back to Lascoux-Leclerc-Thibon conjecture \cite{MR1410572} on modular representations of symmetric groups, which was proved by Ariki \cite{MR1443748}. 
On the other hand, the categorical viewpoint offers profound insights into the features of the quantum group including positivity. 

In recent years, many attempts have been made to strengthen this connection by finding categorical counterparts of intricate structures of the quantum group, such as braid group symmetry and cluster structures. 
For instance, see \cite{MR3165425, MR4717658, MR3758148}. 
However, progress in this direction remains incomplete. 
A major difficulty is caused by the fact that modules over the quiver Hecke algebra are not necessarily completely reducible: 
we need a deeper understanding of nontrivial extensions between modules. 

One classical way for addressing such non-semisimple representation theory is the framework of highest weight categories, developed by Cline-Parshall-Scott \cite{MR961165}. 
A prototypical example is the BGG category $\mathcal{O}$ for a semisimple Lie algebra.
Kleshchev extended this framework to infinite dimensional graded analogues, affine highest weight categories and, more generally, stratified categories \cite{MR3335289}.
In this formulation, the role of Verma modules and dual Verma modules in the category $\mathcal{O}$ are played by standard modules and proper costandard modules respectively. 

Regarding stratifications of the module category over the quiver Hecke algebra, several results have been obtained in some special cases. 
When $\mathfrak{g}$ is of finite ADE type and the base field is of characteristic zero, Kato \cite{MR3165425} proved that it is an affine highest weight category using geometric methods.  
In the case of arbitrary finite types, Brundan-Kleshchev-McNamara \cite{MR3205728} algebraically demonstrated that it is an affine highest weight category without assuming characteristic zero. 
Their proof is based on \cite{MR3403455}, which established the finiteness of the global dimension relying on case-by-case computations. 
When $\mathfrak{g}$ is of symmetric affine type and the characteristic is zero, McNamara \cite{MR3694676} provided stratifications, and   
later Kleshchev-Muth \cite{MR3612467} confirmed that the result still holds when the characteristic is sufficiently large. 

However, beyond these cases, such results were not known. 
Moreover, even in finite or symmetric affine cases, there was no useful algebraic description of standard modules. 

The aim of this paper is to extend these results to arbitrary symmetrizable Lie types over any field. 
We prove that, for any element $w$ of the Weyl group and its reduced expression, the module category over the quiver Hecke algebra has a stratification, 
and that a certain full subcategory corresponding to the quantum unipotent subgroup $U_q(\mathfrak{n}^- \cap w \mathfrak{n}^+)$ is an affine highest weight category. 
Our approach is both distinct from and independent of earlier works.  
We concretely realize standard modules and proper costandard modules using a special family of modules called determinantial modules, and establish our results by analyzing them using R-matrices. 
These modules arise naturally as categorical counterparts of the PBW monomials and dual PBW monomials of $U_q(\mathfrak{n}^- \cap w\mathfrak{n}^+)$. 
In our proof, no case-by-case computation is required. 
We anticipate that these explicit descriptions of standard modules will pave the way for further in-depth study of the homological properties of modules over the quiver Hecke algebra. 

Additionally, our study reveals a new perspective on R-matrices. 
While R-matrices have been primarily used to investigate problems related to simple modules, 
we use them to explore homological properties of modules. 

We explain our results precisely. 
Let $(C = (c_{i,j})_{i,j \in I}, P, \{ \alpha_i \}_{i \in I}, \{h_i\}_{i \in I},  (\cdot, \cdot))$ be a root datum. 
Associated with this datum, we have the symmetrizable Kac-Moody Lie algebra $\mathfrak{g}$, and its quantum group $\quantum{}$.
Let $W$ be the Weyl group, $\Phi_+$ the set of positive roots and $Q_+ = \sum_{i \in I }\mathbb{Z}_{\geq 0} \alpha_i$. 

By selecting some additional data, the quiver Hecke algebra $R(\beta)$ is defined for each $\beta \in Q_+$ over a field $\mathbf{k}$ of arbitrary characteristic. 
It is a Noetherian $\mathbb{Z}$-graded algebra. 
We just say an $R$-module to refer to an $R(\beta)$-module for some $\beta \in Q_+$. 
We want to understand the category of finitely generated graded $R(\beta)$-modules $\gMod{R(\beta)}$. 
Let $q$ be the degree shift functor acting on $\gMod{R(\beta)}$. 
One remarkable feature is that the category $\gMod{R} = \bigoplus_{\beta \in Q_+} \gMod{R(\beta)}$ has a monoidal structure: 
for $M \in \gMod{R(\beta)}$ and $N \in \gMod{R(\gamma)}$, we have a convolution product $M \circ N \in \gMod{R(\beta + \gamma)}$.  
The categorification of the quantum group is the following assertion: 
there is a $\mathbb{Q}(q)$-algebra isomorphism
\[
\bigoplus_{\beta \in Q_+} K(\gproj{R(\beta)}) \otimes_{\mathbb{Z}[q,q^{-1}]} \mathbb{Q}(q) \simeq \quantum{-},
\]
where $\gproj{R(\beta)}$ is the category of finitely generated graded projective $R(\beta)$-modules. 

For $w \in W$, there are two quotient rings of $R(\beta)$, $R_{w,*}(\beta)$ and $R_{*,w}(\beta)$, introduced by Kashiwara-Kim-Oh-Park \cite{MR3771147}. 
In their notation, the categories of finite dimensional modules $\gmod{R_{w,*}} = \mathscr{C}_{w,*}$ and $\gmod{R_{*,w}} = \mathscr{C}_{*,w}$. 
Here, we give a conceptual explanation of these quotient rings, though it may lack full precision. 

For $w \in W$, there are two subalgebras of $\quantum{-}$, $U_q(\mathfrak{n}^- \cap w \mathfrak{n}^+)$ and $U_q(\mathfrak{n}^- \cap w\mathfrak{n}^-)$,
where $\mathfrak{n}^- \cap w \mathfrak{n}^+ = \bigoplus_{\alpha \in \Phi_- \cap w \Phi_+} \mathfrak{g}_{\alpha}$ and $\mathfrak{n}^- \cap w\mathfrak{n}^- = \bigoplus_{\alpha \in \Phi_- \cap w \Phi_-} \mathfrak{g}_{\alpha}$. 
Fix a reduced expression $w = s_{i_1} \cdots s_{i_l}$, and write $w_{\leq k} = s_{i_1} \cdots s_{i_k} \ (1 \leq k\leq l)$. 
Letting $\beta_k = w_{\leq k-1}\alpha_{i_k}$, the set $\Phi_+ \cap w\Phi_-$ is $\{ \beta_1, \ldots, \beta_l\}$. 
We have the root vector $f_{\beta_k}$ and the dual root vector $f_{\beta_k}^{\mathrm{up}}$ for each $1 \leq k \leq l$. 
By multiplying these vectors, we obtain two bases of $U_q(\mathfrak{n}^- \cap w\mathfrak{n}^+)$: the PBW basis $\{ f_{\lambda} \mid \lambda \in \mathbb{Z}_{\geq 0}^l\}$ and the dual PBW basis $\{ f_{\lambda}^{\mathrm{up}} \mid \lambda \in \mathbb{Z}_{\geq 0}^l\}$. 

Under the categorification of the quantum group, module categories over $R_{w,*}$ and $R_{*,w}$ correspond to $U_q(\mathfrak{n}^-\cap w\mathfrak{n}^+)$ and $U_q(\mathfrak{n}^- \cap w \mathfrak{n}^-)$ respectively: 
see Section \ref{sub:quantumunipotent} for rigorous treatment. 
We can lift the PBW basis and the dual PBW basis of $U_q(\mathfrak{n}^- \cap w\mathfrak{n}^+)$ to certain $R_{w,*}$-modules. 
More precisely, we use determinantial modules and their affinizations, introduced by Kang, Kashiwara, Kim, Oh and Park \cite{MR3758148, MR3771147, MR4359265}. 
The determinantial module $M(w_{\leq k}\Lambda_{i_k}, w_{\leq k-1} \Lambda_{i_k})$ corresponds to $f_{\beta_k}^{\mathrm{up}}$ and its affinization $\widehat{M}(w_{\leq k}\Lambda_{i_k}, w_{\leq k-1} \Lambda_{i_k})$ to $f_{\beta_k}$. 
Taking convolution products of these modules, we obtain modules that categorify the dual PBW basis and the PBW basis, denoted by $\overline{\Delta}(\lambda)$ and $\Delta(\lambda)$ respectively. 

An $R$-module is said to be self-dual if it is isomorphic to its dual representation. 
It is known that, for any simple $R$-module $L$, there exists a unique integer $d$ such that $q^dL$ is self-dual. 
Let $\{ L(s) \}_{s \in S}$ be the complete set of self-dual simple $R_{*,w}$-modules up to isomorphism, and take $\beta_s \in Q_+$ so that $L(s)$ is an $R(\beta_s)$-module. 
For each $s \in S$, let $\Delta(s)$ be a projective cover of $L(s)$ in the category of finitely generated graded $R_{*,w}(\beta_s)$-modules. 
For each $(\lambda,s) \in \mathbb{Z}_{\geq 0}^l \times S$, we define
\[
\Delta(\lambda,s) = \Delta(s) \circ \Delta(\lambda), \ \overline{\Delta}(\lambda,s) = L(s) \circ \overline{\Delta}(\lambda). 
\]
The head of $\overline{\Delta}(\lambda,s)$, denoted by $L(\lambda,s)$, is a self-dual simple module and they classify all the self-dual simple $R$-modules. 
Let $P(\lambda,s)$ be the projective cover of $L(\lambda,s)$ in $\gMod{R}$. 

In this setup, the main theorem of this paper is as follows: 

\begin{theorem}[{Theorem \ref{thm:standard}, Theorem \ref{thm:stratification}}]  \label{thm:intro1}
 Let $\beta \in Q_+$. 
 The category $\gMod{R(\beta)}$ of finitely generated graded $R(\beta)$-modules is stratified with standard modules $\Delta(\lambda,s) \ (\sum_{1 \leq k \leq l} \lambda_k\beta_k + \beta_s = \beta)$. 
 Precisely speaking, there exists a certain preorder $\leq$ on the labelling set $\Sigma(\beta) = \{ (\lambda,s) \mid \sum_k \lambda_k\beta_k + \beta_s = \beta\}$ and satisfies the following conditions: 
 \begin{enumerate}
 \item For any $\sigma \in \Sigma(\beta)$, we have $\Delta(\sigma) \simeq P(\sigma) / K(\sigma)$, where 
 \[
 K(\sigma) = \sum_{\sigma' \in \Sigma (\beta), \sigma' \not \leq \sigma, f \colon P(\sigma') \to P(\sigma)} \Image f. 
 \]
 \item For any $\sigma \in \Sigma(\beta)$, the module $K(\sigma)$ has a finite filtration whose successive quotients are of the form $q^d \Delta(\sigma') \ (d \in \mathbb{Z}, \sigma' < \sigma)$. 
 \end{enumerate}
\end{theorem}

By a routine truncation argument, we obtain the former part of the following theorem. 

\begin{theorem}[{Theorem \ref{thm:affinehighest}}] \label{thm:intro2}
  Let $\beta \in Q_+$. 
  The category $\gMod{R_{w,*}(\beta)}$ of finitely generated graded $R_{w,*}(\beta)$-modules is stratified with standard modules $\Delta(\lambda) \ (\sum_{1 \leq k \leq l}\lambda_k\beta_k = \beta)$. 
  Moreover, for any $\lambda$, the ring $\End_{R(\beta)}(\Delta(\lambda))$ is a polynomial algebra, that is, the category $\gMod{R_{w,*}(\beta)}$ is an affine (polynomial) highest weight category. 
\end{theorem}

The latter part is not difficult to prove. 
We immediately have the following corollary: 

\begin{corollary}[{Corollary \ref{cor:gldim}}] 
For any $\beta \in Q_+$, the global dimension of $\gMod{R_{w,*}(\beta)}$ is finite.   
\end{corollary}

We give a brief overview on the proof of Theorem \ref{thm:intro1}. 
There are two key Theorems. 
The first one is 

\begin{theorem}[{Theorem \ref{thm:ext1}}] \label{thm:introext1}
Let $v \in W$ and $i \in I$. 
If $vs_i > v$, we have
\[
\Ext_{R(\beta)}^1 (\widehat{M}(vs_i\Lambda_i, v\Lambda_i), M(vs_i\Lambda_i, v\Lambda_i)) = 0.   
\]
\end{theorem}

In particular, we apply this proposition to $v = w_{\leq k-1}$ and $i = i_k$.
In proving Theorem \ref{thm:introext1}, the following two propositions are crucial. 

\begin{proposition}[{Corollary \ref{cor:M1ext}}] \label{prop:introext1}
For $x \in W$ and $i \in I$, we have 
\[
\Ext_{R(\beta)}^1 (\widehat{M}(x\Lambda_i, \Lambda_i), M(x\Lambda_i, \Lambda_i)) = 0.   
\]
\end{proposition}

\begin{proposition}[{Theorem \ref{thm:SES}}] \label{prop:introSES}
Let $v \in W$ and $i \in I$. 
If $vs_i > v$, we have the following short exact sequence
\begin{align*}
  0 &\to q^{(\alpha_i,\alpha_i)+ (\Lambda_i-v\Lambda_i, v\alpha_i)} \widehat{M}(v\Lambda_i,\Lambda_i) \circ \widehat{M}(vs_i\Lambda_i,v\Lambda_i) \\ 
  & \to \widehat{M}(vs_i\Lambda_i,v\Lambda_i) \circ \widehat{M}(v\Lambda_i,\Lambda_i) \to \widehat{M}(vs_i\Lambda_i,\Lambda_i) \to 0,   
\end{align*}
where the injection is given by a special homomorphism called renormalized R-matrix. 
\end{proposition}

Note that the short exact sequence in Proposition \ref{prop:introSES} consists of only two kinds of modules: 
the one includes $\widehat{M}(v\Lambda_i,\Lambda_i)$ and $\widehat{M}(vs_i\Lambda_i,\Lambda_i)$, which are studied in Proposition \ref{prop:introext1}, 
and the other includes $\widehat{M}(vs_i\Lambda_i,v\Lambda_i)$, which we have to study in Theorem \ref{thm:introext1}. 
Based on this observation, Theorem \ref{thm:introext1} is deduced from these two propositions. 

In the proof of Proposition \ref{prop:introext1}, 
we use a category equivalence deduced from Kang-Kashiwara's categorification of highest weight integrable modules \cite{MR3134909}, 
and 2-representation theory of the 2-categorical $U_q(\mathfrak{sl}_2)$ \cite{rouquier20082kacmoodyalgebras}. 

The second key theorem is the following. 

\begin{theorem}[{Proposition \ref{prop:BGGformula}}] \label{thm:BGG}
We have the following equality in an appropriate Grothendieck group. 
\[
[P(\lambda,s)]_q = \sum_{(\mu, t)} \overline{[\overline{\nabla}(\mu, t): L(\lambda,s)]_q} [\Delta(\mu,t)]_q, 
\]
where $\overline{(\cdot)}$ is a ring involution on $\mathbb{Z}[q,q^{-1}]$ defined by $\overline{q} = q^{-1}$. 
\end{theorem}

This equation should be thought of as a shadow of the BGG reciprocity. 
To prove it, we use a coincidence of $\Ext$-bilinear form with the Kashiwara's bilinear form on $\quantum{-}$ under the categorification. 
In the computation, we also establish short exact sequences that approximate a projective resolution of $R_{*,w}(\beta)$ as an $R(\beta)$-module (Section \ref{sub:projresol}). 

By Theorem \ref{thm:introext1} and Theorem \ref{thm:BGG}, we may apply a general criterion (Proposition \ref{prop:criterion}) to prove Theorem \ref{thm:intro1}.

As an application, we investigate the case where $\mathfrak{g}$ is of arbitrary affine type.
For any convex order, we establish a stratification on $\gMod{R(\beta)}$ (Theorem \ref{thm:finerstratification}), which is finer than that of Theorem \ref{thm:intro1}. 
For any Dynkin node $i$ of the underlying finite type Lie algebra, we study a specific standard module and prove that it categorifies the minimal imaginary root vector of \cite{MR3874704} (Corollary \ref{cor:imaginarymodule}). 
The basic ideas in the proof are similar to \cite{MR3694676}. 
We highlight several crucial differences: 
\begin{enumerate}
  \item From the outset, we have a kind of projectivity about $\widehat{M}(ws_i \Lambda_i, w\Lambda_i)$ by Proposition \ref{prop:introext1}. 
  \item We do not need the growth estimate \cite[Theorem 15.5]{MR3694676}, which was used to prove injectivity of some homomorphisms in \cite[Lemma 16.1]{MR3694676};  
  instead, we utilize the injectivity of renormalized R-matrices (Proposition \ref{prop:reninj}).
  \item We use R-matrix methods to give shortcuts for several arguments.  
  \item We need to address $A_{2l}^{(2)}$ case separately.  
\end{enumerate}

The organization of this paper is as follows.  
In Section \ref{sec:preliminaries}, we review background knowledge on Laurentian algebras and quantum groups. 
In Section \ref{sec:quiverHecke}, we provide a summary of fundamental topics in the representation theory of quiver Hecke algebras. 
It includes a number of refinements and generalizations. 
We particularly highlight the injectivity of renormalized R-matrices (Proposition \ref{prop:reninj}), which plays a crucial role in this paper. 
In Section \ref{sub:cuspidal}, we introduce weakly convex orders and develop cuspidal decompositions in this generalized setting. 

Section \ref{sec:properties} proves several building blocks to establish the stratification. 
In particular, Proposition \ref{prop:introext1} is proved. 
In Section \ref{sub:interlude}, we briefly digress to demonstrate that generalized Schur-Weyl duality functors are exact on modules with finite projective dimension, in particular, on $\gMod{R_{w,*}}$. 

In Section \ref{sec:stratifications}, after a quick review of stratified categories, we present a criterion for determining whether a given category is stratified. 
We apply it to our setup, and establish Theorem \ref{thm:intro1} and Theorem \ref{thm:intro2}. 
In Section \ref{sub:quantumunipotent}, we prove that several subalgebras of $\quantum{-}$ and $K_{\oplus} (\gproj{R_{w,*}})$ all coincide under the categorification. 

Section \ref{sec:affinePBW}  provides a study of the affine type case. 

In the Appendix \ref{app:B}, we summarize in a table how to translate different conventions for quantum groups into our setup. 

\subsection*{Use of AI}
During the preparation of this work, 
the author used Gemini 3.6 Flash in order to improve the English phrasing and exposition of the text. 
After using this tool, the author reviewed and edited the content as needed and takes full responsibility for the content of the publication.

\subsection*{Acknowledgement}
I am grateful to my supervisor, Noriyuki Abe, for his continuous support and invaluable feedback throughout this research. 
I also thank Katsuyuki Naoi for his fruitful suggestions; 
in particular, I learned the criterion presented in Proposition \ref{prop:criterion} from him.  
I am further grateful to Peter McNamara,
who helped me identify and address several issues in an earlier version of this paper. 
Additionally, I thank Yoshiyuki Kimura, Hironori Oya, Masaki Kashiwara and anonymous referees for helpful comments.  
This work was supported by the WINGS-FMSP program at the Graduate School of Mathematical Sciences, the University of Tokyo.

\section{Preliminaries} \label{sec:preliminaries}

\subsection{Notation and convention} 

Let $\Sym_n$\index{$\Sym_n$} be the symmetric group of degree $n$. 
When $n = n' + n''$ with $n', n'' \geq 1$, let $\Sym_n^{n', n''}$\index{$\Sym_n^{n', n''}$} be the set of minimal length coset representatives for $\Sym_n / (\Sym_{n'} \times \Sym_{n''})$. 
We define $w[n', n''] \in \Sym_n^{n', n''}$\index{$w[n',n'']$} as 
\[
w[n', n''](k) = \begin{cases}
  k + n'' & 1 \leq k \leq n', \\
  k - n' & n' + 1 \leq k \leq n.
\end{cases}
\]

Unless otherwise specified, all modules are left modules. 

\subsection{Laurentian algebras} \label{sub:Laurentian}

Let $\mathbf{k}$ be a field. 
Let $H = \bigoplus_{d \in \mathbb{Z}} H_d$ be a $\mathbb{Z}$-graded algebra whose zeroth homogeneous component $H_0$ is a $\mathbf{k}$-algebra. 
We assume that $H$ is left Noetherian. 
It is known that $H$ is left Noetherian as a graded algebra if and only if it is left Noetherian as an ungraded algebra \cite[Theorem 5.4.7]{MR2046303}.

Unless otherwise specified, $H$-modules are assumed to be graded. 
For an $H$-module $M$ and an integer $n \in \mathbb{Z}$, 
let $q^n M$\index{$q^n M$} be the $H$-module whose $d$-th homogeneous component is $M_{d-n}$.
For $H$-modules $M$ and $N$, 
let $\hom_H(M, N)$\index{$\hom_H(M,N)$} denote the set of homomorphisms of degree zero. 
An element of $\hom_H(q^d M, N)$ is called a homomorphism of degree $d$. 
We define a $\mathbb{Z}$-graded $\mathbf{k}$-vector space $\Hom_H (M, N)$ by $\Hom_H (M, N)_d = \hom_H (q^d M, N)$. \index{$\Hom_H(M,N)$}
When $M = N$, we write $\End_H(M) = \Hom_H(M,M)$\index{$\End_H(M)$} and $\grend_H(M) = \hom_H(M,M)$\index{$\grend_H(M)$}.  

We mainly focus on finitely generated graded $H$-modules in this paper. 
Note that a graded $H$-module is finitely generated as a graded $H$-module if and only if it is finitely generated as an ungraded $H$-module.
For finitely generated $H$-modules $M$ and $N$, $\Hom_H(M,N)$ is just the set of homomorphisms as ungraded modules. 
Let $\gMod{H}$\index{$\gMod{H}$} be the category of finitely generated graded $H$-modules whose morphisms are $H$-module homomorphisms of degree zero. 
It is an abelian category with enough projective objects and an autoequivalence $q$. 
For $N \in \gMod{H}$ and $k \in \mathbb{Z}_{\geq 0}$, 
let $\ext_H^k(\cdot, N)$\index{$\ext_H^k(M,N)$} (resp. $\Ext_H^k(\cdot, N)$\index{$\Ext_H^k(M)$}) be the $k$-th right derived functor of $\hom_H(\cdot, N)$ (resp. $\Hom_H(\cdot, N)$).   
We have $\Ext_H^k (M, N) = \bigoplus_{d \in \mathbb{Z}} \ext_H^k (q^d M, N)$. 

A $\mathbb{Z}$-graded $\mathbf{k}$-vector space $V = \bigoplus_{d} V_d$ is said to be Laurentian if $V_d = 0$ for $d \ll 0$ and $\dim V_d < \infty $ for all $d \in \mathbb{Z}$.
If $V$ is Laurentian, the formal series in an indeterminate $q$
\[
\qdim V = \sum_{d\in \mathbb{Z}} (\dim V_d) q^d \index{$\qdim V$}
\] 
is an element of $\mathbb{Z}((q))$. 
In what follows, we assume that $H$ is Laurentian. 
Then, every finitely generated $H$-module is Laurentian. 
Besides, $\Ext_H^k (M, N)$ is a Laurentian $\mathbf{k}$-vector space for any $M, N\in \gMod{H}$ and $k \geq 0$. 

\begin{lemma}[{\cite[Lemma 2.2]{MR3495746}}]
In $\gMod{H}$,  
\begin{enumerate}
 \item all simple objects are finite dimensional, 
 \item there are only finitely many simple objects up to isomorphim and degree shift, and 
 \item every object has a projective cover. 
\end{enumerate}
\end{lemma}

Let $\{L(\lambda) \mid \lambda \in \Lambda \}$ be a set of representatives of simple $H$-modules up to isomorphism and degree shift, where $\Lambda$ is a finite set. 
Let $P(\lambda)$ be the projective cover of $L(\lambda)$ for each $\lambda \in \Lambda$. 
We consider the graded category of finite dimensional $H$-modules $\gmod{H}$\index{$\gmod{H}$}, 
and the graded category of finitely generated projective $H$-modules $\gproj{H}$\index{$\gproj{H}$}.  
The category $\gmod{H}$ is a finite-length abelian category, while $\gproj{H}$ is a Krull-Schmidt category. 
The Grothendieck group $K(\gmod{H})$\index{$K(\gmod{H})$} is a free $\mathbb{Z}[q, q^{-1}]$-module with a basis $\{[L(\lambda)] \mid \lambda \in \Lambda\}$. 
Similarly, the split Grothendieck group $K_{\oplus}(\gproj{H})$\index{$K_{\oplus}(\gproj{H})$} is a free $\mathbb{Z}[q,q^{-1}]$-module with a basis $\{ [P(\lambda)] \mid \lambda \in \Lambda \}$. 

Although $\gMod{H}$ is not of finite-length unless $H$ itself is finite-dimensional, 
we can define ``graded Jordan-H\"{o}lder multiplicity'' as follows:

\begin{definition} \label{def:composition}
  Let $K(\gmod{H})_{\mathbb{Z}((q))} = K(\gmod{H}) \otimes_{\mathbb{Z}[q, q^{-1}]} \mathbb{Z}((q))$. \index{$K(\gmod{H})_{\mathbb{Z}((q))}$}
  For $M \in \gMod{H}$, we define 
  \begin{align*}
  &[M:L(\lambda)]_q = \qdim \Hom_H (P(\lambda), M) \in \mathbb{Z}((q)), \\ \index{$[M:L(\lambda)]_q$}
  &[M]_q = \sum_{\lambda \in \Lambda} [M:L(\lambda)]_q [L(\lambda)] \in K(\gmod{H})_{\mathbb{Z}((q))}.  \index{$[M]_q$}
  \end{align*}
  We say that $L(\lambda)$ is a composition factor of $M$ if $[M:L(\lambda)]_q$ is nonzero. 
\end{definition}

It defines a $\mathbb{Z}[q, q^{-1}]$-module homomorphism $K(\gMod{H}) \to K(\gmod{H})_{\mathbb{Z}((q))}$, which maps $[M]$ to $[M]_q$. 
Note that, if $M$ is finite dimensional, then $[M]_q = [M]$ in $K(\gmod{H})$. 

We give another description of $[M]_q$ as a limit of a sequence in $K(\gmod{H})$. 

\begin{definition}
  Let $M \in \gMod{H}$.
  A descending filtration of graded $H$-submodules $M = M^0 \supset M^1 \supset M^2 \supset \cdots $ is said to be locally-finite if 
  $M^p/M^{p+1}$ is finite dimensional for all $p \geq 0$. 
  It is said to be Hausdorff if $\bigcap_{p \geq 0} M^p = 0$. 
\end{definition}

Every $M \in \gMod{H}$ has a locally-finite Hausdorff filtration: 
for example, $(H M_{\geq p})_{p \geq 0}$ is locally-finite Hausdorff, where $M_{\geq p} = \bigoplus_{d \geq p} M_d$. 
A descending filtration $(M^p)_{p \geq 0}$ of $M$ is locally-finite Hausdorff if and only if $(M^p)_{p \geq 0}$ and $(H M_{\geq p})_{p \geq 0}$ define the same topology on $M$. 

\begin{lemma}
 Let $M \in \gMod{H}$ and $(M^p)_{p \geq 0}$ be a locally-finite Hausdorff filtration of $M$. 
 Then, we have an equality in $K(\gmod{H})_{\mathbb{Z}((q))}$: 
 \[
 [M]_q = \sum_{p=0}^{\infty} [M^p/M^{p+1}] = \lim_{p \to \infty} [M/M^p]. 
 \]
\end{lemma}

\begin{proof}
  It follows from Lemma \ref{lem:limext} (1) below.  
\end{proof}

\begin{lemma} \label{lem:limext}
  Let $M \in \gMod{H}$ and $(M^p)_{p \geq 0}$ be a locally-finite Hausdorff filtration of $M$. Let $k \in \mathbb{Z}_{\geq 0}$. 
  \begin{enumerate}
    \item $M \xrightarrow{\sim} \varprojlim_p M/M^p$ in $\gMod{H}$. 
    \item For $N \in \gMod{H}$, we have $\ext_H^k (N, M) \xrightarrow{\sim} \varprojlim_p \ext_H^k (N, M/M^p)$. 
    \item For finite dimensional $N \in \gmod{H}$, there exists $n \in \mathbb{Z}_{\geq 0}$ such that 
    \[
    \ext_H^k (M/M^n, N) \xrightarrow{\sim} \ext_H^k (M/M^{n+1}, N) \xrightarrow{\sim} \cdots \xrightarrow{\sim} \ext_H^k (M, N). 
    \]
  \end{enumerate}
\end{lemma}

\begin{proof}
  (1) Let $N \in \gMod{H}$.
  It is enough to show that the natural morphism $\hom_H (N, M) \to \hom_H (N, M/M^p)$ is an isomorphism for sufficiently large $p$. 
  By taking a projective presentation of $N$, it is reduced to the case where $N = q^dH$ for some $d \in \mathbb{Z}$. 
  In this case, $\hom_H (q^dH, M) = M_d$ and $\hom_H (H, M/M^p) = M_d/M_d^p$. 
  For sufficiently large $p$, we have $M_d^p = 0$ and the assertion follows. 

  (2) This is a refinement of \cite[Lemma 1.1]{MR3205728}. 
  We take a projective resolution of $N$ in $\gMod{H}$: 
  \[
  \cdots \to P^2 \to P^1 \to P^0 \to N \to 0.  
  \]
  Let $k \in \mathbb{Z}_{\geq 0}$. 
  By the proof of (1), for sufficiently large $p$, we have 
  \[
  \text{$\hom_H (P^j, M) \simeq \hom_H (P^j, M/M^p)$ for $j \in \{k-1, k, k+1\}$.}
  \]
  It follows that $\ext_H^k (N, M) \simeq \ext_H^k (N, M/M^p)$ for sufficiently large $p$. 

  (3) Let $a$ be the largest integer such that $H_{-a} \neq 0$.
  Let $X \in \gMod{H}$, and let $b$ be an integer such that $X_{< b} = 0$. 
  We claim that $X$ has a projective resolution $P^{\bullet}$ that satisfies $P_{< b - la}^l = 0$ for all $l \in \mathbb{Z}_{\geq 0}$.
  Let $\{x_1, \ldots, x_m \}$ be a generating set of $X$ consisting of nonzero homogeneous elements. 
  Then, we have a surjective homomorphism $\bigoplus_{1 \leq s \leq m} q^{\deg x_s} H \to X$. 
  Since $\deg x_s \geq b$ for all $s$, we see that $(\bigoplus_{1 \leq s \leq m} q^{\deg x_s} H)_{< b-a} = 0$.
  By repeating this procedure, we obtain the desired projective resolution.
  
  Now, let $c$ be an integer which satisfies $N_{\geq c} = 0$. 
  Such an integer $c$ exists since $N$ is finite dimensional. 
  There exists $n \in \mathbb{Z}_{\geq 0}$ such that $M_{< c + ka}^n = 0$. 
  For any $n' \geq n$, we also have $M_{<c + ka}^{n'} = 0$. 
  By the claim above, we may take a projective resolution $P^{\bullet}$ of $M^{n'}$ so that it satisfies $P_{<c + (k-l)a}^l = 0$ for all $l \in \mathbb{Z}_{\geq 0}$. 
  In particular, we have $P_{<c}^k = P_{<c+a}^{k-1} = 0$, hence $\hom_H (P^k, N) = \hom_H (P^{k-1}, N) = 0$. 
  It follows that $\ext_H^k (M^{n'}, N) = \ext_H^{k-1} (M^{n'}, N) = 0$. 
  Therefore, the exact sequence
  \[
  \ext_H^{k-1} (M^{n'}, N) \to \ext_H^k (M/M^{n'}, N) \to \ext_H^k (M, N) \to \ext_H^k (M^{n'}, N)
  \]
 proves the assertion. 
\end{proof}

The following corollaries are immediate. 

\begin{corollary} \label{cor:projlim1}
 Let $M, N \in \gMod{H}$ and $k \in \mathbb{Z}_{\geq 0}$. 
 Assume that for any composition factor $X$ of $M$ and $Y$ of $N$, we have $\Ext_H^k (X, Y) = 0$. 
 Then, $\Ext_H^k (M, N) = 0$. 
\end{corollary}

For $M \in \gMod{H}$, we define its projective dimension $\prdim_H M \in \mathbb{Z}_{\geq 0} \cup \{ \infty\}$ as 
\[
\sup \{k \geq 0 \mid \text{there exists $N \in \gMod{H}$ such that $\Ext_H^k (M,N) \neq 0$} \}.
\] 
\index{$\prdim_H M$}

\begin{corollary} \label{cor:projlim2}
  For $M \in \gMod{H}$, we have 
  \[
  \prdim_H M \leq \sup \{ \prdim_H L \mid \text{$L$ is a composition factor of $M$} \}. 
  \]
\end{corollary}

\subsection{Quantum groups} \label{sub:quantumgroups}

We mainly follow the conventions in \cite{MR3758148}. 
Throughout this paper, let $(C, P, \Pi, \Pi^{\lor}, (\cdot, \cdot))$ be a fixed root datum, 
where $C = (c_{i, j})_{i, j \in I}$\index{$C$}\index{$I$} is a symmetrizable generalized Cartan matrix,  
$P$\index{$P$} a free abelian group called the weight lattice,  
$\Pi = \{ \alpha_i \}_{i\in I}$\index{$\Pi$}\index{$\alpha_i$} a subset of $P$, 
$\Pi^{\lor} = \{ h_i \}_{i \in I}$\index{$\Pi^{\lor}$}\index{$h_i$} a subset of $P^{\lor} = \Hom_{\Z}(P, \Z)$,  
and $(\cdot, \cdot)$\index{$(\cdot,\cdot)$ (on $P$)} a $\mathbb{Q}$-valued symmetric bilinear form on $P$, satisfying the following conditions:  
\begin{enumerate}
  \item $c_{i, j} = \langle h_i, \alpha_j \rangle$ for $i, j \in I$, 
  \item $(\alpha_i, \alpha_i) \in 2 \mathbb{Z}_{> 0}$ for $i \in I$, 
  \item $\langle h_i, \lambda \rangle = 2 (\alpha_i, \lambda)/(\alpha_i, \alpha_i)$ for $i\in I$ and $\lambda \in P$, 
  \item $\Pi$ is linearly independent and 
  \item for any $i \in I$, there exists $\Lambda_i \in P$\index{$\Lambda_i$} such that $\langle h_j, \Lambda_i \rangle = \delta_{i, j}$ for all $j \in I$. 
\end{enumerate}

For each $i \in I$, we call $\alpha_i$ the simple root, $h_i$ the simple coroot, and $\Lambda_i$ the fundamental weight. 
Let $\quantum{}$\index{$\quantum{}$} be the quantum group associated with the root datum, 
which is a $\mathbb{Q}(q)$-algebra with generators $\{ e_i, f_i \mid i \in I\} \sqcup \{ q^h \mid h \in P^{\lor} \}$\index{$e_i$}\index{$f_i$}\index{$q^h$}. 
We set $q_i = q^{(\alpha_i, \alpha_i)/2}\index{$q_i$}, [n] = (q^n-q^{-n})/(q-q^{-1})\index{{$[n]$}}, [n]! = [n] [n-1] \cdots [1]\index{$[n]"!$}, [n]_i = (q_i^n - q_i^{-n})/(q_i-q_i^{-1})$\index{{$[n]_i$}}, 
and $[n]_i! = [n]_i [n-1]_i \cdots [1]_i$\index{{$[n]_i "!$}}.    
We define $e_i^{(n)} = e_i^n/([n]_i!)$\index{$e_i^{(n)}$} and $f_i^{(n)} = f_i^n/([n]_i!)$\index{$f_i^{(n)}$}.  
Let $W$\index{$W$} denote the Weyl group, which is generated by the simple reflections $s_i \ (i \in I)$\index{$s_i$}. 
Let $Q\index{$Q$} = \bigoplus_{i \in I} \mathbb{Z} \alpha_i$ be the root lattice. 
The positive root lattice is defined as $Q_+\index{$Q_+$} =\sum_{i \in I} \mathbb{Z}_{\geq 0} \alpha_i \subset P$, and the negative root lattice is $Q_-\index{$Q_-$} = - Q_+$.  

Let $\varphi$\index{$\varphi$} and $(\cdot)^*$\index{$(\cdot)^*$} be $\Q(q)$-algebra antiautomorphisms of $U_q(\g)$ defined by:  
\begin{gather*}
 \varphi(e_i) = f_i, \ \varphi(f_i) = e_i, \ \varphi(q^h) = q^h, \\
 e_i^* = e_i, \ f_i^* = f_i, \ (q^h)^* = q^{-h}. 
\end{gather*}
We define a $\mathbb{Q}$-algebra automorphism $\overline{(\cdot)}$\index{$\overline{(\cdot)}$ (on $\quantum{}$)} by: 
\[
\overline{e_i} = e_i, \ \overline{f_i} = f_i, \ \overline{q^h} = q^{-h}, \ \overline{q} = q^{-1}. 
\]
Note that $\varphi, (\cdot)^*$ and $\overline{(\cdot)}$ commute with each other. 

For a weight $\quantum{}$-module $M = \bigoplus_{\mu \in P} M_{\mu}$, 
we define the restricted dual $M^{\lor}\index{$M^{\lor}$} = \bigoplus_{\mu \in P} M_{\mu}^{\lor}$, 
where $M_{\mu}^{\lor} = \Hom_{\mathbb{Q}}(M_{\mu}, \mathbb{Q})$.
The antiautomorphism $\varphi$ allows us to introduce a left $\quantum{}$ action on $M^{\lor}$ as follows: 
for $x \in \quantum{}, f \in M^{\lor}$ and $m \in M$, the action is given by $(x f)(m) = f(\varphi(x)m)$. 

Let $\quantum{-}$ be the $\mathbb{Q}(q)$-subalgebra of $\quantum{}$ generated by $\{ f_i \mid i \in I \}$. 
For any $i \in I$, 
there exists a unique $\Q(q)$-linear endomorphism $e'_i$\index{$e'_i$} of $\quantum{-}$ such that 
\[
e'_i(1) = 0, \ e_i'(f_j x) = q_i^{(-\alpha_i, \alpha_j)}f_j e'_i (x) + \delta_{i, j} x, 
\]
for all $ x\in \quantum{-}$ and $j \in I$.
There exists a unique nondegenerate symmetric bilinear form $(\cdot,\cdot)$\index{$(\cdot,\cdot)$ (on $\quantum{-}$)} on $\quantum{-}$ such that 
\[
\text{$(1, 1) = 1, \ (f_i x, y) = (x, e'_i y)$ for $x, y \in \quantum{-},i \in I$}. 
\]
We define $\qcoordinate \index{$\qcoordinate$} = \bigoplus_{\beta \in Q_+} \qcoordinate_{\beta}$, where $\qcoordinate_{\beta} = \Hom_{\mathbb{Q}(q)} (\quantum{-}_{-\beta}, \mathbb{Q}(q))$. 
The form $(\cdot, \cdot)$ corresponds to an isomorphism between $\quantum{-}$ and $\qcoordinate$. 
The vector space $\qcoordinate$ is endowed with an algebra structure through this isomorphism, and we call it the quantum unipotent coordinate ring. 

Let $\quantum{}_{\mathbb{Z}[q,q^{-1}]}$ be the $\mathbb{Z}[q,q^{-1}]$-subalgebra of $\quantum{}$ generated by \index{$\quantum{}_{\mathbb{Z}[q,q^{-1}]}$}
\[
e_i^{(n)}, f_i^{(n)}, q^h, \prod_{k=1}^n \frac{q^{1-k}q^h - q^{k-1}q^{-h}}{q^k - q^{-k}}\:\: (i \in I, n \in \Z_{\geq 0}, h \in P). 
\]
We also define $\quantum{-}_{\mathbb{Z}[q,q^{-1}]}$\index{$\quantum{-}_{\mathbb{Z}[q,q^{-1}]}$} as the $\mathbb{Z}[q,q^{-1}]$-subalgebra of $\quantum{-}$ generated by $\{ f_i^{(n)} \mid i \in I, n \in \mathbb{Z}_{\geq 1}\}$. 
Let
\[
\quantum{-}_{\mathbb{Z}[q,q^{-1}]}^{\mathrm{up}} = \{ x \in \quantum{-} \mid (x, \quantum{-}_{\mathbb{Z}[q,q^{-1}]}) \subset \mathbb{Z}[q,q^{-1}] \}.  \index{$\quantum{-}_{\mathbb{Z}[q,q^{-1}]}^{\mathrm{up}}$}
\]
It is a $\mathbb{Z}[q,q^{-1}]$-subalgebra of $\quantum{-}$. 
Let $\qcoordinate_{\mathbb{Z}[q,q^{-1}]}$\index{$\qcoordinate_{\mathbb{Z}[q,q^{-1}]}$} be the subalgebra of $\qcoordinate$ corresponding to $\quantum{-}_{\mathbb{Z}[q,q^{-1}]}^{\mathrm{up}}$ under the isomorphism $\quantum{-} \simeq \qcoordinate$. 

Let $P_+$\index{$P_+$} be the set of dominant weights $\{ \Lambda \in P \mid \text{$\langle h_i, \Lambda \rangle \geq 0$ for all $i \in I$} \}$.   
For each $\Lambda \in P_+$, 
let $V(\Lambda)$\index{$V(\Lambda)$} be the highest weight integrable $U_q(\g)$-module of highest weight $\Lambda$. 
We take a highest weight vector $u_{\Lambda}\index{$u_{\Lambda}$} \in V(\Lambda)_{\Lambda} \setminus \{0\}$. 
Since $V(\Lambda)$ is simple, we have an isomorphism $V(\Lambda) \simeq V(\Lambda)^{\lor}$ that corresponds to a nondegenerate symmetric bilinear form $(\cdot, \cdot)$\index{$(\cdot, \cdot)$ (on $V(\Lambda)$)} on $V(\Lambda)$ such that 
\[
\text{$(u_{\Lambda}, u_{\Lambda}) = 1, \ (xu, v) = (u, \varphi(x)v)$ for $u, v \in V(\Lambda), x \in U_q(\g)$}.
\]
Let 
\begin{align*}
V(\Lambda)_{\mathbb{Z}[q,q^{-1}]}\index{$V(\Lambda)_{\mathbb{Z}[q,q^{-1}]}$} &= \quantum{-}_{\mathbb{Z}[q,q^{-1}]} u_{\Lambda}, \\ 
V(\Lambda)_{\mathbb{Z}[q,q^{-1}]}^{\mathrm{up}}\index{$V(\Lambda)_{\mathbb{Z}[q,q^{-1}]}^{\mathrm{up}}$} &= \{ v\in V(\Lambda) \mid (v, V(\Lambda)_{\mathbb{Z}[q,q^{-1}]}) \subset \mathbb{Z}[q,q^{-1}]\}.
\end{align*}  
They are both $\quantum{}_{\mathbb{Z}[q,q^{-1}]}$-submodules of $V(\Lambda)$.
Through the surjective homomorphism $\quantum{-} \to V(\Lambda), 1 \mapsto u_{\Lambda}$, 
the involution $\overline{(\cdot)}$ on $\quantum{-}$ induces an involution $\overline{(\cdot)}$\index{$\overline{(\cdot)}$ (on $V(\Lambda)$)} on $V(\Lambda)$. 

\subsection{Global bases and quantum minors}

We keep following \cite{MR3758148}. 
Let $(B(\infty)\index{$B(\infty)$}, \wt\index{$\wt$}, \varepsilon_i\index{$\varepsilon_i$}, \varphi\index{$\varphi$}, \tilde{e}_i\index{$\tilde{e}_i$}, \tilde{f}_i\index{$\tilde{f}_i$})$ be the combinatorial crystal associated with $\quantum{-}$. 
The involution $(\cdot)^*$ descends to an involution on $B(\infty)$, also denoted by $(\cdot)^*$\index{$(\cdot)^*$ (on $B(\infty)$)}. 
It produces another crystal structure on $B(\infty)$, denoted by $(B(\infty), \wt^*, \varepsilon_i^*, \varphi_i^*, \e_i^*, \f_i^*)$. \index{$B(\infty)$: $(B(\infty), \wt^*, \varepsilon_i^*, \varphi_i^*, \e_i^*, \f_i^*)$}

Let $\Blow(\quantum{-})$\index{$\Blow(\quantum{-})$} be the lower global basis of $\quantum{-}$, 
also known as the canonical basis. 
It is a bar-invariant $\mathbb{Z}[q,q^{-1}]$-basis of $\quantum{-}_{\mathbb{Z}[q,q^{-1}]}$. 
There is a canonical bijection from $\Blow(\quantum{-})$ to $B(\infty)$ and its inverse map is denoted by $\Glow$\index{$\Glow$}. 

Let $\Bup(\quantum{-}) \subset \quantum{-}$\index{$\Bup(\quantum{-})$} be the upper global basis, also known as the dual canonical basis. 
It is by definition the dual basis of $\Blow(\quantum{-})$ with respect to $(\cdot,\cdot)$. 
It is a $\mathbb{Z}[q,q^{-1}]$-basis of $\quantum{-}_{\mathbb{Z}[q,q^{-1}]}^{\text{up}}$. 
We have a bijection $\Gup\index{$\Gup$} \colon B(\infty) \to \Bup(\quantum{-})$ satisfying $(\Glow(b), \Gup(b')) = \delta_{b, b'}$ for any $b, b' \in B(\infty)$. 
For $x \in \quantum{-}\setminus \{0\}$, 
we define $\varepsilon_i(x)\index{$\varepsilon_i$} = \max \{ m \geq 0 \mid {e'_i}^m x \neq 0 \}$ and $\varepsilon^*_i(x)\index{$\varepsilon_i^*$} = \max \{ m \geq 0 \mid {{e'_i}^*}^m x \neq 0 \}$, 
where ${e'_i}^* (x)\index{${e'_i}^*$} = (e'_i (x^*))^*$.   
Then for any $b \in B(\infty)$, we have 
\[
\varepsilon_i(\Gup(b)) = \varepsilon_i(b),\ \varepsilon_i^*(\Gup(b)) = \varepsilon_i^*(b). 
\]
This follows from \cite[Lemma 5.1.1.]{MR1203234} and Lemma~\ref{lem:globalbasis} (3) below. 

For $\Lambda \in P_+$, 
we also have the crystal $B(\Lambda)$, \index{$B(\Lambda)$}
the lower global basis $\Blow(\Lambda)\index{$\Blow(\Lambda)$} = \{ \Glow_{\Lambda}(b) \mid b \in B(\Lambda) \}\index{$\Glow_{\Lambda}$} \subset V(\Lambda)$  
and the upper global basis $\Bup(\Lambda)\index{$\Bup(\Lambda)$} = \{ \Gup_{\Lambda}(b) \mid b \in B(\Lambda)\}\index{$\Gup_{\Lambda}$} \subset V(\Lambda)$.  
The lower and upper global basis are mutually dual with respect to $(\cdot, \cdot)$: 
we have $(\Glow_{\Lambda}(b), \Gup_{\Lambda}(b')) = \delta_{b, b'}$ for all $b, b' \in B(\Lambda)$. 
Moreover, they are bar-invariant $\mathbb{Z}[q,q^{-1}]$-bases of $V(\Lambda)_{\mathbb{Z}[q,q^{-1}]}$ and $V(\Lambda)^{\text{up}}_{\mathbb{Z}[q,q^{-1}]}$ respectively. 

The global bases of $\quantum{-}$ and of $V(\Lambda)$ are related, as summarized in the following lemma. 

\begin{lemma}[{\cite{MR1115118}}] \label{lem:globalbasis}
Let $\Lambda \in P_+$. Before stating the assertions, we introduce several maps. 
Let $\pi_{\Lambda}\index{$\pi_{\Lambda}$} \colon \quantum{-} \to V(\Lambda)$ be a surjective homomorphism defined by $\pi_{\Lambda}(x) = xu_{\Lambda}$. 
It induces a map $\overline{\pi}_{\Lambda}\index{$\overline{\pi_{\Lambda}}$} \colon B(\infty) \to B(\Lambda) \sqcup \{0\}$. 
We also have an injective homomorphism $\iota_{\Lambda}\index{$\iota_{\Lambda}$} \colon V(\Lambda) \simeq V(\Lambda)^{\lor} \to \qcoordinate \simeq \quantum{-}$, 
where the first and the third isomorphisms are described in the previous section and the second homomorphism is the dual map of $\pi_{\Lambda}$. 
\begin{enumerate}
  \item We have $f_i \circ \pi_{\Lambda} = \pi_{\Lambda} \circ f_i$ and $e'_i \circ \iota_{\Lambda} = \iota_{\Lambda} \circ e_i$ for all $i \in I$. 
  \item The map $\overline{\pi}_{\Lambda}$ restricts to a bijection from $\{ b \in B(\infty) \mid \overline{\pi}_{\Lambda}(b) \neq 0\}$ to $B(\Lambda)$. 
  \item Let $b \in B(\infty)$. If $\overline{\pi}_{\Lambda}(b) \neq 0$, we have $\Glow(b)u_{\Lambda} = \Glow_{\Lambda}(\overline{\pi}_{\Lambda}(b))$ and $\iota_{\Lambda}(\Gup_{\Lambda}(\overline{\pi}_{\Lambda}(b))) = \Gup(b)$. Otherwise, we have $\Glow(b)u_{\Lambda} = 0$. 
\end{enumerate}

\end{lemma}

Note that the assertion for the upper global basis is a consequence of that for the lower global basis. 

There are special elements $D(x\Lambda,y\Lambda)\index{$D(x\Lambda,y\Lambda)$} \ (x,y \in W, x \geq y, \Lambda \in P_+)$ in $\Bup(\quantum{-})$, called unipotent quantum minors. 
They have the following properties. 

\begin{lemma}[{\cite[Lemma 9.1.4, 9.1.5]{MR3758148}}]\label{lem:quantumminor}
Let $\Lambda \in P_+$. 
\begin{enumerate}
  \item For $x \in W$, we have $D(x\Lambda, x\Lambda) = 1$. 
  \item Let $x,y \in W$ and $i \in I$. Assume $s_i x > x \geq y$. Then \begin{align*}
  &\varepsilon_i(D(x\Lambda, y\Lambda)) = 0, \ \varepsilon_i(D(s_ix \Lambda,y\Lambda)) = \langle h_i, x\Lambda \rangle, \\ 
  &{e'_i}^{(\langle h_i, x\Lambda \rangle)}D(s_ix \Lambda, y\Lambda) = D(x\Lambda,y\Lambda). 
  \end{align*}
  \item Let $x,y \in W$ and $i \in I$. Assume $x \geq s_i y > y$. Then \begin{align*}
  &\varepsilon_i^* (D(x\Lambda, s_i y\Lambda)) = 0, \ \varepsilon_i^* (D(x\Lambda,y\Lambda)) = \langle h_i, y\Lambda \rangle, \\ 
  &{e'_i}^{*(\langle h_i, y\Lambda \rangle)}D(x\Lambda, y\Lambda) = D(x\Lambda, s_iy\Lambda). 
  \end{align*}
\end{enumerate}
\end{lemma}

In particular, we can repeatedly use this lemma to obtain a description of $D(x\Lambda, y\Lambda)$ as follows. 
Let $\Lambda \in P_+$ and $w \in W$. 
We write $w = s_{i_1} \cdots s_{i_m}$ in a reduced expression and put $a_k = \langle h_{i_k}, s_{i_{k+1}} \cdots s_{i_m} \Lambda \rangle (1 \leq k \leq m)$. 
Then, we have
\[
D(w\Lambda, \Lambda) = \iota_{\Lambda} (f_{i_1}^{(a_1)} \cdots f_{i_m}^{(a_m)} u_{\Lambda})
\]
Next, let $v \in W$ with $v \leq w$. 
We write $v = s_{j_1} \cdots s_{j_n}$ in a reduced expression and put $b_k = \langle h_{j_k}, s_{j_{k+1}} \cdots s_{j_n} \Lambda \rangle \ (1 \leq k \leq n)$. 
Then, we have 
\[
D(w \Lambda, v \Lambda) = {{e'}_{j_1}^*}^{(b_1)} \cdots {{e'}_{j_n}^*}^{(b_n)} D(w \Lambda, \Lambda)
\]

\section{Quiver Hecke algebras} \label{sec:quiverHecke}

\subsection{Quiver Hecke algebras} \label{sub:quiverHecke}

We review some basic results about quiver Hecke algebras, which is also known as Khovanov-Lauda-Rouquier algebras. 
In the following sections, let $\mathbf{k}$ be a base field of arbitrary characteristic. 

In addition to the root datum $(C, P, \Pi, \Pi^{\lor}, (\cdot, \cdot))$, 
we choose a family of polynomials $(Q_{i,j}(u,v))_{i,j \in I}$\index{$Q_{i,j}(u,v)$} in $\bfk[u,v]$ subject to the following conditions: 
\begin{enumerate}
  \item For any $i \in I$, we have $Q_{i,i} = 0$.  
  \item For any $i, j \in I$, we have $Q_{i,j}(u,v) = Q_{j,i}(v,u)$. 
  \item For any two distinct elements $i, j \in I$, let $\deg(u) = (\alpha_i, \alpha_i)$ and $\deg(v) = (\alpha_j, \alpha_j)$. Then $Q_{i,j}(u,v)$ is homogeneous of degree $-2(\alpha_i, \alpha_j)$. 
  \item For any two distinct elements $i, j \in I$, we have $Q_{i,j}(u,0) \neq 0$ and $Q_{i,j}(0,v) \neq 0$.
\end{enumerate}

For $\beta \in Q_+$ of height $n$, let $I^{\beta} = \{ \nu = (\nu_k)_{1 \leq k \leq n} \in I^n \mid \sum_{k=1}^n \alpha_{\nu_k} = \beta\}$. \index{$I^{\beta}$}

\begin{definition}[{\cite{rouquier20082kacmoodyalgebras, MR2525917, MR2763732}}] \label{def:KLR}
Associated with the root datum $(C, P, \Pi, \Pi^{\lor}, (\cdot, \cdot))$ and the polynomials $(Q_{i,j})_{i,j \in I}$, 
the quiver Hecke algebra $R(\beta)$ is defined, for each $\beta \in Q_+$ with $\height (\beta) = n$, as a unital associative $\mathbf{k}$-algebra by the following generators and relations: \index{$R(\beta)$}
its generators are 
\[
e(\nu) \, (\nu \in I^\beta), \ x_k \, (1 \leq k \leq n), \ \tau_k \, (1 \leq k \leq n-1), \index{$e(\nu)$} \index{$x_k$} \index{$\tau_k$}
\]
and its relations are 
\begin{align*}
& e(\nu)e(\nu') = \delta_{\nu, \nu'} e_{\nu}, \  \sum_{\nu \in I^{\beta}} e(\nu) = 1, \\
& x_k e(\nu) = e(\nu) x_k, \ x_k x_{l} = x_{l} x_k, \\
& \tau_k e(\nu) = e(s_k(\nu)) \tau_k \ (1 \leq k \leq n-1), \\ 
& \tau_k \tau_{l} = \tau_{l} \tau_k \ (1 \leq k, l \leq n-1, \lvert k-l \rvert \geq 2), \\
& (\tau_k x_{k+1} - x_k \tau_k)e(\nu) = (x_{k+1} \tau_k - \tau_k x_k)e(\nu) = \delta_{\nu_k, \nu_{k+1}} e(\nu) \; (1 \leq k \leq n-1), \\
& \tau_k^2 e(\nu) = Q_{\nu_k, \nu_{k+1}} (x_k, x_{k+1}) e(\nu) \; (1 \leq k \leq n-1), \\
& (\tau_{k+1} \tau_k \tau_{k+1} - \tau_k \tau_{k+1} \tau_k) e(\nu) \\ 
& = \overline{Q}_{\nu_k, \nu_{k+1}, \nu_{k+2}}(x_k, x_{k+1}, x_{k+2})e(\nu) \; (1 \leq k \leq n-2), 
\end{align*}
where $\overline{Q}_{i, i', i''}(u,v,w)$\index{$\overline{Q}_{i,i',i''}$} is given by 
\[
\begin{cases}
 \dfrac{Q_{i, i'}(u,v) - Q_{i, i'}(w,v)}{u-w}  & i=i'' \neq i', \\
 0 & \text{otherwise}. 
\end{cases}
\]
\end{definition}

The quiver Hecke algebra $R(\beta)$ is $\mathbb{Z}$-graded with 
\[
\deg e(\nu) = 0,\ \deg x_k e(\nu) = (\alpha_{\nu_k}, \alpha_{\nu_k}),\ \deg \tau_k e(\nu) = -(\alpha_{\nu_k}, \alpha_{\nu_{k+1}}).
\]
Let $P(\beta)\index{$P(\beta)$} = \prod_{\nu \in I^{\beta}} \mathbf{k}[x_1, \ldots, x_n] e(\nu)$. 
For each $w \in \Sym_n$, we choose a reduced expression $w = s_{k_1} \cdots s_{k_l}$ and put $\tau_w = \tau_{k_1} \cdots \tau_{k_l}$. 
Then $R(\beta)$ is a free left (or right) $P(\beta)$ module with a basis $\{\tau_w \mid w \in \Sym_n\}$. 
Hence, it is a Laurentian Noetherian algebra.
We mainly work in the category of finitely generated graded $R(\beta)$-modules, which is denoted by $\gMod{R(\beta)}$. 
Unless otherwise stated, $R(\beta)$-modules are assumed to be graded. 
Let $\gmod{R(\beta)}$ be the full subcategory of finite-dimensional modules, and $\gproj{R(\beta)}$ the full subcategory of projective modules. 

There is a $\mathbf{k}$-algebra anti-involution $\varphi$\index{$\varphi$ (on $R(\beta)$)} of $R(\beta)$ that fixes all the generators $e(\nu), x_k$ and $\tau_k$. 
Using it, we get a duality functor $D$\index{$D$} on $\gmod{R(\beta)}$ given by $D(M) = \Hom_{\mathbf{k}}(M, k)$, on which $R(\beta)$ acts by 
\[
\text{$(a f)(m) = f(\varphi(a)m)$ for $a \in R(\beta),\ f \in D(M), m \in M$. } 
\]
The $d$-th homogeneous component of $D(M)$ is $D(M)_d = \Hom_{\mathbf{k}}(M_{-d}, \mathbf{k})$. 
Similarly, there is a duality functor $D'$\index{$D'$} on $\gproj{R(\beta)}$ defined by $D'(P) = \Hom_{R(\beta)} (P, R(\beta))$. 
A finite-dimensional module $M \in \gmod{R(\beta)}$ is said to be self-dual if $DM \simeq M$. 

There is a $\mathbf{k}$-algebra involution $\psi$\index{$\psi$} on $R(\beta)$ defined by 
\[
\psi(e(\nu_1, \ldots, \nu_n)) = e(\nu_n, \ldots, \nu_1), \ \psi(x_k) = x_{n+1-k}, \ \psi(\tau_k) = -\tau_{n-k}. 
\]
For $M \in \gMod{R(\beta)}$, the involution $\psi$ allows us to twist the $R(\beta)$-module structure on $M$. 
It induces an autofunctor $\psi_*$\index{$\psi_*$} on $\gMod{R(\beta)}$.

Let $\beta, \gamma \in Q_+$ and put $m = \height (\beta), n = \height (\gamma)$. 
We define an idempotent $e(\beta, \gamma)$\index{$e(\beta,\gamma)$} of $R(\beta + \gamma)$ by 
\[
e(\beta, \gamma) = \sum_{\nu \in I^{\beta}, \nu' \in I^{\gamma}} e(\nu, \nu'). 
\]
Then, $R(\beta+ \gamma)e(\beta,\gamma)$ is a right $(R(\beta) \otimes R(\gamma))$-module as follows:
\begin{align*}
 &ue(\beta,\gamma) (e(\nu)\otimes 1) = ue(\nu,\gamma) \ (\nu \in I^{\beta}), ue(\beta, \gamma) (1 \otimes e(\nu)) = ue(\beta, \nu) \ (\nu \in I^{\gamma}), \\
 &ue(\beta,\gamma) (x_k \otimes 1) = ue(\beta,\gamma)x_k \ (1 \leq k \leq m), \\ 
 &ue(\beta,\gamma) (1 \otimes x_k) = ue(\beta,\gamma) x_{k+m} \ (1 \leq k \leq n), \\
 &ue(\beta,\gamma) (\tau_k \otimes 1) = ue(\beta,\gamma)\tau_k \ (1 \leq k \leq m-1), \\ 
 &ue(\beta, \gamma) (1 \otimes \tau_k) = ue(\beta,\gamma) \tau_{k+m} \ (1\leq k \leq n-1). 
\end{align*}
It is both left $R(\beta + \gamma)$-projective and right $(R(\beta) \otimes R(\gamma))$-projective. 
Similar property holds for $e(\beta, \gamma)R(\beta + \gamma)$. 
They produce three exact functors
\begin{align*}
  \Ind_{\beta, \gamma} \colon& \gMod{(R(\beta) \otimes R(\gamma))} \to \gMod{R(\beta + \gamma)}, \\ \index{$\Ind_{\beta,\gamma}$}
  & X \mapsto R(\beta + \gamma)e(\beta, \gamma) \otimes_{R(\beta) \otimes R(\gamma)} X, \\
  \Res_{\beta, \gamma} \colon & \gMod{R(\beta + \gamma)} \to \gMod{(R(\beta) \otimes R(\gamma))},  \\ \index{$\Res_{\beta,\gamma}$}
  & X \mapsto \Hom_{R(\beta+\gamma)} (R(\beta+\gamma)e(\beta,\gamma), X) \simeq e(\beta, \gamma) X, \\
  \Coind_{\beta, \gamma} \colon& \gMod{(R(\beta) \otimes R(\gamma))} \to \gMod{R(\beta + \gamma)}, \\ \index{$\Coind_{\beta,\gamma}$}
  & X \mapsto \Hom_{R(\beta) \otimes R(\gamma)} (e(\beta, \gamma)R(\beta + \gamma), X).
\end{align*}

We have two adjoint pairs, $(\Ind_{\beta, \gamma}, \Res_{\beta, \gamma})$ and $(\Res_{\beta, \gamma}, \Coind_{\beta, \gamma})$. 
For multiple $(\beta_1, \ldots, \beta_m) \in Q_+^m$, we define $\Ind_{\beta_1, \ldots, \beta_m}, \Res_{\beta_1, \ldots, \beta_m}, \Coind_{\beta_1, \ldots, \beta_m}$ in the same manner. 
We usually write $M \circ N$\index{$M\circ N$} instead of $\Ind_{\beta, \gamma} (M \otimes N)$ and call it the convolution product of $M$ and $N$. 
It gives a monoidal structure on $\gMod{R}\index{$\gMod{R}$} = \bigoplus_{\beta \in Q_+} \gMod{R(\beta)}$ with the unit object $\mathbf{k} \in \gMod{R(0)}$. 
Additionally, $\gmod{R}\index{$\gmod{R}$} = \bigoplus_{\beta \in Q_+} \gmod{R(\beta)}$ and $\gproj{R}\index{$\gproj{R}$} = \bigoplus_{\beta \in Q_+} \gproj{R(\beta)}$ are closed under the convolution products.  
We have directly from the definitions that 
\[  
D'(P\circ Q) \simeq D'(P) \circ D'(Q)
\] 
for $P, Q \in \gproj{R}$. We also have 
\[  
\Hom_{R(\beta)}(D'(P), D(M)) \simeq \Hom_{\mathbf{k}}(\Hom_{R(\beta)}(P,M),\mathbf{k}) 
\]   
for $P \in \gproj{R(\beta)}, M \in \gmod{R(\beta)}$. 

The following Mackey filtration \cite[Proposition 2.18]{MR2525917},\cite[Proposition 2.7]{MR3403455}  is fundamental in studying restrictions of induced modules.
Let $\beta_1, \ldots, \beta_m, \gamma_1, \ldots, \gamma_n \in Q_+$ with $\sum_k \beta_k = \sum_l \gamma_l$, and $M_k \in \gMod{R(\beta_k)}$. 
Then, 
\[
\Res_{\gamma_1,\ldots,\gamma_n} \Ind_{\beta_1, \ldots, \beta_m} (M_1 \otimes \cdots \otimes M_m)  
\]
has a filtration whose successive quotients are
\begin{align*}
&q^{m(\alpha_{k,l})}(\Ind_{\alpha_{1,1}, \ldots, \alpha_{m,1}} \otimes \cdots \otimes \Ind_{\alpha_{1,n}, \ldots, \alpha_{m,n}}) (\Res_{\alpha_{1,1},\ldots,\alpha_{1,n}}M_1 \otimes \cdots \otimes \Res_{\alpha_{m,1},\ldots, \alpha_{m,n}}M_m), 
\end{align*}
where $m(\alpha_{k,l}) = -\sum_{1 \leq k < k' \leq m, n \geq l > l' \geq 1} (\alpha_{k,l}, \alpha_{k',l'})$, and $(\alpha_{k,l}) \in Q_+^{mn}$ runs under the condition 
\[
\sum_k \alpha_{k,l} = \gamma_l \ (1 \leq l \leq n), \  \sum_l \alpha_{k,l} = \beta_k \ (1 \leq k \leq m). 
\] 

There is an isomorphism of graded $(R(\beta + \gamma), R(\beta) \otimes R(\gamma))$-modules
\[
q^{(\beta, \gamma)} R(\beta + \gamma)e(\gamma, \beta) \xrightarrow{\sim} \Hom_{R(\beta) \otimes R(\gamma)} (e(\beta, \gamma)R(\beta + \gamma), R(\beta) \otimes R(\gamma)). 
\] 
Be mindful of the order of $\beta$ and $\gamma$ on the right-hand side. 
This isomorphism is induced, using $\otimes$-$\Hom$ adjunction, from 
\begin{align*}
e(\beta, \gamma) R(\beta + \gamma) \otimes R(\beta + \gamma) e(\gamma, \beta) &\to e(\beta, \gamma)R(\beta + \gamma)e(\gamma, \beta) \\ 
&= \bigoplus_{w \in \Sym_{m+n}^{n,m}} e(\beta, \gamma) \tau_w (R(\gamma) \otimes R(\beta)) \\
&\to q^{-(\beta, \gamma)} R(\beta) \otimes R(\gamma), 
\end{align*}
where the first map is the multiplication and the last map is the projection to the direct summand corresponding to $w = w[n,m]$. 
The degree shift comes from the fact that $\deg e(\beta, \gamma) \tau_{w[n,m]} = -(\beta, \gamma)$. 
It yields the following natural isomorphisms \cite[Theorem2.2]{MR2822211}: 
\begin{align*}
 &\text{$\Coind_{\beta, \gamma} (M \otimes N) \simeq q^{(\beta, \gamma)} N \circ M$ for $M \in \gMod{R(\beta)}, N \in \gMod{R(\gamma)}$}, \\
 &\text{$D(M \circ N) \simeq q^{(\beta, \gamma)} DN \circ DM$ for $M \in \gmod{R(\beta)}, N \in \gmod{R(\gamma)}$.} 
\end{align*}

Next, we present a brief overview of the categorification of $\quantum{-}$. 
For each $i \in I$, 
let $L(i)$\index{$L(i)$} be the self-dual simple $R(\alpha_i)$-module, that is, $L(i) \simeq R(\alpha_i)/(x_1) \simeq \mathbf{k}$. 
We consider the following functors
\begin{align*}
F_i &\colon \gmod{R(\beta)} \to \gmod{R(\beta + \alpha_i)}, \ F_i(M) = L(i) \circ M, \\ \index{$F_i$}
F_i^* &\colon \gmod{R(\beta)} \to \gmod{R(\beta + \alpha_i)},\ F_i(M) = M \circ L(i), \\ \index{$F_i^*$}
E_i &\colon \gmod{R(\beta)} \to \gmod{R(\beta - \alpha_i)},\ E_i(M) = e(i, *) M, \\ \index{$E_i$}
E_i^* &\colon \gmod{R(\beta)} \to \gmod{R(\beta - \alpha_i)},\ E_i(M) = e(*, i) M,  \index{$E_i^*$}
\end{align*}
where $e(i,*)\index{$e(i,*)$} = e(\alpha_i, \beta-\alpha_i), e(*,i)\index{$e(*,i)$} = e(\beta-\alpha_i, \alpha_i)$.  
In this paper, we use similar abbreviations for other idempotents. 

Let $K(\gmod{R})$ be the Grothendieck group of the abelian category $\gmod{R}$, and $K_{\oplus}(\gproj{R})$ the spilt Grothendieck group of the additive category $\gproj{R}$.  
They are endowed with $\mathbb{Z}[q, q^{-1}]$-algebra structures by the degree shift functor $q$ and the convolution product. 

Let $L(i^n)\index{$L(i^n)$} = q_i^{n(n-1)/2} L(i)^{\circ n}$ be the self-dual simple $R(n\alpha_i)$-module and $P(i^n)$\index{$P(i^n)$} the projective cover of $L(i^n)$.  

\begin{theorem}[{\cite{MR2525917,MR2763732}}] \label{thm:categorification}
There exist $\mathbb{Z}[q,q^{-1}]$-algebra isomorphisms 
\[
\Psi_1 \colon K_{\oplus}(\gproj{R}) \overset{\sim}{\to} \quantum{-}_{\mathbb{Z}[q,q^{-1}]}, \ \Psi_2 \colon K(\gmod{R}) \overset{\sim}{\to} \quantum{-}_{\mathbb{Z}[q,q^{-1}]}^{\mathrm{up}}, \index{$\Psi_1$} \index{$\Psi_2$}
\]
such that $\Psi_1 ([P(i^n)]) = f_i^{(n)}$ and $\Psi_2 ([L(i)]) = f_i$. 
Moreover, we have
\begin{enumerate}
  \item $\Psi_1 \circ \psi_* = (\cdot)^* \circ \Psi_1, \ \Psi_2 \circ \psi_* = (\cdot)^* \circ \Psi_2$, 
  \item $\Psi_1 \circ D' = \overline{(\cdot)} \circ \Psi_1$, 
  \item $\Psi_2 \circ E_i = e'_i \circ \Psi_2, \ \Psi_2 \circ E_i^* = {e'_i}^* \circ \Psi_2$,
  \item $\{ [P] \mid P \in \gproj{R} : D'P \simeq P \}$ is a $\mathbb{Z}[q,q^{-1}]$-basis of $K_{\oplus}(\gproj{R})$, 
  \item $\{ [L] \mid L \in \gmod{R} : \text{self-dual simple}\}$ is a $\mathbb{Z}[q, q^{-1}]$-basis of $K(\gmod{R})$,    
  \item every simple module is absolutely irreducible, 
  \item for $P \in \gproj{R}$ and $M \in \gmod{R}$, we have $\qdim \Hom_R (P, M) = (\overline{\Psi_1([P])}, \Psi_2 ([M]))$. 
\end{enumerate}
\end{theorem}
We give a more detailed account of (7) in Section \ref{sub:Extform}. 
For $M \in \gMod{R(\beta)}$, we set $\wt M\index{$\wt M$} = -\beta$. 

By considering composition factors as in Definition \ref{def:composition}, we extend $\Psi_2$ to
\[
\Psi \colon K(\gMod{R}) \to \quantum{-}_{\mathbb{Z}((q))}^{\mathrm{up}}. \index{$\Psi$}
\]
It is given by the formula
\[
\Psi ([M]) = \sum_{[L]:\text{$L$ is a self-dual simple $R(\beta)$-module}} [M:L]_q \Psi_2([L])
\]
for $M \in \gMod{R(\beta)}$. 

For a simple self-dual graded $R(\beta)$-module $L$, we define
\begin{align*}
 \varepsilon_i (L) &= \max \{ m \in \mathbb{Z}_{\geq 0} \mid E_i^m L \neq 0 \}, \\
 \varphi_i (L) &= \langle h_i, \wt L \rangle + \varepsilon_i (L) = -\langle h_i, \beta \rangle + \varepsilon_i(L), \\ 
 \tilde{e}_i L &= q_i^{-\varepsilon_i (L)+1} \soc (E_i L), \\ 
 \tilde{f}_i L &= q_i^{\varepsilon_i (L)} \hd (F_i L). 
\end{align*}
It yields a crystal structure on $\mathcal{B} = \{ [L] \mid L \in \gmod{R} : \text{self-dual simple} \}$.  \index{$(\mathcal{B}, \wt, \varepsilon_i,\varphi_i, \tilde{e}_i, \tilde{f}_i)$}
The degree shift can be determined by Lemma \ref{lem:degreeshift} and Lemma \ref{lem:LambdaforLi} below. 

Using $E_i^*$ and $F_i^*$ instead of $E_i$ and $F_i$, we also obtain another crystal $(\mathcal{B}, \wt, \varepsilon_i^*, \varphi_i^*, \tilde{e}_i^*, \tilde{f}_i^*)$. \index{$\varepsilon_i^*$} \index{$\varphi_i^*$} \index{$\tilde{e}_i^*$} \index{$\tilde{f}_i^*$}
It is isomorphic to $(\mathcal{B}, \wt, \varepsilon_i, \varphi_i, \tilde{e}_i, \tilde{f}_i)$ via the involution induced by $\psi_*$. 

\begin{theorem}[{\cite[Theorem 7.4]{MR2822211}}]\label{thm:categoricalcrystal}
 There exists an isomorphism of bicrystals $\mathcal{B} \simeq B(\infty)$. 
\end{theorem}

Moreover, we have $\tilde{e}_i^{\varepsilon_i(L)} L \simeq E_i^{(\varepsilon_i (L))} L$ (\cite[Lemma 3.7]{MR2525917}), where the divided power is defined by
\[
E_i^{(n)} M = \Hom_{R(n\alpha_i)} (P(i^n), e(n\alpha_i, *) M). \index{$E_i^{(n)}$}
\]
$E_i^{*(n)} M$\index{$E_i^{*(n)}$} is defined in the same manner using $E_i^*$.  

\subsection{Cyclotomic quiver Hecke algebras} \label{sec:cyclotomic}

In this section, we review cyclotomic quotients and the categorification of highest weight integrable modules, following \cite{MR2995184}. 
Let $A = \bigoplus_{d\in \mathbb{Z}_{\geq 0}} A_d$ be a $\mathbb{Z}_{\geq 0}$-graded Noetherian Laurentian commutative algebra with $A_0 = \mathbf{k}$. 
For $\Lambda \in P_+$, we fix a family of homogeneous polynomials $a_{\Lambda} = \{ a_{\Lambda, i} (t_i) \}_{i \in I}$\index{$a_{\Lambda} = \{ a_{\Lambda,_i}(t_i) \}_{i \in I}$} such that
\begin{enumerate}
  \item for each $i \in I$, if we set $\deg t_i = (\alpha_i, \alpha_i)$, the polynomial $a_{\Lambda, i}(t_i)$ is homogeneous of degree $2(\alpha_i, \Lambda)$, 
  \item for each $i \in I$, the coefficient of $t_i^{\langle h_i, \Lambda \rangle}$ in $a_{\Lambda, i}(t_i)$ is nonzero. 
\end{enumerate}

\begin{definition} \label{def:cyclotomic}
 For $\lambda = \Lambda - \beta$ with $\beta \in Q_+$, we define the cyclotomic quiver Hecke algebra
 \[
 R_A^{a_{\Lambda}} (\lambda) = (R(\beta) \otimes_{\mathbf{k}} A) / \langle a_{\Lambda, i}(x_n)e(*,i) \ (i \in I) \rangle, \index{$R_A^{a_{\Lambda}}(\lambda)$}
 \]
 where $n = \height \beta$. 
\end{definition}

In this paper, we consider two cases $A = \mathbf{k}$ and $A = \mathbf{k}[z]$. 
When $A= \mathbf{k}$, $a_{\Lambda, i} (t_i) \in \mathbf{k}^{\times} t_i^{\langle h_i, \Lambda \rangle}$ and we just write $R^{\Lambda} (\lambda)\index{$R^{\Lambda}(\lambda)$} = R_{\mathbf{k}}^{a_{\Lambda}} (\lambda)$. 

It is known that $R_A^{a_{\Lambda}}(\lambda)$ is a finitely generated projective $A$-module \cite[Corollary 4.4, Theorem 4.5]{MR2995184}. 
Let $\gMod{R_A^{a_{\Lambda}}(\lambda)}$ be the category of finitely generated graded $R_A^{a_{\Lambda}}(\lambda)$-modules, 
$\gproj{R_A^{a_{\Lambda}}(\lambda)}$ the category of finitely generated graded projective $R_A^{a_{\Lambda}}(\lambda)$-modules, and
$\gmod{R_A^{a_{\Lambda}}(\lambda)}$ the category of finite dimensional graded $R_A^{a_{\Lambda}}(\lambda)$-modules. 
We set 
\begin{align*}
\gMod{R_A^{a_{\Lambda}}} &= \bigoplus_{\lambda \in \Lambda - Q_+} \gMod{R_A^{a_{\Lambda}}(\lambda)}, \\ 
\gmod{R_A^{a_{\Lambda}}} &= \bigoplus_{\lambda \in \Lambda - Q_+} \gmod{R_A^{a_{\Lambda}}(\lambda)}, \\ 
\gproj{R_A^{a_{\Lambda}}} &= \bigoplus_{\lambda \in \Lambda - Q_+} \gproj{R_A^{a_{\Lambda}}(\lambda)}.
\end{align*}

For $i \in I$ and $n \in \mathbb{Z}_{\geq 1}$, we define the following functors
\begin{align*}
 (F_i^{a_{\Lambda}})^{(n)} \colon& \gMod{R_A^{a_{\Lambda}}(\lambda)} \to \gMod{R_A^{a_{\Lambda}}(\lambda - n\alpha_i)}, \\ \index{$(F_i^{a_{\Lambda}})^{(n)}$}
 & X \mapsto R_A^{a_\Lambda} (\lambda - n\alpha_i) e(n\alpha_i, *) \otimes_{R(n\alpha_i) \otimes R_A^{a_{\Lambda}}(\lambda)} (P(i^n) \otimes X) \\
 (E_i^{a_{\Lambda}})^{(n)} \colon& \gMod{R_A^{a_{\Lambda}}(\lambda)} \to \gMod{R_A^{a_{\Lambda}}(\lambda + n\alpha_i)}, \\ \index{$(E_i^{a_{\Lambda}})^{(n)}$}
 & X \mapsto \Hom_{R(n\alpha_i)} (P(i^n), e(n\alpha_i, *) X).
\end{align*}
They are both exact and preserve finite dimensionality. 
Hence they induce endomorphisms on the Grothendieck group $K(\gmod{R_A^{a_{\Lambda}}})$ and on the split Grothendieck group $K_{\oplus}(\gproj{R_A^{a_{\Lambda}}})$, 
denoted by the same letter. 
We write $F_i^{a_{\Lambda}} = (F_i^{a_{\Lambda}})^{(1)}, E_i^{a_{\Lambda}} = (E_i^{a_{\Lambda}})^{(1)}$.

\begin{theorem}[{\cite{MR2995184,MR3709726}}] \label{thm:cyclotomiccategorification}
  There exist $\mathbb{Z}[q, q^{-1}]$-module isomorphisms
  \[
  \Psi_1^{a_{\Lambda}} \colon K(\gproj{R_A^{a_{\Lambda}}}) \overset{\sim}{\to} V(\Lambda)_{\mathbb{Z}[q,q^{-1}]}, \ \Psi_2^{a_{\Lambda}} \colon K(\gmod{R_A^{a_{\Lambda}}}) \overset{\sim}{\to} V(\Lambda)_{\mathbb{Z}[q,q^{-1}]}^{\mathrm{up}}, \index{$\Psi_1^{a_{\Lambda}}$} \index{$\Psi_2^{a_{\Lambda}}$}
  \] 
such that 
\begin{enumerate}
\item $\Psi_1^{a_{\Lambda}} ([A]) = u_{\Lambda}, \Psi_2^{a_{\Lambda}} ([\mathbf{k}]) = u_{\Lambda}$, 
\item For $P \in \gproj{R_A^{a_{\Lambda}}(\lambda)}$, $\Psi_1^{a_{\Lambda}}([F_i^{a_{\Lambda}}P]) = f_i \Psi_1^{a_{\Lambda}}([P]), \Psi_1^{a_{\Lambda}}([E_i^{a_{\Lambda}}P]) = q_i^{1+\langle h_i, \lambda \rangle} e_i \Psi_1^{a_{\Lambda}}([P])$, 
\item For $M \in \gmod{R_A^{a_{\Lambda}}(\lambda)}$, $\Psi_2^{a_{\Lambda}}([F_i^{a_{\Lambda}}M]) = q_i^{-1+\langle h_i, \lambda \rangle} f_i \Psi_2^{a_{\Lambda}}([M]), \Psi_2^{a_{\Lambda}}([E_i^{a_{\Lambda}}M]) = e_i \Psi_2^{a_{\Lambda}}([M])$,
\item For $P \in \gproj{R_A^{a_{\Lambda}}}$ and $M \in \gmod{R_A^{a_{\Lambda}}}$, $\qdim \Hom_R(P, M) = (\overline{\Psi_1^{a_{\Lambda}}([P])}, \Psi_2^{a_{\Lambda}}([M]))$. 
\end{enumerate}
\end{theorem}

In fact, \cite{MR2995184,MR3709726} proved a stronger statement: 
the functors $E_i^{a_{\Lambda}}$ and $F_i^{a_{\Lambda}}$ define an integrable 2-representation of the 2-Kac-Moody algebra \cite{rouquier20082kacmoodyalgebras}. 
As a corollary, we have the following theorem. 

\begin{theorem} \label{thm:sl2categorification}
  Let $\lambda \in W \Lambda$ and assume that $\langle h_i, \lambda \rangle \geq 0$. 
  Then, we have a graded equivalence  
  \[
  \gMod{R_A^{a_{\Lambda}}(\lambda)}  \xtofrom[(E_i^{a_{\Lambda}})^{(\langle h_i, \lambda \rangle)}]{(F_i^{a_{\Lambda}})^{(\langle h_i, \lambda \rangle)}} \gMod{R_A^{a_{\Lambda}}(s_i \lambda)}. 
  \]
\end{theorem}

\begin{proof}
By Theorem \ref{thm:cyclotomiccategorification}, the action of $E_i$ on $\gMod{R_A^{a_{\Lambda}}(\lambda)}$ (resp.\ $F_i$ on $\gMod{R_A^{a_{\Lambda}}(s_i\lambda)})$ vanishes. 
Since the 2-Kac-Moody algebra acts on $\gMod{R_A^{a_{\Lambda}}}$, 
\cite[Lemma 4.12]{rouquier20082kacmoodyalgebras} (or \cite[Theorem 5.2.8]{MR2963085}) shows that the two functors stated in the theorem are mutually quasi-inverse to each other. 
\end{proof}

Theorem \ref{thm:categorification} and Theorem \ref{thm:cyclotomiccategorification} are related as follows. 
Recall from Lemma \ref{lem:globalbasis} the two morphisms $\pi_{\Lambda} \colon \quantum{-}_{\mathbb{Z}[q, q^{-1}]} \to V(\Lambda)_{\mathbb{Z}[q,q^{-1}]}$ and $\iota_{\Lambda} \colon V(\Lambda)_{\mathbb{Z}[q,q^{-1}]}^{\text{up}} \to \quantum{-}_{\mathbb{Z}[q,q^{-1}]}^{\text{up}}$. 
We define two functors $\Pi_{\Lambda}\index{$\Pi_{\Lambda}$} \colon \gMod{R(\beta)} \to \gMod{R^{\Lambda}(\beta)}$ and $I_{\Lambda}\index{$I_{\Lambda}$} \colon \gMod{R^{\Lambda}(\beta)} \to \gMod{R(\beta)}$ as 
\[
\Pi_{\Lambda}(M) = R^{\Lambda}(\beta) \otimes_{R(\beta)} M, \ I_{\Lambda}(M) = \operatorname{Infl}_{R^{\Lambda}(\beta)}^{R(\beta)} M.
\]
Then, we have
$\pi_{\Lambda} \circ \Psi_1 = \Psi_1^{a_{\Lambda}} \circ \Pi_{\Lambda}, \iota_{\Lambda} \circ \Psi_2^{a_{\Lambda}} = \Psi_2 \circ I_{\Lambda}$.

For a simple $R^{\Lambda}(\lambda)$-module $L$, we define 
\begin{align*}
  \wt^{\Lambda} (L) &= \lambda, \\ 
  \varepsilon_i^{\Lambda} (L) &= \max \{m \in \mathbb{Z}_{\geq 0} \mid (E_i^{a_{\Lambda}})^m L \neq 0\}, \\
  \varphi_i^{\Lambda} (L) &= \langle h_i, \wt (L) \rangle + \varepsilon_i (L), \\
  \tilde{e}_i^{\Lambda} L &= q_i^{-\varepsilon_i (L) +1} \soc E_i^{a_{\Lambda}} L, \\
  \tilde{f}_i^{\Lambda} L &= q_i^{\varepsilon_i (L)} \hd F_i^{a_{\Lambda}} L. \\
\end{align*}
It defines a crystal structure on 
\[     
\mathcal{B}^{\Lambda} = \{ [L] \in K(\gmod{R^{\Lambda}}) \mid \text{$L$ is a self-dual simple $R^{\Lambda}$-module} \}, \index{$(\mathcal{B}^{\Lambda}, \wt^{\Lambda}, \varepsilon_i^{\Lambda}, \varphi_i^{\Lambda}, \tilde{e}_i^{\Lambda}, \tilde{f}_i^{\Lambda})$}
\] 
which is isomorphic to $B(\Lambda)$ \cite[Theorem 7.5]{MR2822211}. 
Moreover, the (non-strict) embedding of crystals $\overline{\iota_{\Lambda}} \colon B(\Lambda) \to B(\infty)$ coincides with $I_{\Lambda} \colon \mathcal{B}^{\Lambda} \to \mathcal{B}$ under the isomorphisms $B(\Lambda) \simeq \mathcal{B}^{\Lambda}$ and $B(\infty) \simeq \mathcal{B}$. 

\subsection{R-matrices}

We provide a brief overview on the technique of R-matrices developed by Kang, Kashiwara, Kim, Oh and Park \cite{MR3748315, MR3758148, MR3790066, MR4717658}, 
primarily following \cite{MR3790066}. 
For a comprehensive treatment in a more general context, see \cite{MR4717658}. 

Let $Z(R(\beta))$ be the center of $R(\beta)$. \index{$Z(R(\beta))$}
We have 
\[  
Z(R(\beta)) = (\bigoplus_{\nu \in I^{\beta}} \mathbf{k}[x_1, \ldots, x_{\height \beta}]e(\nu))^{\Sym_{\height \beta}}. 
\]
In particular, $Z(R(\beta))_0 = \mathbf{k}$. 

\begin{lemma}[{\cite[Lemma 1.5]{MR3748315}, \cite[Lemma 1.9]{MR3790066}}]  \label{lem:universalR}
Let $\beta, \gamma \in Q_+$. 
There is a natural homogeneous homomorphism called the universal R-matrix
\[
\universalR{M}{N} \colon M \circ N \to N \circ M \index{$\universalR{M}{N}$}
\]
for $M \in \gMod{R(\beta)}$ and $N \in \gMod{R(\gamma)}$. 
It is of degree $-(\beta,\gamma) + (\beta,\gamma)_n$, \index{$(\beta,\gamma)_n$}
where $(\cdot,\cdot)_n \colon Q \times Q \to \mathbb{Z}$ is a bilinear form defined by $(\alpha_i, \alpha_j)_n = \delta_{i,j}(\alpha_i,\alpha_i) \ (i, j \in I)$. 
It satisfies
\[
\universalR{N}{M} \universalR{M}{N} (u \otimes v) = \mathfrak{p}_{\beta,\gamma} (u\otimes v)
\]
for $u \in M, v \in N$, where $\mathfrak{p}_{\beta,\gamma} \in Z(R(\beta)) \otimes Z(R(\gamma))$ is a homogeneous element defined by 
\[
\mathfrak{p}_{\beta, \gamma} = \sum_{\nu \in I^{\beta} \times I^{\gamma}} \prod_{1 \leq k \leq \height (\beta) < l \leq \height (\beta + \gamma), \nu_k \neq \nu_l} Q_{\nu_k, \nu_l}(x_k, x_l) e(\nu). 
\] \index{$\mathfrak{p}_{\beta,\gamma}$}
Moreover, we have 
\[
\universalR{L}{M\circ N} = (\id_M \otimes \universalR{L}{N})(\universalR{L}{M} \otimes \id_N),\ \universalR{L \circ M}{N} = (\universalR{L}{N} \otimes \id_M)(\id_L \otimes \universalR{M}{N})
\]
for $L, M, N \in \gMod{R}$. 
\end{lemma}

In order to effectively utilize universal R-matrices, we need to introduce the following concept. 

\begin{definition}[{\cite[Definition 2.1]{MR3790066}}] \label{def:affinization}
 Let $\beta \in Q_+$ and $M \in \gmod{R(\beta)}$ be a simple module. 
 An affine object of $M$ is a pair $(\widehat{M}, z = z_{\widehat{M}})$ of an $R(\beta)$-module $\widehat{M}$ and an injective endomorphism $z \in \End_{R(\beta)}(\widehat{M})$ of positive degree satisfying
 \begin{enumerate}
  \item as a $\mathbf{k}[z]$-module, $\widehat{M}$ is free of finite rank, and
  \item we have an isomorphism of $R(\beta)$-modules $\widehat{M}/z\widehat{M} \simeq M$. 
 \end{enumerate}

 An affine object $(\widehat{M}, z)$ of $M$ is said to be an affinization if it additionally satisfies 
 \begin{enumerate}
 \setcounter{enumi}{2}
  \item $\mathfrak{p}_{i, \beta} \widehat{M} \neq 0$ for all $i \in I$, 
 \end{enumerate}
 where $\mathfrak{p}_{i,\beta}$ is an element of $Z(R(\beta))$ defined by 
 \[
 \mathfrak{p}_{i, \beta} = \sum_{\nu \in I^{\beta}} \left( \prod_{1 \leq k \leq \height (\beta), \nu_k = i} x_k \right) e(\nu).  
 \] \index{$\mathfrak{p}_{i,\beta}$}
\end{definition}

\begin{remark}
  Every simple $R(\beta)$-module $M$ has an affine object: for example, set $\widehat{M} = M \otimes \mathbf{k}[z]$ with
  \[
  x (m \otimes f) = xm \otimes f, \ z (m \otimes f) = m \otimes zf \ \text{for $x \in R(\beta), m \in M, f \in \mathbf{k}[z]$}. 
  \]

  However, unless $\beta = 0$, it is not an affinization. 
  In fact, write $\beta = \sum_{i \in I} k_i \alpha_i$, and take $i \in I$ such that $k_i \neq 0$. 
  Then, $\mathfrak{p}_{i,\beta}$ is of positive degree and act on $M$ by $0$. 
  Hence, $\mathfrak{p}_{i,\beta} \widehat{M} = 0$. 

  When $C$ is symmetric and every polynomials $Q_{i,j}$ is a power of $(u-v)$ up to constant multiple, there is a functorial way to construct affinization for every simple module {\cite[Section 1.3.2]{MR3748315}}. 
  On the other hand, the author does not know whether every simple module admits an affinization in other cases. 
\end{remark}

\begin{example} \label{ex:affinization}
 For $i \in I$, $R(\alpha_i) \simeq \mathbf{k}[x_1]$ and $L(i) \simeq R(\alpha_i)/x_1 R(\alpha_i)$. 
 Let $\widehat{L}(i) = R(\alpha_i)$\index{$\widehat{L}(i)$} and $z \in \End_{R(\alpha_i)} (R(\alpha_i))$ be the multiplication by $x_1$. 
 Then, $(\widehat{L}(i),z)$ is an affinization of $L(i)$. 
\end{example}

\begin{lemma} \label{lem:endaffineobj}
  Let $\beta \in Q_+$. 
  If $(\widehat{M},z)$ is an affine object of a simple $R(\beta)$-module $M$, 
  we have $\End_{R(\beta)[z]}(\widehat{M}) \simeq \mathbf{k}[z]$. 
  \end{lemma}
  
  \begin{proof}
  Note that, since $\widehat{M}$ is a finite-rank free $\mathbf{k}[z]$-module, it is a finitely generated $R(\beta)[z]$-module and 
  \[
  \widehat{M}/z\widehat{M} \xrightarrow{z^p, \simeq} z^p \widehat{M}/z^{p+1}\widehat{M} \ (p \in \mathbb{Z}_{\geq 0}),\  \widehat{M} \simeq \varprojlim_p \widehat{M}/z^p \widehat{M}. 
  \]
  
  First, we compute 
  \begin{align*}
  \Hom_{R(\beta)[z]} (\widehat{M}, \widehat{M}/z\widehat{M}) &\simeq \Hom_{R(\beta)} (\widehat{M}/z\widehat{M}, \widehat{M}/z\widehat{M}) \\
  &\simeq \End_{R(\beta)} (M) \quad \text{by condition (2) of Definition \ref{def:affinization}} \\
  & \simeq \mathbf{k}. 
  \end{align*}
  
  Next, let $p \in \mathbb{Z}_{\geq 0}$. 
  By induction, we deduce that 
  \[   
  \dim \Hom_{R(\beta)[z]} (\widehat{M}, \widehat{M}/z^p\widehat{M}) \leq p. 
  \]   
  On the other hand, for $0 \leq k \leq p-1$, we have a nonzero homomorphism of degree $k\deg z$
  \[
  \widehat{M} \twoheadrightarrow \widehat{M}/z^{p-k}\widehat{M} \xrightarrow{z^k} \widehat{M}/z^p\widehat{M}. 
  \]
  Hence, $\Hom_{R(\beta)[z]} (\widehat{M}, \widehat{M}/z^p\widehat{M}) \simeq \mathbf{k}[z]/(z^p)$. 
  
  Now, we obtain
  \[
  \End_{R(\beta)[z]} (\widehat{M}) \simeq \varprojlim_p \Hom_{R(\beta)[z]} (\widehat{M}, \widehat{M}/z^p\widehat{M}) \simeq \varprojlim_p \mathbf{k}[z]/(z^p) \simeq \mathbf{k}[z]. 
  \]
  \end{proof}

\begin{lemma} \label{lem:affinization}
  Let $\beta, \gamma \in Q_+$. 
  Let $M$ be a simple $R(\beta)$-module that admits an affinization $(\widehat{M}, z)$, 
  and $N$ an $R(\gamma)$-module. 
  Then, $\universalR{\widehat{M}}{N}, \universalR{N}{\widehat{M}}$ are injective. 
  Additionally, if $N$ is finite-dimensional, the two homomorphisms obtained by inverting $z$
\begin{align*}
  \universalR{\widehat{M}}{N} &\colon \mathbf{k}[z, z^{-1}] \otimes_{\mathbf{k}[z]} (\widehat{M} \circ N) \to \mathbf{k}[z,z^{-1}] \otimes_{\mathbf{k}[z]} (N \circ \widehat{M}), \\ 
  \universalR{N}{\widehat{M}} &\colon \mathbf{k}[z, z^{-1}] \otimes_{\mathbf{k}[z]} (N \circ \widehat{M}) \to \mathbf{k}[z,z^{-1}] \otimes_{\mathbf{k}[z]} (\widehat{M} \circ N), 
\end{align*}
are isomorphisms. 
\end{lemma}

\begin{proof}
For any $i \in I$, since $\mathfrak{p}_{i,\beta}\widehat{M} \neq 0$ by Definition \ref{def:affinization} (3), it gives a nonzero element of $\End_{R(\beta)[z]}(\widehat{M})$. 
By Lemma \ref{lem:endaffineobj}, it coincides with a multiplication by a power of $z$ up to a nonzero scalar multiple. 

We claim that $\universalR{N}{\widehat{M}}\universalR{\widehat{M}}{N}$ is injective. 
Suppose that it is not the case. 
By Lemma \ref{lem:universalR}, there is a nonzero homogeneous element of degree zero $x = \sum_{d \in \mathbb{Z}} u_d \otimes v_d \ (u_d \in \widehat{M}_d, v_d \in N_{-d})$ such that $\mathfrak{p}_{\beta,\gamma}(x) = 0$.  
Take the largest $d_0$ such that $u_{d_0} \otimes v_{d_0} \neq 0$. 
By the definition, $\mathfrak{p}_{\beta,\gamma} \in Z(R(\beta)) \otimes Z(R(\gamma))$ is homogeneous of degree $-2(\beta,\gamma) + 2(\beta,\gamma)_n$. 
Moreover, the term in $Z(R(\beta))_{-2(\beta,\gamma)+2(\beta,\gamma)_n} \otimes Z(R(\gamma))_0$ is
\[
\left( \prod_{i,j \in I, i\neq j} \mathfrak{p}_{i,\beta}^{-k_j \langle h_i, \alpha_j \rangle} \right) \otimes 1, 
\]  
where $k_j$ is the integer given by $\gamma = \sum_{j \in I} k_j \alpha_j$. 
Hence, 
\[
\mathfrak{p}_{\beta,\gamma} (x) = \left( \prod_{i,j \in I, i\neq j} \mathfrak{p}_{i,\beta}^{-k_j \langle h_i, \alpha_j \rangle}u_{d_0} \right)  \otimes v_{d_0} \mod \bigoplus_{d<d_0} \widehat{M}_{-2(\beta,\gamma)+2(\beta,\gamma)_n + d} \otimes N_{-d}. 
\]
Since $\mathfrak{p}_{i,\beta}$ acts on $\widehat{M}$ as a multiplication by a power of $z$ up to a nonzero scalar multiple, $\mathfrak{p}_{\beta, \gamma} x \neq 0$. 
It yields a contradiction.
Therefore, $\universalR{N}{\widehat{M}}\universalR{\widehat{M}}{N}$ is injective, which proves that $\universalR{\widehat{M}}{N}$ is injective.  

The injectivity of $\universalR{N}{\widehat{M}}$ is proved in the same manner. 
The remaining assertions are \cite[Lemma 2.9]{MR3790066}. 
\end{proof}

\begin{remark}
 The property stated in Lemma \ref{lem:affinization} plays a crucial role in the theory of R-matrices.  
 In fact, Kashiwara-Kim-Oh-Park \cite{MR4717658} generalized the concept of affinization based on this property instead of the third condition in Definition \ref{def:affinization}.  
\end{remark}

\begin{lemma} \label{lem:affinefg}
 Let $\beta \in Q_+ \setminus \{ 0 \}$. 
 Let $(\widehat{M},z)$ be an affinization of simple $R(\beta)$-module $M$. 
 Then, $\widehat{M}$ is finitely generated as an $R(\beta)$-module. 
\end{lemma}

\begin{proof}
  By Definition \ref{def:affinization} (1), $\widehat{M}$ is finitely generated over $R(\beta)[z]$.  
  By Lemma \ref{lem:endaffineobj}, $\End_{R(\beta)[z]} (\widehat{M}) \simeq \mathbf{k}[z]$. 
  We write $\beta =  \sum_{i \in I} k_i \alpha_i \ (k_i \in \mathbf{Z}_{\geq 0})$.
  Since $\beta \neq 0$, we have $k_i \neq 0$ for some $i \in I$. 
  Then, $\mathfrak{p}_{i, \beta}$ is of positive degree. 
  By Definition \ref{def:affinization} (3), $\mathfrak{p}_{i, \beta}$ yields a nonzero element of $\End_{R(\beta)[z]} (\widehat{M})$. 
  Thus, there exist $c \in \mathbf{k}^{\times}$ and $d \in \mathbf{Z}_{\geq 0}$ 
  such that $\mathfrak{p}_{i, \beta}\mid_{\widehat{M}} = cz^d \mid_{\widehat{M}}$, which proves the lemma. 
\end{proof}

A simple module $L$ is said to be real if $L \circ L$ is simple. 

\begin{definition} [{\cite[p.1173]{MR3790066}}] \label{def:renormalizedR} 
  Let $\beta, \gamma \in Q_+$. 
  Let $M$ be a real simple $R(\beta)$-module that admits an affinization $(\widehat{M}, z)$,  
  and $N$ a simple $R(\gamma)$-module. 

  (1) We define the renormalized R-matrix by
  \[
  \renormalizedR{\widehat{M}}{N} = z^{-s}\universalR{\widehat{M}}{N} \colon \widehat{M} \circ N \to N \circ \widehat{M}, \index{$\renormalizedR{\widehat{M}}{N}$}
  \]
  where $s$ is the largest integer such that $\universalR{\widehat{M}}{N}(\widehat{M} \circ N) \subset z^s N \circ \widehat{M}$. 
  $\renormalizedR{\widehat{M}}{N}$ is homogeneous, but not necessarily of degree zero. 
  In a similar manner, we define the renormalized R-matrix $\renormalizedR{N}{\widehat{M}}\index{$\renormalizedR{N}{\widehat{M}}$} \colon N \circ \widehat{M} \to \widehat{M} \circ N$. 

  (2) We define 
  \begin{align*}
  \mathbf{r}_{M,N} &= \renormalizedR{\widehat{M}}{N} \rvert_{z = 0} \colon M \circ N \to N \circ M, \\ \index{$\mathbf{r}_{M,N}$}
  \mathbf{r}_{N,M} &= \renormalizedR{N}{\widehat{M}} \rvert_{z = 0} \colon N \circ M \to M \circ N. 
   \end{align*}

   (3) We define
   \[
   \Lambda(M,N) = \deg \mathbf{r}_{M,N}, \ \Lambda(N,M) = \deg \mathbf{r}_{N,M}.  \index{$\Lambda(M,N)$}
   \]
\end{definition}

Note that $\mathbf{r}_{M,N}$ and $\mathbf{r}_{N,M}$ are nonzero. 
When $N$ is also a real simple module that admits an affinization $(\widehat{N}, w)$,  
there are two ways to define $\mathbf{r}_{M,N}$.
They coincide up to a nonzero scalar multiple by Theorem \ref{thm:rmatrix} (1) below. 

Note that, using the notation in the definition, we have 
\[
\Lambda(M,N) = \deg \universalR{\widehat{M}}{N} - s \deg z \in (\beta,\gamma) - (\beta,\gamma)_n - \mathbb{Z}_{\geq 0} \cdot \deg z. 
\]
Similarly $\Lambda(N,M) \in (\beta,\gamma) - (\beta,\gamma)_n - \mathbb{Z}_{\geq 0} \cdot \deg z$. 

We now present several important results that have been established using R-matrices. 
For a simple $R(\beta)$-module $M$ and a simple $R(\gamma)$-module $N$, we write $M \nabla N\index{$M\nabla N$} = \hd (M \circ N)$ and $M \Delta N\index{$M\Delta N$} = \soc (M \circ N)$.  

\begin{theorem}[{\cite[Proposition 2.10, 2.11]{MR3790066}, \cite[Proposition 3.2.9, Theorem 4.1.1]{MR3758148}}] \label{thm:rmatrix}
Let $\beta, \gamma \in Q_+$. 
 Let $M$ be a real simple $R(\beta)$-module that admits an affinization $(\widehat{M},z)$. 
 Let $N$ be a simple $R(\gamma)$-module. 
 \begin{enumerate}
 \item $\Hom_{R(\beta+\gamma)} (M\circ N, N\circ M) = \mathbf{k} \mathbf{r}_{M,N}, \ \Hom_{R(\beta+\gamma)} (N\circ M, M\circ N) = \mathbf{k} \mathbf{r}_{N,M}$. 
 \item $\End_{R(\beta+\gamma)} (M\circ N) = \mathbf{k} \id_{M\circ N}, \ \End_{R(\beta+\gamma)} (N\circ M) = \mathbf{k} \id_{N\circ M}$. 
 \item $\Hom_{R(\beta+\gamma)[z]} (\widehat{M}\circ N, N \circ \widehat{M}) = \mathbf{k}[z]\renormalizedR{\widehat{M}}{N}, \ \Hom_{R(\beta+\gamma)[z]} (N \circ \widehat{M}, \widehat{M}\circ N) = \mathbf{k}[z]\renormalizedR{N}{\widehat{M}}$. 
 \item $\End_{R(\beta+\gamma)[z]} (\widehat{M}\circ N) = \mathbf{k}[z] \id_{\widehat{M} \circ N},\ \End_{R(\beta+\gamma)[z]} (N \circ \widehat{M}) = \mathbf{k}[z]\id_{N \circ \widehat{M}}$. 
 \item \begin{align*}
 q^{\Lambda(M,N)} M \nabla N &\simeq \Image (\mathbf{r}_{M,N}\colon q^{\Lambda(M,N)} M\circ N \to N\circ M ) = N \Delta M,\\ 
 q^{\Lambda(N,M)} N \nabla M &= \Image (\mathbf{r}_{N,M} \colon q^{\Lambda(N,M)} N\circ M \to M\circ N) = M\Delta N,
 \end{align*} 
 and both of them are simple.  
 \item $\Lambda (M, M \nabla N) = \Lambda (M, N) (= \Lambda (M, N\Delta M))$, and $\Lambda(M, L) \in \Lambda(M, N) - \mathbb{Z}_{> 0} \cdot \deg z$ for any composition factor $L$ of $\Ker \mathbf{r}_{M,N}$ (or of $N \circ M / N \Delta M$). In particular, $[M\circ N: M \nabla N]_q = [N\circ M: N \Delta M]_q = 1$. 
 \item $\Lambda (N \nabla M, M) = \Lambda (N, M) (= \Lambda (M \Delta N, M))$, and $\Lambda(L, M) \in \Lambda(N,M) - \mathbb{Z}_{> 0} \cdot \deg z$ for any composition factor $L$ of $\Ker \mathbf{r}_{N,M}$ (or of $M\circ N / M \Delta N$). In particular, $[N \circ M: N \nabla M]_q = [M \circ N: M \Delta N]_q = 1$. 
 \end{enumerate}
\end{theorem}

\begin{lemma}[{\cite[Lemma 3.1.4]{MR3758148}}] \label{lem:degreeshift}
  Let $M, N$ be as above. 
  Assume that $M$ and $N$ are self-dual. 
  Then, $(\Lambda(M,N) + (\beta,\gamma))/2$ is an integer and $q^{(\Lambda(M,N)+(\beta,\gamma))/2} M\nabla N$ is self-dual. 
\end{lemma}

\begin{proposition} [{\cite[Proposition 2.12, Theorem 2.13]{MR3790066}}] \label{prop:affinermatrix}
 Let $M, (\widehat{M},z), N$ be as above, and $(\widehat{N},w)$ an affine object of $N$. 
 \begin{enumerate}
  \item Both $\Hom_{R(\beta+\gamma)[z,w]} (\widehat{M}\circ \widehat{N}, \widehat{N}\circ \widehat{M})$ and $\Hom_{R(\beta+\gamma)[z,w]} (\widehat{N} \circ \widehat{M}, \widehat{M} \circ \widehat{N})$ are free $\mathbf{k}[z,w]$-module of rank $1$. 
  Let $\renormalizedR{\widehat{M}}{\widehat{N}}$ and $\renormalizedR{\widehat{N}}{\widehat{M}}$ denote their respective generators. \index{$\renormalizedR{\widehat{M}}{\widehat{N}}$}
  Then we have $\renormalizedR{\widehat{M}}{\widehat{N}} \rvert_{w=0} = \renormalizedR{\widehat{M}}{N}$ and $\renormalizedR{\widehat{N}}{\widehat{M}} \rvert_{w=0} = \renormalizedR{N}{\widehat{M}}$
  up to nonzero scalar multiples.
  Hence, $\deg (\renormalizedR{\widehat{M}}{\widehat{N}}) = \Lambda(M,N)$ and $\deg (\renormalizedR{\widehat{N}}{\widehat{M}}) = \Lambda(N,M)$. 
  \item $\End_{R(\beta+\gamma)[z,w]} (\widehat{M}\circ \widehat{N}) = \mathbf{k}[z,w]\id_{\widehat{M}\circ \widehat{N}}, \ \End_{R(\beta+\gamma)[z,w]} (\widehat{N}\circ \widehat{M}) = \mathbf{k}[z,w]\id_{\widehat{N}\circ \widehat{M}}$. 
 \end{enumerate}
\end{proposition}

\begin{remark}
In \cite[Proposition 2.12, Theorem 2.13]{MR3790066}, $(\widehat{N},w)$ is assumed to be an affinization of $N$.  
However, their proof requires only that $(\widehat{N},w)$ is an affine object of $N$.  
See also \cite[Theorem 6.10, Lemma 6.22]{MR4717658}. 
\end{remark}

\begin{remark} \label{rem:univRvsnormR}
By Proposition \ref{prop:affinermatrix} (1), there exists a unique homogeneous polynomial $f(z,w) \in \mathbf{k}[z,w]$ such that $\universalR{\widehat{M}}{\widehat{N}} = f(z,w) \renormalizedR{\widehat{M}}{\widehat{N}}$. 
Moreover, $f(z,0) \neq 0$. 
Similarly, there exists a unique homogeneous polynomial $g(z,w) \in \mathbf{k}[z,w]$ such that $\universalR{\widehat{N}}{\widehat{M}} = g(z,w) \renormalizedR{\widehat{N}}{\widehat{M}}$, 
which satisfies $g(0,w) \neq 0$. 
\end{remark}

\begin{remark} \label{rem:rennotid}
 Let $M$ be a real simple $R(\beta)$-module that admits an affinization $(\widehat{M},z)$. 
 By Theorem \ref{thm:rmatrix} (2), $\mathbf{r}_{M,M}$ is a nonzero scalar multiplication. 
 On the other hand, $\renormalizedR{\widehat{M}}{\widehat{M}}$ is not a scalar multiplication for the following reason. 
 Let $N$ be a copy of $M$, and $(\widehat{N}, w)$ a copy of $(\widehat{M},z)$. 
 By the definition, $\renormalizedR{\widehat{M}}{\widehat{M}} = \renormalizedR{\widehat{M}}{\widehat{N}}$ is an element of $\Hom_{R(2\beta)[z,w]} (\widehat{M}\circ \widehat{N}, \widehat{N}\circ \widehat{M})$. 
 If $\renormalizedR{\widehat{M}}{\widehat{M}}$ is a scalar multiplication, it satisfies
 \[
 \renormalizedR{\widehat{M}}{\widehat{N}}(zv) = wv, \renormalizedR{\widehat{M}}{\widehat{N}}(wv) = zv \ \text{for $v \in \widehat{M} \circ \widehat{N}$}. 
 \]
 Since $\widehat{N}\circ \widehat{M}$ is a free $\mathbf{k}[z,w]$-module, $zv \neq wv$ unless $v = 0$. 
 It contradicts the definition that $\renormalizedR{\widehat{M}}{\widehat{N}}$ is $\mathbf{k}[z,w]$-linear. 
\end{remark}

The following proposition is easy but crucial in this paper. 

\begin{proposition} \label{prop:reninj}
  Let $M$ be a real simple $R(\beta)$-module that admits an affinization $(\widehat{M},z)$. 
  Let $N$ be a simple $R(\gamma)$-module, and $(\widehat{N},w)$ its affine object. 
  Then $\renormalizedR{\widehat{M}}{\widehat{N}}, \renormalizedR{\widehat{N}}{\widehat{M}}, \renormalizedR{\widehat{M}}{N}$ and $\renormalizedR{N}{\widehat{M}}$ are injective. 
\end{proposition}

\begin{proof}
  By Proposition \ref{prop:affinermatrix} (1), there is a polynomial $f(z,w)$ such that $\universalR{\widehat{M}}{\widehat{N}} = f(z,w) \renormalizedR{\widehat{M}}{\widehat{N}}$. 
  Lemma \ref{lem:affinization} shows that $\universalR{\widehat{M}}{\widehat{N}}$ is injective, which implies the assertion for $\renormalizedR{\widehat{M}}{\widehat{N}}$. 
  A similar proof applies to other R-matrices as well.
\end{proof}

\begin{corollary}\label{cor:lambdadiscrete}
  Let $M, N$ be as above. 
  Then, 
  \[ 
  \Lambda(M,N) + \Lambda(N,M) \in \mathbb{Z}_{\geq 0} \deg z. 
  \]
\end{corollary}

\begin{proof} 
 By Proposition \ref{prop:reninj}, the composition $\renormalizedR{N}{\widehat{M}} \renormalizedR{\widehat{M}}{N}$ is a nonzero element of $\End_{R(\beta+\gamma)[z]}(\widehat{M}\circ N)$. 
 On the other hand, Theorem \ref{thm:rmatrix} shows that $\End_{R(\beta+\gamma)[z]} (\widehat{M}\circ N) = \mathbf{k}[z]$.
 Therefore, $\Lambda(M,N) + \Lambda(N,M) = \deg (\renormalizedR{N}{\widehat{M}} \renormalizedR{\widehat{M}}{N})$ belongs to $\mathbb{Z}_{\geq 0}\deg z$. 
\end{proof}

\begin{proposition} [{\cite[Lemma 3.2.3, Proposition 3.2.17.]{MR3758148}}]\label{prop:commuting}
  Let $M$ be a real simple $R(\beta)$-module that admits an affinization $(\widehat{M},z)$, and $N$ a simple $R(\gamma)$-module. 
  \begin{enumerate}
    \item $M \circ N$ is simple if and only if $\Lambda(M,N) + \Lambda(N,M) = 0$.
    \item If $\Lambda(M,N) + \Lambda(N,M) = \deg z$, then $M\circ N$ is of length two with head $M \nabla N$ and socle $M \Delta N$.  
  \end{enumerate}
\end{proposition}

The statements are slightly refined from those in \cite{MR3758148}. 
We provide different proofs using the invariant $\Lambda$, following \cite[Corollary 3.18]{MR4359265}. 

\begin{proof}
(1) If $M \circ N$ is simple, then $N \circ M$ is also simple since it is isomorphic to $D(M\circ N)$ up to degree shift. 
Hence, both $\mathbf{r}_{M,N}$ and $\mathbf{r}_{N,M}$ must be isomorphic, which implies $\mathbf{r}_{N,M} \circ \mathbf{r}_{M,N} = \id_{M\circ N}$ up to a nonzero scalar multiple. 
It follows that $\Lambda(M,N) + \Lambda(N,M) = 0$. 

Next, assume $\Lambda (M,N) + \Lambda(N,M) = 0$. 
Suppose $M \circ N$ is not simple. 
By Lemma \ref{thm:rmatrix} (6) (7), $M \nabla N$ and $M \Delta N$ are not isomorphic and  
we have $\Lambda (M, M \Delta N) \leq \Lambda(M, N) - \deg z, \Lambda (M \Delta N, M) = \Lambda (N, M)$. 
Hence, $\Lambda (M, M \Delta N) + \Lambda (M \Delta N, M) \leq - \deg z < 0$, which is impossible by Corollary \ref{cor:lambdadiscrete}. 

(2) If $\Lambda(M,N) + \Lambda(N,M) = \deg z$, $M\circ N$ is not simple by (1). 
It is enough to show that $\rad (M\circ N) / M \Delta N = 0$. 
Suppose to the contrary $\rad (M \circ N)/M \Delta N \neq 0$ and take its composition factor $L$. 
By Lemma \ref{thm:rmatrix} (6) (7), we have $\Lambda(M,L) \leq \Lambda (M,N) - \deg z, \Lambda(L,M) \leq \Lambda(N,M) - \deg z$. 
Hence $\Lambda(M,L) + \Lambda(L,M) \leq - \deg z < 0$, which is impossible by Corollary \ref{cor:lambdadiscrete}. 
\end{proof}

\subsection{Unmixing sequences}

\begin{definition} [{\cite[Definition 2.5]{MR3542489}}]
  Let $m$ be a positive integer. 
  Let $\beta_k \in Q_+$ and $M_k \in \gMod{R(\beta_k)}$ for each $1 \leq k \leq m$. 
  We say $(M_1, M_2, \ldots, M_m)$ is unmixing if 
  \[
  \Res_{\beta_1, \beta_2, \ldots, \beta_m} (M_1 \circ M_2 \circ \cdots \circ M_m) = M_1 \otimes M_2 \otimes \cdots \otimes M_m. 
  \]
\end{definition}

Note that $(M_1, \ldots, M_m)$ is unmixing if and only if for any composition factor $L_k$ of $M_k \ (1 \leq k \leq m)$, $(L_1, \ldots, L_m)$ is unmixing. 
The following lemma is proved using the Mackey filtration. 

\begin{lemma} [{\cite[Lemma 2.6]{MR3542489}, \cite[Proposition 2.10]{MR3771147}}] \label{lem:unmixing}
  Assume $(M_1, \ldots, M_m)$ is unmixing. 
\begin{enumerate}
  \item For any $w \in \Sym_m$, we have 
  \[
  \Res_{\beta_1, \ldots, \beta_m} (M_{w(1)} \circ \cdots \circ M_{w(m)}) \simeq q^{t_w} M_1 \otimes \cdots \otimes M_m, 
  \]
  where $t_w = -\sum_{1 \leq a<b \leq m, w^{-1}(a)>w^{-1}(b)} (\beta_a, \beta_b)$. 
 \item 
 Every proper submodule of $M_1 \circ \cdots \circ M_m$ is annihilated by $e(\beta_1, \ldots, \beta_m)$. 
 Hence, the head $L$ of $M_1 \circ \cdots \circ M_m$ is simple and $\Res_{\beta_1, \ldots, \beta_m}L = M_1 \otimes \cdots \otimes M_m$. 
 Moreover, $[M_1 \circ \cdots \circ M_m : L]_q = 1$. 
 \item In (2), if $M_k$ are all self-dual $(1 \leq k \leq m)$, then $L$ is also self-dual. 
\end{enumerate}
\end{lemma}

Here, (3) follows from 
\[   
\Res_{\beta_1, \ldots, \beta_m} DL \simeq D (\Res_{\beta_1, \ldots,\beta_m}L) \simeq D (M_1 \otimes \cdots \otimes M_m) \simeq M_1 \otimes \cdots \otimes M_m.
\]

Now, we consider the case $m = 2$. 
For an $R(\beta)$-module $M$, we define
\begin{align*}
  \W(M) = \{ \gamma \in Q_+ \cap (\beta - Q_+) \mid e(\gamma, \beta-\gamma)M \neq 0 \}, \\ \index{$\W(M)$}
  \W^* (M) = \{ \gamma \in Q_+ \cap (\beta - Q_+) \mid e(\beta - \gamma, \gamma)M \neq 0 \}. \index{$\W^*(M)$}
\end{align*} 

Let $\beta, \gamma \in Q_+$ and $M \in \gMod{R(\beta)}, N \in \gMod{R(\gamma)}$. 
By the Mackey filtration, we have 
\[
\W(M\circ N) = \W(M) + \W(N), \ \W^*(M \circ N) = \W^*(M) + \W^*(N),
\]
see \cite[Lemma 2.2]{MR3771147}. 
The Mackey filtration also shows that $(M,N)$ is unmixing if and only if $\W^*(M) \cap \W(N) = \{0\}$. 

Assume that $(M,N)$ is unmixing. 
Then, Lemma \ref{lem:unmixing} (1) implies that $\Res_{\beta,\gamma} N\circ M \simeq q^{-(\beta, \gamma)} M \otimes N$. 
By the induction-restriction adjunction, we obtain a nonzero homomorphism $\Rmatrix{M}{N} \colon q^{-(\beta,\gamma)}M \circ N \to  N \circ M$ such that 
\[
\Rmatrix{M}{N} (u\otimes v) = \tau_{w[\height \gamma,\height \beta]} (v \otimes u) \ (u \in M, v \in N). \index{$\Rmatrix{M}{N}$}
\]


\begin{lemma} \label{lem:unmixingr}
Let $M$ be a simple $R(\beta)$-module and $N$ a simple $R(\gamma)$-module. 
Assume that one of them is real and admits an affinization. 
If $(M,N)$ is unmixing, we have 
\[
 \Rmatrix{M}{N} \in \mathbf{k}^{\times}\mathbf{r}_{M,N},\ \Lambda(M,N) = -(\beta,\gamma). 
\] 
\end{lemma}

\begin{proof}
Since $\Rmatrix{M}{N}$ is a nonzero homomorphism of degree $-(\beta,\gamma)$, the assertion follows from Theorem \ref{thm:rmatrix} (1). 
\end{proof}

\begin{lemma}\cite[Corollary 10.1.4]{MR3758148} \label{lem:LambdaforLi}
 Let $i \in I$, $\beta \in Q_+$ and $M$ a simple $R(\beta)$-module. 
 Then,
 \[
 \Lambda(L(i),M) = \varepsilon_i(M) (\alpha_i,\alpha_i) - (\alpha_i,\beta),\ \Lambda(M, L(i)) = \varepsilon_i^* (M(\alpha_i, \alpha_i)) - (\alpha_i,\beta). 
 \]
\end{lemma}

\begin{proof}
 Note that $L(i)$ admits an affinization $\widehat{L(i)} = R(\alpha_i)$ as in Example \ref{ex:affinization} and is real. 
 Hence, Definition \ref{def:renormalizedR} and Theorem \ref{thm:rmatrix} are applicable. 
 We prove the first assertion. 
 The second one is proved in the same manner. 
 We proceed by induction on $n = \varepsilon_i(M)$. 
 If $n = 0$, then $(L(i),M)$ is unmixing, so Lemma \ref{lem:unmixingr} implies the assertion. 
 If $n \geq 1$, we have  
 \[
 M \simeq \tilde{f}_i \tilde{e}_i M \simeq L(i) \nabla (\tilde{e}_i M) \ \text{up to degree shift}.  
 \] 
 Hence, we obtain
 \begin{align*}
 \Lambda(L(i),M) &= \Lambda(L(i), \tilde{e}_i M) \quad \text{by Theorem \ref{thm:rmatrix} (6)} \\
 &= (n-1) (\alpha_i,\alpha_i) - (\alpha_i, \beta-\alpha_i) \quad \text{by the induction hypothesis} \\
 &= n(\alpha_i, \alpha_i) - (\alpha_i, \beta). 
 \end{align*}
\end{proof}

\begin{proposition} \label{prop:unmixingren}
  Let $(M,N)$ be an unmixing pair of simple modules. 
  Let $(\widehat{M},z)$ be an affine object of $M$, $(\widehat{N},w)$ an affine object of $N$. 
\begin{enumerate}
  \item The pairs $(\widehat{M},N), (M, \widehat{N})$ and $(\widehat{M},\widehat{N})$ are unmixing. 
  \item If $M$ is real and $(\widehat{M},z)$ is an affinization, we have
  \[
 \Rmatrix{\widehat{M}}{N} = \renormalizedR{\widehat{M}}{N} , \  \Rmatrix{\widehat{M}}{\widehat{N}} = \renormalizedR{\widehat{M}}{\widehat{N}}
 \]
 up to nonzero scalar multiples. 
 \item If $N$ is real and $(\widehat{N},w)$ is an affinization, we have
\[
\Rmatrix{M}{\widehat{N}} = \renormalizedR{M}{\widehat{N}} , \  \Rmatrix{\widehat{M}}{\widehat{N}} = \renormalizedR{\widehat{M}}{\widehat{N}}
\]
up to nonzero scalar multiples. 
 \end{enumerate}
\end{proposition}

\begin{proof}
(1) 
By the same argument as in Lemma \ref{lem:limext}, we have $\widehat{N} \simeq \varprojlim_{p} \widehat{N}/z^p \widehat{N}$ in $\gMod{R}$. 
Since $z^p\widehat{N}/z^{p+1}\widehat{N} \simeq N$, we deduce that $\W(\widehat{N}) = \W(N)$. 
Similarly, $\W^*(\widehat{M}) = \W^*(M)$. 
Hence, $(\widehat{M},N)$ and $(\widehat{M},\widehat{N})$ are unmixing. 

(2)
By the construction, $\Rmatrix{\widehat{M}}{N}$ is a nonzero element of $\End_{R(\beta+\gamma)[z]} (\widehat{M}\circ N, N \circ \widehat{M})$ with degree $-(\beta,\gamma)$. 
On the other hand, Theorem \ref{thm:rmatrix} shows that $\End_{R(\beta+\gamma)[z]} (\widehat{M}\circ N, N \circ \widehat{M}) = \mathbf{k}[z] \renormalizedR{\widehat{M}}{N}$. 
Besides, $\deg (\renormalizedR{\widehat{M}}{N}) = \Lambda(M,N)$, which is $-(\beta,\gamma)$ by Lemma \ref{lem:unmixingr}. 
It proves $\Rmatrix{\widehat{M}}{N} = \renormalizedR{\widehat{M}}{N}$ up to a nonzero scalar multiple. 
The assertion for $\Rmatrix{\widehat{M}}{\widehat{N}}$ and (3) are proved in the same manner. 
\end{proof}
 
\begin{proposition} \label{prop:unmixingreninj}
  Let $M \in \gMod{R(\beta)}$ be a real simple module that admits an affinization $(\widehat{M},z)$. 
  Let $N \in \gMod{R(\gamma)}$ and assume that $(M,N)$ is unmixing. 
  Then, $\Rmatrix{\widehat{M}}{N}$ is injective.
\end{proposition}

\begin{proof}
  Note that $\Rmatrix{\widehat{M}}{N}$ is natural in $N$. 
  Since $N$ is a projective limit of finite dimensional $R(\gamma)$-modules, we may assume that $N$ is finite dimensional. 
  By induction on the length of $N$, we may reduce to the case where $N$ is simple. 
  In this case, $\Rmatrix{\widehat{M}}{N} \in \mathbf{k}^{\times} \renormalizedR{\widehat{M}}{N}$ by Proposition \ref{prop:unmixingren}, 
  and the assertion follows from Proposition \ref{prop:reninj}. 
\end{proof}

\subsection{Cuspidal decomposition} \label{sub:cuspidal}

We review cuspidal decompositions of simple modules developed in \cite{MR3542489}.
We also consult \cite{MR3771147}. 
Let $\prroot$ be the set of positive real roots, and 
\[
\minroot = \{ \alpha \in \Phi_+ \mid \text{for any $0 < a < 1$, $a\alpha \not\in \Phi_+$}\}.  \index{$\prroot$} \index{$\minroot$}
\]
It is known that, if $\alpha \in \minroot$, then $\mathbb{Q}\alpha \cap \Phi \subset \mathbb{Z} \alpha$ \cite[Proposition 5.5]{MR1104219}.

Let $\preceq$ be a total preorder on a set $X$. 
We define an equivalence relation $\sim$ on $X$ by $x \sim y \Leftrightarrow x \preceq y$ and $y \preceq x$. 
We call it the $\preceq$-equivalence relation. 
Then, $\preceq$ induces a total order on the set of $\preceq$-equivalence classes $X/{\sim}$, which is denoted by the  same symbol $\preceq$.

\begin{definition}[{\cite[Definition 1.8]{MR3542489}, \cite[Definition 1.16]{MR3771147}}] \label{def:convexpreorder}
  Let $V$ be an $\mathbb{R}$-vector space, and $X$ be a subset of $V$. 
  A total preorder $\preceq$ on $X$ is said to be convex if, for any $\preceq$-equivalence class $C$, the following two conditions are satisfied: 
 \begin{enumerate}
  \item $(\spn_{\mathbb{R}_{\geq 0}} \{ x \in X \mid x \prec C\} + \spn_{\mathbb{R}} C) \cap \spn_{\mathbb{R}_{\geq 0}} \{ x \in X \mid x \succ C \} = \{0\}$,
  \item $(\spn_{\mathbb{R}_{\geq 0}} \{x \in X \mid x \succ C \} + \spn_{\mathbb{R}} C) \cap \spn_{\mathbb{R}_{\geq 0}} \{ x \in X \mid x \prec C \} = \{0\}$. 
 \end{enumerate}
\end{definition}

Let $\preceq$ be a convex preorder on $\minroot$. 
It uniquely extends to a convex preorder on $\mathbb{Z}_{>0} \minroot = \{ a\alpha \mid a \in \mathbb{Z}_{>0}, \alpha \in \minroot\}$: \index{$\mathbb{Z}_{> 0}\minroot$}
$a\alpha \preceq a'\alpha'$ if $\alpha \preceq \alpha'$. 
Note that $\Phi_+$ is a subset of $\mathbb{Z}_{> 0} \minroot$ by \cite[Proposition 5.5]{MR1104219}. 

 \begin{definition} \label{def:cuspidal}
  Let $C \in \minroot/{\sim}$ and $\beta \in \spn_{\mathbb{Z}_{\geq 0}} C$. (In this case, we write $C = C_{\beta}$\index{$C_{\beta}$}.) 
  Set $\Phi_{+,\preceq C}^{\text{min}} = \{ \gamma \in \minroot \mid C_{\gamma} \preceq C \}, \Phi_{+,\succeq C}^{\text{min}} = \{ \gamma \in \minroot \mid C_{\gamma} \succeq C \}$. 
  A simple $R(\beta)$-module $L$ is said to be $\preceq$-cuspidal if 
  \[
  \W(L) \subset \spn_{\mathbb{Z}_{\geq 0}} \Phi_{+,\preceq C}^{\text{min}}, \ \W^*(L) \subset \spn_{\mathbb{Z}_{\geq 0}} \Phi_{+, \succeq C}^{\text{min}}.
  \]
  If the convex preorder is clear from the context, we just say cuspidal. 
 \end{definition}

 Cuspidal modules are called semi-cuspidal modules in \cite{MR3542489}. 

 Let $C \in \Phi_+/{\sim}$ and $\beta \in \spn_{\mathbb{Z}_{\geq 0}} C$. 
 Let $I^{\preceq}(\beta)$\index{$I^{\preceq}(\beta)$} be the two-sided ideal of $R(\beta)$ generated by 
 \[
 e(\gamma,*) \ (\gamma \not \in \spn_{\mathbb{Z}_{\geq 0}} \Phi_{+,\preceq C}^{\text{min}}),\ e(*, \gamma') \ (\gamma' \not \in \spn_{\mathbb{Z}_{\geq 0}} \Phi_{+, \succeq C}^{\text{min}}). 
 \]
 We define $R^{\preceq}(\beta)\index{$R^{\preceq}(\beta)$} = R(\beta)/I^{\preceq}(\beta)$. 
 Then, $\gMod{R^{\preceq}(\beta)}$ is isomorphic to
 \[
 \{ M \in \gMod{R(\beta)} \mid \text{every composition factor of $M$ is $\preceq$-cuspidal} \}.
 \] 

\begin{lemma}[{\cite[Lemma 2.18]{MR3542489}}] \label{lem:cuspidalseq}
  Let $C_1 \prec C_2 \prec \cdots \prec C_m \in \minroot/{\sim}$, 
  and $\beta_k \in \spn_{\mathbb{Z}_{\geq 0}} C_k$. 
  For each $1 \leq k \leq m$, let $L_k$ be a cuspidal $R(\beta_k)$-module.  
  Then, $(L_m, \ldots, L_1)$ is unmixing. 
\end{lemma}
 
 \begin{theorem}[{\cite[Corollary 2.17, Theorem 2.19]{MR3542489}}] \label{thm:cuspidaldecomp}
  \begin{enumerate}
   \item Let $L$ be a simple $R(\beta)$-module. 
   Then, there exist $m \in \mathbb{Z}_{\geq 0}, C_1 \prec C_2 \prec \cdots \prec C_m \in \minroot/{\sim}, 0 \neq \beta_k \in \spn_{\mathbb{Z}_{\geq 0}} C_k \ (1 \leq k \leq m)$, 
   and a cuspidal $R(\beta_k)$-module $L_k$ $(1 \leq k \leq m)$ such that $L \simeq \hd (L_m \circ \cdots \circ L_1)$. 
   $L_k \ (1 \leq k \leq m)$ are unique up to isomorphism and degree shift. 
   Moreover, if $L$ is self-dual, we may choose each $L_k$ to be self-dual. 
   \item Let $C \in \minroot/{\sim}$. 
   Then we have \begin{align*}
   &\sum_{\beta \in \spn_{\mathbb{Z}_{\geq 0}}C}  \lvert \{ \text{self-dual cuspidal $R(\beta)$-module} \}/{\simeq} \rvert e^{\beta} = \prod_{\beta \in \Phi_+ \cap \spn_{\mathbb{Z}_{\geq 0}}C} \frac{1}{(1-e^{\beta})^{\dim \mathfrak{g}_{\beta}^+}}. 
   \end{align*}
 \end{enumerate}
 \end{theorem}

 We call the expression $L \simeq \hd (L_m \circ \cdots \circ L_1)$ in (1) the cuspidal decomposition of $L$. 
 Note that $[L_m \circ \cdots \circ L_1 : L]_q = 1$ by Lemma \ref{lem:unmixing} (2). 

 \begin{remark}
 The second assertion is a correction of \cite[Corollary 2.17]{MR3542489}, \cite[Proposition 2.6]{MR3771147}. 
 In their proof, the computation of $\dim U(\mathfrak{g})_{n\alpha}$ is incorrect.
 It should be derived from the generating function
 \[
 \sum_{\beta \in Q_+} \dim U(\mathfrak{g})_{\beta} e^{\beta} = \prod_{\beta \in \Phi_+} \frac{1}{(1-e^{\beta})^{\dim \mathfrak{g}_{\beta}}}. 
 \]
 \end{remark}

\begin{corollary} \label{cor:cuspidaldecomp}
  Let $L$ be a simple $R(\beta)$-module and take the cuspidal decomposition $L \simeq \hd (L_m \circ \cdots \circ L_1)$ as in Theorem \ref{thm:cuspidaldecomp} (1). 
  Let $C$ be a $\preceq$-equivalence class.
  \begin{enumerate}
    \item $\W(L) \subset \spn_{\mathbb{Z}_{\geq 0}} \Phi_{+, \preceq C}^{\mathrm{min}}$ if and only if $C_m \preceq C$. 
    \item $\W^*(L) \subset \spn_{\mathbb{Z}_{\geq 0}} \Phi_{+, \succeq C}^{\mathrm{min}}$ if and only if $C_1 \succeq C$.  
  \end{enumerate}
\end{corollary}

\begin{proof}
(1) First, assume $\W(L) \subset \spn_{\mathbb{Z}_{\geq 0}} \Phi_{+, \preceq C}^{\text{min}}$. 
By Lemma \ref{lem:cuspidalseq}, $\Res_{\beta_m, \ldots, \beta_1} L \simeq L_m \otimes \cdots \otimes L_1$. 
In particular, $\beta_m \in \W(L)$, hence $\beta_m \in \spn_{\mathbb{Z}_{\geq 0}} \Phi_{+, \preceq C}^{\mathrm{min}}$. 
By convexity, $C_m \preceq C$. 

Next, assume $C_m \preceq C$. 
Then, we have 
\[
\W(L) \subset \W(L_m \circ \cdots \circ L_1) = \W(L_m) + \cdots + \W(L_1)
\]
by the Mackey filtration. 
Since $L_k$ is cuspidal, $\W(L_k) \subset \spn_{\mathbb{Z}_{\geq 0}} \Phi_{+, \preceq C_k}^{\text{min}} \subset \spn_{\mathbb{Z}_{\geq 0}} \Phi_{+, \preceq C}^{\text{mint}}$. 
The assertion follows. 

(2) is similar to (1). 
\end{proof}

 Let $\beta \in Q_+$. 
 We define 
 \begin{align*}
 \mathcal{P}(\beta) &= \{ (\beta_1, \ldots, \beta_n) \mid n \geq 1, \beta_k \in Q_+ \ (1 \leq k \leq n), \sum_k \beta_k = \beta \}, \\ \index{$\mathcal{P}(\beta)$}
 \mathcal{P}^{\preceq} (\beta) &= \left \{ (\beta_1, \ldots, \beta_n) \in \mathcal{P}(\beta) \mathrel{}\middle|\mathrel{} \begin{aligned}  \index{$\mathcal{P}^{\preceq}(\beta)$}
    &\text{there exists $C_1 \prec \cdots \prec C_n \in \minroot/{\sim}$} \\ 
    &\text{such that for each $1 \leq k \leq n$,} \\
    &\text{$0 \neq \beta_k \in \spn_{\mathbb{Z}_{\geq 0}}C_k$.}
 \end{aligned}
    \right \}. 
\end{align*}

 \begin{definition} \label{def:bilexicographic}
   For $\underline{\beta} = (\beta_1, \ldots, \beta_m) \in \mathcal{P}^{\preceq}(\beta), \underline{\beta'} = (\beta'_1, \ldots, \beta'_n) \in \mathcal{P}(\beta)$, 
   we write $\underline{\beta'} < \underline{\beta}$\index{$\underline{\beta'} < \underline{\beta}$} if the following two conditions hold. 
 \begin{enumerate}
 \item there exists $1 \leq k \leq \min \{ m, n\}$ such that $\beta_1 = \beta'_1 , \ldots, \beta_{k-1} = \beta'_{k-1}$ and either (i) $\beta'_k \in (\spn_{\mathbb{Z}_{\geq 0}} \Phi_{+,\succeq C_{\beta_k}}^{\text{min}}) \setminus (\spn_{\mathbb{Z}_{\geq 0}} C_{\beta_k})$ or (ii) $\beta'_k \in \spn_{\mathbb{Z}_{\geq 0}}C_{\beta_k}, \beta_k - \beta'_k \in Q_+ \setminus \{0\}$ holds. 
 \item there exists $1 \leq k \leq \min \{ m, n\}$ such that $\beta_m = \beta'_n , \ldots, \beta_{m-k+1} = \beta'_{n-k+1}$ and either (i) $\beta'_{n-k} \in (\spn_{\mathbb{Z}_{\geq 0}} \Phi_{+, \preceq C_{\beta_{m-k}}}^{\text{min}}) \setminus (\spn_{\mathbb{Z}_{\geq 0}} C_{\beta_m-k})$ or (ii) $\beta'_{n-k} \in \spn_{\mathbb{Z}_{\geq 0}} C_{\beta_{m-k}}, \beta_{m-k} - \beta'_{n-k} \in Q_+ \setminus \{0 \}$ holds.
 \end{enumerate}
 It defines a partial order on $\mathcal{P}^{\preceq}(\beta)$. 
 \end{definition}
 
 Let $L$ be a simple $R(\beta)$-module and take its cuspidal decomposition $L = \hd (L_m \circ \cdots \circ L_1)$, where $L_k$ is a cuspidal $R(\beta_k)$-module. 
 We write $\underline{\beta}_L\index{$\underline{\beta}_L$} = (\beta_1, \ldots, \beta_m) \in \mathcal{P}^{\preceq}(\beta)$. 
   
 \begin{proposition} \label{prop:properstandardcomposition}
  Let $\beta \in Q_+$.
  Let $L$ be a simple $R(\beta)$-module and take its cuspidal decomposition $L \simeq \hd (L_m \circ \cdots \circ L_1)$, where $L_k$ is a cuspidal $R(\beta_k)$-module. 
  \begin{enumerate}
   \item Let $\underline{\beta'} = (\beta'_1, \ldots, \beta'_n) \in \mathcal{P}(\beta)$. 
   If $\Res_{\beta'_n, \ldots, \beta'_1} (L_m \circ \cdots \circ L_1) \neq 0$, then $\underline{\beta'} \leq \underline{\beta}_L$. 
   \item For any composition factor $L'$ of $L_m \circ \cdots \circ L_1$ other than $\hd (L_m \circ \cdots \circ L_1)$, we have $\underline{\beta}_{L'} < \underline{\beta}_L$. 
  \end{enumerate}
 \end{proposition}

 \begin{proof}
 This is a generalization of \cite[Proposition 2.15]{MR3771147}. 

 (1) Assume that $\Res_{\beta'_n, \ldots, \beta'_1} (L_m \circ \cdots \circ L_1) \neq 0$. 

 Then, by the Mackey filtration, there exists $\gamma_{k,l} \in Q_+ \ (1\leq k \leq m, 1 \leq l \leq n)$ such that 
 \begin{align*}
 &\sum_k \gamma_{k,l} = \beta'_l \ (1 \leq l \leq n), \ \sum_l \gamma_{k,l} = \beta_k \ (1 \leq k \leq m),\\ 
 &\Res_{\gamma_{k,n}, \ldots, \gamma_{k,1}} L_k \neq 0 \ (1 \leq k \leq m). 
 \end{align*}

 We may assume that $\underline{\beta} \neq \underline{\beta'}$. 
 Take the least number $k_0$ such that $\beta_{k_0} \neq \beta'_{k_0}$. 

 First, we claim that, for $1 \leq k < k_0$, 
 \[
 \gamma_{k,k} = \beta_k, \gamma_{k,l} = 0 \ (k < l \leq n), \gamma_{l,k} = 0 \ (k < l \leq m). 
 \]
 We prove it by induction on $k$. 
 Assume that the claim is proved for smaller numbers. 
 In particular, $\gamma_{k',k} = 0$ for $1 \leq k' <k$, hence $\beta_k = \beta'_k \sum_{k' = k}^m \gamma_{k',k}$. 
 We also have $\gamma_{k',l} = 0$ for $k \leq k' \leq m, 1 \leq l < k$. 
 Since $L_{k'}$ is cuspidal, we deduce that $\gamma_{k',k} \in \spn_{\mathbb{Z}_{\geq 0}} \Phi_{+, \succeq C_{\beta_{k'}}}^{\text{min}} \ (k \leq k' \leq m)$, 
 which is a subset of $\spn_{\mathbb{Z}_{\geq 0}} \Phi_{+, \succ C_{\beta_k}}^{\text{min}}$ if $k' > k$. 
 By the convexity of $\preceq$, we must have $\gamma_{k,k} = \beta_k$ and $\gamma_{k',k} = 0 \ (k < k' \leq m)$. 
 Since $\sum_{1 \leq l \leq n} \gamma_{k,l} = \beta_k$, we obtain $\gamma_{k,l} = 0 \ (l \neq k)$. 
 Therefore, the induction proceeds.  

 By a similar argument, we have $\gamma_{k',k_0} \in \spn_{\mathbb{Z}_{\geq 0}} \Phi_{+, \succeq C_{\beta_{k'}}}^{\text{min}}$ for $k_0 \leq k' \leq m$, 
 which is a subset of $\spn_{\mathbb{Z}_{\geq 0}} \Phi_{+, \succ C_{\beta_{k_0}}}^{\text{min}}$ if $k' > k_0$.
 If $\gamma_{k',k_0} \neq 0$ for some $k' > k_0$, we have 
 \[
 \beta'_{k_0} = \sum_{k' = k_0}^m \gamma_{k', k_0} \in (\spn_{\mathbb{Z}_{\geq 0}} \Phi_{+, \succeq C_{\beta_{k_0}}}^{\text{min}}) \setminus (\spn_{\mathbb{Z}_{\geq 0}} C_{\beta_{k_0}}). 
 \]
 Otherwise, we have $\beta'_{k_0} = \gamma_{k_0,k_0} \in \W^*(L_{k_0})$, hence 
 \[
 \beta'_{k_0} \in \spn_{\mathbb{Z}_{\geq 0}} \Phi_{+, \succeq C_{\beta_{k_0}}}^{\text{min}},\ \beta_{k_0} - \beta'_{k_0} \in Q_+. 
 \]
 In any case, the condition (1) of Definition \ref{def:bilexicographic} is satisfied. 

 A similar argument shows that the condition (2) is also satisfied. 
 Therefore, $\underline{\beta'} < \underline{\beta}$. 

 (2) Put $\underline{\beta}_{L'} = (\beta'_1, \ldots, \beta'_n)$. 
 Since $\Res_{\beta'_n, \ldots, \beta'_1} L' \neq 0$ by Lemma \ref{lem:unmixing}, $\Res_{\beta'_n, \ldots, \beta'_1} (L_m \circ \cdots \circ L_1) \neq 0$. 
 Hence, (1) shows $\underline{\beta}_{L'} \leq \underline{\beta}_L$.
 The strict inequality follows from Lemma \ref{lem:unmixing} (2).   
 \end{proof}
 
 We discuss the cuspidal decompositions in the case of $\preceq$ being an order, not just a preorder. 
 Theorem \ref{thm:cuspidaldecomp} (2) says, for each $\alpha \in \minroot$,
  \[
   \sum_{n \geq 0} \lvert \{ \text{self-dual cuspidal $R(n\alpha)$-module} \}/{\simeq} \rvert t^n = \prod_{n \geq 1} \frac{1}{(1-t^n)^{\dim \mathfrak{g}_{n\alpha}}}. 
 \]
 In particular, if $\alpha$ is real, there is exactly one self-dual cuspidal $R(n\alpha)$-module for each $n \geq 0$ up to isomorphism.
 Let $L(n\alpha) = L^{\preceq}(n\alpha)$ denote this unique self-dual cuspidal $R(n\alpha)$-module. \index{$L(n\alpha) = L^{\preceq}(n\alpha)$}
 Let $q_{\alpha} = q^{(\alpha,\alpha)/2}$. \index{$q_{\alpha}$}

 \begin{proposition}[{\cite[Proposition 2.21]{MR3542489}}] \label{prop:realcuspidal}
   Assume that $\preceq$ is a total order. 
   Then, we have 
   \[
   L(n\alpha) \simeq q_{\alpha}^{n(n-1)/2} L(\alpha)^{\circ n}. 
   \]
 \end{proposition}
 
 Here, the degree shift is determined by 
 \[
 D(L(\alpha)^{\circ n}) \simeq q^{(\alpha,\alpha)n(n-1)/2} D(L(\alpha))^{\circ n} \simeq q^{(\alpha,\alpha)n(n-1)/2} L(\alpha)^{\circ n}. 
 \]
 
 We have a bijection between $\mathcal{P}^{\preceq}(\beta)$ and the set
 \[
   \left\{ f \colon \minroot \to \mathbb{Z}_{\geq 0} \middle| \begin{aligned} 
   &\text{$f(\alpha) = 0$ for all but finitely many $\alpha \in \minroot$}, \\
   &\sum_{\alpha \in \minroot} f(\alpha)\alpha = \beta 
   \end{aligned} \right\}. 
 \]
 The correspondence is given as follows. 
 Let $(\beta_1, \ldots,\beta_m) \in \mathcal{P}^{\preceq} (\beta)$. 
 For each $1 \leq k \leq m$, there uniquely exist $\alpha_k \in \minroot$ and $a_k \in \mathbb{Z}_{>0}$ such that $\beta_k = a_k\alpha_k$. 
 We define $f \colon \minroot \to \mathbb{Z}_{\geq 0}$ by 
 \[
 f(\alpha) = \begin{cases}
  a_k & \text{if $\alpha = \alpha_k \ (1 \leq k \leq m)$}, \\
  0 & \text{otherwise}. 
 \end{cases}
 \]
 
 It is convenient to identify these sets. 
 In terms of $f$, the partial order $\leq$ on $\mathcal{P}^{\preceq}(\beta)$ from Definition \ref{def:bilexicographic} is the bilexicographic order, that is, the lexicographic order where we read the values from both sides. 

Now, we drop the assumption that $\preceq$ is an order. 
It is known that we can always refine the convex preorder $\preceq$ into a convex order by choosing a convex order on each $\preceq$-equivalence class \cite[Proposition 1.21]{MR3771147}.
Let $\preceq'$ be such a refinement. 
Then, $\preceq$-cuspidal decomposition and $\preceq'$-cuspidal decomposition are related as follows. 

 Let $L$ be a simple $R(\beta)$-module. 
 Let $\sim$ be the $\preceq$-equivalence relation on $\minroot$. 
 Take the $\preceq'$-cuspidal decomposition $L \simeq \hd(L_m \circ \cdots \circ L_1)$, where $L_k$ is a $\preceq'$-cuspidal $R(a_k \alpha_k)$-module ($a_k \in \mathbb{Z}_{>0}, \alpha_k \in \minroot$). 
 There is a unique sequence $1 \leq m_1 < \cdots < m_r = m$ such that 
 \[
 \alpha_1 \sim \cdots \sim \alpha_{m_1} \prec \alpha_{m_1+1} \sim \cdots \sim \alpha_{m_2} \prec \cdots \prec \alpha_{m_{r-1}+1} \sim \cdots \sim \alpha_m. 
 \]
 We set $m_0 = 0$.
 
 \begin{proposition} \label{prop:refinedcuspidal}
  In this setup, for each $1 \leq s \leq r$, the simple module $M_s = \hd(L_{m_s} \circ \cdots \circ L_{m_{s-1}+1})$ is $\preceq$-cuspidal. 
  Moreover, $L \simeq \hd (M_r \circ \cdots \circ M_1)$ is the $\preceq$-cuspidal decomposition of $L$. 
 \end{proposition}

 \begin{proof}
 Let $1 \leq s \leq r$. 
 Let $C_s$ be the $\preceq$-equivalence class that contains $\alpha_{m_{s-1}+1}, \ldots, \alpha_{m_s}$. 
 Note that 
 \[   
 \W(M_s) \subset \W(L_{m_s} \circ \cdots \circ L_{m_{s-1}+1}) = \W(L_{m_s}) + \cdots + \W(L_{m_{s-1}+1}). 
 \] 
 For $m_{s-1}+1 \leq k \leq m_s$, since $L_k$ is $\preceq'$-cuspidal, $\W(L_k) \subset \spn_{\mathbb{Z}_{\geq 0}} \Phi_{+, \preceq' \alpha_k}^{\text{min}} \subset \spn_{\mathbb{Z}_{\geq 0}} \Phi_{+, \preceq C_s}^{\text{min}}$. 
 Hence, $\W(M_s) \subset \spn_{\mathbb{Z}_{\geq 0}} \Phi_{+, \preceq C_s}^{\text{min}}$. 
 Similarly, we have $\W^*(M_s) \subset \spn_{\mathbb{Z}_{\geq 0}} \Phi_{+, \succeq C_s}^{\mathrm{min}}$. 
 Hence, $M_s$ is $\preceq$-cuspidal.

 We have an obvious surjective homomorphism $L_m \circ \cdots \circ L_1 \to M_r \circ \cdots \circ M_1$, so $\hd(M_r \circ \cdots \circ M_1) \simeq L$. 
 It gives the $\preceq$-cuspidal decomposition of $L$. 
\end{proof}

This proposition leads us to a generalization of cuspidal decompositions. 

\begin{definition} \label{def:weakconvex}
 A preorder $\preceq$ on $\minroot$ is said to be weakly convex if it can be refined into a convex order. 
\end{definition}

As is explained above, every convex preorder is a weakly convex preorder. 
It would be interesting to explore an intrinsic characterization of weakly convex preorders. 

\begin{example} \label{ex:weakconvex}
 Consider the $A_2$ case. 
 Then, $\minroot = \Phi_+ = \{ \alpha_1, \alpha_2, \alpha_3 = \alpha_1 + \alpha_2 \}$. 
 If we define a preorder $\preceq$ by $\alpha_1 \prec \alpha_2 \sim \alpha_3$, then it is weakly convex, but not convex. 
 It can be refined into a convex order $\alpha_1 \prec \alpha_3 \prec \alpha_2$, so it is weakly convex. 
 On the other hand, we have $ \alpha_3 - \alpha_1 = \alpha_2 \neq 0$, so $\preceq$ is not convex.  
\end{example}

\begin{definition} \label{def:weakcuspidal}
Let $\preceq$ be a weakly convex preorder on $\minroot$. 
Let $C \in \minroot/{\sim}$ and $\beta \in \spn_{\mathbb{Z}_{\geq 0}} C$. 
Set $\Phi_{+,\preceq C}^{\text{min}} = \{ \gamma \in \minroot \mid C_{\gamma} \preceq C \}, \Phi_{+,\succeq C}^{\text{min}} = \{ \gamma \in \minroot \mid C_{\gamma} \succeq C \}$. 
A simple $R(\beta)$-module $L$ is said to be $\preceq$-cuspidal if 
\begin{align*}
\W(L) \subset \spn_{\mathbb{Z}_{\geq 0}} \Phi_{+,\preceq C}^{\text{min}}, \ \W^*(L) \subset \spn_{\mathbb{Z}_{\geq 0}} \Phi_{+, \succeq C}^{\text{min}}.
\end{align*}
If the weakly convex preorder is clear from the context, we just say cuspidal. 
\end{definition}

Let $\preceq$ be a weakly convex preorder on $\minroot$, $C \in \minroot/{\sim}$ and $\beta \in \spn_{\mathbb{Z}_{\geq 0}} C$. 
As in the case of a convex preorder, we define $I^{\preceq}(\beta)$ to be the two-sided ideal of $R(\beta)$ generated by 
\[
e(\gamma,*) \ (\gamma \not \in \spn_{\mathbb{Z}_{\geq 0}} \Phi_{+, \preceq C}^{\mathrm{min}}), \ e(*,\gamma') \ (\gamma' \not \in \spn_{\mathbb{Z}_{\geq 0}} \Phi_{+, \succeq C}^{\mathrm{min}}),
\]
and put $R^{\preceq}(\beta) = R(\beta) / I^{\preceq}(\beta)$. 

\begin{lemma} \label{lem:weakcuspidal}
Let $\preceq$ be a weakly convex preorder on $\minroot$ and take a refinement into a convex order $\preceq'$. 
Let $\beta \in Q_+$ and $L$ be a simple $R(\beta)$-module.
We use the same notation in the paragraph before Proposition \ref{prop:refinedcuspidal}. 
Then, $L$ is $\preceq$-cuspidal if and only if $r = 1$. 
\end{lemma}

\begin{proof}
If $r =1$, the argument in the first paragraph of the proof of Proposition \ref{prop:refinedcuspidal} shows that $L$ is $\preceq$-cuspidal. 

Next, assume that $L$ is $\preceq$-cuspidal. 
Let $C$ be the $\preceq$-equivalence class containing $\beta$. 
Since $L$ is $\preceq$-cuspidal, we have $\alpha_m \in \spn_{\mathbb{R}_{\geq 0}} \Phi_{+, \preceq C}^{\text{min}}$. 

We claim that $C_{\alpha_m} \preceq C$. 
Suppose that $C_{\alpha_m} \succ C$. 
Then, $\alpha_m \succ' \gamma$ for any $\gamma \in \Phi_{+, \preceq C}^{\mathrm{min}}$, hence $\alpha_m \in \spn_{\mathbb{Z}_{\geq 0}} \Phi_{+, \prec' \alpha_m}^{\mathrm{min}}$.   
It contradicts the convexity of $\preceq'$.
Hence, the claim follows.

Similarly, we have $C_{\alpha_1} \succeq C$. 
Therefore, $C \preceq C_{\alpha_1} \preceq C_{\alpha_m} \preceq C$, which implies $\alpha_1 \sim \alpha_2 \sim \cdots \sim \alpha_m$.
By the definition of $r$, we deduce $r = 1$. 
\end{proof}

\begin{proposition} \label{prop:weakcuspidaldecomp}
Let $\preceq$ be a weakly convex preorder on $\minroot$. 
  \begin{enumerate}
    \item Let $r \in \mathbb{Z}_{\geq 1}, C_1 \prec C_2 \prec \cdots \prec C_r \in \minroot/{\sim}, 0 \neq \beta_k \in \spn_{\mathbb{Z}_{\geq 0}} C_k \ (1 \leq k \leq r)$, 
    and $L_k$ be a cuspidal $R(\beta_k)$-module for each $1 \leq k \leq r$. 
    Then $(L_r, \ldots, L_1)$ is unmixing. 
    \item Let $L$ be a simple $R(\beta)$-module. 
    Then, there uniquely (up to degree shifts) exist $r \in \mathbb{Z}_{\geq 0}, C_1 \prec C_2 \prec \cdots \prec C_r \in \minroot/{\sim}, 0 \neq \beta_k \in \spn_{\mathbb{Z}_{\geq 0}} C_k \ (1 \leq k \leq r)$, 
    and a cuspidal $R(\beta_k)$-module $L_k$ $(1 \leq k \leq r)$ such that $L \simeq \hd (L_r \circ \cdots \circ L_1)$. 
    Moreover, if $L$ is self-dual, we may choose each $L_k$ to be self-dual. 
    \item Let $C \in \minroot/{\sim}$. 
    Then we have \begin{align*}
    &\sum_{\beta \in \spn_{\mathbb{Z}_{\geq 0}}C} \lvert \{ \text{self-dual cuspidal $R(\beta)$-module} \}/{\simeq} \rvert \ e^{\beta} = \prod_{\beta \in \Phi_+ \cap \spn_{\mathbb{Z}_{\geq 0}}C} \frac{1}{(1-e^{\beta})^{\dim \mathfrak{g}_{\beta}}}. 
    \end{align*}
    \item Proposition \ref{prop:refinedcuspidal} holds replacing the convex preorder by the weakly convex preorder $\preceq$.  
  \end{enumerate}
\end{proposition}

\begin{proof}
(1) Let $\preceq'$ be a convex order that refines $\preceq$. 
Lemma \ref{lem:weakcuspidal} shows that, for each $1 \leq k \leq r$, the $\preceq'$-cuspidal decomposition of $L_k$ is of the form $L_k \simeq \hd(L_{k,m_k} \circ \cdots \circ L_{k, 1})$ for some $\preceq'$-cuspidal $R(\alpha'_{k,s})$ module $L_{k,s} \ (1 \leq s \leq m_k)$, 
where $\alpha_{k,1} \prec' \cdots \prec' \alpha_{k,m_k}$. 
Since $C_1 \prec \cdots \prec C_r$, we have 
\[
\alpha_{1,1} \prec' \cdots \prec' \alpha_{1,m_1} \prec' \alpha_{2,1} \prec' \cdots \prec' \alpha_{r, m_r}. 
\]
Hence, $(L_{r,m_r}, \ldots, L_{r,1}, \ldots, L_{2,1}, L_{1,m_1}, \ldots, L_{1,1})$ is unmixing. 
Since $\W(L_k) \subset \W(L_{k,m_k}) + \cdots + \W(L_{k,1}), \W^*(L_k) \subset \W^*(L_{k,m_k}) + \cdots + \W^*(L_{k,1})$, the assertion follows. 

Using (1), the remaining assertion follow from Theorem \ref{thm:cuspidaldecomp} and Lemma \ref{lem:weakcuspidal}. 
\end{proof}

\begin{remark} \label{rem:failure}
For a weakly convex preorder $\preceq$, we may define $\mathcal{P}^{\preceq}(\beta)$ and a partial order on it as in the case of a convex preorder. 
 However, Proposition \ref{prop:properstandardcomposition} cannot be directly generalized for weakly convex preorders. 
 We give an example in type $A_2$. 
 Let $\preceq$ be the weakly convex preorder given by $\alpha_1 \preceq \alpha_2 \sim \alpha_3 = \alpha_1 + \alpha_2$. 
 Obviously, both $L(1^2)$ and $L(2)$ are $\preceq$-cuspidal. 

 We claim that $\Res_{\alpha_3, \alpha_1} L(2) \circ L(1^2) \neq 0$. 
 In fact, ignoring degree shifts, we have $L(2) \circ L(1^2) = L(2) \circ L(1) \circ L(1)$ so the claim follows from the Mackey filtration. 
 However, we have $(\alpha_1, \alpha_3) \not \leq (2\alpha_1, \alpha_2)$, since $\alpha_2-\alpha_3 \not \in Q_+$. 
 Hence, Proposition \ref{prop:properstandardcomposition} does not hold in this case. 
\end{remark}

For a $\preceq$-equivalence class $C$, we define $\Phi_{+, \preceq C}$ (resp. $\Phi_{+, \succeq C}$) to be $\Phi_+ \cap \{a\gamma \mid a\in \mathbb{Z}_{\geq 0}, \gamma \in \Phi_{+, \preceq C}^{\text{min}} (\text{resp. $\in \Phi_{+, \succeq C}^{\text{min}}$}) \}$. 

 \begin{corollary} \label{cor:cuspidalcriterion}
  Let $\preceq$ be a weakly convex preorder on $\minroot$. 
   Let $C \in \minroot/{\sim}, \beta \in \spn_{\mathbb{Z}_{\geq 0}} C$ and $L$ a simple $R(\beta)$-module.
   Then the following three conditions are equivalent.  
   \begin{enumerate}
   \item $L$ is a cuspidal module. 
   \item $\W(L) \cap \Phi_+ \subset \Phi_{+, \preceq C}$. 
   \item $\W^*(L) \cap \Phi_+ \subset \Phi_{+, \succeq C}$. 
   \end{enumerate}
 \end{corollary}

 \begin{proof}
 Let $\preceq'$ be a convex order that refines a weakly convex order $\preceq$. 
 (1) immediately implies (2) and (3) by the definition of cuspidal modules and the convexity of $\preceq'$.  
 We prove that (2) implies (1). 
 The same proof applies to show that (3) implies (1). 

 We use the same notation as in Proposition \ref{prop:refinedcuspidal}. 
 Assume (2). 
 We ignore degree shifts. 
 Note that $\Res_{a_m\alpha_m, \ldots, a_1 \alpha_1} L = L_m \otimes \cdots \otimes L_1$ since $(L_m, \ldots, L_1)$ is unmixing (Lemma \ref{lem:cuspidalseq}). 
 If $\alpha_m$ is a real root, then $L_m \simeq L^{\preceq'}(\alpha_m)^{\circ a_m}$ by Proposition \ref{prop:realcuspidal}. 
 Hence, $\Res_{\alpha_m^{a_m}} L_m \neq 0$ by the Mackey filtration and $\alpha_m \in \W(L) \cap \Phi_+$. 
 If $\alpha_m$ is an imaginary root, then $a_m\alpha_m$ is also a root \cite[Proposition 5.5]{MR1104219} and $a_m\alpha_m \in \W(L) \cap \Phi_+$. 
 In any case, we obtain $a_m \alpha_m \in \spn_{\mathbb{Z}_{\geq 0}} \Phi_{+, \preceq C}^{\text{min}}$. 
 It follows that $C \succeq C_r \succ C_{r-1} \succ \cdots \succ C_1$. 
 On the other hand, we have $\beta = a_m\alpha_m + \cdots + a_1 \alpha_1$. 
 Therefore, the convexity of $\preceq'$ shows that $r = 1$. 
 It proves $L \simeq M_1$, which is $\preceq$-cuspidal. 
 \end{proof}

 \begin{corollary} \label{cor:independenceoforder}
  Let $\preceq, \preceq'$ be two weakly convex preorders. 
  Let $C$ be a $\preceq$-equivalence class and $C'$ a $\preceq'$-equivalence class.
  Let $\beta \in \spn_{\mathbb{Z}_{\geq 0}}C \cap \spn_{\mathbb{Z}_{\geq 0}}C'$. 
  Assume that $\Phi_{+, \preceq C}^{\mathrm{min}} \subset \Phi_{+, \preceq' C'}^{\mathrm{min}}$ (or $\Phi_{+, \succeq C}^{\mathrm{min}} \subset \Phi_{+, \succeq' C'}^{\mathrm{min}}$). 
  Then, every $\preceq$-cuspidal $R(\beta)$-module is $\preceq'$-cuspidal.    
 \end{corollary}

\begin{proof}
  It follows from Corollary \ref{cor:cuspidalcriterion}. 
\end{proof}
 
 \begin{corollary}\label{cor:cuspidalcriterion2}
  Let $\preceq$ be a weakly convex preorder on $\Phi_+^{\mathrm{min}}$, $\beta \in Q_+$ and $L$ be a simple $R(\beta)$-module. 
  Let $C$ be a $\preceq$-equivalence class.  
  Then, the following conditions are equivalent. 
  \begin{enumerate}
  \item $\beta \in \spn_{\mathbb{Z}_{\geq 0}}C$ and $L$ is a cuspidal $R(\beta)$-module. 
  \item $\W(L) \subset \spn_{\mathbb{Z}_{\geq 0}} \Phi_{+, \preceq C}^{\mathrm{min}}$ and $\W^*(L) \subset \spn_{\mathbb{Z}_{\geq 0}} \Phi_{+,\succeq C}^{\mathrm{min}}$. 
  \item $\W(L) \cap \Phi_+ \subset \Phi_{+, \preceq C}$ and $\W^*(L)\cap \Phi_+ \subset \Phi_{+,\succeq C}$. 
  \end{enumerate}
 \end{corollary}

 \begin{proof}
  Let $\preceq'$ be a convex order that refines $\preceq$. 
  We use the same notation as in Proposition \ref{prop:refinedcuspidal}. 
  (1) implies (2) by the definition of cuspidal modules. 
  (2) implies (3) by the convexity of $\preceq'$. 
  It remains to prove that (3) implies (1). 
  By the same argument as in Corollary \ref{cor:cuspidalcriterion}, $\W(L) \cap \Phi_+ \subset \Phi_{+, \preceq C}$ implies $C \succeq C_r \succ \cdots \succ C_1$. 
  Similarly, $\W^*(L) \cap \Phi_+ \subset \Phi_{+, \succeq C}$ implies $C \preceq C_1 \prec \cdots \prec C_r$. 
  Hence, we must have $r = 1$ and $L \simeq M_1$ is $\preceq$-cuspidal. 
 \end{proof}

\subsection{Determinantial modules} \label{sec:determinantial}

We review some preliminary results on determinantial modules. 

\begin{lemma}[{\cite[Proposition 4.1]{MR3771147}}] \label{lem:determinantial}
 Let $\Lambda \in P_+$ and $w,v \in W$ with $w \geq v$. 
 Then, there exists a self-dual simple $R(v\Lambda - w\Lambda)$-module $M(w\Lambda, v \Lambda)$\index{$M(w\Lambda,v\Lambda)$} such that 
 \[
 \Psi_2 ([M(w \Lambda, v \Lambda)]) = D(w \Lambda, v\Lambda). 
 \]
 Moreover, it is unique up to isomorphism. 
\end{lemma}

For the sake of readers' convenience, we include the proof below.

\begin{proof}
  We construct $M(w\Lambda, v\Lambda)$ inductively.
  We refer to the description of unipotent quantum minors given after Lemma \ref{lem:quantumminor}. 
  We use the notation there. 

  First, consider the case $v = e$. 
  Then, $D(w\Lambda, \Lambda) = \iota_{\Lambda} (f_{i_1}^{(a_1)} \cdots f_{i_m}^{(a_m)} v_{\Lambda})$. 
  Accordingly, we define 
  \[  
  M(w\Lambda, \Lambda) = I_{\Lambda} (F_{i_1}^{\Lambda})^{(a_1)} \cdots (F_{i_m}^{\Lambda})^{(a_m)} \mathbf{k}. 
  \]
  By Theorem \ref{thm:categorification} and Theorem \ref{thm:cyclotomiccategorification}, we have $\Psi_2([M(w\Lambda,\Lambda)]) = D(w\Lambda, \Lambda)$. 
  By applying Theorem \ref{thm:sl2categorification} repeatedly, 
 we have an equivalence 
  \[
  \gMod{R^{\Lambda}(w\Lambda)} \simeq \gMod{R^{\Lambda}(\Lambda)} \simeq \gMod{\mathbf{k}}, 
  \]
 under which $M(w\Lambda,\Lambda)$ corresponds to $\mathbf{k}$. 
  Hence, $M(w\Lambda,\Lambda)$ is simple, which completes the proof in this case. 

  Next, we consider the case $v \neq e$. 
  Take $i \in I$ such that $s_i v < v$ and assume that $M(w\Lambda, s_iv\Lambda)$ has already been constructed. 
  By Lemma \ref{lem:quantumminor}, we have $\varepsilon_i^*(M(w\Lambda, s_i v\Lambda)) = \varepsilon_i^*(D(w\Lambda, s_i v\Lambda)) = \langle h_i, s_i v \Lambda \rangle \geq 0$. 
  Let $m$ be this integer. 
  Set, $M(w\Lambda,v\Lambda) = \tilde{e}_i^{*m} M(w\Lambda, s_i v\Lambda) = e_i^{*(m)} M(w\Lambda, s_i v\Lambda)$ (see the remark after Theorem \ref{thm:categoricalcrystal}).
  It is a self-dual simple $R(v\Lambda - w\Lambda)$-module, and we have 
  \[
  \Psi_2([L]) = e_i^{*(m)}D(w\Lambda, s_iv\Lambda) = D(w\Lambda, v\Lambda) \quad \text{by Lemma \ref{lem:quantumminor}}. 
  \]
  Therefore, it fulfills the requirements. 
  The existence is proved. 

  Since $\Psi_2$ is an isomorphism, the uniqueness follows from the simplicity of $M(w\Lambda, v\Lambda)$. 
\end{proof}

We call $M(w\Lambda, v\Lambda)$ a determinantial module.
Through the categorification (Theorem \ref{thm:categorification} and Theorem \ref{thm:cyclotomiccategorification}), the properties of unipotent quantum minors presented in Lemma \ref{lem:quantumminor} implies the following lemma. 

\begin{lemma}[{\cite[Proposition 10.2.4]{MR3758148}}]\label{lem:determinantialepsilon}
  Let $\Lambda \in P_+$. 
  \begin{enumerate}
    \item For $x \in W$, we have $M(x\Lambda, x\Lambda) = R(0) = \mathbf{k}$. 
    \item Let $x,y \in W$ and $i \in I$. Assume $s_i x > x \geq y$. Then\begin{align*}
    &\varepsilon_i(M(x\Lambda, y\Lambda)) = 0,\ \varepsilon_i(M(s_ix \Lambda,y\Lambda)) = \langle h_i, x\Lambda \rangle, \\ 
    &{e'_i}^{(\langle h_i, x\Lambda \rangle)}M(s_ix \Lambda, y\Lambda) = M(x\Lambda,y\Lambda). 
    \end{align*}
    \item Let $x,y \in W$ and $i \in I$. Assume $x \geq s_i y > y$. Then \begin{align*}
    &\varepsilon_i^* (M(x\Lambda, s_i y\Lambda)) = 0,\ \varepsilon_i^* (M(x\Lambda,y\Lambda)) = \langle h_i, y\Lambda \rangle, \\
    &{e'_i}^{*(\langle h_i, y\Lambda \rangle)}M(x\Lambda, y\Lambda) = M(x\Lambda, s_iy\Lambda). 
   \end{align*}
  \end{enumerate}
  \end{lemma}

It is known that $M(w\Lambda, v\Lambda) \circ M(w\Lambda', v\Lambda') \simeq q^{-(v\Lambda, v\Lambda'- w\Lambda')} M(w(\Lambda + \Lambda'), v(\Lambda + \Lambda'))$ \cite[Proposition 4.2]{MR3771147}. 
In particular, determinantial modules are real. 

\begin{lemma}[{\cite[Theorem 4.5]{MR3771147}}] \label{lem:detsupport}
  Let $\Lambda \in P_+, w,v \in W$ and assume that $w \geq v$. 
  Then, we have 
  \[
  \W(M(w\Lambda,v\Lambda)) \subset Q_+ \cap wQ_-,\ \W^*(M(w\Lambda,v\Lambda)) \subset Q_+ \cap vQ_+. 
   \]
\end{lemma}

Determinantial modules admit affinizations. 
We demonstrate how to construct them. 
Let $\Lambda \in P_+$ and $w,v \in W$ with $w \geq v$. 
We consider the cyclotomic quiver Hecke algebra $R_{\mathbf{k}[z]}^{a_{\Lambda}} = R_z^{a_{\Lambda}}$ over $A = \mathbf{k}[z]$, where $z$ is an indeterminate of positive degree. 
Since $z$ is a central element of positive degree, 
it acts trivially on every graded simple module over $R_z^{a_{\Lambda}}(\lambda)$. 
In other words, every simple module are obtained from a simple module over $R_z^{a_{\Lambda}}(\lambda)/zR_z^{a_{\Lambda}}(\lambda) \simeq R^{a_{\Lambda}}(\lambda)$ via inflation. 
Hence, we regard $M(w\Lambda,\Lambda)$ as the unique self-dual simple $R_z^{a_{\Lambda}}(w\Lambda)$-module. 
We impose an additional condition on $a_{\Lambda}$ that $a_{\Lambda, i} (0) \neq 0$ for all $i \in I$
(which implies that $\deg z$ must divide $2(\alpha_i, \Lambda)$ for all $i \in I$).
Using the notation introduced after Lemma \ref{lem:quantumminor}, we define
\[
\widehat{M}(w \Lambda, v\Lambda) = E_{j_1}^{*(b_1)} \cdots E_{j_n}^{*(b_n)} (F_{i_1}^{a_{\Lambda}})^{(a_1)}\cdots (F_{i_m}^{a_{\Lambda}})^{(a_m)} \mathbf{k}[z].
\] \index{$\widehat{M}(w\Lambda,v\Lambda)$}
By applying Theorem \ref{thm:sl2categorification} repeatedly, 
we have an equivalence 
\[
\gMod{R_z^{a_{\Lambda}}(w\Lambda)} \simeq \gMod{R_z^{a_{\Lambda}}(\Lambda)} = \gMod{\mathbf{k}[z]}, 
\] 
under which $\widehat{M}(w\Lambda,\Lambda)$ corresponds to the projective module $\mathbf{k}[z]$. 
Hence, $\widehat{M}(w \Lambda, \Lambda)$ is the projective cover of the unique self-dual simple $R_z^{a_{\Lambda}}(w\Lambda)$-module $M(w\Lambda,\Lambda)$. 

\begin{proposition} [{\cite[Theorem 3.26]{MR4359265}}] \label{prop:detaffinization}
  $(\widehat{M}(w\Lambda, v\Lambda), z)$ is an affinization of $M(w\Lambda, v\Lambda)$. 
\end{proposition}

In this paper, we only need this proposition for the case where $\Lambda = \Lambda_i$ for some $i\in I$. 
Let $z = z_i$ be an indeterminate of degree $(\alpha_i, \alpha_i)$
and put $a_{\Lambda_i, j} (t_j) = (t_j - z_i)^{\delta_{i,j}}$.  
Throughout this paper, let $R_{z}^{\Lambda_i}(\lambda)$ denote the cyclotomic quiver Hecke algebra over $\mathbf{k}[z]$ associated with this datum.  

\begin{proposition}[{\cite[Proposition 4.6]{MR3771147}}] \label{prop:dethead}
  Let $\Lambda \in P_+, w,v,x\in W$ and assume that $w \geq v \geq x$. 
  Then, we have 
  \[
  M(w\Lambda,v\Lambda) \nabla M(v\Lambda,x\Lambda) \simeq M(w\Lambda,x\Lambda). 
  \]
\end{proposition}

\begin{remark} \label{rem:deterrata}
The proof of this proposition given in \cite{MR3771147} contains an error. 
We provide a corrected proof in Appendix \ref{app:correction}. 
\end{remark}

\subsection{Subcategories corresponding to quantum unipotent subgroups} \label{sub:Cwv}

\begin{definition}
Let $w, v \in W$ and $\beta \in Q_+$.
We define $I_{w,*}(\beta)$\index{$I_{w,*}(\beta)$} as a two-sided ideal of $R(\beta)$ generated by $e(\nu) \ (\nu \in I^{\beta})$ such that 
\[
\text{there exists $1 \leq m \leq \height \beta$ such that $\sum_{1 \leq k \leq m} \alpha_{\nu_k} \not \in Q_+ \cap wQ_-$}.  
\]
We define $I_{*,v}(\beta)$\index{$I_{*,v}(\beta)$} as a two-sided ideal of $R(\beta)$ generated by $e(\nu) \ (\nu \in I^{\beta})$ such that 
\[
\text{there exists $1 \leq m \leq \height \beta$ such that $\sum_{m \leq k \leq \height \beta} \alpha_{\nu_k} \not \in Q_+ \cap v Q_+$}. 
\]
We define $I_{w,v}(\beta) = I_{w,*}(\beta) + I_{*,v}(\beta)$.  \index{$I_{w,v}(\beta)$}
We define $R_{w,*} (\beta) = R(\beta)/I_{w,*}(\beta)$\index{$R_{w,*}(\beta)$}, $R_{*,v}(\beta)\index{$R_{*,v}(\beta)$} = R(\beta)/I_{*,v}(\beta)$ and $R_{w,v}(\beta)\index{$R_{w,v}(\beta)$} = R(\beta)/I_{w,v}(\beta)$.   
\end{definition}

These quotient rings are characterized as follows:
\begin{align*}
 \gMod{R_{w,*}(\beta)} &\simeq \{ M \in \gMod{R(\beta)} \mid W(M) \subset Q_+ \cap w Q_-\}, \\
 \gMod{R_{*,v}(\beta)} &\simeq \{ M \in \gMod{R(\beta)} \mid W^*(M) \subset Q_+ \cap v Q_+\}, \\
 \gMod{R_{w,v}(\beta)} &\simeq \left\{ M \in \gMod{R(\beta)} \middle| \begin{aligned}
 &X(M) \subset Q_+ \cap w Q_-, \\
 &W^*(M) \subset Q_+ \cap v Q_+
 \end{aligned} \right\}. 
\end{align*}

Put 
\[
\gMod{R_{w,v}} = \bigoplus_{\beta \in Q_+} \gMod{R_{w,v}(\beta)}, \gmod{R_{w,v}} = \bigoplus_{\beta \in Q_+ } \gmod{R_{w,v}(\beta)}.
\]
Categories $\gMod{R_{w,*}}$, $\gMod{R_{*,v}}$, $\gmod{R_{w,v}}$, $\gmod{R_{w,*}}$ and $\gmod{R_{*,v}}$ are defined similarly. 
They are isomprhic to Serre full subcategories of $\gMod{R}$. 
Moreover, since $\W(M\circ N) = \W(M) + \W(N), \W^*(M\circ N) = \W^*(M) + \W^*(N)$, they are closed under convolution products.
The category $\gmod{R_{w,v}}$ is precisely the category $\mathscr{C}_{w,v}$ introduced in \cite{MR3771147}. 
By Lemma \ref{lem:detsupport}, $M(w\Lambda,v\Lambda) \in \gmod{R_{w,v}}$. 

We demonstrate the relationship of these subcategories to cuspidal decompositions.
Let $w \in W$ and take a reduced expression $w = s_{i_1} \cdots s_{i_l}$. 
Many constructions below depend on the choice of $w$ and its reduced expression, but we usually exclude them from the notation for simplicity. 
Put $w_{\leq k}\index{$w_{\leq k}$} = s_{i_1} \cdots s_{i_k} \ (0 \leq k \leq l)$ and $\beta_k\index{$\beta_k$} = w_{\leq k-1} \alpha_{i_k} \ (1 \leq k \leq l)$.  
We define a preorder $\preceq$ on $\minroot$ by 
\begin{enumerate}
\item $\beta_1 \prec \beta_2 \prec \cdots \prec \beta_l \prec \beta$ for any $\beta \in \minroot \cap w\Phi_+$, and 
\item $\minroot \cap w \Phi_+$ is a single $\preceq$-equivalence class. 
\end{enumerate}
By \cite[Proposition 1.24]{MR3771147}, it is a weakly convex preorder. 

\begin{lemma} \label{lem:Rwvcuspidal}
Let $\beta \in Q_+$ and $ 1 \leq k \leq l$. 
\begin{align*}
  R_{w_{\leq k}, w_{\leq k-1}}(\beta) &= \begin{cases}
     R^{\preceq}(\beta) & \text{if $\beta \in \mathbb{Z}_{\geq 0}\beta_k$}, \\
     0 & \text{otherwise},  
  \end{cases} \\
 R_{*,w} (\beta) &= \begin{cases}
  R^{\preceq}(\beta) & \text{if $\beta \in \spn_{\mathbb{Z}_{\geq 0}} (\minroot \cap w \Phi_+)$}, \\
  0 & \text{otherwise}. 
\end{cases}
\end{align*}
\end{lemma}

\begin{proof}
Assume that $R_{w_{\leq k}, w_{\leq k-1}}(\beta) \neq 0$. 
Take a simple $R_{w_{\leq k}, w_{\leq k-1}}(\beta)$-module $L$. 
Then, $\W(L) \cap \Phi_+ \subset \Phi_+ \cap w_{\leq k}\Phi_- = \Phi_{+, \preceq \beta_k}$ and $\W^*(L) \cap \Phi_+  \subset \Phi_+ \cap w_{\leq k-1}\Phi_+ = \Phi_{+, \succeq \beta_{k}}$. 
By Corollary \ref{cor:cuspidalcriterion2}, $\beta \in \mathbb{Z}_{\geq 0} \beta_k$ and $L$ is cuspidal. 
In particular, $R_{w_{\leq k}, w_{\leq k-1}}(\beta) = 0$ unless $\beta \in \mathbb{Z}_{\geq 0} \beta_k$. 

Now, assume that $\beta = a\beta_k$ for some $a \in \mathbb{Z}_{\geq 0}$. 
We claim that a simple $R(\beta)$-module $L$ is an $R_{w_{\leq k}, w_{\leq k-1}}(\beta)$-module if and only if it is an $R^{\preceq}(\beta)$-module. 
If $L$ is an $R_{w_{\leq k}, w_{\leq k-1}}(\beta)$-module, it is an $R^{\preceq}(\beta)$-module by the discussion in the previous paragraph. 
If $L$ is an $R^{\preceq}(\beta)$-module, we have $\W(L) \subset \spn_{\mathbb{Z}_{\geq 0}} \Phi_{+, \preceq \beta_k}^{\text{min}} \subset Q_+ \cap w_{\leq k} Q_-$. 
Similarly, $\W^*(L) \subset Q_+ \cap w_{\leq k-1}Q_+$. 
Hence, $L$ is an $R_{w_{\leq k}, w_{\leq k-1}}(\beta)$-module and the claim is proved. 
Since both of these quotient rings are obtained by ideals generated by idempotents, the former half of the lemma follows from this claim. 

To prove the latter assertion, assume that $R_{*,w}(\beta) \neq 0$. 
Set $C = \minroot \cap w \Phi_+$. 
Take a simple $R_{*,w}(\beta)$-module $L$. 
Then, $\W(L) \cap \Phi_+ \subset \Phi_{+, \preceq C}, \W^*(L) \cap \Phi_+ \subset \Phi_+ \cap w \Phi_+ = \Phi_{+, \succeq C}$. 
By Corollary \ref{cor:cuspidalcriterion2}, $\beta \in \spn_{\mathbb{Z}_{\geq 0}}C$ and $L$ is cuspidal.
In particular, $R_{*,w}(\beta) = 0$ unless $\beta \in \spn_{\mathbb{Z}_{\geq 0}} C$. 

Now, assume that $\beta \in \spn_{\mathbb{Z}_{\geq 0}}C$.
We can prove that a simple $R(\beta)$-module $L$ is an $R_{*,w}(\beta)$-module if and only if it is an $R^{\preceq}(\beta)$-module as in the second paragraph. 
It proves the latter assertion. 
\end{proof}

By Proposition \ref{prop:weakcuspidaldecomp}, there exists a unique self-dual $\preceq$-cuspidal $R(a\beta_k)$-module up to isomorphism, which is denoted by $L(a\beta_k) = L^{\preceq}(a\beta_k)$.
As in Proposition \ref{prop:realcuspidal}, we have $L(a\beta_k) \simeq q_{i_k}^{a(a-1)/2} L(\beta_k)^{\circ a}$. 
Lemma \ref{lem:Rwvcuspidal} implies the following proposition. 

\begin{proposition} [{\cite[Theorem 4.5(ii)]{MR3771147}}] \label{prop:determsupport}
For $1 \leq k \leq l$, we have 
\[
L(\beta_k) = L^{\preceq}(\beta_k) \simeq M(w_{\leq k}\Lambda_{i_k}, w_{\leq k-1}\Lambda_{i_k}). 
\]
\end{proposition}

This proposition is significant because it gives a concrete realization of the cuspidal module.  

Put $S = \{ \text{self-dual simples in $\gMod{R_{*,w}}$} \} / {\simeq}$. \index{$S$}
Let $s_0$ denote the isomorphism class of $\mathbf{k} \in \gMod{R_{*,w}(0)}$. \index{$s_0$}
For each $s \in S$, we choose a representative $L(s)\index{$L(s)$} \in \gMod{R_{*,w}}$ and set $\beta_s\index{$\beta_s$} = - \wt L(s)$.  
We define
\[ 
\Sigma (\beta) =  \{ (\lambda = (\lambda_1, \ldots, \lambda_l), s) \in \mathbb{Z}_{\geq 0}^l \times S \mid \sum_{1 \leq k \leq l} \lambda_k \beta_k + \beta_s = \beta \}
\] \index{$\Sigma(\beta)$}
for each $\beta \in Q_+$. 
There is a map $\rho\index{$\rho$} \colon \Sigma (\beta) \to \mathcal{P}^{\preceq}(\beta)$ given by 
\[
\rho(\lambda,s) = (\lambda_1 \beta_1, \ldots, \lambda_l \beta_l, \beta_s). 
\]
For $(\lambda, s) \in \Sigma(\beta)$, we define
\begin{align*}
\overline{\Delta}(\lambda,s) &= L(s) \circ L(\lambda_l \beta_l) \circ \cdots \circ L(\lambda_1 \beta_1) \\ \index{$\overline{\Delta}(\lambda,s)$}
&\simeq q^{m_{\lambda}} L(s) \circ M(w_{\leq l}\Lambda_{i_l}, w_{\leq l-1}\Lambda_{i_l})^{\circ \lambda_l} \circ \cdots \circ M(s_{i_1}\Lambda_{i_1}, \Lambda_{i_1})^{\circ \lambda_1}, \\
L(\lambda,s) &= \hd \overline{\Delta}(\lambda,s), \index{$L(\lambda,s)$}
\end{align*}
where $m_{\lambda} = \sum_{1 \leq k \leq l} (\alpha_{i_k}, \alpha_{i_k}) \lambda_k (\lambda_k -1)/4$. 

As pointed out in Remark \ref{rem:failure}, Proposition \ref{prop:properstandardcomposition} does not hold for weakly convex preorders.
To remedy this in the current situation, we introduce the following partial order on $\mathcal{P}^{\preceq}(\beta)$, which is weaker than the one in Definition \ref{def:bilexicographic}.

\begin{definition} \label{def:modifiedlexico}
  For $\underline{\gamma} = (\gamma_1, \ldots, \gamma_m) \in \mathcal{P}^{\preceq}(\beta), \underline{\gamma'} = (\gamma'_1, \ldots, \gamma'_n) \in \mathcal{P}(\beta)$, we write $\underline{\gamma'} < \underline{\gamma}$\index{$\underline{\gamma'} < \underline{\gamma}$} if the following two conditions hold: 
\begin{enumerate}
\item there exists $1 \leq k \leq \min \{ m, n\}$ such that $\gamma_1 = \gamma'_1 , \ldots, \gamma_{k-1} = \gamma'_{k-1}$ and either (i) $\gamma'_k \in (\spn_{\mathbb{Z}_{\geq 0}} \Phi_{+,\succeq C_{\gamma_k}}^{\text{min}}) \setminus (\spn_{\mathbb{Z}_{\geq 0}} C_{\gamma_k})$ or (ii) $\gamma'_k \in \spn_{\mathbb{Z}_{\geq 0}}C_{\gamma_k}, \gamma_k - \gamma'_k \in Q_+ \setminus \{0\}$ holds. 
\item if $\gamma_m = a\beta_r$ for some $a \in \mathbb{Z}_{>0}$ and $1 \leq  r \leq l$, then there exists $1 \leq k \leq \min \{ m, n\}$ such that $\gamma_m = \gamma'_n , \ldots, \gamma_{m-k+1} = \gamma'_{n-k+1}$ and either (i) $\gamma'_{n-k} \in (\spn_{\mathbb{Z}_{\geq 0}} \Phi_{+, \preceq C_{\gamma_{m-k}}}^{\text{min}}) \setminus (\spn_{\mathbb{Z}_{\geq 0}} C_{\gamma_k})$ or (ii) $\gamma'_{n-k} \in \spn_{\mathbb{Z}_{\geq 0}} C_{\gamma_{m-k}}, \gamma_{m-k} - \gamma'_{n-k} \in Q_+ \setminus \{0 \}$ holds.
\end{enumerate}
It defines a partial order on $\mathcal{P}^{\preceq}(\beta)$. 
\end{definition}

Note that, since $\sum_{r} \gamma_r = \sum_{r} \gamma'_r$, the index $k$ in the condition (1) must satisfy $k \leq \min \{ m,n\} -1$. 
Similarly, the index $k$ in the condition (2) must satisfy $k \geq 2$. 

It is often convenient to identify $\mathcal{P}^{\preceq}(\beta)$ with the set 
\[
 \left\{ (\lambda_1 \beta_1, \ldots, \lambda_l \beta_l, \gamma) \in \mathcal{P}(\beta) \middle| \begin{aligned}  
  &\lambda_1, \ldots, \lambda_l \in \mathbb{Z}_{\geq 0}, \gamma \in \spn_{\mathbb{Z}_{\geq 0}} (\minroot \cap w \Phi_+), \\
  &\sum_{1\leq k \leq l}\lambda_k\beta_k + \gamma = \beta
 \end{aligned}\right\}.
\] 
Then, Definition \ref{def:modifiedlexico} simply says that $(\lambda_1 \beta_1, \ldots, \lambda_l \beta_l, \gamma) > (\mu_1 \beta_1, \ldots, \mu_l \beta_l, \gamma')$ in $\mathcal{P}^{\preceq}(\beta)$ if the following two conditions hold: 
\begin{enumerate}
  \item there exists $1 \leq k \leq l$ such that $\lambda_1 = \mu_1, \ldots, \lambda_{k-1} = \mu_{k-1}, \lambda_k > \mu_k$, 
  \item if $\gamma = 0$, then $\gamma' = 0$ and there exists $1 \leq k \leq l$ such that $\lambda_l = \mu_l, \ldots, \lambda_{k+1} = \mu_{k+1}, \lambda_k > \mu_k$. 
\end{enumerate}

Note that, $\preceq$ is a convex preorder only when $\mathfrak{g}$ is finite dimensional and $w$ is the longest element $w_0$. 
In this case, Definition \ref{def:bilexicographic} and Definition \ref{def:modifiedlexico} coincide.  

The following theorem summarizes the cuspidal decomposition (Section \ref{sub:cuspidal}) within the current setup. 

  \begin{theorem} \label{thm:cuspidalCw}
  Let $\beta \in Q_+$. We have a bijection
  \begin{align*}
  \Sigma (\beta) &\to \{ \text{self-dual simple module in $\gMod{R(\beta)}$} \}/{\simeq}, \\
 (\lambda, s) & \mapsto L(\lambda, s). 
  \end{align*}
  Moreover, we have 
  \begin{enumerate}
   \item $\Res_{\beta_s, \lambda_l \beta_l, \ldots, \lambda_1\beta_1} \overline{\Delta}(\lambda, s) = \Res_{\beta_s, \lambda_l\beta_l, \ldots, \lambda_1 \beta_1} L(\lambda, s) = L(s) \otimes L(\lambda_l \beta_l) \otimes \cdots \otimes L(\lambda_1\beta_1)$, 
   \item For $\underline{\gamma} = (\gamma_1, \ldots, \gamma_n) \in \mathcal{P}(\beta)$, $\Res_{\gamma_n, \ldots, \gamma_1} \overline{\Delta}(\lambda,s) = \Res_{\gamma_n, \ldots, \gamma_1} L(\lambda, s) = 0$ if $(\gamma_1, \ldots,\gamma_n) \not \leq (\lambda_1 \beta_1, \ldots, \lambda_l\beta_l, \beta_s)$, 
   \item $[\overline{\Delta} (\lambda,s): L(\lambda,t)]_q = \delta_{s,t}$, 
   \item If $[\overline{\Delta}(\lambda, s): L(\mu, t)]_q \neq 0$, then $\rho(\mu,t) <\rho(\lambda, s)$ or $(\mu,t) = (\lambda,s)$,
   \item $L(\lambda, s) \in \gMod{R_{w,*}}$ if and only if $s = s_0$, 
   \item $L(\lambda, s) \in \gMod{R_{*,w}}$ if and only if $\lambda = 0$. 
  \end{enumerate}
\end{theorem}

\begin{proof}
  In view of Lemma \ref{lem:Rwvcuspidal} and Proposition \ref{prop:determsupport}, the bijection is Proposition \ref{prop:weakcuspidaldecomp}. 
  Since $(L_s, L(\lambda_l\beta_l), \ldots, L(\lambda_1\beta_1))$ is unmixing, (1) follows. 

  (2)
 By disregarding entries of zero, write $(\lambda_1\beta_1, \ldots, \lambda_l\beta_l, \beta_s)$ in the form $\underline{\gamma'} = (\gamma'_1, \ldots, \gamma'_m)$.
 Assume $\underline{\gamma} \not \leq \underline{\gamma'}$. 

 First, assume that $\beta_s \neq 0$. 
 By Definition \ref{def:modifiedlexico} and the remark after it, we have $\underline{\gamma} \neq \underline{\gamma'}$ and either (i) or (ii): 
 \begin{enumerate}
 \item[(i)] there exists $1 \leq k\leq \min \{m,n\} -1$ such that $\gamma_1 = \gamma'_1, \ldots, \gamma_{k-1} = \gamma'_{k-1}$ and $\gamma \not \in \spn_{\mathbb{Z}_{\geq 0}} \Phi_{+, \preceq C_{\gamma'_k}}^{\mathrm{min}}$, 
 \item[(ii)] there exists $1 \leq k\leq \min\{m,n\} -1$ such that $\gamma_1 = \gamma'_1, \ldots, \gamma_{k-1} = \gamma'_{k-1}$ and $\gamma_k \in \mathbb{Z}_{\geq 2}\gamma'_k$. 
 \end{enumerate}
Note that $k < m$, hence $\gamma'_1, \ldots, \gamma'_k \in \{ \lambda_1\beta_1, \ldots, \lambda_l\beta_l\}$. 
By the same argument as in Proposition \ref{prop:costandard}, we deduce the assertion. 
(The key is that the restriction of $\preceq$ to $\{\beta_1,\ldots,\beta_l\}$ is an order. )

Next, assume that $\beta_s = 0$. 
Note that $\{\gamma'_1, \ldots, \gamma'_m\} = \{\lambda_k\beta_k \mid 1 \leq k \leq l, \lambda_k \neq 0 \}$. 
By Definition \ref{def:modifiedlexico}, we have $\underline{\gamma} \neq \underline{\gamma'}$ and either of (i) to (iv): 
\begin{enumerate}
\item[(iii)] there exists $1 \leq k \leq \min \{m,n\}$ such that $\gamma_n = \gamma'_m, \ldots, \gamma_{n-k+1} = \gamma'_{m-k+1}$, and $\gamma_{n-k} \not \in \spn_{\mathbb{Z}_{\geq 0}} \Phi_{+, \succeq_{\gamma'_{m-k}}}^{\mathrm{min}}$, 
\item[(iv)] there exists $1 \leq k \leq \min\{m,n\}$ such that $\gamma_n = \gamma'_m, \ldots, \gamma_{n-k+1} = \gamma'_{m-k+1}$, and $\gamma_{n-k} \in \mathbb{Z}_{\geq 2} \gamma_{m-k}$. 
\end{enumerate}
In any case, the same argument as in Proposition \ref{prop:costandard} proves the assertion. 

  (4) 
  Note that, by (1), $\Res_{\beta_t, \mu_l\beta_l, \ldots, \mu_1\beta_1} L(\mu,t) = L(t) \otimes L(\mu_l\beta_l) \otimes \cdots \otimes L(\mu_1 \beta_1)$. 
  If $\rho(\mu,t) \not \leq \rho(\lambda,s)$, then (2) shows $[\overline{\Delta}(\lambda,s): L(\mu,t)]_q = 0$. 
  Assume $\rho(\mu,t) = \rho(\lambda,s)$.
  Then, $\Res_{\beta_t, \mu_l\beta_l, \ldots, \mu_1 \beta_1} \Delta(\lambda,s) = L(s) \otimes L(\lambda_l\beta_l) \otimes \cdots \otimes L(\lambda_1 \beta_1)$ by (1).
  Hence, $[\overline{\Delta}(\lambda,s): L(\mu,t)]_q = 0$ unless $(\mu,t) = (\lambda,s)$.    

(3)
It follows from (1) and (4). 

  (5) Assume $s = s_0$. 
  By Lemma \ref{lem:detsupport}, $\W(L(\beta_k)) \subset Q_+ \cap w_{\leq k} Q_- \subset Q_+ \cap w Q_-$. 
  Hence $L(\beta_k)$ is an $R_{w,*}(\beta_k)$-module. 
  Since $\gMod{R_{w,*}}$ is closed under convolution products, we deduce that $\overline{\Delta}(\lambda,s_0)$ and its head $L(\lambda,s_0)$ are $R_{w,*}(\beta)$-modules. 

  Assume $s \neq s_0$. 
  Then, $\beta_s$ is a nonzero element of $\spn_{\mathbb{Z}_{\geq 0}} \minroot \cap w \minroot \subset Q_+ \cap w Q_+$.
  Since $\beta_s \in \W(L(\lambda,s))$, $L(\lambda,s)$ is not an $R_{w,*}(\beta)$-module. 

  (6) is proved in the same manner as (5). 
\end{proof}

For each $\beta \in Q_+$, we define $\Lambda(\beta)\index{$\Lambda(\beta)$} = \{ \lambda \in \mathbb{Z}_{\geq 0}^l \mid \sum_{1 \leq k \leq l} \lambda_k \beta_k = \beta \}$. 
By the Theorem above, it parametrizes the isomorphism classes of self-dual simple $R_{w,*}(\beta)$-modules. 
We write $L(\lambda)\index{$L(\lambda)$} = L(\lambda, s_0), \overline{\Delta}(\lambda)\index{$\overline{\Delta}(\lambda)$} = \overline{\Delta}(\lambda,s_0)$.

\section{Properties of affinizations of determinantial modules} \label{sec:properties}

\subsection{Homological property of $\widehat{M}(w\Lambda_i, \Lambda_i)$ } \label{sub:homologicalM1} 

Let $i \in I$ and $w \in W$. 
Assume $w \Lambda_i \neq \Lambda_i$. 
Recall $R_z^{\Lambda_i}$ and $\widehat{M}(w\Lambda_i,\Lambda_i)$ introduced after Proposition \ref{prop:detaffinization}. 
In this section, we present a homological property of $\widehat{M} (w\Lambda_i, \Lambda_i)$.   

\begin{proposition} \label{prop:cyclotomic}
 The canonical algebra morphism $R(\Lambda_i -w\Lambda_i) \hookrightarrow R(\Lambda_i - w\Lambda_i) \otimes \mathbf{k}[z] \twoheadrightarrow R_z^{a_{\Lambda_i}} (w\Lambda_i)$ induces an algebra isomorphism 
 \[
  R(\Lambda_i - w\Lambda_i) / \langle e(*, j), e(*, i, i) \mid j \in I \setminus \{i \} \rangle \simeq R_z^{a_{\Lambda_i}} (w\Lambda_i), 
 \]
 where $e(*,j)$ (resp. $e(*,i,i)$) is understood to be zero if $\Lambda_i - w\Lambda_i -\alpha_j \not \in Q_+$ (resp. $\Lambda_i - w\Lambda_i - 2\alpha_i \not \in Q_+$). 
\end{proposition}

\begin{proof}
  Put $R'(w\Lambda_i) = R(\Lambda_i - w\Lambda_i)/\langle e(*, j), e(*, i, i) \mid j \in I \setminus \{i \} \rangle$ and $n = \height (\Lambda_i - w\Lambda_i)$. \index{$R'(w\Lambda_i)$}
  By Theorem \ref{thm:sl2categorification}, $R_z^{a_{\Lambda_i}}(w\Lambda_i)$ is Morita equivalent to $\mathbf{k}[z]$ and it has a unique self-dual simple module $M(w\Lambda_i, \Lambda_i)$.
  Since $w\Lambda_i \neq \Lambda_i$, Lemma \ref{lem:determinantialepsilon} shows $\varepsilon_j^*(M(w\Lambda_i, \Lambda_i)) = \langle h_j, \Lambda_i \rangle = \delta_{i, j}$.  
  Hence, we have $e(*,j), e(*,i,i)=0 \ (j \in I \setminus \{i\})$ in $R_z^{a_{\Lambda_i}}(w\Lambda_i)$,  
  and it follows that the morphism $R(\Lambda_i - w\Lambda_i) \to R_z^{a_{\Lambda_i}}(w\Lambda_i)$ factors through $R'(w\Lambda_i)$. 

  We compute several relations in $R'(w\Lambda_i)$. 
  By the definition, $\sum_{j \neq i} e(*, j, i) = 1$. 
  Since 
  \[
  \tau_{n-1}e(*,j,i) = e(*,i,j)\tau_{n-1} = 0 
  \]
  for any $j \neq i$, we obtain $\tau_{n-1} = 0$. 
  It implies that $x_n$ is a central element of $R'(w\Lambda_i)$. 

  By the definition, $a_{\Lambda_i,i} (x_n) = x_n - z$. 
  Therefore, we obtain
  \[
  R_z^{a_{\Lambda_i}} (w\Lambda_i) \simeq (R'(w\Lambda_i) \otimes \mathbf{k}[z]) / \langle x_n - z \rangle \simeq R'(w\Lambda_i). 
  \]
\end{proof}

\begin{remark} 
  It is also possible to derive the relation $e(*, i, i) = 0$ in $R_z^{a_{\Lambda_i}}(w\Lambda_i)$ by a direct calculation. 
  In fact, $(x_n - z) e(*, i, i) = 0$ implies $\partial_{n-1}(x_n - z) e(*, i, i)= 0$ \cite[Lemma 4.2]{MR2995184}, where $\partial_{n-1}$ is the Demazure operator. 
  Therefore, we obtain $e(*, i, i) = 0$. 
\end{remark}

\begin{remark} \label{rem:morita}
  Theorem \ref{thm:sl2categorification} shows a graded Morita equivalence from $R'(w\Lambda_i) \simeq R_z^{a_{\Lambda_i}}(w\Lambda_i)$ to $\mathbf{k}[z]$. 
  Under this equivalence, $M(w\Lambda_i, \Lambda_i)$ corresponds to the simple module $\mathbf{k}$, and $\widehat{M}(w\Lambda_i, \Lambda_i)$ to the projective module $\mathbf{k}[z]$. 
\end{remark}

\begin{corollary} \label{cor:M1ext}
 $\Ext_{R(\Lambda_i - w\Lambda_i)}^1 (\widehat{M}(w\Lambda_i, \Lambda_i), M(w\Lambda_i,\Lambda_i)) = 0$. 
\end{corollary}

Note that we compute the extension group over $R(\Lambda_i-w\Lambda_i)$, not over $R_z^{a_{\Lambda_i}}(w\Lambda_i)$. 

\begin{proof}
 By Remark \ref{rem:morita}, $\widehat{M}(w\Lambda_i, \Lambda_i)$ is the unique projective $R'(\Lambda_i - w\Lambda_i)$-module. 
 Since $R'(\Lambda_i-w\Lambda_i)$ is a quotient of $R(\Lambda_i-w\Lambda_i)$ by an ideal generated by some idempotents, $\gMod{R'(\Lambda_i - w\Lambda_i)}$ is closed under extensions in $\gMod{R(\Lambda_i-w\Lambda_i)}$. 
 It proves the corollary. 
\end{proof}

\subsection{Braider structure of $\widehat{M}(w\Lambda_i, \Lambda_i)$} \label{sub:braider}

Let $i \in I, w\in W$ and assume that $w\Lambda_i \neq \Lambda_i$.
Put $M_1\index{$M_1$} = M(w\Lambda_i, \Lambda_i)$, $\widehat{M_1}\index{$\widehat{M_1}$} = \widehat{M}(w\Lambda_i, \Lambda_i)$ and $\gamma_1\index{$\gamma_1$} = -\wt M_1 = \Lambda_i - w\Lambda_i$.   
Note that $\widehat{M_1}$ is finitely generated $R(\gamma_1)$-module by Lemma \ref{lem:affinefg}. 
Recall that, for each $j \in I$, $R(\alpha_j)$ coupled with $w_j = x_1 \in \End_{R(\alpha_j)}(R(\alpha_j))$ is an affinization of the simple $R(\alpha_j)$-module $L(j)$.
We fix $\renormalizedR{\widehat{M_1}}{R(\alpha_j)}$ for each $j \in I$.  
It is of degree $\Lambda (M_1, L(j)) = \delta_{i,j} (\alpha_i, \alpha_i) - (\gamma_1, \alpha_j)$ by Lemma \ref{lem:LambdaforLi} and Lemma \ref{lem:determinantialepsilon}. 
Let $\phi\index{$\phi(\beta)$} \colon Q \to \mathbb{Z}$ be a $\mathbb{Z}$-linear homomorphism defined by $\phi (\alpha_j) = \Lambda (M_1, L(j)) \ (j \in I)$. 
The next lemma is a variation of \cite[Proposition 4.1]{MR4359265}. 

\begin{lemma} \label{lem:braider}
There exists a natural morphism 
 \[
 \braider{\widehat{M_1}}{X} \colon q^{\phi(\beta)}\widehat{M_1} \circ X \to X \circ \widehat{M_1}  \index{$\braider{\widehat{M_1}}{X}$}
 \]
 for $\beta \in Q_+, X \in \gMod{R(\beta)}$ such that
\[
  \braider{\widehat{M_1}}{R(\alpha_j)} = \renormalizedR{\widehat{M_1}}{R(\alpha_j)} 
\]
for all $j \in I$, and the diagram
\[
\begin{tikzcd}[column sep = 5cm]
q^{\phi(\beta + \gamma)} \widehat{M_1} \circ X \circ Y \arrow[r, "q^{\phi(\gamma)}\braider{\widehat{M_1}}{X}"]\arrow[rd, "\braider{\widehat{M_1}}{X \circ Y}"'] & q^{\phi(\gamma)} X \circ \widehat{M_1} \circ Y \arrow[d, "\braider{\widehat{M_1}}{Y}"] \\
 & X \circ Y \circ \widehat{M_1}
\end{tikzcd}
\]
commutes for $\beta, \gamma \in Q_+, X \in \gMod{R(\beta)}, Y \in \gMod{R(\gamma)}$. 
Moreover, $\braider{\widehat{M_1}}{X}$ is injective for any $\beta \in Q_+, X \in \gMod{R(\beta)}$. 
\end{lemma}

\begin{proof}
We ignore degree shifts in this proof. 
As remarked after Proposition \ref{prop:affinermatrix}, there is a polynomial $f_j(z, w_j) \in \mathbf{k}[z, w_j]$ uniquely determined by $\universalR{\widehat{M_1}}{R(\alpha_j)} = f_j(z, w_j) \renormalizedR{\widehat{M_1}}{R(\alpha_j)}$. 
It is homogeneous if we set $\deg w_j = (\alpha_j,\alpha_j)$, and its leading coefficient with respect to $z$ is in $\mathbf{k}^{\times}$. 
Let $\beta \in Q_+$ and set $n =\height (\beta)$. 
We have $R(\beta) = \bigoplus_{\nu \in I^{\beta}} R(\beta)e(\nu) = \bigoplus_{\nu \in I^{\beta}} R(\alpha_{\nu_1}) \circ \cdots \circ R(\alpha_{\nu_n})$. 
We define $\braider{\widehat{M_1}}{R(\beta)} = \bigoplus_{\nu \in I^{\beta}} \braider{\widehat{M_1}}{R(\beta)e(\nu)}$, 
where $\braider{\widehat{M_1}}{R(\beta)e(\nu)}$ is the composite
\begin{align*}
 \widehat{M_1} \circ R(\alpha_{\nu_1}) \circ \cdots \circ  R(\alpha_{\nu_n}) \xrightarrow{\renormalizedR{\widehat{M_1}}{R(\alpha_{\nu_1})}} & R(\alpha_{\nu_1}) \circ \widehat{M_1} \circ \cdots \circ R(\alpha_{\nu_n}) \\
 \xrightarrow{\renormalizedR{\widehat{M_1}}{R(\alpha_{\nu_2})}} & \cdots \\ 
 \xrightarrow{\renormalizedR{\widehat{M_1}}{R(\alpha_{\nu_n})}} & R(\alpha_{\nu_1}) \circ \cdots \circ R(\alpha_{\nu_n}) \circ \widehat{M_1}.
\end{align*}

Note that $\widehat{M_1}\circ R(\beta)$ and $R(\beta) \circ \widehat{M_1}$ are right $R(\beta)$-modules by the right multiplication.
Now, we claim that $\braider{\widehat{M_1}}{R(\beta)}$ is a morphism of right $R(\beta)$-modules. 
Let
\[
f_{\beta} = \sum_{\nu \in I^{\beta}} f(z, x_1) \cdots f(z, x_n) e(\nu) \in Z(R(\beta))[z], 
\]
where $Z(R(\beta))$ is the center of $R(\beta)$. 
Then, we have 
\[
\universalR{\widehat{M_1}}{R(\beta)} =  f_{\beta} \braider{\widehat{M_1}}{R(\beta)}
\]
Note that the leading coefficient of $f_{\beta}$ with respect to $z$ is in $Z(R(\beta))_0 \setminus \{0\} = \mathbf{k}^{\times}$
By a similar argument to the proof of Lemma \ref{lem:affinization}, $f_{\beta}$ acts on both $\widehat{M_1} \circ R(\beta)$ and $R(\beta) \circ \widehat{M_1}$ injectively. 
Since $\universalR{\widehat{M_1}}{R(\beta)}$ is right $R(\beta)$-linear, so is $\braider{\widehat{M_1}}{R(\beta)}$. 

For $X \in \gMod{R(\beta)}$, we define $\braider{\widehat{M_1}}{X}$ as 
\[
\widehat{M_1} \circ X \simeq (\widehat{M_1} \circ R(\beta)) \otimes_{R(\beta)} X \xrightarrow{\braider{\widehat{M_1}}{R(\beta)} \otimes \id_X} (R(\beta) \circ \widehat{M_1}) \otimes_{R(\beta)} X \simeq X \circ \widehat{M_1}. 
\]
It is functorial in $X$ and fulfills the first half of the requirements. 
Moreover, we have 
\[
\universalR{\widehat{M_1}}{X} = f_{\beta} \braider{\widehat{M_1}}{X}.
\]
Since $\universalR{\widehat{M_1}}{X}$ is injective for all $X \in \gMod{R(\beta)}$ by Lemma \ref{lem:affinization}, $\braider{\widehat{M_1}}{X}$ is also injective. 
\end{proof}

\begin{corollary} \label{cor:braidervsren}
Let $L$ be a simple $R(\beta)$-module. 
Then, 
\[
\Lambda(M_1, L) \in \phi(\beta) - \mathbb{Z}_{\geq 0} (\alpha_i, \alpha_i). 
\]
\end{corollary}

\begin{proof}
By the construction, $\braider{\widehat{M}_1}{L}$ is a nonzero element of $\Hom_{R(\gamma_1 + \beta)[z]} (\widehat{M}_1 \circ L, L \circ \widehat{M}_1)$. 
Theorem \ref{thm:rmatrix} (3) shows that $\braider{\widehat{M}_1}{L} \in \mathbf{k}[z] \renormalizedR{\widehat{M}_1}{L}$, which proves the corollary. 
\end{proof}

\begin{definition} \label{def:Gamma}
  We define a functor $\Gamma \colon \gMod{R(\beta)} \to \gMod{R(\beta + \gamma_1)}$ by
  \[
  \Gamma (X) = \operatorname{Cok} (\braider{\widehat{M_1}}{X} \colon q^{\phi(\beta)}\widehat{M_1} \circ X \to X \circ \widehat{M_1}). 
  \] \index{$\Gamma(X)$}
\end{definition}

\begin{corollary}
  The functor $\Gamma$ is exact. 
\end{corollary}

\begin{proof}
 By Lemma \ref{lem:braider}, $\braider{\widehat{M_1}}{X}$ is injective for all $X \in \gMod{R(\beta)}$. 
 The assertion follows from the exactness of convolution products and the snake lemma. 
\end{proof}

\subsection{Foundational short exact sequence} \label{sub:SES}

In this section, we provide a key short exact sequence. 
It allows us to relate $\widehat{M}(ws_i \Lambda_i, w\Lambda_i)$ to already analyzed modules $\widehat{M}(w\Lambda_i, \Lambda_i)$ and $\widehat{M}(ws_i\Lambda_i, \Lambda_i)$. 
Let $w \in W, i \in I$ and assume that $ws_i > w, w\Lambda_i \neq \Lambda_i$. 
Set 
\begin{align*} 
&M_1 = M(w\Lambda_i , \Lambda_i), \widehat{M_1} = \widehat{M}(w\Lambda_i, \Lambda_i), \\
&M_2 = M(ws_i\Lambda_i, w\Lambda_i), \widehat{M_2} = \widehat{M}(ws_i \Lambda_i, w \Lambda_i), \\ \index{$M_2$} \index{$\widehat{M_2}$}
&M_3 = M(ws_i \Lambda_i, \Lambda_i), \widehat{M_3} = \widehat{M}(ws_i \Lambda_i, \Lambda_i), \\ \index{$M_3$} \index{$\widehat{M_3}$}
&\gamma_1 = -\wt M_1 = \Lambda_i - w \Lambda_i, \gamma_2 = - \wt M_2 = w \alpha_i, \\
&\gamma_3 = - \wt M_3 = \Lambda_i - ws_i \Lambda_i. \index{$\gamma_2$} \index{$\gamma_3$}
\end{align*}
Here, we use the affinizations specified after Proposition \ref{prop:detaffinization}. 
Note that Lemma \ref{lem:affinefg} ensures that all the modules are finitely generated. 
For each $k \in \{ 1,2,3 \}$, let $z_k$ denote the endomorphism of the affinization $\widehat{M_k}$. 
They all have degree $(\alpha_i, \alpha_i)$.  

\begin{lemma} \label{lem:lambda}
 We have 
 \begin{enumerate}
  \item $\Lambda (M_1, M_2) = \phi(\gamma_2) =  (\alpha_i, \alpha_i) + (\gamma_1, \gamma_2)$, $(M_2,M_1)$ is unmixing and $\Lambda (M_2, M_1) =  -(\gamma_1, \gamma_2)$. 
  \item Let $L = q^{(\alpha_i, \alpha_i)/2 + (\gamma_1, \gamma_2)} M_1 \nabla M_2$. 
  It is a self-dual simple $R(\gamma_3)$-module not isomorphic to $M_3$, and we have non-split short exact sequences
  \begin{align*}
    0 \to q^{(\alpha_i, \alpha_i)/2} L \to &M_2 \circ M_1 \to M_3 \to 0, \\
    0 \to q^{-(\gamma_1, \gamma_2)} M_3 \to &M_1 \circ M_2 \to q^{-(\alpha_i, \alpha_i)/2 - (\gamma_1, \gamma_2)} L \to 0. 
  \end{align*}
  \item $\renormalizedR{\widehat{M_1}}{\widehat{M_2}} = \braider{\widehat{M_1}}{\widehat{M_2}}$ up to a nonzero scalar multiple. 
 \end{enumerate}
\end{lemma}

\begin{remark}  
  Under the assumptions that $R$ is symmetric and $\mathbf{k}$ is of characteristic zero, 
  this lemma was proved in \cite[Proposition 10.3.3]{MR3758148} via global basis computations.  
  Their proof relied on a geometric result from \cite[Theorem 10.3.1]{MR3758148}, which necessitated these hypotheses. 
  Although \cite[Proposition 4.6]{MR3771147} later generalized this auxiliary result \cite[Theorem 10.3.1]{MR3758148} to arbitrary quiver Hecke algebras over fields of arbitrary characteristic,
  and \cite[Corollary 11.5]{mcnamara2021clustermonomialsdualcanonical} established the lemma for the symmetric case in arbitrary characteristic, 
  a complete proof for the non-symmetric case over an arbitrary field has not been explicitly written in the literature.
  To fill this gap, we present here a proof using $R$-matrices. 
\end{remark}

\begin{proof}
  (1) We first show that $\phi(\gamma_2) = (\alpha_i, \alpha_i) + (\gamma_1, \gamma_2)$. 
  Since $1 = \langle h_i, \Lambda_i \rangle = 2(\alpha_i, \Lambda_i) / (\alpha_i, \alpha_i)$, we have $(\alpha_i, \Lambda_i) = (\alpha_i, \alpha_i)/2$.   
  Writing $\gamma_2 = w\alpha_i = \sum_{j \in I} n_j \alpha_j \ (n_j \in \mathbb{Z}_{\geq 0})$, we compute 
  \[
  (\gamma_1, \gamma_2) = (\Lambda_i - w\Lambda_i, \gamma_2) = n_i(\Lambda_i, \alpha_i) - (w\Lambda_i, w\alpha_i) = \frac{n_i -1}{2} (\alpha_i, \alpha_i).  
  \]
  Hence, we obtain 
  \[
  \phi(\gamma_2) = n_i (\alpha_i, \alpha_i) - (\gamma_1, \gamma_2) = \frac{n_i + 1}{2} (\alpha_i, \alpha_i), 
  \]
  which proves $\phi(\gamma_2) = (\alpha_i, \alpha_i) + (\gamma_1, \gamma_2)$.
  
  By Proposition \ref{lem:detsupport}, $(M_2, M_1)$ is unmixing. 
  Hence, Lemma \ref{lem:unmixingr} shows that $\Lambda(M_2, M_1) = -(\gamma_1, \gamma_2)$. 
  By Corollary \ref{cor:braidervsren}, $\Lambda (M_1, M_2) \in \phi(\gamma_2) - \mathbb{Z}_{\geq 0} (\alpha_i, \alpha_i) = (\alpha_i,\alpha_i) + (\gamma_1,\gamma_2) - \mathbb{Z}_{\geq 0} (\alpha_i,\alpha_i)$.
  On the other hand, Corollary \ref{cor:lambdadiscrete} shows $\Lambda (M_1, M_2) + \Lambda (M_2, M_1) \geq 0$.
  Therefore, there are only two possibilities: 
  $\Lambda(M_1, M_2) = (\gamma_1, \gamma_2) + (\alpha_i, \alpha_i)$ or $\Lambda(M_1, M_2) = (\gamma_1, \gamma_2)$. 

  We claim that the second case cannot occur. 
  We repeat the argument from \cite[Proposition 10.3.3]{MR3758148} verbatim.  
  Suppose that $ \Lambda (M_1, M_2) = (\gamma_1, \gamma_2)$.
  Then, we have $\Lambda(M_1, M_2) + \Lambda(M_2,M_1) = 0$, which implies that $M_2 \circ M_1$ is simple by Proposition \ref{prop:commuting} (1).
  Hence, $M_2 \circ M_1 \simeq M_2 \nabla M_1$. 
  Moreover, Proposition \ref{prop:dethead} shows that $M_2 \nabla M_1 \simeq M_3$. 
  Using the Mackey filtration, we obtain $\varepsilon_j^* (M_3) = \varepsilon_j^* (M_2 \circ M_1) = \varepsilon_j^* (M_2) + \varepsilon_j^* (M_1)$ for all $j \in I$. 
  Here, we have $\varepsilon_j^* (M_3) = \delta_{i, j} = \varepsilon_j^* (M_1) \ (j \in I)$ by Lemma \ref{lem:determinantialepsilon}, hence $\varepsilon_j^*(M_2) = 0$ for all $j \in I$. 
  It contradicts $\wt M_2 = - w\alpha_i \neq 0$ by Theorem \ref{thm:categoricalcrystal}. 

  (2) By (1), $\Lambda (M_1, M_2) + \Lambda (M_2, M_1) = (\alpha_i, \alpha_i)$, which is equal to $\deg z_1$. 
  Thus, the length of $M_1 \circ M_2$ is exactly 2 by Proposition \ref{prop:commuting} (2). 
  The degree shifts are determined by Lemma \ref{lem:degreeshift} and the fact that $D(M_1 \circ M_2) \simeq q^{(\gamma_1,\gamma_2)}M_2 \circ M_1$. 
  By Theorem \ref{thm:rmatrix} (6), $L$ is not isomorphic to $M_3$. 
 
  (3) By (1), we have $\deg (\renormalizedR{\widehat{M_1}}{\widehat{M_2}}) = \Lambda (M_1, M_2) = \phi (\gamma_2) = \deg (\braider{\widehat{M_1}}{\widehat{M_2}})$. 
  The assertion follows from Theorem \ref{thm:rmatrix} (3). 
\end{proof}

\begin{theorem} \label{thm:SES}
 Let $i \in I$ and $w \in W$ with $ws_i > w$. 
 Assume that $w\Lambda_i \neq \Lambda_i$. 
 Then, there exists a short exact sequence
\[
0 \to q^{(\alpha_i, \alpha_i) + (\gamma_1, \gamma_2)}\widehat{M_1} \circ \widehat{M_2} \to \widehat{M_2} \circ \widehat{M_1} \to \widehat{M_3} \to 0, 
\]
where the injection is $\renormalizedR{\widehat{M_1}}{\widehat{M_2}}$. 
\end{theorem}

\begin{remark}
  Using the exact functor $\Gamma$ from Definition \ref{def:Gamma}, the proposition is rephrased as 
  \[
  \Gamma (\widehat{M_2}) \simeq \widehat{M_3}. 
  \]
\end{remark}

\begin{proof}
 We prove the proposition in three steps. 
 First, we directly construct a short exact sequence in the case of type $A_2$ quiver Hecke algebra $R^{A_2}$. 
 Second, we introduce an exact monoidal functor $F$ from $\gMod{R^{A_2}}$ to $\gMod{R}$. 
 Finally, we show that $F$ sends the short exact sequence obtained in Step 1 to the desired short exact sequence in $\gMod{R}$.  

 Step 1. We consider the root datum of type $A_2$: $I = \{ 1,2\}$ and $(\alpha_1, \alpha_1) = (\alpha_2, \alpha_2) = 2, (\alpha_1,\alpha_2) = -1$. 
 Set $\alpha_3 = \alpha_1 + \alpha_2$, the highest root.  
 Let $Q_{1,2} (u,v) = Q_{2,1} (v,u) = au + bv$ for some $a, b \in \mathbf{k}^{\times}$. 
 Let $R^{A_2}$ denote the quiver Hecke algebra associated with these datum.
 We need to introduce a slightly generalized grading on $R^{A_2}$ following \cite[Section 7.1]{MR4717658}. 
 Let $d \in \mathbb{Z}_{> 0}$ and 
 $(\lambda_{1,2}, \lambda_{2,1})$ be a pair of integers satisfying $\lambda_{1,2} + \lambda_{2,1} = d$.  
 Then, $R^{A_2}$ is graded by setting 
 \[
 \deg x_1 = \deg x_2 = d, \deg e(1,2) \tau_1 = \lambda_{1,2}, \deg e(2,1) \tau_1 = \lambda_{2,1}.   
 \]

Note that $R^{A_2}(\alpha_1) \circ R^{A_2}(\alpha_2) = R^{A_2}(\alpha_3)e(1,2)$ and $R^{A_2}(\alpha_2) \circ R^{A_2}(\alpha_1) = R^{A_2}(\alpha_3) e(2,1)$. 
Put $L^{A_2}(\alpha_3) = L^{A_2}(2) \nabla L^{A_2}(1)$. 
It is a self-dual one dimensional module on which $e(1,2)$ acts by zero. 
Note that $(L^{A_2}(1), L^{A_1}(2))$ is unmixing. 
Set $\operatorname{R} = \Rmatrix{R^{A_2}(\alpha_1)}{R^{A_2}(\alpha_2)} \colon q^{\lambda_{1,2}} R^{A_2}(\alpha_1) \circ R^{A_2}(\alpha_2) \to R^{A_2}(\alpha_2) \circ R^{A_2}(\alpha_1)$.
By the definition, it is given by the right multiplication by $e(1,2)\tau_1$. 
As left $\mathbf{k}[x_1,x_2]$-modules, 
\begin{align*}
R^{A_2}(\alpha_1) \circ R^{A_2}(\alpha_2) = \mathbf{k}[x_1,x_2] e(1,2) \oplus \mathbf{k}[x_1,x_2] \tau_1 e(1,2), \\
R^{A_2}(\alpha_2) \circ R^{A_2}(\alpha_1) = \mathbf{k}[x_1,x_2] e(2,1) \oplus \mathbf{k}[x_2,x_1] \tau_1 e(2,1). 
\end{align*}
Since $\tau_1 e(1,2) \tau_1 = Q_{2,1} (x_1, x_2) e(2,1) = (ax_2 + bx_1) e(2,1)$, we obtain
\[
\operatorname{Cok} (\operatorname{R}) \simeq \mathbf{k}[x_1, x_2]e(2,1) /(ax_2 + bx_1)\mathbf{k}[x_1, x_2] e(2,1). 
\]  
Hence, $\operatorname{R}$ is injective and  $\operatorname{Cok} (\operatorname{R})$ is an affine object of $L^{A_2}(\alpha_3)$. 
Setting $\widehat{L}^{A_2}(\alpha_3) = \Cok (\operatorname{R})$, we obtain a short exact sequence 
\[
 0 \to q^{\lambda_{1,2}} R^{A_2}(\alpha_1) \circ R^{A_2}(\alpha_2) \xrightarrow{R} R^{A_2}(\alpha_2) \circ R^{A_2}(\alpha_1) \to \widehat{L}^{A_2}(\alpha_3) \to 0.
\]

 Step 2. 
 Based on Lemma \ref{lem:lambda}, we apply a general method developed in \cite[Section 7.4]{MR4717658} to define a monoidal functor $F \colon \gMod{R^{A_2}} \to \gMod{R}$. 
 For the sake of completeness, we outline how it works in the current situation. 
 We need a duality datum \cite[Section 7.4]{MR4717658}, which consists of a family of affinizations and R-matrices between them. 
 In our setup, the affinizations are $\widehat{M_1}$ and $\widehat{M_2}$.
 Renormalized R-matrices are unique up to nonzero scalar multiple and we make the following choice: 
 let $\renormalizedR{\widehat{M_1}}{\widehat{M_2}} = \braider{\widehat{M_1}}{\widehat{M_2}}$,  
 and fix $\renormalizedR{\widehat{M_2}}{\widehat{M_1}}$ in any way. 
 For $r \in \{1,2\}$, we choose $\renormalizedR{\widehat{M_r}}{\widehat{M_r}}$ so that it yields $\id_{M_r \circ M_r}$ when specialized to $z_r = 0$. 
 Note that $\renormalizedR{\widehat{M_r}}{\widehat{M_r}} \neq \id_{\widehat{M_r} \circ \widehat{M_r}}$, as explained in Remark \ref{rem:rennotid}. 
 There uniquely exists $T_r \in \End_{R(2\gamma_r)}(\widehat{M_r} \circ \widehat{M_r})$ such that $\renormalizedR{\widehat{M_r}}{\widehat{M_r}} - \id = (z_r \otimes 1 - 1 \otimes z_r) \circ T_r$ by \cite[Proposition 6.13]{MR4717658}. 
 We put $T_{1,2} = \renormalizedR{\widehat{M_1}}{\widehat{M_2}}, T_{2,1} = \renormalizedR{\widehat{M_2}}{\widehat{M_1}}, T_{1,1} = T_1, T_{2,2} = T_2$. 
 By Proposition \ref{prop:affinermatrix}, there exists $Q_{1,2}(u,v), Q_{2,1}(u,v) \in \mathbf{k}[u,v]$ uniquely determined by 
 \[
 \renormalizedR{\widehat{M}_2}{\widehat{M}_1} \renormalizedR{\widehat{M}_1}{\widehat{M}_2} = Q_{1,2}(z_1, z_2) \id_{\widehat{M}_1 \circ \widehat{M}_2}, \renormalizedR{\widehat{M}_1}{\widehat{M}_2} \renormalizedR{\widehat{M}_2}{\widehat{M}_1} = Q_{2,1}(z_2,z_1) \id_{\widehat{M}_2 \circ \widehat{M}_1}. 
 \]
 Since $Q_{1,2}(z_1,z_2) \renormalizedR{\widehat{M}_2}{\widehat{M}_1} = \renormalizedR{\widehat{M}_2}{\widehat{M}_1} \renormalizedR{\widehat{M}_1}{\widehat{M}_2} \renormalizedR{\widehat{M}_2}{\widehat{M}_1} = Q_{2,1}(z_2,z_1) \renormalizedR{\widehat{M}_2}{\widehat{M}_1}$, Proposition \ref{prop:affinermatrix} imply $Q_{1,2}(u,v) = Q_{2,1}(v,u)$. 
 Since both $\widehat{M}_1$ and $\widehat{M}_2$ are affinizations, Remark \ref{rem:univRvsnormR} shows $Q_{1,2}(u,0) \neq 0, Q_{1,2}(0,v) \neq 0$.  
 Moreover, $Q_{1,2}(z_1, z_2)$ is a homogeneous polynomial of degree $(\alpha_i, \alpha_i)$ by Lemma \ref{lem:lambda}. 
 Hence $Q_{1,2}(u,v)$ is of the form $au + bv$ for some $a,b \in \mathbf{k}^{\times}$. 
 We consider the quiver Hecke algebra $R^{A_2}$ associated with this $(Q_{1,2}, Q_{2,1})$.

 Using the duality datum, we define $(R, R^{A_2})$-modules.
 For $\alpha = m_1 \alpha_1 + m_2 \alpha_2 \in Q_+^{A_2} = \mathbb{Z}_{\geq 0} \alpha_1 + \mathbb{Z}_{\geq 0} \alpha_2$ of height $n = m_1 + m_2$, we put
 \[
 K(\alpha) = \bigoplus_{\nu \in \{1,2\}^{\alpha}} K(\nu)\in \gMod{R(\beta)}, \ K(\nu) = \widehat{M_{\nu_1}} \circ \widehat{M_{\nu_2}} \circ \cdots \circ \widehat{M_{\nu_n}}, 
 \]
 where $\beta = m_1 \gamma_1 + m_2 \gamma_2$. 
 The right $R^{A_2} (\alpha)$-action is given as follows: 
 $e(\nu)$ is the projection $K(\alpha) \to K(\nu) \hookrightarrow K(\alpha)$, 
 $e(\nu) x_k \colon K(\nu) \to K(\nu)$ is $\id_{K_{\nu_1}} \otimes \cdots \otimes z_{\nu_k} \otimes \cdots \otimes \id_{K_{\nu_n}}$ and 
 $e(\nu) \tau_k \colon K(\nu) \to K(s_k \nu)$ is $\id_{K_{\nu_1}} \otimes \cdots \otimes T_{\nu_k, \nu_{k+1}} \otimes \cdots \otimes \id_{K_{\nu_n}}$. 
 We put a $\mathbb{Z}$-grading on $R^{A_2}$ by setting $d = (\alpha_i, \alpha_i), \lambda_{1,2} = \Lambda (M_1, M_2), \lambda_{2,1} = \Lambda (M_2, M_1)$. 
 Note that the condition $\lambda_{1,2} + \lambda_{2,1} = d$ is satisfied by Lemma \ref{lem:lambda}. 
 Then, $K(\alpha)$ is a graded $(R(\beta), R^{A_2}(\alpha))$-module: in fact, it is similar to the polynomial representation \cite[Proposition 2.3]{MR2525917}. 
 Moreover, Lemma \ref{lem:affinefg} shows that $K(\alpha)$ is finitely generated as a left $R(\beta)$-module. 
 Hence, we obtain a right exact functor $F \colon \gMod{R^{A_2} (\alpha)} \to \gMod{R(\beta)}$ defined by $F(X) = K(\alpha) \otimes_{R^{A_2}(\alpha)} X$. 

 The functor $F$ is exact. 
 In fact, by \cite[Proposition  3.7]{MR3748315} \cite[Proposition 7.6]{MR4717658}, it follows from the fact that the global dimension of $\gMod{R^{A_2}(\alpha)}$ is finite \cite[Theorem 4.7]{MR3205728}.  
 (See also Section \ref{sub:interlude}, where we offer a generalization of this argument.)
 Actually, we only need the exactness of $F$ for $\alpha \in \{0, \alpha_1, \alpha_2, \alpha_3\}$. 
 For these $\alpha$, it is possible to directly verify that $\gldim \gMod{R^{A_2}(\alpha)}$ is finite without referring to \cite[Theorem 4.7]{MR3205728}. 

Step 3. Applying the exact functor $F$ to the short exact sequence in Step 1, we obtain a short exact sequence in $\gMod{R(\gamma_3)}$
\[
0 \to q^{\lambda_{1,2}} \widehat{M_1} \circ \widehat{M_2} \xrightarrow{\renormalizedR{\widehat{M_1}}{\widehat{M_2}}} \widehat{M_2} \circ \widehat{M_1} \to F(\widehat{L}^{A_2}(\alpha_3)) \to 0. 
\]
We claim that $F(\widehat{L}^{A_2}(\alpha_3))$ is isomorphic to $\widehat{M_3}$. 
Since $F(x_1\vert_{R^{A_2}(\alpha_1)}) = z_1 \vert_{\widehat{M_1}}$, we have 
\[
F(L^{A_2}(1)) \simeq F(R^{A_2}(\alpha_1)/x_1R^{A_2}(\alpha_1)) \simeq \widehat{M_1}/z_1 \widehat{M_1} \simeq M_1. 
\]
Similarly, we have $F(L^{A_2}(2)) \simeq M_2$. 
Through the surjections $R^{A_2}(\alpha_1) \to L^{A_2}(1)$ and $R^{A_2}(\alpha_2) \to L^{A_2}(2)$, 
the homomorphism 
\[
\Rmatrix{R^{A_2}(\alpha_2)}{R^{A_2}(\alpha_1)} \colon q^{\lambda_{2,1}} R^{A_2}(\alpha_2) \circ R^{A_2}(\alpha_1) \to R^{A_2}(\alpha_1) \circ R^{A_2}(\alpha_2) 
\]
induces
\[
\Rmatrix{L^{A_2}(2)}{L^{A_2}(1)} \colon q^{\lambda_{2,1}} L^{A_2}(2) \circ L^{A_2}(1) \to L^{A_2} (1) \circ L^{A_2} (2). 
\] 
Hence, we obtain $F(\Rmatrix{L^{A_2}(2)}{L^{A_2}(1)}) = \renormalizedR{\widehat{M_2}}{\widehat{M_1}} \vert_{z_1 =  z_2 = 0} = \mathbf{r}_{M_2, M_1}$, which implies that 
\[
F(L^{A_2}(\alpha_3)) = F(\Image (\Rmatrix{L^{A_2}(2)}{L^{A_2}(1)})) \simeq \Image (\mathbf{r}_{M_2, M_1}) = M_2 \nabla M_1 = M_3.  
\]
Since $\widehat{L}^{A_2}(\alpha_3) \in \gMod{R^{A_2}(\alpha_3)}$ is an affine object of $L^{A_2}(\alpha_3)$, we have $[\widehat{L}^{A_2}(\alpha_3)]_q = \sum_{k=0}^{\infty} q^{dk} [L^{A_2}(\alpha_3)]$.
Hence, $[F(\widehat{L}^{A_2}(\alpha_3))]_q = \sum_{k=0}^{\infty} q^{dk} [M_3]$. 
It implies that $F(\widehat{L}^{A_2}(\alpha_3))$ belongs to $\gMod{R'(ws_i\Lambda_i)}$, where $R'(ws_i \Lambda_i)$ is the algebra defined in the proof of Proposition \ref{prop:cyclotomic}. 
By Remark ~\ref{rem:morita}, $R'(ws_i \Lambda_i)$ is Morita equivalent to $\mathbf{k}[z]$ as a graded ring, where $\deg z = d = (\alpha_i, \alpha_i)$. 
Let $S$ denote the graded $\mathbf{k}[z]$-module corresponding to $F(\widehat{L}^{A_2}(\alpha_3))$ through this Morita equivalence. 
It is enough to show that $S \simeq \mathbf{k}[z]$ as graded $\mathbf{k}[z]$-modules.

Since $\widehat{L}^{A_2}(\alpha_3)$ is an affine object of $L^{A_2}(\alpha_3)$, we have an endomorphism $f$ of $\widehat{L}^{A_2}(\alpha_3)$ such that
\[
0 \to q^d \widehat{L}^{A_2}(\alpha_3) \xrightarrow{f} \widehat{L}^{A_2}(\alpha_3) \to L^{A_2}(\alpha_3) \to 0
\]
is an exact sequence. 
Applying $F$ and the Morita equivalence to this short exact sequence, we obtain an endomorphism $g$ of $S$ such that
\[
0 \to q^d S \xrightarrow{g} S \to \mathbf{k} \to 0
\]
is a short exact sequence. 
Since we are working in the category of finitely generated graded modules, $S$ is a finite direct sum of modules of the form $q^l \mathbf{k}[z] \ (l \in \mathbb{Z})$ and $q^l \mathbf{k}[z]/(z^m) \ (l, m\in \mathbb{Z}, m \geq 1)$. 
Let $S_{\text{tor}}$ be the torsion part of $S$.
Then, we have $g(S_{\text{tor}}) \subset S_{\text{tor}}$. 
On the other hand, $S_{\text{tor}}$ is bounded above graded module and $g$ is an injective morphism of positive degree.
Therefore, we must have $S_{\text{tor}}=0$. 
It follows that $S$ is torsion free, that is, a finite direct sum of $q^{l} \mathbf{k}[z] \ (l \in \mathbb{Z})$. 
Using $[S]_q = \sum_{k=0}^{\infty} q^{dk} [\mathbf{k}] = [\mathbf{k}[z]]_q$, we obtain $S \simeq \mathbf{k}[z]$ and the Proposition is proved. 
\end{proof}

\begin{remark}
 The short exact sequence constructed in Step 1 is a special case of this theorem. 
 Assume that our quiver Hecke algebra $R$ is $R^{A_2}$ with the usual grading: $d = 2, \lambda_{1,2} = \lambda_{2,1} = 1$.  
 Additionally, assume that $i = 1, w = s_1 s_2$. 
 Then, $M_1 \simeq L^{A_2}(1), M_2 \simeq L^{A_2}(2), M_3 \simeq L^{A_2}(\alpha_3), \widehat{M_1} \simeq R^{A_2}(\alpha_1)$ and $\widehat{M_2} \simeq R^{A_2}(\alpha_2)$.
 It follows that the functor $F \colon \gMod{R^{A_2}} \to \gMod{R} = \gMod{R^{A_2}}$ is the identity functor. 
 Hence, Step 3 shows that $\widehat{L}^{A_2}(\alpha_3) \simeq F(\widehat{L}^{A_2}(\alpha_3)) \simeq \widehat{M_3}$ and the short exact sequence constructed in Step 1 is of the required form. 
\end{remark}

\subsection{$\Ext^1$-vanishing for $\widehat{M}(ws_i \Lambda_i, w \Lambda_i)$} \label{sub:Ext1vanishing}

We maintain the setting from the previous section. 
In this section, we establish a homological property of $\widehat{M_2} = \widehat{M}(ws_i \Lambda_i, w\Lambda_i)$.  
We begin with examining the head of $\widehat{M_2}$. 

\begin{proposition} \label{prop:headM2}
We have $\hd (\widehat{M_2} \circ \widehat{M_1}) = M_2 \nabla M_1$ and $\hd \widehat{M_2} \simeq M_2$, 
where $\hd$ stands for the head in $\gMod{R}$. 
\end{proposition}

Note that if we consider the head of $\widehat{M_2}$ as an $R(\gamma_2) \otimes \mathbf{k}[z_2]$-module not as an $R(\gamma_2)$-module, the assertion is obvious. 
We need the following lemma. 

\begin{lemma} \label{lem:headM2}
Let $J$ be the two-sided ideal of $R(\gamma_3)$ generated by $e(*, j) \ (j \in I \setminus \{i\})$ and $e(*, i, i)$. 
Then 
\begin{enumerate}
  \item $\widehat{M_1} \circ \widehat{M_2} = J(\widehat{M_1} \circ \widehat{M_2})$. 
  \item $\Image (\renormalizedR{\widehat{M_1}}{\widehat{M_2}}) = J(\widehat{M_2}\circ \widehat{M_1}),\ (\widehat{M_2}\circ \widehat{M_1}) / J(\widehat{M_2}\circ \widehat{M_1}) \simeq \widehat{M_3}$. 
\end{enumerate}
\end{lemma}

\begin{proof}
  Note that $R(\gamma_3)/J$ is $R'(ws_i \Lambda_i)$ from Section \ref{sub:homologicalM1}.
  
(1) Recall that $M_1 \circ M_2$ is of length two with two different composition factors $M_3$ and $M_1 \nabla M_2$ by Lemma \ref{lem:lambda}.
By Remark \ref{rem:morita}, $M_3$ is the unique simple graded $R'(ws_i\Lambda_i)$-module up to isomorphism and degree shift. 
It follows that $M_1 \nabla M_2$ is not an $R'(ws_i \Lambda_i)$-module, that is, $J (M_1 \nabla M_2) \neq 0$. 
Since $M_1 \nabla M_2$ is the head of $M_1 \circ M_2$ and is simple, we have $J(M_1 \circ M_2) = M_1 \circ M_2$. 
  
Next, consider the short exact sequence 
\[
0 \to q^d \widehat{M_1} \circ M_2 \xrightarrow{z_1} \widehat{M_1} \circ M_2 \to M_1 \circ M_2 \to 0. 
\]
Applying the functor $R'(ws_i\Lambda_i) \otimes_{R(\gamma_3)} (\cdot)$ to it, we obtain an exact sequence
\[
q^d \widehat{M_1} \circ M_2 / J(\widehat{M_1} \circ M_2) \xrightarrow{z_1} \widehat{M_1} \circ M_2 / J(\widehat{M_1} \circ M_2) \to M_1 \circ M_2 / J(M_1 \circ M_2) = 0. 
\]
Since $\widehat{M_1} \circ M_2 / J(\widehat{M_1} \circ M_2)$ is bounded below graded module, it must be zero. 
To complete the proof, we repeat the same argument for the short exact sequence
\[
0 \to q^d \widehat{M_1} \circ \widehat{M_2} \xrightarrow{z_2} \widehat{M_1} \circ \widehat{M_2} \to \widehat{M_1} \circ M_2 \to 0. 
\]

(2) By (1), we have $\Image {\renormalizedR{\widehat{M_1}}{\widehat{M_2}}} \subset J(\widehat{M_2} \circ \widehat{M_1})$. 
On the other hand, we have $\Cok \renormalizedR{\widehat{M_1}}{\widehat{M_2}} \simeq \widehat{M_3}$ by Proposition \ref{thm:SES}. 
Since $J \widehat{M_3} = 0$, it follows that $J (\widehat{M_2} \circ \widehat{M_1}) \subset \Image \renormalizedR{\widehat{M_1}}{\widehat{M_2}}$. 
\end{proof}

\begin{proof}[Proof of Proposition \ref{prop:headM2}]
 We prove the first assertion. The second one follows from it. 

 We claim that $\Hom_{R(\gamma_3)} (\widehat{M_2} \circ \widehat{M_1}, M_1 \nabla M_2) = 0$. 
 In fact, 
 \begin{align*}
 &\Hom_{R(\gamma_3)} (\widehat{M_2}\circ \widehat{M_1}, M_1 \nabla M_2) \\
 &= \Hom_{R(\gamma_2) \otimes R(\gamma_1)} (\widehat{M_2} \otimes \widehat{M_1}, \Res_{\gamma_2, \gamma_1} M_1 \nabla M_2) \quad \text{by the $\Ind$-$\Res$ adjunction} \\ 
 &= \Hom_{R(\gamma_2) \otimes R(\gamma_1)} (\widehat{M_2} \otimes \widehat{M_1}, 0) \quad \text{by Lemma \ref{lem:unmixing} (2) and Lemma \ref{lem:lambda}} \\
 &= 0. 
\end{align*}

On the other hand, we compute using Lemma \ref{lem:headM2}
\begin{align*}
\Hom_{R(\gamma_3)} (\widehat{M_2}\circ \widehat{M_1}, M_3) &= \Hom_{R(\gamma_3)} (\widehat{M_2}\circ \widehat{M_1}/J(\widehat{M_2}\circ \widehat{M_1}), M_3) \\ 
&\simeq \Hom_{R(\gamma_3)} ( \widehat{M_3}, M_3), 
\end{align*}
which is one dimensional by Remark \ref{rem:morita}. 

Since the composition factors of $\widehat{M_2} \circ \widehat{M_1}$ are $M_1 \nabla M_2$ and $M_3$, the first assertion follows. 
\end{proof}

Now, we prove the following homological property of $\widehat{M_2}$.

\begin{theorem} \label{thm:ext1}
 $\widehat{M_2}$ is the projective cover of $M_2$ in $\gMod{R_{ws_i, w}(\gamma_2)}$. 
 In particular, we have 
 \[
\Ext_{R(\gamma_2)}^1(\widehat{M_2}, M_2) = 0, \ \qdim \Hom_{R(\gamma_2)}^1 (\widehat{M_2},M_2) = 1. 
 \]
\end{theorem}

\begin{proof}
We make use of the exact functor $\Gamma \colon \gMod{R(\gamma_2)} \to \gMod{R(\gamma_3)}$ from Definition \ref{def:Gamma}. 
Let $P$ be the projective cover of $M_2$ in $\gMod{R_{ws_i,w}(\gamma_2)}$. 
It is enough to show that $\widehat{M_2}$ is isomorphic to $P$.

By Theorem \ref{thm:SES}, we have $\Gamma (\widehat{M_2}) \simeq \widehat{M_3}$. 
Applying $\Gamma$ to the short exact sequence
\[
0 \to q^d \widehat{M_2} \xrightarrow{z_2} \widehat{M_2} \to M_2 \to 0, 
\]
we obtain a short exact sequence
\[
0 \to q^d \widehat{M_3} \xrightarrow{\Gamma (z_2)} \widehat{M_3} \to \Gamma (M_2) \to 0. 
\]
By Remark \ref{rem:morita}, $\End_{R(\gamma_3)} (\widehat{M_3}) \simeq \mathbf{k}[z_3]$. 
Thus, $\Gamma (z_2)$ coincides with $z_3$ up to a nonzero scalar multiple, which implies $\Gamma (M_2) \simeq M_3$. 

Next, we examine $\Gamma (P)$. 
Since $\Gamma(M_2) \simeq M_3$ and $M_2$ is the unique composition factor of $P$, 
the composition factors of $\Gamma(P)$ consist of only $M_3$. 
In other words, we have $\Gamma(P) \in \gMod{R'(ws_i \Lambda_i)}$. 

We claim that $\hd \Gamma(P) \simeq M_3$. 
Since there are surjective homomorphisms 
\[
P \circ \widehat{M_1} \twoheadrightarrow \Gamma(P) \twoheadrightarrow \Gamma (M_2) \simeq M_3, 
\]
it is enough to show that $\hd (P \circ \widehat{M_1}) \simeq M_3$. 
To see this, we compute
\begin{align*}
  &\Hom_{R(\gamma_3)} (P \circ \widehat{M_1}, M_3) \\
  &\simeq \Hom_{R(\gamma_2) \otimes R(\gamma_1)} (P \otimes \widehat{M_1}, \Res_{\gamma_2, \gamma_1} M_3)  \quad \text{by the $\Ind$-$\Res$ adjunction} \\
  &\simeq \Hom_{R(\gamma_2)} (P, M_2) \otimes \Hom_{R(\gamma_1)} (\widehat{M_1}, M_1)  \quad \text{by Lemma \ref{lem:unmixing} and Lemma \ref{lem:lambda}} \\
  &\simeq \mathbf{k}, \\
  & \Hom_{R(\gamma_3)} (P \circ \widehat{M_1}, M_1 \nabla M_2)  \\
  &\simeq \Hom_{R(\gamma_2) \otimes R(\gamma_1)} (P \otimes \widehat{M_1}, \Res_{\gamma_2, \gamma_1} M_1 \nabla M_2)  \quad  \text{by the $\Ind$-$\Res$ adjunction} \\
  &\simeq \hom_{R(\gamma_2) \otimes R(\gamma_1)} (P \otimes \widehat{M_1}, 0) \quad \text{by Lemma \ref{lem:unmixing} and Lemma \ref{lem:lambda}} \\
  &= 0. 
\end{align*}
Since the composition factors of $P \circ \widehat{M_1}$ is $M_3$ and $M_1 \nabla M_2$, it proves the claim. 

Hence, we have a surjective homomorphism $f \colon \widehat{M_3} \to \Gamma(P)$ by Remark \ref{rem:morita}. 
By the definition of $P$, there exists a surjective homomorphism $g \colon P \to \widehat{M_2}$. 
Applying $\Gamma$, we obtain a surjective homomorphism $\Gamma (g) \colon \Gamma (P) \to \widehat{M_3}$. 
Then, $\Gamma (g) \circ f \in \End_{R(\gamma_3)} (\widehat{M_3})$ is a surjective homomorphism of degree zero.
Since $\End_{R(\gamma_3)} (\widehat{M_3}) \simeq \mathbf{k}[z]$ by Remark \ref{rem:morita}, it must be an isomorphism. 
It implies that both $f$ and $\Gamma(g)$ are isomorphisms.
In particular, $\Ker \Gamma(g) \simeq \Gamma (\Ker (g))$ is zero. 
Since $\Ker (g)$ is a finitely generated module whose composition factors are all $M_2$, 
the results $\Gamma(M_2) \simeq M_3 \neq 0$ and $\Gamma(\Ker (g)) = 0$ imply that $\Ker (g) = 0$.
Therefore, $g \colon P \to \widehat{M_2}$ is an isomorphism and the assertion follows. 
\end{proof}

\subsection{Finiteness of projective dimensions} \label{sub:finiteness}

In this section, we demonstrate that every object in $\gMod{R_{w,*}}$ has finite projective dimension as an object of $\gMod{R}$.  

\begin{lemma} \label{lem:projdimind}
  Let $\beta, \gamma \in Q_+$, $X$ an $R(\beta)$-module and $Y$ an $R(\gamma)$-module. 
  If both $\prdim_{R(\beta)} X$ and $\prdim_{R(\gamma)} Y$ are finite, then $\prdim_{R(\beta+\gamma)} (X\circ Y)$ is also finite. 
\end{lemma}

\begin{proof}
It follows from the fact that convolution product of two projective modules is projective and the exactness of the induction. 
\end{proof}

\begin{lemma} \label{lem:prdim1}
 Let $A = \oplus_{d \geq 0} A_d$ be a Noetherian graded ring with $A_0 = \mathbf{k}$, and assume that $A$ is Laurentian. 
 Then, every cyclotomic quiver Hecke algebra $R_A^{a_{\Lambda}}(\lambda)$ from Definition \ref{def:cyclotomic} has finite projective dimension as a graded $R(\Lambda - \lambda)$-module. 
\end{lemma}

Put $R_A(\beta) = R(\beta) \otimes A$. 
In the proof, we consider convolution products over $A$. 
Namely, for an $R_A(\beta)$-module $X$ and an $R_A(\gamma)$-module $Y$, we define
\[
X \circ_A Y = R_A(\beta + \gamma) e(\beta, \gamma) \otimes_{R_A(\beta) \otimes_A R_A(\gamma)} (X \otimes_A Y).  
\]

\begin{proof}
 We proceed by induction on the weight $\lambda$. 
 If $\lambda = \Lambda$, $R(\Lambda - \lambda) = \mathbf{k}$, so the assertion is trivial. 
 Let $\lambda \neq \Lambda$. 
 By \cite[Theorem 4.7]{MR2995184}, we have a short exact sequence
 \begin{align*}
 0 &\to q^{(\alpha_j, \Lambda + \lambda + \alpha_j)} R_A^{a_{\Lambda}}(\lambda +\alpha_j) \circ_A R_A(\alpha_j) \to R_A(\alpha_j) \circ_A R_A^{a_{\Lambda}}(\lambda+\alpha_j) \\ 
 &\to R_A^{a_{\Lambda}}(\lambda)e(j,*) \to 0, 
 \end{align*}
for any $j \in I$ that satisfies $\lambda + \alpha_j \in \Lambda - Q_+$.  
Since $R_A^{a_{\Lambda}}(\lambda + \alpha_j) \circ_A R_A(\alpha_j) \simeq R_A^{a_{\Lambda}} (\lambda + \alpha_j) \circ R(\alpha_j)$ and 
$R_A(\alpha_j) \circ_A R_A^{a_{\Lambda}}(\lambda + \alpha_j) \simeq R(\alpha_j) \circ R_A^{a_{\Lambda}}(\lambda + \alpha_j)$, 
they have finite projective dimension by the induction hypothesis and Lemma \ref{lem:projdimind}. 
Hence, the short exact sequence implies that the projective dimension of $R_A^{a_{\Lambda}} (\lambda)e(j,*)$ is also finite. 
Since it holds for arbitrary $j \in I$ that satisfies $\lambda + \alpha_j \in \Lambda - Q_+$, we see that the  projective dimension of $R_A^{a_{\Lambda}} (\lambda)$ is finite. 
\end{proof}

\begin{lemma} \label{lem:prdim2}
  Let $w \in W, i\in I$ and assume that $ws_i > w$. 
  Then, both $\prdim_{R(w\alpha_i)} \widehat{M}(ws_i \Lambda_i, w\Lambda_i)$ and $\prdim_{R(w\alpha_i)} M(ws_i \Lambda_i, w\Lambda_i)$ are finite. 
\end{lemma}

\begin{proof}
If $w \Lambda_i = \Lambda_i$, it is reduced to Lemma \ref{lem:prdim1}. 
 We assume $w\Lambda_i \neq \Lambda_i$. 
 We use the notation in Section \ref{sub:SES}. 
 Let $X \in \gMod{R(\gamma_2)}$. 
 The short exact sequence in Theorem \ref{thm:SES} yields the following exact sequence for $k \in \mathbb{Z}_{\geq 0}$: 
 \begin{align*}
 &\Ext_{R(\gamma_3)}^k (\widehat{M_3}, M_1 \circ X) \\
 &\to \Ext_{R(\gamma_3)}^k (\widehat{M_2}\circ \widehat{M_1}, M_1 \circ X) \to \Ext_{R(\gamma_3)}^k (q^{(\alpha_i, \alpha_i) + (\gamma_1, \gamma_2)}\widehat{M_1}\circ \widehat{M_2}, M_1 \circ X) \\
 &\to \Ext_{R(\gamma_3)}^{k+1} (\widehat{M_3}, M_1 \circ X).  
 \end{align*}
 Note that by Remark \ref{rem:morita} and Lemma \ref{lem:prdim1}, $\prdim_{R(\gamma_3)} (\widehat{M_3})$ is finite. 
 Assume that $k > \prdim_{R(\gamma_3)} (\widehat{M_3})$. 
 Then, both the leftmost term and the rightmost term vanish.
 Hence, we obtain an isomorphism
 \[
 \Ext_{R(\gamma_3)}^k (\widehat{M_2}\circ \widehat{M_1}, M_1 \circ X) \simeq \Ext_{R(\gamma_3)}^k (q^{(\alpha_i, \alpha_i) + (\gamma_1, \gamma_2)}\widehat{M_1}\circ \widehat{M_2}, M_1 \circ X). 
 \]  
 We compute these extension groups using $M_1 \circ X \simeq q^{-(\gamma_1, \gamma_2)}\Coind (X \otimes M_1)$.
 The left-hand side is 
 \begin{align*}
 &q^{-(\gamma_1,\gamma_2)}\Ext_{R(\gamma_2) \otimes R(\gamma_1)}^k (\Res_{\gamma_2,\gamma_1} (\widehat{M}_2 \circ \widehat{M}_1), X \otimes M_1) \\
 &\quad \text{by the restriction-coinduction adjunction} \\ 
 &\simeq q^{-(\gamma_1, \gamma_2)} \Ext_{R(\gamma_2)\otimes R(\gamma_1)}^k (\widehat{M_2} \otimes \widehat{M_1}, X \otimes M_1) \\ 
 &\quad \text{by Lemma \ref{lem:unmixing} and Lemma \ref{lem:lambda}}. 
 \end{align*}
 Similarly, the right-hand side is 
\begin{align*}
  &q^{-(\gamma_1,\gamma_2)}\Ext_{R(\gamma_2) \otimes R(\gamma_1)}^k (\Res_{\gamma_2,\gamma_1} q^{(\alpha_i,\alpha_i)+(\gamma_1,\gamma_2)}(\widehat{M}_1 \circ \widehat{M}_2), X \otimes M_1) \\
  & \quad \text{by the restriction-coinduction adjunction} \\ 
  &\simeq q^{-(\gamma_1, \gamma_2)} \Ext_{R(\gamma_2)\otimes R(\gamma_1)}^k (q^{(\alpha_i,\alpha_i)+(\gamma_1,\gamma_2)-(\gamma_1,\gamma_2)} \widehat{M_2} \otimes \widehat{M_1}, X \otimes M_1) \\
  & \quad \text{by Lemma \ref{lem:unmixing} and Lemma \ref{lem:lambda}} \\
  &\simeq q^{-(\alpha_i,\alpha_i)-(\gamma_1, \gamma_2)} \Ext_{R(\gamma_2)\otimes R(\gamma_1)}^k (\widehat{M_2} \otimes \widehat{M_1}, X \otimes M_1). 
\end{align*}

Hence, the left hand side and the right hand side only differ by a degree shift $q^{(\alpha_i, \alpha_i)}$. 
However, both of them are bounded below graded modules, it implies that they vanish. 

Since $\Ext_{R(\gamma_2)}^k (\widehat{M_2}, X) \simeq \Ext_{R(\gamma_2)}^k(\widehat{M_2}, X) \otimes \Hom_{R(\gamma_1)} (\widehat{M_1}, M_1)$ is canonically a direct summand of $\Ext_{R(\gamma_2)\otimes R(\gamma_1)}^k (\widehat{M_2}~\otimes \widehat{M_1}, X \otimes M_1) = 0$,
it follows that $\Ext_{R(\gamma_2)}^k (\widehat{M_2}, X) = 0$. 
Therefore, we obtain $\prdim_{R(\gamma_2)} \widehat{M_2} \leq \prdim_{R(\gamma_3)} \widehat{M_3} < \infty$. 

Using the short exact sequence
\[
0 \to q^d \widehat{M_2} \to \widehat{M_2} \to M_2 \to 0, 
\]
we deduce that $\prdim_{R(\gamma_2)} M_2$ is also finite. 
\end{proof}

\begin{proposition} \label{prop:prdim}
  Let $w \in W, \beta \in Q_+$. 
  Then, every $R_w(\beta)$-module has finite projective dimension as a graded $R(\beta)$-module. 
\end{proposition}

\begin{proof}
 We adopt the notation in Section \ref{sub:Cwv}, depending on the choice of a reduced expression of $w$. 
 By Lemma \ref{lem:projdimind} and Lemma \ref{lem:prdim2}, $\prdim_{R(\beta)} \overline{\Delta} (\lambda)$ is finite for every $\lambda \in \Lambda(\beta)$. 
 Then using Theorem \ref{thm:cuspidalCw} (3), (4), we may inductively show that $\prdim_{R(\beta)} L(\lambda)$ is finite for every $\lambda \in \Lambda (\beta)$. 
 It follows from Corollary \ref{cor:projlim2} that every object in $\gMod{R_w(\beta)}$ has finite projective dimension. 
\end{proof}

\subsection{The Ext bilinear form} \label{sub:Extform}

Let $\fpdgMod{R(\beta)}$\index{$\fpdgMod{R(\beta)}$} denote the full subcategory of $\gMod{R(\beta)}$ consisting of modules whose projective dimension is finite. 
It is closed under extensions, so it is an exact category for the exact sequences of $\gMod{R(\beta)}$. 
Moreover, $\fpdgMod{R} = \bigoplus_{\beta \in Q_+} \fpdgMod{R(\beta)}$ is a monoidal subcategory of $\gMod{R}$, by Lemma \ref{lem:projdimind}. 
Let $K(\fpdgMod{R})$ denote the Grothendieck group of the exact category $\fpdgMod{R}$. 
Taking projective resolutions, we obtain a canonical isomorphism of $\mathbb{Z}[q,q^{-1}]$-algebras $K(\fpdgMod{R}) \simeq K_{\oplus} (\gproj{R})$. 

\begin{definition} \label{def:extform}
  For $M \in \fpdgMod{R(\beta)}$ and $N \in \gMod{R(\beta)}$, we define
  \[
  \Extform{M}{N} = \sum_{k \geq 0} (-1)^k \qdim \Ext_{R(\beta)}^k (M, N). \index{$\Extform{M}{N}$}
  \]
Since $\prdim M < \infty$, the sum is finite and defines an element of $\mathbb{Z}((q))$. 
\end{definition}

It induces a $\mathbb{Z}$-bilinear form 
\[
K (\fpdgMod{R(\beta)}) \times K(\gMod{R(\beta)}) \to \mathbb{Z}((q)), 
\]
which is also denoted by $\Extform{\cdot}{\cdot}$. 
Since projective modules span the abelian group $K(\fpdgMod{R(\beta)})$, we have 
\[
\Extform{M}{N} = \sum_{[L]: \text{$L$ is a self-dual simple $R(\beta)$-module}} [N:L]_q \Extform{M}{L}. 
\]

Recall three morphisms introduced in Section \ref{sub:quiverHecke}: 
\begin{align*}
&\Psi_1 \colon K_{\oplus}(\gproj{R}) \xrightarrow{\sim} \quantum{-}_{\mathbb{Z}[q,q^{-1}]}, \\
&\Psi_2 \colon K(\gmod{R}) \xrightarrow{\sim} \quantum{-}_{\mathbb{Z}[q,q^{-1}]}^{\mathrm{up}}, \\
&\Psi \colon K(\gMod{R}) \to \quantum{-}_{\mathbb{Z}((q))}^{\mathrm{up}}. 
\end{align*}
Through the isomorphism $K_{\oplus}(\gproj{R}) \simeq K(\fpdgMod{R})$, 
we have an isomorphism 
\[
K(\fpdgMod{R}) \xrightarrow{\sim} \quantum{-}_{\mathbb{Z}[q,q^{-1}]},
\]
which is also denoted by $\Psi_1$\index{$\Psi_1$}. 

Let $\theta$ be the $\mathbb{Q}(q)$-algebra automorphism of $\quantum{-}$ defined by \index{$\theta$}
\[
\theta (f_i) = (1-q_i^2) f_i \ (i \in I). 
\]
We extend it to $\quantum{-}_{\mathbb{Q}((q))}$ by base change. 

Following \cite[Proposition 1.2.5.]{MR2759715}, we define $\Lusform{\cdot}{\cdot}$\index{$\Lusform{\cdot}{\cdot}$} to be the $\mathbb{Q}(q)$-bilinear form on $\quantum{-}_{\mathbb{Q}(q)}$ given by 
\[
\text{$\Lusform{1}{1} = 1, \Lusform{f_i x}{y} = \frac{1}{1-q_i^2} \Lusform{x}{e'_i y}$ for $i \in I$ and $x, y \in \quantum{-}$.}
\]
It is related to $(\cdot,\cdot)$ in section \ref{sub:quantumgroups}: 
using $\theta e'_i \theta^{-1} = e'_i/(1-q_i^2) \ (i \in I)$, we have 
\[
(x,y) = \Lusform{x}{\theta(y)} 
\]
for any $x, y \in \quantum{-}$. 
We may extend these bilinear forms to $\mathbb{Q}((q))$-bilinear forms on $\quantum{-}_{\mathbb{Q}((q))}$. 
We use the same symbols for them.

\begin{lemma} \label{lem:bilinearforms}
  Let $\beta \in Q_+$. 
  We have 
  \[
  \Extform{M}{N} = (\overline{\Psi_1([M])}, \Psi([N])) = \Lusform{\overline{\Psi_1([M])}}{\theta \Psi([N])}
  \]
  for any $M \in \fpdgMod{R(\beta)}$ and $N \in \gMod{R(\beta)}$. 
\end{lemma}

\begin{proof}
  The second equality follows from the discussion preceding the lemma. 
  We have to prove the first equality. 
  By the discussion after Definition \ref{def:extform}, we may assume that $N$ is finite dimensional. 
  The $\mathbb{Z}$-bilinear form $(\overline{\cdot}, \cdot) \colon \quantum{-}_{\mathbb{Z}[q,q^{-1}]} \times \quantum{-}_{\mathbb{Z}[q,q^{-1}]}^{\mathrm{up}} \to \mathbb{Z}[q,q^{-1}]$ is characterized by
  \begin{enumerate}
  \item $(\overline{x},y) = 0$ unless $\beta = \gamma$, 
  \item $(\overline{qx}, y) = (\overline{x}, q^{-1}y) = q^{-1}(\overline{x}, y)$, 
  \item $(\overline{1}, 1) = 1$,
  \item $(\overline{f_ix}, y) = (\overline{x}, e'_i y)$, 
  \end{enumerate}
  for $\beta, \gamma \in Q_+, x \in \quantum{-}_{-\beta}, y \in \quantum{-}_{-\gamma}$. 
  It suffices to verify that the $\mathbb{Z}$-bilinear form 
  \[
  \Extform{{\Psi_1}^{-1}(\cdot)}{{\Psi_2}^{-1}(\cdot)}
  \]
  also satisfies these conditions.
  
  (1) and (2) are immediate. 
  (3) follows from $\Extform{P(i)}{L(i)} = 1$. 
  It remains to prove (4). 
  Let $i \in I$. 
  By the induction-restriction adjunction, we have 
  \[
  \Extform{R(\alpha_i)\circ M}{N} = \Extform{R(\alpha_i) \otimes M}{\Res_{\alpha_i, *}N} = \Extform{M}{E_i N}
  \]
  for all $M \in \fpdgMod{R(\beta)}$ and $N \in \gmod{R(\beta+\alpha_i)} \ (\beta \in Q_+)$. 
  By Theorem \ref{thm:categorification}, we deduce that 
  \[
  \Extform{\Psi_1^{-1}(f_ix)}{\Psi_2^{-1}(y)} = \Extform{\Psi_1^{-1}(x)}{\Psi_2^{-1}(e'_i y)} 
  \]
  for $x \in \quantum{-}_{-\beta}, y \in \quantum{-}_{-\beta-\alpha_i}, \beta\in Q_+$. 
 Hence, the lemma follows.
 \end{proof}

\begin{lemma} \label{lem:character}
  We have the following commutative diagram: 
  \[
  \begin{tikzcd}
  K(\fpdgMod{R}) \arrow[r, "\Psi_1"]\arrow[d] & \quantum{-} \arrow[d] \\
  K(\gMod{R}) \arrow[r,"\theta \Psi"] & \quantum{-}_{\mathbb{Q}((q))}, 
  \end{tikzcd}
  \]
  where the left vertical arrow is induced from embedding of abelian categories $\fpdgMod{R} \to \gMod{R}$ and the right one is the canonical inclusion. 
  \end{lemma}
  
\begin{proof}
    It is enough to prove the lemma after extending scalars of all terms to $\mathbb{Q}((q))$. 
    Since $\{ [P(i)] \mid i \in I\}$ generates the algebra $K(\fpdgMod{R})_{\mathbb{Q}((q))} \simeq K_{\oplus}(\gproj{R})_{\mathbb{Q}((q))} \simeq \quantum{-}_{\mathbb{Q}((q))}$, 
    it suffices to verify that $\Psi_1([P(i)]) = \theta\Psi([P(i)]_q)$ for each $i \in I$. 
    By Theorem \ref{thm:categorification}, $\Psi_1([P(i)]) = f_i$. 
    On the other hand, 
    \[  
    \theta \Psi([P(i)]_q) = \theta \Psi_2 \left( \sum_{k=0}^{\infty} q_i^{2k}[L(i)] \right) = \sum_{k=0}^{\infty} q_i^{2k} \theta(f_i) = \frac{1}{1-q_i^2} (1-q_i^2) f_i = f_i. 
    \]
    The lemma is proved. 
\end{proof}

Following \cite[Chapter 1]{MR2759715}, endow $\quantum{-} \otimes \quantum{-}$ with an algebra structure given by
\[
(x_1 \otimes x_2) (y_1 \otimes y_2) = q^{-(\wt x_2, \wt y_1)} x_1y_1 \otimes x_2 y_2, 
\]
and define an algebra homomorphism $r\index{$r$} \colon \quantum{-} \to \quantum{-} \otimes \quantum{-}$ by 
\[
r(f_i) = f_i \otimes 1 + 1 \otimes f_i \ (i \in I).
\]
It satisfies $\Lusform{xx'}{y} = \Lusform{x \otimes x'}{r(y)}$ for any $x,x', y \in \quantum{-}$. 
Note that, for $x \in \quantum{-}_{-\beta} \ (\beta \in Q_+)$, we have 
\[
r(x) \in \bigoplus_{\beta_1, \beta_2 \in Q_+, \beta_1 + \beta_2 = \beta} \quantum{-}_{-\beta_1}\otimes \quantum{-}_{-\beta_2}. 
\]
Write $r(x)$ along this direct sum decomposition and let $r_{\beta_1, \beta_2}(x)$\index{$r_{\beta_1,\beta_2}$} be the term in $\quantum{-}_{-\beta_1} \otimes \quantum{-}_{-\beta_2}$. 
We have $r(x) = \sum_{\beta_1, \beta_2 \in Q_+, \beta_1 + \beta_2 = \beta} r_{\beta_1, \beta_2} (x)$.  

\begin{lemma}
  Let $\beta, \beta_1, \beta_2 \in Q_+$ and assume $\beta = \beta_1 + \beta_2$. 
\begin{enumerate}
  \item If $M \in \fpdgMod{R(\beta)}$, $\Res_{\beta_1,\beta_2} M \in \fpdgMod{(R(\beta_1) \otimes R(\beta_2))}$. 
  \item If $P \in \gproj{R(\beta)}$, $\Res_{\beta_1,\beta_2} P \in \gproj{(R(\beta_1)\otimes R(\beta_2))}$. 
\end{enumerate}
\end{lemma}

\begin{proof}
  It follows from the restriction-coinduction adjunction. 
\end{proof}

\begin{lemma} \label{lem:rescompatible}
  Let $\beta, \beta_1, \beta_2 \in Q_+$ and assume $\beta = \beta_1 + \beta_2$. 
 \begin{enumerate}
  \item For $M \in \gMod{R(\beta)}$, $(\theta \Psi \otimes \theta \Psi) (\Res_{\beta_1, \beta_2} [M]) = r_{\beta_1,\beta_2} \theta \Psi [M]$. 
  \item For $M \in \fpdgMod{R(\beta)}$ (or $M \in \gproj{R(\beta)}$), $(\Psi_1 \otimes \Psi_1) (\Res_{\beta_1, \beta_2} [M]) = r_{\beta_1, \beta_2} \Psi_1 [M]$.  
 \end{enumerate}
\end{lemma}

\begin{proof}
(1) For $M \in \gMod{R(\beta)}, M_k \in \gproj{R(\beta_k)} \ (k=1,2)$, we have the induction-restriction adjunction
\[
\Extform{M_1 \circ M_2}{M} = \Extform{M_1 \otimes M_2}{\Res_{\beta_1, \beta_2}M}. 
\]
On the other hand, for $x \in \quantum{-}_{-\beta}, x_k \in \quantum{-}_{-\beta_k} \ (k=1,2)$, we have
\[
\Lusform{x_1x_2}{x} = \Lusform{x_1 \otimes x_2}{r_{\beta_1,\beta_2}x}. 
\]
Hence, the assertion follows from the nondegeneracy of $\Lusform{\cdot}{\cdot}$ and Lemma \ref{lem:bilinearforms}. 

(2) follows from (1) using Lemma \ref{lem:character}. 
\end{proof}

\subsection{Projective resolution of $R_{*,w}$} \label{sub:projresol}

Let $w \in W$ and write $w = s_{i_1} \cdots s_{i_l}$ in a reduced expression. 
We adopt the notation in Section \ref{sub:Cwv}.
In particular, $\beta_k = w_{\leq k-1}\alpha_{i_k}$ and $\preceq$ is a convex preorder such that $\beta_1 \prec \beta_2 \prec \cdots \prec \beta_l \prec \minroot \cap w\Phi_+$.
By Proposition \ref{prop:determsupport}, $L(\beta_k) = L^{\preceq}(\beta_k) \simeq M(w_{\leq k}\Lambda_{i_k}, w_{\leq k-1} \Lambda_{i_k})$. 
Recall that, for $n \in \mathbb{Z}_{\geq 0}$, $L(n\beta_k) = L^{\preceq}(n\beta_k) \simeq q_{i_k}^{n(n-1)/2} L(\beta_k)^{\circ n}$ by Proposition \ref{prop:realcuspidal}. 
Lemma \ref{lem:Rwvcuspidal} and Proposition \ref{prop:determsupport} shows that it is the unique self-dual simple $R_{w_{\leq k}, w_{\leq k-1}}(n\beta_k)$-module up to isomorphism.
Let $\widehat{L}(\beta_k) = \widehat{M}(w_{\leq k}\Lambda_{i_k}, w_{\leq k-1}\Lambda_{i_k})$, the affinization defined after Proposition \ref{prop:detaffinization}. 

\begin{proposition} \label{prop:endoring}
 For $n \in \mathbb{Z}_{\geq 0}$, we have an isomorphism of graded rings $\End_{R(n\beta_k)} (\widehat{L}(\beta_k)^{\circ n}) \simeq R(n\alpha_{i_k})$. 
\end{proposition}

\begin{proof}
The same argument as in \cite[Section 3.7]{MR3205728} applies to this proposition. 
It is also possible to introduce the action of $R(n\alpha_{i_k})$ using renormalized R-matrices as in \cite[Section 7.4]{MR4717658}. 
\end{proof}

\begin{definition}
 For $n \in \mathbb{Z}_{\geq 0}$, we define 
 \[
 \widehat{L}(\beta_k)^{\circ (n)} = \Hom_{R(n\alpha_{i_k})} (P(i_k^n), \widehat{L}(\beta_k)^{\circ n}). 
 \]  \index{$\widehat{L}(\beta_k)^{\circ (n)}$}
\end{definition}

Since $R(n\alpha_{i_k}) \simeq P(i_k^n)^{\oplus [n]_{i_k}!}$, we have 
\begin{equation} \label{eq:dividedpower}
\widehat{L}(\beta_k)^{\circ n} \simeq (\widehat{L}(\beta_k)^{\circ (n)})^{\oplus [n]_{i_k}!}. 
\end{equation}

\begin{theorem} \label{thm:dividedpower}
For $n \in \mathbb{Z}_{\geq 0}$, $\widehat{L}(\beta_k)^{\circ (n)}$ is the projective cover of $L(n\beta_k)$ in $\gMod{R_{w_{\leq k}, w_{\leq k-1}}(n\beta_k)}$. 
\end{theorem}

\begin{proof}
Since $L(\beta_k)$ is $\preceq$-cuspidal, the Mackey filtration shows 
\[
[\Res_{\beta_k^n} L(n\beta_k)] = [n]_{i_k}! [L(\beta_k)\otimes \cdots \otimes L(\beta_k)].
\]
See the proof of \cite[Lemma 2.11]{MR3205728} for detail. 
Hence, the theorem follows from Theorem \ref{thm:ext1} and the induction-restriction adjunction. 
\end{proof}

Let $\beta \in Q_+$. 
For each $1 \leq k \leq l$ and $m \geq 0$, we define 
\[
R(\beta)_{k,m} = R(\beta)/\langle e(*,\beta_{k'}) \ (1 \leq k' \leq k-1), e(*,(m+1)\beta_k) \rangle, 
\] \index{$R(\beta)_{k,m}$}
where we regard $e(*, \gamma)$ as $0$ if $\beta - \gamma \not \in Q_+$. 
Theorem \ref{thm:cuspidalCw} shows that, for each $(\lambda,s) \in \Sigma (\beta)$, $L (\lambda,s) \in \gMod{R(\beta)_{k,m}}$ if and only if 
\[
\lambda_{k'} = 0 \ (1 \leq k' \leq k-1), \lambda_k \leq m. 
\] 
In particular, $R(\beta)_{k,0} = R(\beta)_{*,w_{\leq k}}$. 

Note that $R(\beta)_{k,m}$ is a quotient ring of $R(\beta)_{k,m+1}$.
For sufficiently large $m$, 
$\height m\beta_1 > \height \beta$ and $R(\beta)_{1,m}$ coincides with $R(\beta)$. 
Similarly, for each $2 \leq k \leq l$ and sufficiently large $m$, we have $R(\beta)_{k,m} = R(\beta)_{k-1,0}$. 
Moreover, $R(\beta)_{l,0} = R_{*,w}(\beta)$. 

\begin{proposition} \label{prop:projresol}
Let $1 \leq k \leq l$ and $m \geq 1$. Set $\beta' = \beta - m \beta_k$. 
  \begin{enumerate} 
    \item For an $R(\beta)_{k,m}$-module $M$, $\Res_{\beta', m\beta_k}M$ is an $(R(\beta')_{k,0} \otimes R(m\beta_k)_{w_{\leq k}, w_{\leq k-1}})$-module. 
    \item For $X \in \gMod{R(\beta')_{k,0}}$ and $Y \in \gMod{R(m\beta_k)_{w_{\leq k}, w_{\leq k-1}}}$, both $X \circ Y$ and $Y \circ X$ belong to $\gMod{R(\beta)_{k,m}}$. 
    \item For a projective $R(\beta')_{k,0}$-module $P$, $P \circ \widehat{L}(\beta_k)^{\circ (m)}$ is a projective $R(\beta)_{k,m}$-module. 
    \item The kernel of the canonical surjection $R(\beta)_{k,m} \to R(\beta)_{k,m-1}$ is isomorphic to $P \circ \widehat{L}(\beta_k)^{\circ (m)}$ for some projective $R(\beta')_{k,0}$-module $P$. 
  \end{enumerate}
\end{proposition}

\begin{proof} 
  In the proof, we repeatedly use Theorem \ref{thm:cuspidalCw} without explicitly referring to it. 
  We suppress degree shifts. 



  (1)
  It suffices to show the assertion when $M$ is simple. 
  By the discussion preceding this proposition, we may assume $M = L(\lambda, s)$ for some $(\lambda, s) \in \Sigma(\beta)$ with 
  \[
  \lambda_{k'} = 0 \ (1 \leq k' \leq k-1), \ \lambda_k \leq m. 
  \]
  Recall $\beta_s = -\wt L(s)$. 
  If $\lambda_k \leq m-1$, we have $(m \beta_k, \beta') \not \leq (\lambda_k \beta_k, \ldots, \lambda_l\beta_l, \beta_s)$, hence $\Res_{\beta', m\beta_k} L(\lambda,s) = 0$. 
  Assume $\lambda_k = m$. 
  We have $\Res_{\beta',m\beta_k} \overline{\Delta}(\lambda,s) = \overline{\Delta}(\lambda_{k+1}\beta_{k+1}, \ldots, \lambda_l\beta_l,\beta_s)$ by Lemma \ref{lem:unmixing}. 
  Since $\Res_{\beta',m\beta_k} L(\lambda,s)$ is a quotient of this module and $\overline{\Delta}(\lambda_{k+1}\beta_{k+1}, \ldots, \lambda_l\beta_l,\beta_s)$ is an $R(\beta')_{k,0}$-module, (1) follows. 
 
  (2) 
  It suffices to prove the assertion when $X$ and $Y$ are simple. 
  Hence, we may assume $X = L(\lambda,s)$ for some $(\lambda, s) \in \Sigma(\beta')$ with $\lambda_{k'} = 0 \ (1 \leq k' \leq k)$ and $Y = L(\beta_k)^{\circ m}$. 
  For each $1 \leq k' \leq k-1$, we have $(\beta_{k'},\beta-\beta_{k'}) \not \leq (m\beta_k, \lambda_{k+1}\beta_{k+1}, \ldots, \lambda_l\beta_l, \beta_s)$, hence $\Res_{\beta-\beta_{k'}, \beta_{k'}} X \circ Y = 0$. 
  We also have $((m+1)\beta_k, \beta- (m+1) \beta_k) \not \leq (m\beta_k, \lambda_{k+1}\beta_{k+1},\ldots, \lambda_l\beta_l, \beta_s)$, hence $\Res_{\beta-(m+1) \beta_k, (m+1)\beta_k} X \circ Y = 0$. 
  It follows that $X \circ Y$ is an $R(\beta)_{k,m}$-module. 

  Since $Y \circ X$ is isomorphic to $D(X \circ Y)$, its composition factors coincide with those of $X \circ Y$. 
  Therefore, $Y \circ X$ is also an $R(\beta)_{k,m}$-module.  

  (3) By (2), $P \circ \widehat{L}(\beta_k)^{\circ (m)}$ is an $R(\beta)_{k,m}$-module.  
  It remains to show that $\Ext_{R(\beta)}^1 (P \circ \widehat{L}(\beta_k)^{\circ (m)}, M) = 0$ for all $M \in \gMod{R(\beta)_{k,m}}$. 
  By the induction-restriction adjunction, we have 
  \[
  \Ext_{R(\beta)}^1 (P \circ \widehat{L}(\beta_k)^{\circ (m)}, M) \simeq \Ext_{R(\beta') \otimes R(m\beta_k)}^1 (P \otimes \widehat{L}(\beta_k)^{\circ (m)}, \Res_{\beta', m\beta_k} M),
  \]
  and it vanishes by (1) and Theorem \ref{thm:dividedpower}. 

  (4) First, we claim that $\Res_{\beta', m\beta_k} R(\beta)_{k,m}$ is a projective $R(\beta')_{k, 0} \otimes R(m\beta_k)_{w_{\leq k}, w_{\leq k-1}}$-module. 
  It suffices to show that 
  \[
  \Ext_{R(\beta')\otimes R(m\beta_k)}^1 (\Res_{\beta', m\beta_k} R(\beta)_{k,m}, X \otimes Y) = 0
  \]
  for all $X \in \gMod{R(\beta')_{k,0}}$ and $Y \in \gMod{R(m\beta_k)_{w_{\leq k}, w_{\leq k-1}}}$. 
  By the restriction-coinduction adjunction, it is isomorphic to $\Ext_{R(\beta)}^1 (R(\beta)_{k,m}, Y \circ X)$. 
  (2) ensures that $Y \circ X$ is an $R(\beta)_{k,m}$-module, hence the $\Ext^1$ vanishes and the claim follows. 

  Since $R(m\beta_k)_{w_{\leq k}, w_{\leq k-1}}$ has a unique indecomposable projective module $\widehat{L}(\beta_k)^{\circ (m)}$ by Theorem \ref{thm:dividedpower}, 
  the projective module $\Res_{\beta', m\beta_k} R(\beta)_{k,m}$ is isomorphic to $P \otimes \widehat{L}(\beta_k)^{\circ (m)}$ for some projective $R(\beta')_{k, 0}$-module $P$. 
  By the induction-restriction adjunction, it induces 
  \[
  P \circ \widehat{L}(\beta_k)^{\circ (n)} \to R(\beta)_{k,m}. 
  \]
  The cokernel of this homomorphism is $R(\beta)_{k,m}/R(\beta)_{k,m} e(\beta', m\beta_k) R(\beta)_{k,m} = R(\beta)_{k,m-1}$. 
  Hence, it is enough to show that this homomorphism is injective. 

  The isomorphism $P \otimes \widehat{L}(\beta_k)^{\circ (n)} \simeq \Res_{\beta', m\beta_k} R(\beta)_{k,m}$ also induces, by the restriction-coinduction adjunction, a homomorphism $R(\beta)_{k,m} \to \widehat{L}(\beta_k)^{\circ (m)} \circ P$. 
  We claim that the composition 
  \[
    P \circ \widehat{L}(\beta_k)^{\circ (m)} \to R(\beta)_{k,m} \to \widehat{L}(\beta_k)^{\circ (m)} \circ P
  \]
  is injective.

  Observe that the composition is given by 
  \[
  u \otimes v \mapsto \tau_{w[mn, n']} v \otimes u \ (u \in P, v \in \widehat{L}(\beta_k)^{\circ (m)}), 
  \]
  where $n = \height \beta_k$ and $n' = \height \beta'$. 
  By taking a direct sum, we obtain a homomorphism $P \circ \widehat{L}(\beta_k)^{\circ m} \to \widehat{L}(\beta_k)^{\circ m} \circ P$ . 
  It coincides with a composition of $m$ homomorphisms 
  \[
  P \circ \widehat{L}(\beta_k)^{\circ m} \to \widehat{L}(\beta_k) \circ P \circ \widehat{L}(\beta_k)^{\circ m-1} \to \cdots \to \widehat{L}(\beta_k)^{\circ m} \circ P, 
  \]
  where each homomorphism comes from $P \circ \widehat{L}(\beta_k) \to \widehat{L}(\beta_k) \circ P, u \otimes v \mapsto \tau_{w[n,n']} v \otimes u$. 
  This homomorphism is nothing but $\Rmatrix{P}{\widehat{L}(\beta_k)}$ for the unmixing pair $(P, \widehat{L}(\beta_k))$, which is injective by Proposition \ref{prop:unmixingreninj}. 
  It proves the assertion.
 \end{proof}

\begin{theorem} \label{thm:projresol}
  Let $1 \leq k \leq l$ and $m \geq 0$. 
  For any $M \in \gMod{R(\beta)_{k,m}}$ and $n \geq 1$, we have 
  \[
  \Ext_{R(\beta)}^n (R(\beta)_{k,m}, M) = 0. 
  \]
  In particular, for any $M \in \gMod{R(\beta)_{*,w}}$ and $n \geq 1$, we have 
  \[
  \Ext_{R(\beta)}^n (R(\beta)_{*,w}, M) = 0. 
  \]
\end{theorem}

Note that we compute the extension groups over $R(\beta)$, not over $R(\beta)_{k,m}$. 

\begin{proof}
 We proceed by induction on $(k, m)$. 
 By the remark before Proposition \ref{prop:projresol}, when $k =1$ and $m$ is sufficiently large, $R(\beta)_{1,m} =R(\beta)$ and the assertion is obvious. 
 In addition, taking $m$ sufficiently large for each $1 \leq k \leq l$, we have $R(\beta)_{k,m} = R(\beta)_{k-1,0}$.  
 Hence, it suffices to show the assertion for $(k,m-1)$ assuming that for $(k, m)\ (m \geq 1)$.  

 Let $M \in \gMod{R(\beta)_{k,m-1}}$. 
 By Proposition \ref{prop:projresol} (4), we have a short exact sequence 
 \[
 0 \to P \circ \widehat{L}(\beta_k)^{\circ (m)} \to R(\beta)_{k,m} \to R(\beta)_{k,m-1} \to 0
 \]
for some projective $R(\beta - m \beta_k)_{k,0}$-module $P$.  
Using the induction-restriction adjunction and the fact that $\Res_{\beta-m\beta_k, m\beta_k}M = 0$, we obtain
\[
\Ext_{R(\beta)}^n (P \circ \widehat{L}(\beta_k)^{\circ (m)}, M) = 0 \ (n \geq 0). 
\]
Hence, the long exact sequence induced from the short exact sequence shows 
\[
\Ext_{R(\beta)}^n (R(\beta)_{k,m-1}, M ) \simeq \Ext_{R(\beta)}^n (R(\beta)_{k,m}, M)
\]
for all $n \geq 0$. 
Since $M$ is an $R(\beta)_{k,m}$-module, the assertion follows from the induction hypothesis. 
\end{proof}

\begin{proposition} \label{prop:prdimR*w}
Let $1 \leq k \leq l$ and $m \geq 0$. 
Then, the projective dimension of $R(\beta)_{k,m}$ as an $R(\beta)$-module is finite.   
In particular, the projective dimension of $R_{*,w}(\beta)$ as an $R(\beta)$-module is finite. 
\end{proposition}

\begin{proof}
It follows from Lemma \ref{lem:projdimind}, Lemma \ref{lem:prdim2} and Proposition \ref{prop:projresol} (4) by induction. 
\end{proof}

\subsection{Interlude: Exactness of generalized Schur-Weyl duality functors} \label{sub:interlude}

Let $F \colon \gMod{R} \to \Pro(\mathcal{C})$ be a functor constructed by the procedure in \cite[7.4]{MR4717658}, 
where $\mathcal{C}$ is a monoidal category that satisfies certain requirements \cite[(6.2)]{MR4717658}. 
They proved that if $R$ is of finite type, $F$ is exact using the fact that $R(\beta)$ has finite global dimension \cite[Proposition 7.6]{MR4717658}.
In this section, we make a slight modification to their argument to show the following generalization. 

\begin{proposition}
For a quiver Hecke algebra of arbitrary type, the functor $F$ is exact when restricted to $\fpdgMod{R}$.
In particular, it is exact on $\gMod{R_{w,*}}$ for any $w \in W$. 
\end{proposition}

Note that the latter assertion follows from the former assertion using Proposition \ref{prop:prdim}. 

Similarly to the proof of \cite[Proposition 7.6]{MR4717658}, it is a consequence of the following lemma, which is a variation of \cite[Proposition 3.7]{MR3748315}. 
For a graded ring $A$, let $A^{\text{op}}$ be the opposite ring.
For an $A^{\text{op}}$-module $M$, we define $\flatdim_{A^{op}} M$ as
\[
\sup \{ d \geq 0 \mid \text{There exists an $A$-module $N$ such that $\Tor_d^A (M,N) \neq 0$} \}.  
\]

\begin{lemma}
  Let $A, B$ be graded rings, and $A \to B$ a graded ring homomorphism. 
  Let $M$ be a $B$-module, and assume the following conditions:
  \begin{enumerate}
  \item $B$ is a finitely generated projective $A^{\mathrm{op}}$-module, where the $A^{\mathrm{op}}$-module structure on $B$ is given by the right multiplication,  
  \item $\Hom_{A^{\mathrm{op}}} (B, A)$ is a projective $B^{\mathrm{op}}$-module, 
  \item $\prdim_B M$ is finite.  
  \end{enumerate}
  Then, for any $B^{\mathrm{op}}$-module $K$, we have 
  \[
  \text{$\Tor_d^B (K,M) = 0$ for $d > \flatdim_{A^{\mathrm{op}}} K$.}
  \]
\end{lemma}

\begin{proof}
Fix $M$ and consider the following assertion for each $n \geq 0$:  
\begin{enumerate}
 \item[($P_n$)] for any $B^{\text{op}}$-module $K$ with $\flatdim_{A^{\text{op}}} K \leq n$, we have 
 \[
 \Tor_d^B(K,M) = 0 
 \]
 for any $d > n$. 
\end{enumerate}
We prove ($P_n$) by a descending induction on $n$. 
By the third assumption, ($P_n$) holds if $n \geq \prdim_B M$. 
Now, assuming ($P_{n+1}$) we prove ($P_n$).
Take a $B^{\text{op}}$-module $K$ with $\flatdim_{A^{\text{op}}} K \leq n$. 
Since $\flatdim_{A^{\text{op}}} \leq n+1$, ($P_{n+1}$) shows that $\Tor_d^B(K,M) = 0$ for $d > n+1$. 
It remains to verify that $\Tor_{n+1}^B(K,M) = 0$. 

There exists a canonical injective $B^{\text{op}}$-linear map
\[
f \colon K \to \Hom_{A^{\text{op}}} (B,K). 
\]
We apply $\Tor_{\bullet}^B(\cdot, M)$ to the short exact sequence 
\[
0 \to K \xrightarrow{f} \Hom_{A^{\text{op}}}(B,K) \to \Cok(f) \to 0 
\]
to obtain an exact sequence
\[
\Tor_{n+2}^B(\Cok (f), M) \to \Tor_{n+1}^B (K, M) \to \Tor_{n+1}^B (\Hom_{A^{\mathrm{op}}} (B,K), M).  
\]
We use it later in the proof. 

We claim that 
\[
\flatdim_{A^{\text{op}}}\Hom_{A^{\text{op}}}(B,K) \leq \flatdim_{B^{\text{op}}} \Hom_{A^{\text{op}}} (B,K) \leq \flatdim_{A^{\text{op}}} K \leq n. 
\]
The first inequality follows from assumption (1), and the third one follows from our assumption on $K$. 
We prove the second one. 
By assumption (1), $\Hom_{A^{\text{op}}}(B, \cdot) \colon \Mod{A^{\text{op}}} \to \Mod{B^{\text{op}}}$ is a exact functor,
and it is naturally isomorphic to $\Hom_{A^{\text{op}}}(B,A)\otimes_{A^{\text{op}}} (\cdot)$. 
Then, assumption (2) shows that this functor sends flat $A^{\text{op}}$-modules to flat $B^{\text{op}}$-modules. 
It proves the claim. 

By the claim, we have $\flatdim_{A^{\text{op}}} \Hom_{A^{\text{op}}}(B,K), \flatdim_{A^{\text{op}}} K \leq n$.
Hence, the short exact sequence shows $\flatdim_{A^{\text{op}}} \Cok (f) \leq n+1$. 
By the induction hypothesis ($P_{n+1}$), we deduce $\flatdim_{B^{\text{op}}} \Cok(f) \leq n+1$. 
We also have $\flatdim_{B^{\text{op}}}\Hom_{A^{\text{op}}}(B,K) \leq n$ by the claim. 
These inequalities combined with the exact sequence indicates $\Tor_{n+1}^B(K,M) = 0$, which completes the induction step. 
\end{proof}

\section{Stratifications} \label{sec:stratifications}

\subsection{Definition of stratified categories} \label{sub:stratified}

Let $H$ be a left Noetherian Laurentian algebra. 
We consider the category $\gMod{H}$ of finitely generated graded left $H$-modules. 
Take a complete set of representatives of simple modules $\{ L(\lambda) \mid \lambda \in \Lambda \}$ up to isomorphism and degree shift, where $\Lambda$ is a finite set. 
For each $\lambda \in \Lambda$, let $P(\lambda)$ be the projective cover of $L(\lambda)$. 

Let $(\Xi, \leq)$ be a partially ordered set and $\rho \colon \Lambda \to \Xi$ be a map. 
For each $\lambda \in \Lambda$, we define the standard object
\[
\Delta(\lambda) = P(\lambda) / K(\lambda),\ K(\lambda) = \sum_{\mu: \rho(\mu) \not \leq \rho(\lambda), f \in \Hom_H(P(\mu),P(\lambda))} \Image f. 
\]

We say $M \in \gMod{H}$ has a $\Delta$-filtration if there exists a decreasing filtration $M = M^0 \supset M^1 \supset M^2 \supset \cdots $ such that 
each nontrivial successive quotient is of the form $q^n \Delta(\lambda)$  for some $n \in \mathbb{Z}, \lambda \in \Lambda$, and $\bigcap_p M^p =0$.

\begin{definition}[{\cite{MR3335289}}] \label{def:stratified}
The category $\gMod{H}$ is said to be stratified if 
\begin{enumerate}
\item for any $\lambda \in \Lambda$, $K(\lambda)$ has an $\Delta$-filtration whose successive quotients are of the form $q^n \Delta(\mu)$ for some $n \in \mathbb{Z}, \rho(\mu) > \rho(\lambda)$. 
\end{enumerate}
Moreover, $\gMod{H}$ is said to be an affine highest weight category if it satisfies the following additional conditions:
\begin{enumerate} 
  \setcounter{enumi}{1}
  \item for any $\lambda \in \Lambda$, $\End_H(\Delta(\lambda))$ is isomorphic to some polynomial algebra $\mathbf{k}[x_1,\ldots,x_n] \ (n \in \mathbb{Z}_{\geq 0})$ with each $x_k$ being homogeneous of positive degree, 
  \item for any $\lambda \in \Lambda$, the right $\End_H(\Delta(\lambda))$-module $\Delta(\lambda)$ is free of finite rank. 
\end{enumerate}
\end{definition}

Note that our definition of affine highest weight category is slightly stronger than Kleshchev's original definition:
in (2), he only assumed that each $\End_H(\Delta(\lambda))$ is a quotient of a polynomial algebra. 

\begin{example}
  Let $\Xi$ be a singleton. 
  Then for all $\lambda \in \Lambda$, we have $P(\lambda) = \Delta(\lambda)$. 
  Hence with respect to this trivial $\rho$, $\gMod{H}$ is always stratified. 
\end{example}

\begin{remark}
 If $\gMod{H}$ is stratified and $M \in \gMod{H}$ has a $\Delta$-filtration $M = M^0 \supset M^1 \supset M^2 \supset \cdots$, then $M^p = 0$ for sufficiently large $p$ \cite[Lemma 5.12.]{MR3335289}.  
 We define 
 \[
 (M:\Delta(\mu))_q = \sum_{d \in \mathbb{Z}} \lvert \{p \geq 1 \mid M^{p-1}/M^p \simeq q^d \Delta(\mu) \} \rvert q^d.  
 \]
 This is independent of the choice of the $\Delta$-filtration by Proposition \ref{prop:costandard} (2) below. 
\end{remark}

\begin{definition}[{cf.\cite[Definition 2.8]{MR5005930}}] \label{def:propercostd}
  Let $\lambda \in \Lambda$. 
  A finite dimensional module $\overline{\nabla}(\lambda) \in \gmod{H}$ is called a proper costandard module (with respect to $\rho$) if 
  \begin{enumerate}
    \item $\soc (\overline{\nabla}(\lambda)) = L(\lambda)$, 
    \item all the composition factors of $\overline{\nabla}(\lambda)/\soc (\overline{\nabla}(\lambda))$ is of the form $L(\mu)$ with $\rho (\mu) < \rho (\lambda)$, 
    \item $\Ext_H^1 (L(\mu), \overline{\nabla}(\lambda)) = 0$ if $\rho(\mu) \not \geq \rho(\lambda)$. 
  \end{enumerate}
\end{definition}

A routine argument shows that a proper costandard module is unique up to isomorphism if it exists.  
Let $\overline{(\cdot)}$ be a $\mathbb{Z}$-algebra automorphism of $\mathbb{Z}[q,q^{-1}]$ defined by $\overline{q} = q^{-1}$. 

\begin{proposition}[{\cite{MR3335289}}] \label{prop:costandard}
 Assume that $\gMod{H}$ is stratified. 
 Then for each $\lambda \in \Lambda$, there exists a proper costandard module $\overline{\nabla}(\lambda)$ and it satisfies the following properties: 
  \begin{enumerate}
 \item $\qdim \Hom_H (\Delta(\mu), \overline{\nabla}(\lambda)) = \delta_{\mu, \lambda}, \Ext_H^n (\Delta(\mu), \overline{\nabla}(\lambda)) = 0, \ (\mu \in \Lambda, n \geq 1)$, 
 \item $M \in \gMod{H}$ has a $\Delta$-filtration if and only if $\Ext_H^1 (M, \overline{\nabla}(\lambda)) = 0$ for all $\lambda \in \Lambda$. 
 In this case, $(M:\Delta(\lambda))_q = \overline{\qdim \Hom_H (M, \overline{\nabla}(\lambda))}$. 
 \item $(P(\lambda):\Delta(\mu))_q = \overline{[\overline{\nabla}(\mu):L(\lambda)]_q}$ (generalized BGG reciprocity). 
 \end{enumerate}
\end{proposition}

\subsection{Criterion to be a stratified category} \label{sub:criterion}

We keep the notation of the previous section. 
Write $\Xi = \{ \xi_1, \xi_2, \ldots, \xi_n \}$ so that $a < b$ implies $\xi_a \not \leq \xi_b$. 

\begin{lemma} \label{lem:filtration}
 Assume that there exists a proper costandard module $\overline{\nabla}(\lambda) \in \gmod{H}$ for each $\lambda \in \Lambda$. 
 Let $M \in \gMod{H}$ and put $m_{\lambda} = \qdim \Hom_H (M, \overline{\nabla}(\lambda)) \in \mathbb{Z}[q,q^{-1}]$. 
 Then, there exists a finite filtration $0 = M^0 \subset M^1 \subset \cdots \subset M^n = M$
 such that, for every $1 \leq a \leq n$, $M^a/M^{a-1}$ is a quotient of $\bigoplus_{\lambda \in \rho^{-1}(\xi_a)} \Delta(\lambda)^{\oplus \overline{m_{\lambda}}}$. 
\end{lemma}

\begin{proof}
We may assume that $M \neq 0$. 
Let $1 \leq a \leq n$ be the least number such that $[M:L(\lambda)]_q \neq 0$ for some $\lambda \in \rho^{-1}(\xi_a)$. 
Set $M^0 = M^1 = \cdots = M^{a-1} = 0$ and 
\[
M^a = \sum_{\lambda \in \rho^{-1}(\xi_a), f \in \Hom_H (\Delta(\lambda), M)} \Image f.  
\]
Then, by the definition of $a$, every composition factor $L(\mu)$ of $M/M^a$ satisfies $\mu \in \rho^{-1} (\{ \xi_{a+1}, \cdots, \xi_n \})$. 
Moreover, $M^a$ is a quotient of $\bigoplus_{\lambda \in \rho^{-1}(\xi_a)} \Delta(\lambda)^{\oplus h_{\lambda}}$ for some $h_{\lambda} \in \mathbb{Z}[q, q^{-1}] \ (\lambda \in \rho^{-1}(\xi_a))$. 

We claim that the head of $M^a$ is $\bigoplus_{\lambda \in \rho^{-1}(\xi_a)} L(\lambda)^{\oplus \overline{m_{\lambda}}}$, 
which implies that $M^a$ is a quotient of $\bigoplus_{\lambda \in \rho^{-1}(\xi_a)} \Delta(\lambda)^{\oplus \overline{m_{\lambda}}}$. 
Let $\lambda \in \rho^{-1}(\xi_a)$. 
By the definition of proper costandard modules, we have $\Hom_H(L(\mu), \overline{\nabla}(\lambda)) = \Ext_H^1(L(\mu),\overline{\nabla}(\lambda)) = 0$ for any $\mu \in \rho^{-1}(\{ \xi_{a+1}, \ldots, \xi_{n}\})$. 
Hence, Corollary \ref{cor:projlim1} shows 
\[
\Hom_H (M/M^a, \overline{\nabla}(\lambda)) = \Ext_H^1 (M/M^a, \overline{\nabla}(\lambda)) = 0.
\] 
Applying $\Ext_H^{\bullet} (\cdot, \overline{\nabla}(\lambda))$ to the short exact sequence $0 \to M^a \to M \to M/M^a \to 0$, 
we see that $\Hom_H (M, \overline{\nabla}(\lambda)) \to \Hom_H (M^a, \overline{\nabla}(\lambda))$ is an isomorphism. 
On the other hand, since every composition factor $L(\mu)$ of $\overline{\nabla}(\lambda)/L(\lambda)$ satisfies $\mu \in \rho^{-1}(\{ \xi_{a+1}, \cdots, \xi_n \})$, 
we have $\Hom_H (M^a, \overline{\nabla}(\lambda)/L(\lambda)) = 0$. 
Hence, $\Hom_H (M^a, L(\lambda)) \to \Hom_H (M^a, \overline{\nabla}(\lambda))$ is an isomorphism.
We deduce from these two isomorphisms 
\[
\qdim \Hom_H (M^a, L(\lambda)) = \qdim \Hom_H (M, \overline{\nabla}(\lambda)) = m_{\lambda}, 
\]
which proves the claim. 

Next, we verify that for any $\lambda \in \Lambda$, we have 
\[
\qdim \Hom_H (M/M^a, \overline{\nabla}(\lambda)) = \begin{cases}
m_{\lambda} & \lambda \not \in \rho^{-1}(\xi_a), \\
0 & \lambda \in \rho^{-1}(\xi_a). 
\end{cases}
\]
If $\lambda \in \rho^{-1} (\{ \xi_1, \ldots, \xi_{a-1}\})$, 
$L(\lambda)$ is not a composition factor of $M$, so both $\qdim \Hom_H (M/M^a, \overline{\nabla}(\lambda))$ and $m_{\lambda}$ are $0$. 
If $\lambda \in \rho^{-1} (\xi_a)$, 
$L(\lambda)$ is not a composition factor of $M/M^a$, so $\qdim \Hom_H (M/M^a, \overline{\nabla}(\lambda)) = 0$. 
It remains to consider the case $\lambda \in \rho^{-1} (\{ \xi_{a+1}, \ldots, \xi_n\})$. 
In this case, for any $\mu \in \rho^{-1}(\xi_a)$, $L(\mu)$ is not a composition factor of $\overline{\nabla}(\lambda)$. 
Hence, $\Hom_H (M^a, \overline{\nabla}(\lambda)) = 0$ and it follows that $\Hom_H (M/M^a, \overline{\nabla}(\lambda)) \simeq \Hom_H (M, \overline{\nabla}(\lambda))$. 

Therefore, the lemma is proved by induction on $a$. 
\end{proof}

\begin{proposition} \label{prop:criterion}
 The category $\gMod{H}$ is stratified (with respect to $\rho$) if 
 there exists a proper costandard module $\overline{\nabla}(\lambda)$ for each $\lambda \in \Lambda$ and the following equalities in $K(\gmod{H})_{\mathbb{Z}((q))}$ hold: 
 \[
 [P(\lambda)]_q = \sum_{\mu \in \Lambda} \overline{[\overline{\nabla}(\mu):L(\lambda)]_q} [\Delta(\mu)]_q \ (\lambda \in \Lambda). 
 \]
\end{proposition}

\begin{proof}
  Let $\lambda \in \Lambda$. 
  Put $P = P(\lambda)$ and $m_{\mu} = \qdim \Hom_H (P(\lambda),\overline{\nabla}(\mu)) = [\overline{\nabla}(\mu) : L(\lambda)]_q$ for each $\mu \in \Lambda$. 
  By Lemma \ref{lem:filtration}, we have a filtration $0 = P^0 \subset P^1 \subset \cdots \subset P^n = P$ 
  such that $P^a/P^{a-1}$ is a quotient of $\bigoplus_{\mu \in \rho^{-1}(\xi_a)} \Delta(\mu)^{\oplus \overline{m_{\mu}}}$.
  The assumption means that the surjection $\bigoplus_{\mu \in \rho^{-1}(\xi_a)} \Delta(\mu)^{\oplus \overline{m_{\mu}}} \to P^a/P^{a-1}$ must be an isomorphism. 
  Hence, we obtain a $\Delta$-filtration of $P = P(\lambda)$.

  Assume that $\rho(\lambda) = \xi_k$. 
  Then for any $a > k$ and $\mu \in \rho^{-1}(\xi_a)$, we have $m_{\mu} = 0$, hence $P^k = P$. 
  For $\mu \in \rho^{-1}(\xi_k)$, we have $m_{\mu} = \delta_{\mu, \lambda}$, hence $P^k/P^{k-1} \simeq \Delta(\lambda)$. 
  For $a < k$ and $\mu \in \rho^{-1}(\xi_a)$, we have $m_{\mu} = 0$ unless $\xi_a > \xi_k$. 
  It follows that $P^{k-1}$ has a $\Delta$-filtration whose successive quotients are $q^n\Delta(\mu) \ (n \in \mathbb{Z}, \rho(\mu) > \rho(\lambda))$. 
  Hence, $\gMod{H}$ is stratified. 
\end{proof}

\begin{remark}
 This criterion is used in \cite{MR3177927, MR3371494} for current algebras. 
\end{remark}

\subsection{Root vectors and determinantial modules} \label{sub:PBW}

First, we briefly review the construction of the PBW basis \cite[Part VI.]{MR2759715}. 
See Appendix \ref{app:B}, where we demonstrate how to translate different conventions in other literatures into our setup. 

For each $i \in I$, we define $T_i$\index{$T_i$} to be the $\mathbb{Q}(q)$-algebra automorphism of $\quantum{}$ defined by 
\begin{align*}
  T_i(q^h) = q^{s_i h}, T_i (e_i) = - q_i^{-h_i}f_i, T_i(f_i) = - e_i q_i^{h_i}, \\
  T_i(e_j) = \sum_{r+s = -\langle h_i \alpha_j \rangle}(-1)^r q_i^{-r} e_i^{(r)}e_je_i^{(s)} \ (j \neq i), \\
  T_i(f_j) = \sum_{r+s = -\langle h_i, \alpha_j \rangle}(-1)^r q_i^{r}f_i^{(s)}f_jf_i^{(r)} \ (j \neq i).  
\end{align*}
They satisfy the braid relation, hence we can define an automorphism $T_w$ for each $w \in W$. 

Let $w \in W$ and fix its reduced expression. 
We use the notation in Section \ref{sub:Cwv} and Section \ref{sub:Extform}.  
Let $c \colon Q \to \mathbb{Q}(q)^{\times}$\index{$c(\beta)$} be a group homomorphism defined by  
\[
c(\alpha_i) = 1 - q_i^2 \ (i \in I). 
\]
Then, we have $\theta (x) = c(\beta) x$ for $x \in \quantum{-}_{-\beta} \ (\beta \in Q_+)$. 

\begin{definition} \label{def:rootvec}
For each $1 \leq k \leq l$, we define the root vector and the dual root vector 
\[
f_{\beta_k} = T_{i_1} \cdots T_{i_{k-1}} f_{i_k}, \ f_{\beta_k}^{\mathrm{up}} = \frac{1-q_{i_k}^2}{c(\beta_k)} f_{\beta_k} \in \quantum{-}_{-\beta_k}, 
\] \index{$f_{\beta_k}$} \index{$f_{\beta_k}^{\mathrm{up}}$}
respectively. 
For $\lambda \in \mathbb{Z}_{\geq 0}^l$, we define 
\[
f_{\lambda} = f_{\beta_l}^{(\lambda_l)} \cdots f_{\beta_1}^{(\lambda_1)}, \ f_{\lambda}^{\mathrm{up}} = q^{m_{\lambda}} f_{\beta_l}^{\mathrm{up} \ \lambda_l} \cdots f_{\beta_1}^{\mathrm{up} \ \lambda_1}, 
\] \index{$f_{\lambda}$} \index{$f_{\lambda}^{\mathrm{up}}$}
where $f_{\beta_k}^{(n)} = f_{\beta_k}^n/[n]_{i_k}!, m_{\lambda} = \sum_{1 \leq k \leq l} (\alpha_{i_k}, \alpha_{i_k}) \lambda_k (\lambda_k-1)/4$. 
\end{definition}

Let $\quantum{\geq 0}$ be a $\mathbb{Q}(q)$-subalgebra of $\quantum{}$ generated by $\{q^h \mid h \in P^{\lor} \} \cup \{ e_i \mid i \in I\}$.
It is known that the sets $\{ f_{\lambda} \}_{\lambda}, \{ f_{\lambda}^{\mathrm{up}} \}_{\lambda}$ form $\mathbb{Q}(q)$-bases of $\quantum{-}\cap T_w \quantum{\geq 0}$, 
which are called the PBW basis and the dual PBW basis respectively: see Proposition \ref{prop:subspaceUw} and the references therein. 
We offer a more detailed discussion later in Section \ref{sub:quantumunipotent}. 

\begin{lemma} \label{lem:rootvec}
  For any $1 \leq k \leq l$, we have
\begin{enumerate}
\item $f_{\beta_k}^{\mathrm{up}} = D(w_{\leq k}\Lambda_{i_k}, w_{\leq k-1}\Lambda_{i_k})$, 
\item $(f_{\beta_k}, f_{\beta_k}^{\mathrm{up}}) = 1, (\overline{f_{\beta_k}}, f_{\beta_k}^{\mathrm{up}}) = 1$. 
\end{enumerate}
\end{lemma}

\begin{proof}
(1) This is \cite[Proposition 7.4]{MR3090232}. 

(2) Since $\Lusform{T_i(x)}{T_i(y)} = \Lusform{x}{y}$ for any $x, y \in \quantum{-} \cap T_i^{-1}\quantum{-}$ \cite[Proposition 38.2.1]{MR2759715}, we have 
\[
\Lusform{f_{\beta_k}}{f_{\beta_k}} = \Lusform{f_{i_k}}{f_{i_k}} = \frac{1}{1-q_{i_k}^2}. 
\]
Hence, we obtain
\[
(f_{\beta_k}, f_{\beta_k}^{\mathrm{up}}) = c(\beta_k) \Lusform{f_{\beta_k}}{f_{\beta_k}^{\mathrm{up}}} = (1-q_{i_k}^2) \Lusform{f_{\beta_k}}{f_{\beta_k}} = 1. 
\] 

To prove the second equality, we use a $\mathbb{Q}$-linear involution $\sigma$\index{$\sigma$} on $\quantum{-}$ characterized by 
\[
(x,\sigma(y)) = \overline{(\overline{x}, y)} \ \text{for $x,y \in \quantum{-}$}. 
\]
It is known that $\sigma (\Gup(b)) = \Gup(b)$ for any $b \in B(\infty)$ \cite[Proposition 16]{MR1985725}, \cite[Corollary 3.4]{MR2914878}. 
In particular, (1) shows $f_{\beta_k}^{\mathrm{up}} \in \Bup(\quantum{-})$, hence $\sigma(f_{\beta_k}^{\mathrm{up}}) = f_{\beta_k}^{\mathrm{up}}$. 
We obtain 
\[
(\overline{f_{\beta_k}}, f_{\beta_k}^{\mathrm{up}}) = \overline{(f_{\beta_k}, \sigma(f_{\beta_k}^{\mathrm{up}}))} = \overline{1} = 1. 
\]
\end{proof}

Now, we establish a relation between root vectors and determinantial modules.
Recall 
\[
L(\beta_k) = M(w_{\leq k}\Lambda_{i_k},w_{\leq k-1}\Lambda_{i_k}), \widehat{L}(\beta_k) = \widehat{M}(w_{\leq k}\Lambda_{i_k}, w_{\leq k-1}\Lambda_{i_k}).
\]  
The following proposition is well-known, but we prove it for completeness. 

\begin{proposition} \label{prop:rootmodule}
For any $1 \leq k \leq l$, we have 
\[
\Psi_1([\widehat{L}(\beta_k)]) = f_{\beta_k}, \ \Psi_2([L(\beta_k)]) = f_{\beta_k}^{\mathrm{up}}. 
\]
\end{proposition}

Note that $\widehat{L}(\beta_k)$ has finite projective dimension by Lemma \ref{lem:prdim2}, so $\Psi_1([\widehat{L}(\beta_k)])$ is well-defined.  

\begin{proof}
The second assertion follows from Lemma \ref{lem:determinantial} and Lemma \ref{lem:rootvec}. 
By Lemma \ref{lem:character}, we have  
\[
\Psi_1([\widehat{L}(\beta_k)]) = \theta \Psi ([\widehat{L}(\beta_k)]) = \frac{1}{1-q_{i_k}^2} \theta \Psi_2([L(\beta_k)]) = \frac{c(\beta_k)}{1-q_{i_k}^2} \Psi_2([L(\beta_k)]). 
\]
The second assertion implies that it coincides with $f_{\beta_k}$, which proves the first assertion. 
\end{proof}

\begin{definition} \label{def:standardw}
For $\beta \in Q_+$ and $\lambda \in \Lambda(\beta)$, we define graded $R_{w,*}(\beta)$-modules
\begin{align*}
\Delta(\lambda) &= \widehat{L}(\beta_l)^{\circ (\lambda_l)} \circ \cdots \circ \widehat{L}(\beta_1)^{\circ (\lambda_1)}, \\ \index{$\Delta(\lambda)$}
\overline{\nabla}(\lambda) &= D(\overline{\Delta}(\lambda)). \index{$\overline{\nabla}(\lambda)$}
\end{align*}
\end{definition}

\begin{corollary} \label{cor:PBWmodule}
For $\lambda \in \mathbb{Z}_{\geq 0}^l$, we have 
\[
\Psi_1([\Delta(\lambda)]) = f_{\lambda}, \ \Psi_2([\overline{\Delta}(\lambda)]) = f_{\lambda}^{\mathrm{up}}. 
\]
\end{corollary}

\begin{proof}
It follows from (\ref{eq:dividedpower}), Definition \ref{def:rootvec} and Proposition \ref{prop:rootmodule}. 
\end{proof}

\subsection{Evidence of the BGG reciprocity} \label{sub:BGG}

We maintain the setting in the previous section. 
We prove an equation in $\quantum{-}$, which can be regarded as a shadow of the $\Delta$-filtration on projective modules. 
For each $s \in S$, let $\Delta(s)$\index{$\Delta(s)$} be the projective cover of $L(s)$ in $\gMod{R_{*,w}(\beta_s)}$. 

\begin{definition} \label{def:standard}
For $\beta \in Q_+$ and $(\lambda, s) \in \Sigma (\beta)$, we define graded $R(\beta)$-modules
\[
  \Delta(\lambda, s) = \Delta(s) \circ \Delta(\lambda), \ \overline{\nabla} (\lambda, s) = D(\overline{\Delta}(\lambda,s)). 
\] \index{$\Delta(\lambda,s)$} \index{$\overline{\nabla}(\lambda,s)$}
Let $P(\lambda,s)$ be the projective cover of $L(\lambda,s)$ in $\gMod{R(\beta)}$. 
\end{definition}

Note that, by Lemma \ref{lem:projdimind}, Proposition \ref{prop:prdim} and Proposition \ref{prop:prdimR*w}, the projective dimension of $\Delta(\lambda,s)$ is finite. 
At the moment, we do not know whether $\Delta(\lambda,s)$ is the standard module defined in Section \ref{sub:stratified} and $\overline{\nabla}(\lambda,s)$ is the proper costandard module of Definition \ref{def:propercostd}. 
We shall verify that they are in Theorem \ref{thm:standard}. 

When $s = s_0$, we have
\[
\Delta(\lambda,s_0) = \Delta(\lambda), \  \overline{\Delta}(\lambda, s_0) = \overline{\Delta}(\lambda),\ \overline{\nabla}(\lambda, s_0) = \overline{\nabla}(\lambda). 
\]

\begin{lemma} \label{lem:dualbasis}
  Let $\beta \in Q_+$. 
  For $(\lambda,s), (\mu,t) \in \Sigma(\beta)$, we have 
\[
  \Extform{\Delta(\lambda,s)}{\overline{\nabla}(\mu, t)} = \delta_{(\lambda,s),(\mu,t)}. 
 \]
\end{lemma}

\begin{proof}
 In the proof, we repeatedly use Theorem \ref{thm:cuspidalCw} without explicitly referring to it. 

 First, consider the case $\rho(\lambda, s) \not \leq \rho(\mu, t)$.
 Then, we have 
 \[
 \Res_{\beta_s, \lambda_l \beta_l, \ldots, \lambda_1\beta_1} \overline{\Delta}(\mu, t) = 0.
 \] 
 Since $\overline{\nabla}(\mu,t) = D(\overline{\Delta}(\mu, t)) $, we also have $\Res_{\beta_s, \lambda_l \beta_l, \ldots, \lambda_1\beta_1} \overline{\nabla}(\mu, t) = 0$. 
 By the induction-restriction adjunction, we obtain $\Extform{\Delta(\lambda,s)}{\overline{\nabla}(\mu,t)} = 0$. 

 Next, assume $\rho(\lambda,s) \not \geq \rho(\mu, t)$. 
 We have $\Res_{\beta_t, \mu_l\beta_l, \ldots, \mu_1\beta_1} \Delta(\lambda,s) = 0$. 
 Since $\overline{\nabla}(\mu,t) = L(\mu_1\beta_1) \circ \cdots \circ L(\mu_l\beta_l) \circ L(t)$ up to degree shift, using the restriction-coinduction adjunction we deduce $\Extform{\Delta(\lambda,s)}{\overline{\nabla}(\mu,t)} = 0$. 

 It remains to address the case $\rho(\lambda,s) = \rho(\mu,t)$. 
 Since $\Res_{\beta_s, \lambda_l\beta_l, \ldots, \lambda_1\beta_1} \overline{\Delta}(\lambda,t) = L(t) \otimes L(\lambda_l\beta_l) \otimes \cdots \otimes L(\lambda_1\beta_1)$,  
 we have $\Res_{\beta_s, \lambda_l\beta_l, \ldots, \lambda_1\beta_1} \overline{\nabla}(\lambda,t) = L(t) \otimes L(\lambda_l\beta_l) \otimes \cdots \otimes L(\lambda_1\beta_1)$. 
 Using the induction-restriction adjunction, we obtain
 \begin{align*}
 &\Extform{\Delta(\lambda,s)}{\overline{\nabla}(\lambda,t)} \\
 &= \Extform{\Delta(s)}{L(t)} \Extform{\widehat{L}(\beta_l)^{\circ (\lambda_l)}}{L(\lambda_l \beta_l)} \cdots \Extform{\widehat{L}(\beta_1)^{\circ (\lambda_1)}}{L(\lambda_1 \beta_1)}. 
 \end{align*}
 By Theorem \ref{thm:projresol}, $\Extform{\Delta(s)}{L(t)} = \delta_{s,t}$. 
 It suffices to show 
 \[
 \Extform{\widehat{L}(\beta_k)^{\circ (\lambda_k)}}{L(\lambda_k \beta_k)} = 1
 \]
 for each $1 \leq k \leq l$. 
 Using the Mackey filtration as in \cite[Lemma 2.11.]{MR3205728}, we obtain 
 \[ 
 [\Res_{\beta_k^{\lambda_k}} L(\lambda_k \beta_k)] = q_{i_k}^{\lambda_k(\lambda_k - 1)/2} [\Res_{\beta_k^{\lambda_k}} L(\beta_k)^{\circ \lambda_k}] = [\lambda_k]_{i_k}! [L(\beta_k)^{\otimes \lambda_k}]. 
 \]
 Hence, $\Extform{\widehat{L}(\beta_k)^{\circ (\lambda_k)}}{L(\lambda_k\beta_k)} = \Extform{\widehat{L}(\beta_k)}{L(\beta_k)}^{\lambda_k}$. 
 By Lemma \ref{lem:bilinearforms}, Lemma \ref{lem:rootvec} and Proposition \ref{prop:rootmodule}, 
 \[
 \Extform{\widehat{L}(\beta_k)}{L(\beta_k)} = (\overline{f_{\beta_k}}, f_{\beta_k}^{\mathrm{up}}) = 1, 
 \]
 which proves the assertion. 
 \end{proof}

 Now, we demonstrate that, in the Grothendieck group, we have an equation corresponding to the BGG reciprocity. 
 
 \begin{proposition} \label{prop:BGGformula}
 For $(\lambda,s) \in \Sigma(\beta)$, we have the following equation in $K(\gmod{R})_{\mathbb{Q}((q))}$: 
 \[
 [P(\lambda,s)]_q = \sum_{(\mu, t) \in \Sigma(\beta)} \overline{[\overline{\nabla}(\mu,t):L(\lambda,s)]_q}  [\Delta(\mu,t)]_q. 
 \]
 \end{proposition}

 \begin{proof}
  Recall the isomorphism $\Psi_2 \colon K(\gmod{R})_{\mathbb{Q}((q))} \simeq \quantum{-}_{\mathbb{Q}((q))}$. 
  Note that $\{ [L(\lambda,s)] \mid (\lambda,s ) \in \Sigma(\beta) \}$ is a basis of $(\quantum{-}_{\mathbb{Q}((q))})_{-\beta}$. 
  Since the decomposition matrix $([\overline{\nabla}(\lambda,s):L(\mu,t)]_q)_{(\lambda,s),(\mu,t) \in \Sigma(\beta)}$ is unitriangular by Theorem \ref{thm:cuspidalCw}, 
  we see that $\{ [\overline{\nabla}(\lambda,s)]_q \mid (\lambda,s) \in \Sigma(\beta) \}$ is also a basis. 
  Lemma \ref{lem:dualbasis} shows that even $\{ [\Delta(\lambda,s)]_q \mid (\lambda,s) \in \Sigma (\beta) \}$ is a basis. 

 Hence, we may expand $[P(\lambda,s)]_q = \sum_{(\mu,t) \in \Sigma(\beta)} a_{(\mu,t)} [\Delta(\mu,t)]_q$, where $a_{(\mu,t)} \in \mathbb{Q}((q))$.
  By Lemma \ref{lem:bilinearforms}, we have for any $(\lambda',s') \in \Sigma(\beta)$, 
  \[
  [\overline{\nabla}(\lambda',s'):L(\lambda,s)]_q = \Extform{P(\lambda,s)}{\overline{\nabla}(\lambda',s')} = \sum_{(\mu,t)} \overline{a_{(\mu,t)}} \Extform{\Delta(\mu,t)}{\overline{\nabla}(\lambda',s')}. 
  \]
  By Lemma \ref{lem:dualbasis}, it is $\overline{a_{(\lambda',s')}}$, which proves the proposition. 
 \end{proof}

\subsection{Main results} \label{sub:mainresults}

Let $\beta \in Q_+, w \in W$, and take a reduced expression of $w$. 
We adopt the notation from section \ref{sub:Cwv}. 
Recall that simple objects of $\gMod{R(\beta)}$ is parametrized by $\Sigma(\beta)$. 
We have a map $\rho \colon \Sigma(\beta) \to \mathcal{P}^{\preceq}(\beta)$ given by $\rho(\lambda,s) = (\lambda_1\beta_1, \ldots, \lambda_l\beta_l, \beta_s)$. 
We also have a partial order $\leq$ on $\mathcal{P}^{\preceq}(\beta)$ from Definition \ref{def:modifiedlexico}. 

First, we determine the standard modules and the proper costandard modules in $\gMod{R(\beta)}$. 

\begin{theorem} \label{thm:standard}
 With respect to the map $\rho \colon \Sigma(\beta) \to \mathcal{P}^{\preceq}(\beta)$ and the partial order $\leq$ on $\mathcal{P}^{\preceq}(\beta)$, 
 $\Delta(\lambda,s)$ and $\overline{\nabla}(\lambda,s) \ ((\lambda,s) \in \Sigma(\beta))$ as defined in Definition \ref{def:standard}, are the standard modules and the proper costandard modules of $\gMod{R(\beta)}$, respectively. 
\end{theorem}

\begin{proof}
 By Theorem \ref{thm:cuspidalCw}, all the composition factors of $\Delta(\lambda,s)$ is of the form $L(\mu,t) \ (\rho(\mu,t) \leq \rho(\lambda,s))$. 
 In order to prove that $\Delta(\lambda,s)$ is the standard module, 
 it suffices to prove the following claim: 
 for $(\lambda,s), (\mu,t) \in \Sigma(\beta)$ such that $\rho(\lambda,s) \geq \rho(\mu,t)$, 
 we have 
 \[
 \Ext_{R(\beta)}^1(\Delta(\lambda,s), L(\mu,t)) = 0, \ \qdim \Hom_{R(\beta)}(\Delta(\lambda,s), L(\mu,t)) = \delta_{(\lambda,s), (\mu,t)}. 
 \]
By definition, we have 
\[
\Delta(\lambda,s) = \Delta(s) \circ \widehat{L}(\beta_l)^{\circ (\lambda_l)} \circ \cdots \circ \widehat{L}(\beta_1)^{\circ (\lambda_1)}. 
\]
The induction-restriction adjunction shows 
\[
\Ext_{R(\beta)}^k (\Delta(\lambda,s), L(\mu,t)) \simeq \Ext^k (\Delta(s) \otimes \widehat{L}(\beta_l)^{\circ (\lambda_l)} \otimes \cdots \otimes \widehat{L}(\beta_1)^{\circ (\lambda_1)}, \Res_{\beta_s,\lambda_l\beta_l, \ldots, \lambda_1\beta_1}L(\mu,t)), 
\]
where the Ext group on the right-hand side is taken over the algebra ${R(\beta_s) \otimes R(\lambda_l\beta_l) \otimes \cdots R(\beta_1\lambda_1)}$. 
By Theorem \ref{thm:cuspidalCw}, 
the restriction on the right-hand side is zero unless $\rho(\lambda,s) = \rho(\mu,t)$.
In the cases where $\rho(\lambda,s) = \rho(\mu,t)$, Theorem \ref{thm:cuspidalCw} shows that the restriction coincides with 
\[
L(t) \otimes L(\lambda_l\beta_l) \otimes L(\lambda_1 \beta_1), 
\]
which proves the claim using Theorem \ref{thm:ext1}.

To prove that $\overline{\nabla}(\lambda,s)$ is the proper costandard module, 
we check the three conditions in Definition \ref{def:propercostd}. 
Recall that $\overline{\nabla}(\lambda,s)$ is defined as the dual of $\overline{\Delta}(\lambda,s)$. 
The first two conditions follow from Theorem \ref{thm:cuspidalCw}. 
The third condition follows from the restriction-coinduction adjunction in the same manner as in the third paragraph of the proof of Lemma \ref{lem:dualbasis}. 
\end{proof}

Now, we arrive at the main theorem of this paper.   

\begin{theorem} \label{thm:stratification}
 With respect to the map $\rho \colon \Sigma(\beta) \to \mathcal{P}^{\preceq}(\beta)$ and the partial order $\leq$ on $\mathcal{P}^{\preceq}(\beta)$, 
 the category $\gMod{R(\beta)}$ is a stratified category. 
\end{theorem}

\begin{proof}
By Proposition \ref{prop:BGGformula} and Theorem \ref{thm:standard}, we can apply the criterion given in Proposition \ref{prop:criterion}. 
\end{proof}

\begin{corollary} \label{cor:generalhigherextvanishing}
 In the setting above, we have 
 \[
 \Ext_{R(\beta)}^k (\Delta(\sigma), L(\sigma')) = 0
 \]
 for any $k \geq 1, \sigma, \sigma' \in \Sigma(\beta)$ with $\rho(\sigma) \not < \rho(\sigma')$. 
\end{corollary}

\begin{proof}
By Theorem \ref{thm:stratification}, we have a projective resolution $\cdots \to P^2 \to P^1 \to P^0$ of $\Delta(\sigma)$ such that 
\begin{enumerate}
  \item $P^0 = P(\sigma)$, and
  \item for $k \geq 1$, $P^k$ is a direct sum of $P(\sigma')$ with $\rho(\sigma') > \rho(\sigma)$. 
\end{enumerate} 
It proves the assertion. 
\end{proof}

\begin{corollary} \label{cor:higherextvanishing}
  Let $v \in W$ and $i \in I$.
  Assume that $vs_i \geq v$. 
  Then, 
  \[ 
  \Ext_{R(v\alpha_i)}^k (\widehat{M}(vs_i \Lambda_i, v\Lambda_i), M(vs_i \Lambda_i, v\Lambda_i)) = 0
  \] 
  for any $k \geq 1$.   
\end{corollary}

\begin{proof}
This is a special case of Corollary \ref{cor:generalhigherextvanishing}. 
\end{proof}

Next, we consider the truncation from $\gMod{R(\beta)}$ to $\gMod{R_w(\beta)}$. 
Recall the set $\Lambda(\beta)$ that can be identified with $\{ (\lambda,s_0) \in \Sigma(\beta) \}$.
It parametrizes the simple modules of $\gMod{R_{w,*}(\beta)}$. 
Note that it has the following property: for $\sigma, \tau \in \Sigma(\beta)$ with $\rho(\sigma) \leq \rho(\tau)$, $\tau \in \Lambda(\beta)$ implies $\sigma \in \Lambda(\beta)$. 
The map $\rho$ is injective on $\Lambda(\beta)$ and the partial order on $\Lambda(\beta)$ induced from $\leq$ is the bilexicographic, namely the lexicographic order where we read the values from both sides. 

\begin{theorem} \label{thm:affinehighest}
  With respect to the bilexicographic order on $\Lambda(\beta)$, the category $\gMod{R_{w,*}(\beta)}$ is affine highest weight.
  Its standard modules and proper costandard modules are, respectively, $\Delta(\lambda)$ and $\overline{\nabla}(\lambda)$ from Definition \ref{def:standardw}. 
  We also have 
  \[
 \End_{R_{w,*}(\beta)} (\Delta(\lambda)) \simeq \bigotimes_{1 \leq k \leq l} \mathbf{k}[z_{k,1}, \ldots, z_{k,\lambda_k}]^{\Sym_{\lambda_k}}, 
 \]
 where $\deg (z_{k,r}) = (\alpha_{i_k}, \alpha_{i_k})$. 
\end{theorem}

\begin{proof}
  By a truncation argument (\cite[Proposition 5.16]{MR3335289}), we deduce from Theorem \ref{thm:standard} and Theorem \ref{thm:stratification} that $\gMod{R_{w,*}(\beta)}$ is stratified with standard modules $\Delta(\lambda)$ and proper costandard modules $\overline{\nabla}(\lambda)$. 

  We need to determine $\End_{R_{w,*}(\beta)} (\Delta(\lambda))$. 
  Since $(\widehat{L}(\beta_l)^{\circ (\lambda_l)}, \ldots, \widehat{L}(\beta_1)^{\circ (\lambda_1)})$ is unmixing by Lemma \ref{lem:cuspidalseq}, Lemma \ref{lem:unmixing} shows 
  \[ 
  \Res_{\lambda_l\beta_l, \ldots, \lambda_1 \beta_1}\Delta(\lambda) = \widehat{L}(\beta_l)^{\circ (\lambda_l)} \otimes \cdots \otimes \widehat{L}(\beta_1)^{\circ (\lambda_1)}.
  \] 
  Hence, we can compute the endomorphism ring by the induction-restriction adjunction and Proposition \ref{prop:endoring}. 
  
  Since $\widehat{L}(\beta_k)$ is free of finite rank over $\mathbf{k}[z]$, 
  $\widehat{L}(\beta_k)^{\circ \lambda_k}$ is free of finite rank over $\mathbf{k}[z_1, \ldots, z_{\lambda_k}]$.  
  Hence, $\widehat{L}(\beta_k)^{\circ \lambda_k}$ is free of finite rank over $\mathbf{k}[z_1, \ldots, z_{\lambda_k}]^{\Sym_{\lambda_k}}$, so is $\widehat{L}(\beta_k)^{\circ (\lambda_k)}$. 
  It follows that $\Delta(\lambda)$ is free of finite rank over $\End_{R_{w,*}(\beta)} (\Delta(\lambda))$. 

  Therefore, the three requirements in the definition of affine highest weight categories (Definition \ref{def:stratified}) are fulfilled.  
\end{proof}

\begin{remark}
  Our results reproduce the affine highest weight structure for finite-type quiver Hecke algebras \cite{MR3205728}. 
  In fact, our standard modules are the same as theirs, 
  since $\widehat{L}(\beta_k) = \Delta(\beta_k)$ is characterized as a projective cover of $L(\beta_k)$ in the full subcategory 
  \[
  \{M \in \gMod{R(\beta_k)} \mid \text{all the composition factors of $M$ are $L(\beta_k)$} \}.
  \] 
  Here, we used the fact that $(\beta_k)$ is the least element of $\Sigma(\beta_k)$.
\end{remark}

\begin{corollary} \label{cor:gldim}
The global dimension of $\gMod{R_{w,*}(\beta)}$ is finite. 
\end{corollary}

\begin{proof}
  It follows from Theorem \ref{thm:affinehighest} and \cite[Corollary 5.25]{MR3335289}. 
\end{proof}

The following proposition is proved by the same argument as in \cite[Theorem A]{MR3705237}. 

\begin{proposition}
  For different $\lambda, \mu \in \Lambda(\beta)$, we have 
  \[
  \Hom_{R(\beta)} (\Delta(\lambda), \Delta(\mu)) = 0. 
  \]
\end{proposition}

\subsection{Quiver Hecke algebra and quantum unipotent subgroups} \label{sub:quantumunipotent}

Let $w \in W$. 
We define subalgebras of $\quantum{-}$: 
\begin{align*}
\quantum{-}^{w,*} &= \quantum{-} \cap T_w(\quantum{\geq 0}), \\ \index{$\quantum{-}^{w,*}$}
\quantum{-}^{*,w} &= \quantum{-} \cap T_w \quantum{-},  \index{$\quantum{-}^{*,w}$}
\end{align*}
where $\quantum{\geq 0}$ is a $\mathbb{Q}(q)$-subalgebra of $\quantum{}$ generated by $\{q^h \mid h \in P^{\lor} \} \cup \{ e_i \mid i \in I\}$. 
We also define 
\begin{align*}
\quantum{-}_{\mathbb{Z}[q,q^{-1}]}^{w,*} &= \quantum{-}^{w,*} \cap \quantum{-}_{\mathbb{Z}[q,q^{-1}]}, \\ \index{$\quantum{-}_{\mathbb{Z}[q,q^{-1}]}^{w,*}$}
\quantum{-}_{\mathbb{Z}[q,q^{-1}]}^{*,w} &= \quantum{-}^{*,w} \cap \quantum{-}_{\mathbb{Z}[q,q^{-1}]}. \index{$\quantum{-}_{\mathbb{Z}[q,q^{-1}]}^{*,w}$}
\end{align*}
We identify  $\qcoordinate$ with $\quantum{-}$, and $\qcoordinate_{\mathbb{Z}[q,q^{-1}]}$ with the subspace $\quantum{-}_{\mathbb{Z}[q,q^{-1}]}^{\text{up}}$ of $\quantum{-}$.  
We define
\begin{align*}
\qcoordinate_{\mathbb{Z}[q,q^{-1}]}^{w,*} &= \quantum{-}^{w,*} \cap A_q(\mathfrak{n})_{\mathbb{Z}[q,q^{-1}]}, \\ \index{$\qcoordinate_{\mathbb{Z}[q,q^{-1}]}^{w,*}$}
\qcoordinate_{\mathbb{Z}[q,q^{-1}]}^{*,w} &= \quantum{-}^{*,w} \cap A_q(\mathfrak{n})_{\mathbb{Z}[q,q^{-1}]}. \\ \index{$\qcoordinate_{\mathbb{Z}[q,q^{-1}]}^{*,w}$}
\end{align*}

First, we consider $\quantum{-}^{w,*}$. 

\begin{proposition} \label{prop:subspaceUw}
Choose a reduced expression $w = s_{i_1} \cdots s_{i_l}$ to introduce PBW basis $\{f_{\lambda}\}$ and dual PBW basis $\{f_{\lambda}^{\mathrm{up}}\}$ as in Definition \ref{def:rootvec}. 
Then, we have 
\begin{enumerate}
\item $\quantum{-}^{w,*} = \bigoplus_{\lambda \in \mathbb{Z}_{\geq 0}^l} \mathbb{Q}(q) f_{\lambda}, \ \quantum{-}_{\mathbb{Z}[q,q^{-1}]}^{w,*} = \bigoplus_{\lambda \in \mathbb{Z}_{\geq 0}^l} \mathbb{Z}[q,q^{-1}] f_{\lambda}$, \\
\item $A_q(\mathfrak{n})_{\mathbb{Z}[q,q^{-1}]}^{w,*} = \bigoplus_{\lambda \in \mathbb{Z}_{\geq 0}^l} \mathbb{Z}[q,q^{-1}] f_{\lambda}^{\mathrm{up}}$. 
\end{enumerate}
\end{proposition}

\begin{proof}
  (1) is \cite[Proposition 2.3]{MR1712630}, \cite[Theorem 2.18]{MR3582403}. 
  (2) follows from \cite[Theorem 4.25, Theorem 4.29]{MR2914878}. 
\end{proof}

Since the global dimension of $\gMod{R_{w,*}(\beta)}$ is finite for each $\beta \in Q_+$ by Corollary \ref{cor:gldim}, we have an isomorphism of $\mathbb{Z}[q,q^{-1}]$-modules
\[
K_{\oplus}(\gproj{R_{w,*}}) \simeq K(\gMod{R_{w,*}}). 
\]
While $\gproj{R_{w,*}}$ is not closed under convolution products, $\gMod{R_{w,*}}$ is closed under convolution products and $K(\gMod{R_{w,*}})$ has an algebra structure. 
Through the above isomorphism, $K_{\oplus} (\gproj{R_{w,*}})$ inherits an algebra structure. 

\begin{theorem} \label{thm:categorifyUw*}
The isomorphism $\Psi_2 \colon K(\gmod{R}) \xrightarrow{\sim} A_q(\mathfrak{n})_{\mathbb{Z}[q,q^{-1}]}$ restricts to an isomorphism of algebras
\[
K(\gmod{R_{w,*}}) \xrightarrow{\sim} \qcoordinate_{\mathbb{Z}[q,q^{-1}]}^{w,*}. 
\]
The isomorphism $\Psi_1 \colon K(\fpdgMod{R}) \to \quantum{-}_{\mathbb{Z}[q,q^{-1}]}$ restricts to an isomorphism of algebras
\[
K_{\oplus}(\gproj{R_{w,*}}) \simeq K(\gMod{R_{w,*}}) \to \quantum{-}_{\mathbb{Z}[q,q^{-1}]}^{w,*}. 
\]
Moreover, for $L \in \gmod{R_{w,*}(\beta)}$ and $P \in \gproj{R_{w,*}(\beta)}$, we have 
\[
\qdim \Hom_{R_{w,*}(\beta)} (P,L) = (\overline{\Psi_1(P)}, \Psi_2(L)). 
\]
\end{theorem}

\begin{proof}
We take a reduced expression $w = s_{i_1} \cdots s_{i_l}$, and use the notation in Section \ref{sub:mainresults}. 
We prove the theorem using the stratification on $\gMod{R_{w,*}}$ given in Theorem \ref{thm:affinehighest}. 
The first assertion follows from Proposition \ref{prop:subspaceUw} (2) and 
\[
K(\gmod{R_{w,*}}) = \bigoplus_{\lambda \in \mathbb{Z}_{\geq 0}^l} \mathbb{Z}[q,q^{-1}][\overline{\Delta}(\lambda)], \ \Psi_2 ([\overline{\Delta}(\lambda)]) = f_{\lambda}^* \ \text{(Corollary \ref{cor:PBWmodule})}. 
\]
The second assertion follows from Proposition \ref{prop:subspaceUw} (1) and 
\[
K(\gMod{R_{w,*}}) = \bigoplus_{\lambda \in \mathbb{Z}_{\geq 0}^l} \mathbb{Z}[q,q^{-1}][\Delta(\lambda)], \ \Psi_1([\Delta(\lambda)]) = f_{\lambda} \ \text{(Corollary \ref{cor:PBWmodule})}.  
\]
To prove the last equality, recall the discussion before Theorem \ref{thm:affinehighest}. 
By Theorem \ref{thm:stratification}, it implies $\Ext_R^k (P, L) = 0$ for $k \geq 1$. 
Combined with Lemma \ref{lem:bilinearforms}, it yields the assertion.  
\end{proof}

Next, we consider $\quantum{-}^{*,w}$. 

\begin{theorem} \label{thm:categorifyU*w}
Let $w \in W$ and choose a reduced expression $w = s_{i_1} \cdots s_{i_l}$, which specifies the root vectors $f_{\beta} \ (\beta \in \Phi_+ \cap w \Phi_-)$. 
Then, the following four subspaces of $\qcoordinate_{\mathbb{Z}[q,q^{-1}]}$ coincide: 
\begin{enumerate}
 \item $\qcoordinate_{\mathbb{Z}[q,q^{-1}]}^{*,w}$,  
 \item $\Psi_2(K(\gmod{R_{*,w}}))$, 
 \item $\{ x \in \qcoordinate_{\mathbb{Z}[q,q^{-1}]} \mid \text{${e'}_{\nu_1}^* \cdots {e'}_{\nu_{\height \beta}}^* x = 0$ for any $\beta \in \Phi_+ \cap w \Phi_-, \nu \in I^{\beta}$} \}$, and 
 \item $\{ x \in \qcoordinate_{\mathbb{Z}[q,q^{-1}]} \mid \text{$(\quantum{-} f_{\beta}, x) = 0$ for any $\beta \in \Phi_+ \cap w \Phi_-$} \}$. 
\end{enumerate}

Similarly, the following four subspaces of $\quantum{-}$ coincide: 
\begin{enumerate}
  \item $\quantum{-}^{*,w}$,  
  \item $\{ x \in \quantum{-} \mid \text{${e'}_{\nu_1}^* \cdots {e'}_{\nu_{\height \beta}}^* x = 0$ for any $\beta \in \Phi_+ \cap w \Phi_-, \nu \in I^{\beta}$} \}$, and 
  \item $\{ x \in \quantum{-} \mid \text{$(\quantum{-} f_{\beta}, x) = 0$ for any $\beta \in \Phi_+ \cap w \Phi_-$} \}$. 
 \end{enumerate}
\end{theorem}

\begin{proof}
We prove the former assertion. 
The latter assertion is proved in the same manner. 

By \cite[Theorem 2.20]{MR3771147}, the second and the third spaces are the same. 

We prove that the third and the forth spaces are the same using Theorem \ref{thm:categorification}. 
Through the isomorphism $\Psi_2 \colon K(\gmod{R}) \simeq \qcoordinate_{\mathbb{Z}[q,q^{-1}]}$, 
the third space coincides with 
\[
\bigoplus_{\beta \in Q_+} \{x\in K(\gmod{R(\beta)}) \mid \text{${E'}_{\nu_1}^* \cdots {E'}_{\nu_{\height \beta_k}}^* x = 0$ for any $1\leq k\leq l, \nu \in I^{\beta_k}$} \},  
\]
while the forth space coincides with 
\[
\bigoplus_{\beta \in Q_+} \left\{ x \in K(\gmod{R(\beta)}) \middle| \begin{aligned} 
 &\text{$\Extform{P \circ \widehat{M}(\beta_k)}{x} =0$} \\
 &\text{for any $1 \leq k \leq l, P \in \gproj{R(\beta-\beta_k)}$} 
\end{aligned} \right\}. 
\]
Let $x \in K(\gmod{R(\beta)})$ and expand it in $[\overline{\Delta}(\lambda,s)]$ (Theorem \ref{thm:cuspidalCw}): 
\[
x = \sum_{(\lambda,s) \in \Sigma(\beta)} a_{(\lambda,s)} [\overline{\Delta}(\lambda,s)] \quad (a_{(\lambda,s)} \in \mathbb{Q}((q))). 
\]
By the induction-restriction adjunction, for $1 \leq k \leq l$ and $P \in \gproj{R(\beta-\beta_k)}$, we have $\Extform{P \circ \widehat{M}(\beta_k)}{x} = \Extform{P \otimes \widehat{M}(\beta_k)}{\Res_{\beta-\beta_k,\beta_k}x}$. 
By Theorem \ref{thm:cuspidalCw}, for $(\lambda,s) \in \Sigma(\beta)$ that satisfies $\lambda_1 = \cdots \lambda_{k-1} = 0$, we have is equal to
\begin{align*}
&[\Res_{\beta-\beta_k,\beta_k}\overline{\Delta}(\lambda,s)] \\
&= \begin{cases}
  0 & \text{if $\lambda_k = 0$}, \\
  [\lambda_k]_{i_k} [\overline{\Delta}(0, \ldots, 0, \lambda_k-1, \lambda_{k+1}, \ldots, \lambda_l, s) \otimes L(\beta_k)] & \text{if $\lambda_k \neq 0$}.
\end{cases}
\end{align*}
Hence, $x \in K(\gmod{R(\beta)})$ belongs to the forth space if and only if $\Res_{\beta-\beta_k, \beta_k}x = 0$ for all $1 \leq k\leq l$, that is, $x$ belongs to the third one. 

Finally, we prove that the first one coincides with others by induction on $l$. 
If $l=0$, it is trivial. 
If $l=1$, it is \cite[Proposition 38.1.6]{MR2759715}. 
Assume $l \geq 2$. 
Put $i = i_1$. 
We prove that the first space is contained in the forth one. 
Let $x \in \qcoordinate_{\mathbb{Z}[q,q^{-1}]}^{*,w}$. 
By \cite[Proposition 3.4, Lemma 3.6]{MR3582403}, we have $\quantum{-}^{*,w} = T_i \quantum{-}^{*, s_i w} \cap \quantum{-}$.
In particular, $x \in \qcoordinate_{\mathbb{Z}[q,q^{-1}]}^{*,s_i}$, so the induction hypothesis shows $(\quantum{-} f_i, x) = 0$. 
On the other hand, for any $2 \leq k \leq l$ and $y \in \quantum{-}$, we compute up to a nonzero scalar multiple
\begin{align*}
(y f_{\beta_k}, x) &\equiv (T_i^{-1}(y) T_i^{-1}(f_{\beta_k}), T_i^{-1}(x)) \quad \text{by \cite[Proposition 38.2.1]{MR2759715}}\\
 &= (T_i^{-1}(y) T_{s_iw}(f_{i_k}), T_i^{-1}(x)) \\
 &= 0 \quad \text{by applying the induction hypothesis to $T_i^{-1}(x) \in \qcoordinate_{\mathbb{Z}[q,q^{-1}]}^{*,s_iw}$}. 
\end{align*}
Hence, $x$ belongs to the forth space.
The opposite inclusion is proved similarly, which completes the induction step. 
\end{proof}

The following corollary is one of the main theorems of \cite{MR3582403}. 
We give an alternative proof in terms of the categorification. 

\begin{corollary} \label{cor:tensordecomp}
 The multiplication induces isomorphisms
 \begin{align*}
 \quantum{-}^{*,w} \otimes_{\mathbb{Q}(q)} \quantum{-}^{w,*} &\to \quantum{-}, \\
 \quantum{-}_{\mathbb{Z}[q,q^{-1}]}^{*,w} \otimes_{\mathbb{Z}[q,q^{-1}]} \quantum{-}_{\mathbb{Z}[q,q^{-1}]}^{w,*} &\to \quantum{-}_{\mathbb{Z}[q,q^{-1}]}, \\
 \qcoordinate_{\mathbb{Z}[q,q^{-1}]}^{*,w} \otimes_{\mathbb{Z}[q,q^{-1}]} \qcoordinate_{\mathbb{Z}[q,q^{-1}]}^{w,*} &\to \qcoordinate_{\mathbb{Z}[q,q^{-1}]}. 
 \end{align*}
\end{corollary}

\begin{proof}
We use Theorem \ref{thm:stratification}, \ref{thm:categorifyUw*} and \ref{thm:categorifyU*w}. 
The first and the third isomorphisms are deduced from
\[
\overline{\Delta}(\lambda,s) = L(s) \circ \overline{\Delta}(\lambda). 
\]

It remains to prove the second isomorphism. 
By the first isomorphism, it is injective. 
Let $s \in S$. 
Since $\Delta(s)$ is an projective $R_{*,w}(\beta_s)$-module, it has finite projective dimension over $R(\beta_s)$ by Proposition \ref{prop:prdimR*w}. 
By considering the composition factors of $\Delta(s)$, we deduce from Lemma \ref{lem:character} and Theorem \ref{thm:categorifyU*w} that $\Psi_1([\Delta(s)]) \in \quantum{-}_{\mathbb{Z}[q,q^{-1}]}^{*,w}$. 
Hence the surjectivity follows from 
\[
\Delta(\lambda,s) = \Delta(s) \circ \Delta(\lambda). 
\]
\end{proof}

\begin{corollary}
  Let $w \in W$ and choose a reduced expression $w = s_{i_1} \cdots s_{i_l}$, which specifies the root vectors $f_{\beta} \ (\beta \in \Phi_+ \cap w \Phi_-)$. 
  Then, the following four subspaces of $\quantum{-}_{\mathbb{Z}[q,q^{-1}]}$ coincide: 
  \begin{enumerate}
   \item $\quantum{-}_{\mathbb{Z}[q,q^{-1}]}^{*,w}$,  
   \item $\Psi_1(K_{\oplus}(\gproj{R_{*,w}}))$, 
   \item $\{ x \in \quantum{-}_{\mathbb{Z}[q,q^{-1}]} \mid \text{${e'}_{\nu_1}^* \cdots {e'}_{\nu_{\height \beta}}^* x = 0$ for any $\beta \in \Phi_+ \cap w \Phi_-, \nu \in I^{\beta}$} \}$, and 
   \item $\{ x \in \quantum{-}_{\mathbb{Z}[q,q^{-1}]} \mid \text{$(\quantum{-} f_{\beta}, x) = 0$ for any $\beta \in \Phi_+ \cap w \Phi_-$} \}$. 
  \end{enumerate}
\end{corollary}

\begin{proof}
The coincidence of the first, third and fourth spaces is obtained by Theorem \ref{thm:categorifyU*w}. 
By Corollary \ref{cor:tensordecomp}, $\{ \Psi_1([\Delta(s)]) \mid s \in S\}$ is a $\mathbb{Z}[q,q^{-1}]$-basis of $\quantum{-}_{\mathbb{Z}[q,q^{-1}]}^{w,*}$. 
The corollary is proved. 
\end{proof}

\section{Affine Lie types} \label{sec:affinePBW}

Let $C = (c_{i,j})_{i,j\in I}$ be a generalized Cartan matrix of affine type $X_N^{(r)}$ and take a root datum $(C,P,\Pi,\Pi^{\lor},(\cdot,\cdot))$. 
Let $0 \in I$ be the exceptional Dynkin node as in \cite{MR2066942} and put $\mathring{I}\index{$\mathring{I}$} = I \setminus \{ 0\}$. 
It is a natural choice when realizing $\mathfrak{g}$ as a loop algebra. 
Let $\mathring{\mathfrak{g}}$\index{$\mathring{\mathfrak{g}}$} be the simple Lie algebra corresponding to $\mathring{I}$, 
$\mathring{\Phi}$\index{$\mathring{\Phi}$} be its root system, 
$\mathring{Q}_+\index{$\mathring{Q}_+$} = \sum_{i \in \mathring{I}} \mathbb{Z}_{\geq 0} \alpha_i$, 
$\mathring{Q}_-\index{$\mathring{Q}_-$} = - \mathring{Q}_+$, 
and $\mathring{W}$\index{$\mathring{W}$} its Weyl group.     
Let $\prroot$ be the set of positive real roots and $\delta$\index{$\delta$} be the minimal imaginary root. 
Then we have $\minroot = \prroot \sqcup \{ \delta \}$. 
Note that our normalization of $(\cdot, \cdot)$ can be different from \cite{MR2066942}. 
For each $\alpha \in \mathring{\Phi}$ or $\alpha \in \Phi^{\text{re}}$, we define 
\[
d_{\alpha} = \begin{cases}
  r & \text{if $r \geq 2, X_N^{(r)} \neq A_{2l}^{(2)}$ and $\alpha$ is a long root}, \\
  2 & \text{if $X_N^{(r)} = A_{2l}^{(2)}, \alpha \in W\alpha_0$ }, \\
  1 & \text{otherwise}. 
\end{cases}
\] \index{$d_{\alpha}$}
Then, $(\alpha + \mathbb{Z}\delta) \cap \Phi = \alpha + d_{\alpha} \mathbb{Z}\delta$. 
For $i \in \mathring{I}$, we write $d_i\index{$d_i$} = d_{\alpha_i}$.

We fix a family of polynomials $(Q_{i,j})_{i,j \in I}$ to define the quiver Hecke algebra.

\subsection{Convex orders and cuspidal decomposition} \label{sub:affinecuspdecomp}

First, we demonstrate some useful properties of convex orders for the affine root system.

\begin{lemma} [{\cite[Proposition 2.3]{MR3874704}}] \label{lem:onerow} 
 Let $\preceq$ be a convex order on $\minroot$. 
 Assume that every real root $\alpha$ is finitely far from one end of the order, that is, either $\Phi_{+, \preceq \alpha}^{\mathrm{min}}$ or $\Phi_{+, \succeq \alpha}^{\mathrm{min}}$ is a finite set. 
 Then, there exists an infinite word $\cdots, i_{-2}, i_{-1}, i_0, i_1, i_2, \cdots$ such that every finite successive subword is reduced and, by setting 
 \[
 \beta_k = \begin{cases}
  s_{i_1} s_{i_2} \cdots s_{i_{k-1}} \alpha_{i_k} & k \geq 1, \\
  s_{i_0} s_{i_{-1}} \cdots s_{i_{k+1}} \alpha_{i_k} & k \leq 0,
 \end{cases}
 \]
 for each $k \in \mathbb{Z}$, we have $\prroot = \{\beta_k\}_{k \in \mathbb{Z}}$ and 
 \[
 \beta_1 \prec \beta_2 \prec \cdots \prec \delta \prec \cdots \prec \beta_{-1} \prec \beta_0. 
 \] 
\end{lemma}

\begin{proof}
  It is proved by the same argument as in \cite[Lemma 3.8]{MR3694676}. 
\end{proof}

We call the convex order in this proposition a one-row order. 

\begin{lemma} [{\cite[Lemma 2.8]{MR3874704}}]\label{lem:convexextension} 
 Let $A \subset \minroot$ be a finite subset and $\preceq$ be a convex order on $A$. 
 Then, it extends to a one-row order on $\minroot$. 
\end{lemma}

Let $\preceq$ be a convex order on $\minroot$. 
Let $p\index{$p$} \colon \mathfrak{h}^* \to \mathring{\mathfrak{h}}^*$ denote the projection, 
where $\mathfrak{h}$\index{$\mathfrak{h}$} is the Cartan subalgebra of $\mathfrak{g}$, 
and $\mathring{\mathfrak{h}}$\index{$\mathring{\mathfrak{h}}$} is the Cartan subalgebra of $\mathring{\mathfrak{g}}$.   
Then, there exists a unique $w  \in \mathring{W}$ such that $p(\Phi_{+, \succ \delta}^{\text{min}}) \subset w \mathring{Q}_+, p(\Phi_{+, \prec \delta}^{\text{min}}) \subset w \mathring{Q}_-$ \cite[Lemma 3.7]{MR3694676}. 

\begin{definition} [{\cite[Definition 2.9]{MR3874704}}]
  We call the element $w$ defined above the coarse type of $\preceq$. 
\end{definition}

\begin{remark}
\cite[Definition 2.9]{MR3874704} contains a typographical error: 
their definition should be modified by a right multiplication by $w_0$.
Then it is compatible with our definition under the correspondence given in Appendix \ref{app:B}. 
\end{remark}

By the definition, the algebra $R^{\preceq}(n\delta)$ depends only on the coarse type of $\preceq$, 
so we write $R^w(n\delta)\index{$R^w(n\delta)$} = R^{\preceq}(n\delta)$. 

For $\alpha \in \mathring{\Phi}$, let $\widetilde{\alpha}\index{$\widetilde{\alpha}$} \in \Phi_+$ be the positive root such that 
\[
p(\widetilde{\alpha}) = \alpha, \text{$\widetilde{\alpha} - a\delta \not \in \Phi_+$ for any $a > 0$. }
\]

\begin{lemma} \label{lem:onerowexist}
  For any $w \in \mathring{W}$, there exists a one-row order of coarse type $w$. 
\end{lemma}

\begin{proof}
 Let $\preceq_0$ be a convex order on $\mathring{\Phi}_+$, and write $\mathring{\Phi}_+ = \{ \beta_1 \prec_0 \beta_2 \prec_0 \cdots \prec_0 \beta_m \}$. 
 For instance, take a reduced expression of the longest element of $\mathring{W}$, $w_0 = s_{i_1} \cdots s_{i_m}$, and define $\alpha_{i_1} \prec_0 s_{i_1}\alpha_{i_2} \prec_0 \cdots \prec_0 s_{i_1}\cdots s_{i_{m-1}}\alpha_{i_m}$. 
  We define a convex order $\preceq_1$ on $\{ \widetilde{w\alpha} \mid \alpha \in \mathring{\Phi}_+\} \cup \{\delta\}$ by
 \[
 \delta \prec_1 \widetilde{w\beta_1} \prec_1 \cdots \prec_1 \widetilde{w\beta_m}. 
 \]
 By Lemma \ref{lem:convexextension}, we can extend $\preceq_1$ to a one-row order $\preceq_2$ on $\minroot$. 
 Since $w\mathring{\Phi}_+ \subset p(\Phi_{+, \succ_2 \delta}^{\text{min}})$, the coarse type of $\preceq_2$ must be $w$. 
 The lemma is proved. 
\end{proof}

Next, we explain how simple modules decompose into cuspidal modules in affine cases. 
Let $w \in \mathring{W}$. 
For a multipartition $\underline{\lambda} = (\lambda^{(i)})_{i \in \mathring{I}}$, we define \index{$\| \underline{\lambda} \|$}
\[
\| \underline{\lambda} \| = \sum_{i \in \mathring{I}} d_i \lvert \lambda^{(i)}\rvert = \sum_{i \in \mathring{I}, k \geq 1} d_i \lambda_k^{(i)}. 
\]

\begin{proposition} \label{prop:imaginarycuspnumber}
 Let $n \in \mathbb{Z}_{\geq 0}$. 
 The number of isomorphism classes of self-dual simple $R^w(n\delta)$-modules is equal to the number of multipartitions $\underline{\lambda} = (\lambda^{(i)})_{i \in I_0}$ with $\| \underline{\lambda}\| = n$. 
\end{proposition}

\begin{proof}
It follows from Theorem \ref{thm:cuspidaldecomp} and the fact that 
\[ 
\dim \mathfrak{g}_{m\delta}^+ = \begin{cases}
  \lvert \mathring{I} \rvert & \text{if $\mathfrak{g} = A_{2l}^{(2)}$ or $m \equiv 0 \mod r$}, \\
  \lvert \mathring{I}^{\text{short}} \rvert & \text{if $\mathfrak{g} \neq A_{2l}^{(2)}$ and $m \not \equiv 0 \mod r$},  
\end{cases}
\]
where $\mathring{I}^{\text{short}} = \{ i \in \mathring{I} \mid \text{$\alpha_i$ is a short root of $\mathring{\Phi}$} \}$.
\end{proof}

Put $T_n^w\index{$T_n^w$} = \{ \text{self-dual simple $R^w(n\delta)$-module}\}/{\simeq}$ for each $n \geq 0$. 
For each $t \in T_n^w$, let $L(t)$\index{$L(t)$} denote a representative of $t$. 
We set $T^w\index{$T^w$} = \bigsqcup_{n \geq 0} T_n^w$ and define $\beta_t\index{$\beta_t$} = n\delta$ for each $t \in T_n^w$.  

Let $\preceq$ be a convex order on $\minroot$ of coarse type $w$. 
For $\beta \in Q_+$, we define $\Omega^{\preceq} (\beta)$\index{$\Omega^{\preceq}(\beta)$} as the set 
\[
\left\{ (\boldsymbol{c}_-, t, \boldsymbol{c}_+) \mathrel{}\middle|\mathrel{} \begin{gathered} 
  \boldsymbol{c}_- \colon \Phi_{+, \prec \delta}^{\text{min}} \to \mathbb{Z}_{\geq 0}, t \in T^w, \boldsymbol{c}_+ \colon \Phi_{+, \succ \delta}^{\text{min}} \to \mathbb{Z}_{\geq 0}, \\
  \text{$\boldsymbol{c}_{+} (\alpha) = 0$ for all but finitely many $\alpha \in \Phi_{+, \prec \delta}^{\mathrm{min}}$}, \\
  \text{$\boldsymbol{c}_{-} (\alpha) = 0$ for all but finitely many $\alpha \in \Phi_{+, \succ \delta}^{\mathrm{min}}$}, \\
  \left(\sum_{\alpha \in \Phi_{+, \prec \delta}^{\text{min}}} \boldsymbol{c}_-(\alpha)\alpha \right) + \beta_t + \left( \sum_{\alpha \in \Phi_{+, \succ \delta}^{\text{min}}} \boldsymbol{c}_+(\alpha)\alpha \right) = \beta.    
\end{gathered} \right\}. 
\]

Let $\omega = (\boldsymbol{c}_-, t, \boldsymbol{c}_+) \in \Omega^{\preceq}(\beta)$.
We write 
\begin{align*}
\{\alpha \in \Phi_{+, \prec \delta}^{\text{min}} \mid \boldsymbol{c}_-(\alpha) \neq 0 \} &= \{ \alpha_{-,1} \prec \alpha_{-,2} \prec \cdots \prec \alpha_{-,n_-}\}, \\
\{\alpha \in \Phi_{+, \succ \delta}^{\text{min}} \mid \boldsymbol{c}_+(\alpha) \neq 0 \} &= \{ \alpha_{+,1} \prec \alpha_{+,2} \prec \cdots \prec \alpha_{+,n_+}\}. 
\end{align*}
We define
\begin{align*}
\overline{\Delta}(\omega) =& q^{m_{\omega}}L(\alpha_{+,n_+})^{\circ \boldsymbol{c}_+(\alpha_{+,n_+})} \circ \cdots \circ L(\alpha_{+,1})^{\circ \boldsymbol{c}_+(\alpha_{+,1})} \circ L(t) \\
&\circ L(\alpha_{-,n_-})^{\circ \boldsymbol{c}_-(\alpha_{-,n_-})} \circ \cdots \circ L(\alpha_{-,1})^{\circ \boldsymbol{c}_-(\alpha_{-,1})}, 
\end{align*} \index{$\overline{\Delta}(\omega)$}
where 
\begin{align*}
m_{\omega} =& \sum_{1 \leq k\leq n_+} \frac{\boldsymbol{c}_+(\alpha_{+,k})(\boldsymbol{c}_+(\alpha_{+,k})-1) (\alpha_{+,k},\alpha_{+,k})}{4} \\
&+ \sum_{1 \leq k\leq n_-} \frac{\boldsymbol{c}_-(\alpha_{-,k})(\boldsymbol{c}_-(\alpha_{-,k})-1) (\alpha_{-,k},\alpha_{-,k})}{4}.  
\end{align*}
Let $L(\omega)$\index{$L(\omega)$} denote the head of $\overline{\Delta}(\omega)$. 
Applying Theorem \ref{thm:cuspidaldecomp} to this situation, we obtain the following proposition. 

\begin{proposition} \label{prop:affinecuspdecomp}
 Let $\beta \in Q_+$.
 Then, we have a bijection
 \[
 \Omega^{\preceq}(\beta) \to \{ \text{self-dual simple $R(\beta)$-module} \} / {\simeq} ; \omega \mapsto L(\omega). 
 \]
\end{proposition}

\subsection{Real root modules} \label{sub:realrootmodules}

Let $\preceq$ be a convex order on $\minroot$. 
For each $\alpha \in \prroot$, we shall describe the cuspidal module $L(\alpha) = L^{\preceq}(\alpha)$ using a determinantial module.  
Let $A = \{ \gamma \in \Phi_+ \mid \alpha -\gamma \in Q_+ \} \cup \{\delta\}$, which is a finite set. 
By Lemma \ref{lem:convexextension}, there exists a one-row order $\preceq'$ on $\minroot$ that coincides with $\preceq$ when restricted to $A$. 
Let $\cdots, i_{-1},i_0, i_1, \cdots$ be the infinite word that yields $\preceq'$ and define $\beta_k$ as in Lemma \ref{lem:onerow}. 
Then, there exists $k \in \mathbb{Z}$ such that $\beta_k = \alpha$. 
Since both $\beta_k = \alpha$ and $\delta$ belong $A$, it follows that 
\begin{enumerate}
  \item if $\alpha \prec \delta$, we have $k \geq 1$, 
  \item if $\alpha \succ \delta$, we have $k \leq 0$. 
\end{enumerate}

Recall the automorphism $\psi$ of $R(\alpha)$ from Section \ref{sub:quiverHecke}. 

\begin{proposition}[{cf. \cite[Theorem 9.1]{MR3694676}}] \label{prop:realrootmodule}
 We have 
 \[ 
 L(\alpha) \simeq \begin{cases}
 M(s_{i_1} s_{i_2} \cdots s_{i_k}\Lambda_{i_k}, s_{i_1} s_{i_2}\cdots s_{i_{k-1}}\Lambda_{i_k}) & \text{if $\alpha \prec \delta$}, \\
 \psi_* M(s_{i_0} s_{i_{-1}} \cdots s_{i_{k}}\Lambda_{i_k}, s_{i_0} s_{i_{-1}} \cdots s_{i_{k+1}}\Lambda_{i_k}) & \text{if $\alpha \succ \delta$}. 
 \end{cases}
 \]
\end{proposition}

\begin{proof}
 First, consider the case $\alpha \prec \delta$. 
 Put $M = M(s_{i_1} \cdots s_{i_k}\Lambda_{i_k}, s_{i_1} \cdots s_{i_{k-1}}\Lambda_{i_k})$. 
 By Corollary \ref{cor:cuspidalcriterion}, it is enough to show that $\W(M) \cap \Phi_+ \subset \Phi_{+, \preceq \alpha}$. 
 Take $\gamma \in \W(M) \cap \Phi_+$. 
 By Lemma \ref{lem:detsupport}, we have $\W(M) \subset Q_+ \cap s_{i_1} \cdots s_{i_k} Q_-$. 
 Hence, we deduce $\gamma \in \{ \beta_1, \ldots, \beta_k \}$, that is, $\gamma \preceq' \alpha$. 
 We also have $\gamma \in \W(M) \subset A$ and $\alpha \in A$. 
 Since $\preceq$ and $\preceq'$ coincide on $A$, we deduce $\gamma \preceq \alpha$. 

 The case $\alpha \succ \delta$ is proved in a similar manner, noting that $\W^* (\psi_* (X)) = \W(X)$ for an $R(\alpha)$-module $X$. 
\end{proof}

We define $\widehat{L}(\alpha)$ as the projective cover of $L(\alpha)$ in \index{$\widehat{L}(\alpha)$}
\[ 
\{ M \in \gMod{R(\alpha)} \mid \text{every composition factor of $M$ is $L(\alpha)$}\}.
\] 
By Theorem \ref{thm:ext1} and the proposition above, we have 
\[
\widehat{L}(\alpha) \simeq \begin{cases}
\widehat{M}(s_{i_1} \cdots s_{i_k} \Lambda_{i_k}, s_{i_1}\cdots s_{i_{k-1}}\Lambda_{i_k}) & \text{if $\alpha \prec \delta$}, \\
\psi_* \widehat{M}(s_{i_0} s_{i_{-1}} \cdots s_{i_{k}}\Lambda_{i_k}, s_{i_0} s_{i_{-1}} \cdots s_{i_{k+1}}\Lambda_{i_k}) & \text{if $\alpha \succ \delta$}. 
\end{cases} 
\]

\begin{remark}
For $\alpha \in \prroot$, $\Psi_1([\widehat{L}(\alpha)])$ coincides with the real root vector with respect to the convex order $\preceq$ defined in \cite[Definition 4.5]{MR3874704}. 
\end{remark}

\subsection{Stratification for affine Lie types} \label{sub:affinestratification}

Let $\preceq$ be a convex order on $\minroot$ of coarse type $w$, and $\beta \in Q_+$. 
Recall $\Omega^{\preceq}(\beta)$ from Section \ref{sub:affinecuspdecomp}. 
The set $\mathcal{P}^{\preceq}(\beta)$ is identified with the set consisting of maps from $\minroot$ to $\mathbb{Z}_{\geq 0}$, as explained in Section \ref{sub:cuspidal}. 
There is a map $\rho \colon \Omega^{\preceq}(\beta) \to \mathcal{P}^{\preceq}(\beta)$ given by
\[
\rho(\omega) \rvert_{\Phi_{+, \prec \delta}^{\text{min}}} = \boldsymbol{c}_-, \ \rho(\omega) \rvert_{\Phi_{+, \succ \delta}^{\text{min}}} = \boldsymbol{c}_+, \ \rho (\omega)(\delta) = \| \boldsymbol{c}_{\delta}\|, 
\]
for $\omega = (\boldsymbol{c}_-, t, \boldsymbol{c}_+) \in \Omega^{\preceq} (\beta)$. 

For $t \in T_n^w$, let $\Delta(t)$ denote the projective cover of $L(t)$ in $\gMod{R^w(n\delta)}$. 
Using the notation in Section \ref{sub:affinecuspdecomp}, we define, for $\omega \in \Omega^{\preceq}(\beta)$,
\begin{align*}
\Delta(\omega) =& \widehat{L}(\alpha_{+,n_+})^{\circ (\boldsymbol{c}_+(\alpha_{+,n_+}))} \circ \cdots \circ \widehat{L}(\alpha_{+,1})^{\circ (\boldsymbol{c}_+(\alpha_{+,1}))} \circ \Delta(t) \\
&\circ \widehat{L}(\alpha_{-,n_-})^{\circ (\boldsymbol{c}_-(\alpha_{-,n_-}))} \circ \cdots \circ \widehat{L}(\alpha_{-,1})^{\circ (\boldsymbol{c}_-(\alpha_{-,1}))}, \\ \index{$\Delta(\omega)$}
\overline{\nabla}(\omega) =& D(\overline{\Delta}(\omega)). \index{$\overline{\nabla}(\omega)$}
\end{align*} 

\begin{theorem} \label{thm:finerstratification}
With respect to the map $\rho \colon \Omega^{\preceq}(\beta) \to \mathcal{P}^{\preceq}(\beta)$ and the partial order $\leq$ on $\mathcal{P}^{\preceq}(\beta)$, 
the category $\gMod{R(\beta)}$ is stratified. 
Its standard modules are $\Delta(\omega)$ and its proper costandard modules are $\overline{\nabla}(\omega)$. 
\end{theorem}

The proof is similar to Theorem \ref{thm:stratification}, so we omit the details. 
In the proof, we need the following analogue to Theorem \ref{thm:projresol} and Proposition \ref{prop:prdimR*w}.

\begin{proposition} \label{prop:affineprojresol}
 Let $w \in \mathring{W}, n \geq 0$. 
 \begin{enumerate}
\item For $s \in T_n^w$, $\Delta(s)$ has finite projective dimension as an $R(n\delta)$-module.
\item For $s,t \in T_n^w$, we have $\Ext_{R(n\delta)}^k (\Delta(s), L(t)) = 0$ for any $k \geq 1$. 
 \end{enumerate}
\end{proposition}

\begin{proof}
By Lemma \ref{lem:onerowexist}, there exists a one-row convex order $\preceq$ of coarse type $w$.
Then, $R^w(n\delta) = R^{\preceq}(n\delta)$. 
Define $\beta_k \ (k \in \mathbb{Z})$ as in Lemma \ref{lem:onerow}. 
Note that for sufficiently large $M \in \mathbb{Z}$, we have $\height \beta_k > \height (n\delta)$ for any $\lvert k \rvert > M$. 
Hence, Corollary \ref{cor:cuspidalcriterion} shows that 
\[
R^w (n\delta) = R^{\preceq}(n\delta) = R(n\delta)/\langle e(*, \beta_k), e(\beta_{k'},*) \mid 1 \leq k\leq M, -M \leq k' \leq 0 \rangle. 
\]
The proposition is proved by induction similar to Proposition \ref{prop:projresol}, Theorem \ref{thm:projresol} and Proposition \ref{prop:prdimR*w}. 
Precisely speaking, we use the following quotient rings $R(n\delta)_{k,m} \ (k\in \mathbb{Z}, m \geq 0)$ as intermediate objects: 
\[ 
\begin{cases} 
R(n\delta)/\langle e(*,\beta_{k'}) \ (1 \leq k' \leq k-1), \  e(*, (m+1)\beta_k) \rangle & \text{if $k \geq 1, m \geq 0$}, \\
R(n\delta)/\langle e(*,\beta_{k'}) \ (k' \geq 1), e(\beta_{k'},*) \ (k+1 \leq k' \leq 0), \ e((m+1)\beta_k,*) \rangle & \text{if $k \leq 0, m \geq 0$}. 
\end{cases}
\]
Taking $m$ sufficiently large, we have a sequence of quotients
\begin{align*}
R(n\delta) &= R(n\delta)_{1,m} \twoheadrightarrow R(n\delta)_{1,m-1} \twoheadrightarrow \cdots \twoheadrightarrow R(n\delta)_{1,0} \\
&= R(n\delta)_{2,m} \twoheadrightarrow R(n\delta)_{2,m-1} \twoheadrightarrow \cdots \twoheadrightarrow R(n\delta)_{2,0} \\
&= \cdots \\
&= R(n\delta)_{M,m} \twoheadrightarrow R(n\delta)_{M,m-1} \twoheadrightarrow \cdots \twoheadrightarrow R(n\delta)_{M,0} \\
&= R(n\delta)_{0,m} \twoheadrightarrow R(n\delta)_{0,m-1} \twoheadrightarrow \cdots \twoheadrightarrow R(n\delta)_{0,0} \\
&= R(n\delta)_{-1,m} \twoheadrightarrow R(n\delta)_{-1,m-1} \twoheadrightarrow \cdots \twoheadrightarrow R(n\delta)_{-1,0} \\
&= \cdots \\
&= R(n\delta)_{-M,m} \twoheadrightarrow R(n\delta)_{-M,m-1} \twoheadrightarrow \cdots \twoheadrightarrow R(n\delta)_{-M,0} \\
&= R^w(m\delta). 
\end{align*}
We can argue by induction along this sequence. 
When passing from $R(n\delta)_{k,m'}$ to $R(n\delta)_{k,m'-1}$ for $k \leq 0$, 
we use the left-right reversed version of Proposition \ref{prop:projresol} by applying $\psi_*$. 
\end{proof}

\subsection{Minimal imaginary root modules} \label{sub:minimalimaginary}

Let $w \in \mathring{W}$ and $i \in \mathring{I}$.
In this section, we introduce certain cuspidal $R(d_i\delta)$-module $L_i^w(d_i\delta)$ and prove that the corresponding standard module categorifies the imaginary root vector of \cite{MR3874704}. 

\begin{proposition} [{cf.\cite[Theorem 13.1]{MR3694676} \cite[Corollary 4.6]{MR3874704}}] \label{prop:chambercoweight}
 Let $\preceq, \preceq'$ be two convex orders with the same coarse type $w$. 
 For each $\epsilon \in \{+,-\}$ and $n \in \mathbb{Z}_{\geq 0}$, we have 
 \[
 L^{\preceq}( \widetilde{\epsilon w\alpha_i} + d_i n \delta) = L^{\preceq'}(\widetilde{\epsilon w\alpha_i} + d_i n \delta). 
\]
\end{proposition}

\begin{proof}
  We only prove the assertion for $\widetilde{w\alpha_i} + d_i n\delta$. 
  The other assertion is proved in the same manner. 

  Put $\beta = \widetilde{w\alpha_i} + d_i n \delta$. 
  Then by the definition of coarse type, we have $\beta \succ \delta$ and $\beta \succ' \delta$. 
  Let $\gamma \in \W(L^{\preceq}(\beta)) \cap \Phi_+$. 
  Since $L^{\preceq}(\beta)$ is $\preceq$-cuspidal, we have $\gamma \preceq \beta$ and $\beta-\gamma \in \spn_{\mathbb{Z}_{\geq 0}} \Phi_{+, \succeq \beta}^{\text{min}}$. 
  Note that if $\alpha \in \Phi_{+, \succeq \beta}^{\text{min}}$, then $\alpha \succeq \beta \succ \delta$ and we obtain $p(\alpha) \in w \mathring{Q}_+$. 
  Since $\beta-\gamma$ is a $\mathbb{Z}_{\geq 0}$ linear combination of such $\alpha$, we have $p(\beta-\gamma) \in w \mathring{Q}_+$. 

  Assume $p(\beta-\gamma) = 0$, that is, $\beta - \gamma \in \mathbb{Z}_{\geq 0}\delta$. 
  Since $\beta - \gamma \in \spn_{\mathbb{Z}_{\geq 0}} \Phi_{+, \succeq \beta}^{\mathrm{min}}$ and $\beta \succ \delta$, we must have $\beta- \gamma = 0$ by convexity. 
  Hence, $\gamma = \beta$ in this case. 

  Next, assume that $p(\beta-\gamma) \neq 0$.  
  Since $w\alpha_i$ is a simple root of $\mathring{\Phi}$ with respect to the positive system $w\mathring{\Phi}_+$ and we have a decomposition $w\alpha_i = p(\gamma) + p(\beta - \gamma)$, 
  we deduce that $p(\gamma) \in w \mathring{Q}_-$. 
  Since the coarse type of $\preceq'$ is $w$, it follows that $\gamma \preceq' \delta \preceq' \beta$.

  We have proved that $\W(L^{\preceq}(\beta)) \cap \Phi_+ \subset \Phi_{+, \preceq' \beta}$. 
  Hence, the assertion follows from Corollary \ref{cor:cuspidalcriterion}. 
\end{proof}

Let $\beta_{\epsilon}\index{$\beta_+, \beta_-$}  = \widetilde{\epsilon w\alpha_i} \ (\epsilon \in \{+, -\})$. 
Although they depend on $w$ and $i$, we do not make it explicit as it should be clear from the context. 
For a convex order $\preceq$ of coarse type $w$, the cuspidal module $L^{\preceq}(\beta_{\epsilon})$ only depends on $w$ by Proposition \ref{prop:chambercoweight}. 
Thus, let $L^{w}(\beta_{\epsilon})$\index{$L^w(\beta_+)$}\index{$L^w(\beta_-)$} denote this module.    
Similarly, $\widehat{L}^{\preceq}(\beta_{\epsilon})$ only depend on the coarse type $w$, 
so let $\widehat{L}^w(\beta_{\epsilon})$\index{$\widehat{L}^w(\beta_+)$}\index{$\widehat{L}^w(\beta_-)$} denote this module.  

\begin{lemma}[{\cite[Section 3.4]{MR3542489}}] \label{lem:adaptedconvex} 
  Put $F\index{$F$} = \spn_{\mathbb{R}} \{\beta_+, \beta_-\}$. 
  Then, there exists a convex preorder $\preceq$ on $\minroot$ such that 
  \begin{enumerate}
  \item all the roots in $F \cap \minroot$ belong to the same $\preceq$-equivalence class, 
  \item for any $\alpha_{\epsilon} \in (p^{-1}(w\mathring{Q}_{\epsilon}) \setminus F) \cap \minroot \ (\epsilon \in \{+, -\})$ and $\beta \in F \cap \minroot$, we have $\alpha_- \prec \beta \prec \alpha_+$. 
  \end{enumerate}
\end{lemma}

\begin{proof}
  Let $\{\omega_j^{\lor}\}_{j \in \mathring{I}} \subset \mathring{\mathfrak{h}}$ be the set of fundamental coweights for the simple system $\{\alpha_j\}_{j \in \mathring{I}}$. 
  Let $\rho^{\lor} \in \mathfrak{h}$ be a coweight such that $\langle \rho^{\lor}, \alpha_j \rangle = 1$ for all $j \in I$.  
  We introduce an $\mathbb{R}$-linear map $c: \spn_{\mathbb{R}} \Phi \to \mathbb{C}$ given by  
  \[
  c(\beta) = - \sum_{j \in \mathring{I}\setminus \{i \}}  \langle  w \omega_j^{\lor}, \beta \rangle + \sqrt{-1} \langle \rho^{\lor}, \beta \rangle. 
  \]
  As discussed in \cite[Section 1.2]{MR3542489}, we can define a convex preorder $\preceq^c$ on $\minroot$ by 
  \[
  \alpha \preceq^c \beta \Leftrightarrow \arg c(\alpha) \leq \arg c(\beta),
  \]
where the arguments take value in $(0, \pi)$. 
Then, $\preceq^c$ satisfies the requirements. 
\end{proof}


Note that $F \cap \Phi_+$ is an affine root system of rank two as indicated in Figure \ref{fig:positiveroots}.  
We specify its simple roots $\{ \alpha'_-, \alpha'_+ \}$ as follows.   \index{$\alpha'_+, \alpha'_-$}
If $(X_N^{(r)}, i) \neq (A_{2l}^{(2)}, l)$, then it is of type $A_1^{(1)}$ with simple system $\{\alpha'_- = \beta_-, \alpha'_+ = \beta_+ \}$ as in the left figure.
If $(X_N^{(r)}, i) = (A_{2l}^{(2)}, l)$ and $w\alpha_i \in \mathring{\Phi}_+$, then it is of type $A_2^{(2)}$ with simple system $\{\alpha'_- = \widetilde{-2w\alpha_i} = -2w\alpha_i + \delta, \alpha'_+ = \beta_+ = w\alpha_i \}$ as in the middle figure. 
If $(X_N^{(r)}, i) = (A_{2l}^{(2)},l)$ and $w\alpha_i \in \mathring{\Phi}_-$, then it is of type $A_2^{(2)}$ with simple system $\{\alpha'_- = \beta_- = -w\alpha_i, \alpha'_+ = \widetilde{2w\alpha_i} = 2w\alpha_i + \delta \}$ as in the right figure. 

\begin{figure} 
  \centering
  \begin{tikzpicture}[scale = 0.6]
   \begin{scope}
   \draw[->] (0,0) -- (0,1) node[above] {$d_i\delta$} ; 
   \draw[->] (0,1) -- (0,2) node[above] {$2d_i \delta$} ; 
   \draw[->] (0,2) -- (0,3) node[above] {$3d_i\delta$}; 
   \draw[dotted] (0, 3.5) -- (0, 4) ; 
   \draw[->] (0,0) -- (1,0.5) node[right] {$\alpha'_+ = \beta_+$} ;
   \draw[->] (0,0) -- (1,1.5) ;
   \draw[->] (0,0) -- (1,2.5) ;
   \draw[->] (0,0) -- (1,3.5) ;
   \draw[->] (0,0) -- (-1,0.5) node[left] {$\alpha'_- = \beta_-$} ;
   \draw[->] (0,0) -- (-1,1.5) ;
   \draw[->] (0,0) -- (-1,2.5) ;
   \draw[->] (0,0) -- (-1,3.5) ; 
   \draw (0,0) node[below] {$0$} ; 
   \draw[dotted] (1,4) -- (1,4.5) ;
   \draw[dotted] (-1,4) -- (-1,4.5) ; 
   \end{scope}
  
   \begin{scope}[xshift = 6.5cm]
   \draw (0,0) node[below] {$0$}; 
   \draw[->] (0,0) -- (0,2) node[above] {$\delta$} ; 
   \draw[->] (0,2) -- (0,4) node[above] {$2\delta$}; 
   \draw[dotted] (0,4.5) -- (0,5) ;
   \draw[->] (0,0) -- (1, 0.5) node[right] {$\alpha'_+ = \beta_+$} ; 
   \draw[->] (0,0) -- (1,2.5) ;
   \draw[->] (0,0) -- (1,4.5) ; 
   \draw[->] (0,0) -- (2,3) ;
   \draw[->] (0,0) -- (2,7) ;
   \draw[->] (0,0) -- (-1,1.5) node[left] {$\beta_-$} ;
   \draw[->] (0,0) -- (-1,3.5) ;
   \draw[->] (0,0) -- (-1,5.5) ;
   \draw[->] (0,0) -- (-2,1) node[left] {$\alpha'_-$} ;
   \draw[->] (0,0) -- (-2,5) ;
   \draw[dotted] (1.5,6) -- (1.5,6.5) ;
   \draw[dotted] (-1.5,5.5) -- (-1.5,6) ;
   \end{scope}
  
   \begin{scope}[xshift = 14cm]
      \draw (0,0) node[below] {$0$}; 
      \draw[->] (0,0) -- (0,2) node[above] {$\delta$} ; 
      \draw[->] (0,2) -- (0,4) node[above] {$2\delta$}; 
      \draw[dotted] (0,4.5) -- (0,5) ;
      \draw[->] (0,0) -- (-1, 0.5) node[left] {$\alpha'_- = \beta_-$} ; 
      \draw[->] (0,0) -- (-1,2.5) ;
      \draw[->] (0,0) -- (-1,4.5) ; 
      \draw[->] (0,0) -- (-2,3) ;
      \draw[->] (0,0) -- (-2,7) ;
      \draw[->] (0,0) -- (1,1.5) node[right] {$\beta_+$} ;
      \draw[->] (0,0) -- (1,3.5) ;
      \draw[->] (0,0) -- (1,5.5) ;
      \draw[->] (0,0) -- (2,1) node[right] {$\alpha'_+$} ;
      \draw[->] (0,0) -- (2,5) ;
      \draw[dotted] (1.5,5.5) -- (1.5,6) ;
      \draw[dotted] (-1.5,6) -- (-1.5,6.5) ;
   \end{scope}
  \end{tikzpicture}
  \alt{Images of the positive roots in rank 2 affine root systems}
  \caption{Affine rank 2 positive roots} \label{fig:positiveroots}
  \end{figure}
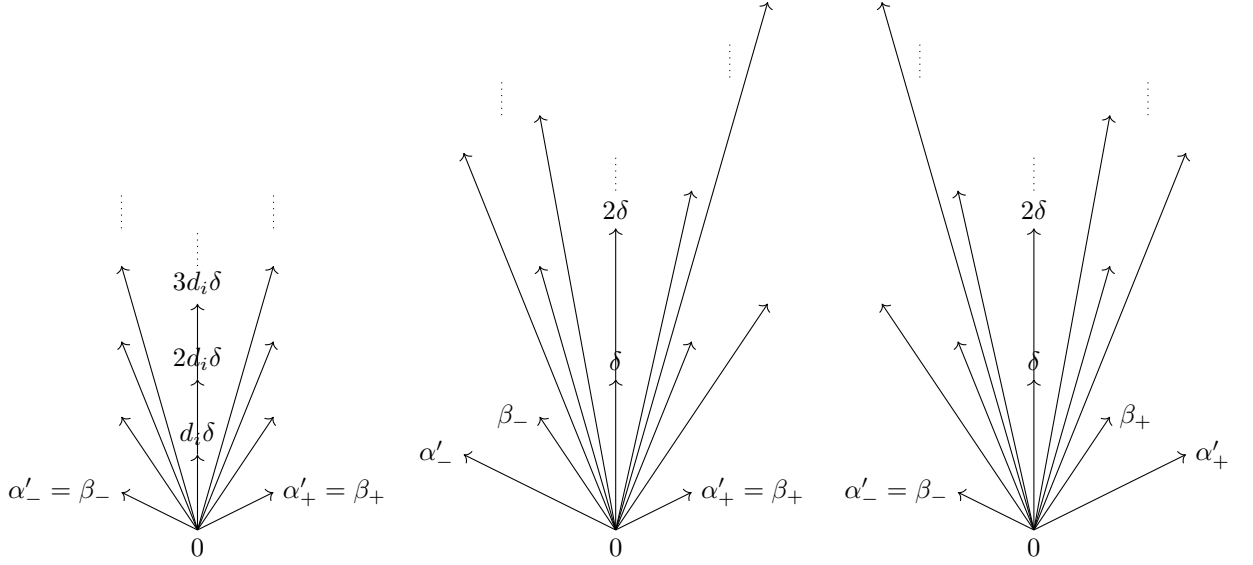 

We have two convex orders $\preceq_0^+, \preceq_0^-$ on $F \cap \minroot$ determined by 
\[
\alpha'_- \prec_0^+ \alpha'_+, \alpha'_+ \prec_0^- \alpha'_-.  
\]
In Figure \ref{fig:positiveroots}, as we proceed clockwise, the vectors get bigger with respect to $\preceq_0^+$. 
Let $\preceq$ be a convex preorder on $\minroot$ as in Lemma \ref{lem:adaptedconvex}. 
The convex orders $\preceq_0^+$ and $\preceq_0^-$ above allow us to refine $\preceq$ on $F \cap \minroot$, as in \cite[Lemma 1.20]{MR3771147}. 
Making them even finer, we obtain convex orders $\preceq^+$\index{$\preceq^+$} and $\preceq^-$\index{$\preceq^-$} on $\minroot$ \cite[Proposition 1.21]{MR3771147}: 
they satisfy the second condition of Lemma \ref{lem:adaptedconvex} and 
\[ 
\alpha'_- \prec^+ \delta \prec^+ \alpha'_+, \alpha'_+ \prec^- \delta \prec^- \alpha'_-. 
\]
The coarse type of $\preceq^+$ is $w$ and that of $\preceq^-$ is $ws_i$. 
Note that the convex order introduced in \cite[Example 3.6]{MR3694676} is an example of $\preceq^+$.

\begin{lemma} \label{lem:chambercoweightmodule}
  Consider convex orders $\preceq^+$ and $\preceq^-$ on $\minroot$ as above.  
  Then, $L^{\preceq^+}(\alpha'_{\epsilon}) \simeq L^{\preceq^-} (\alpha'_{\epsilon})$ for $\epsilon \in \{+,-\}$.  
\end{lemma}

\begin{proof}
  Observe that 
  \[
  \Phi_{+, \preceq^+ \alpha'_-} \subset \Phi_{+, \preceq^- \alpha'_-}, \Phi_{+, \preceq^- \alpha'_+} \subset \Phi_{+, \preceq^+ \alpha'_+}. 
  \]
  Hence, Corollary \ref{cor:independenceoforder} shows that $L^{\preceq^+}(\alpha'_-)$ is $\preceq^-$-cuspidal and $L^{\preceq^-}(\alpha'_+)$ is $\preceq^+$-cuspidal. 
  For a real root, self-dual cuspidal module is unique by Proposition \ref{prop:realrootmodule}, hence the lemma follows.  
\end{proof}

\begin{lemma} \label{lem:minimalpair}
  Consider convex orders $\preceq^+$ and $\preceq^-$ on $\minroot$ as above.  
  Then, 
  \begin{enumerate}
  \item $(\beta_-, \beta_+)$ is a minimal element of $\{ \underline{\gamma} \in \mathcal{P}^{\preceq^+}(d_i \delta) \mid \underline{\gamma} > (d_i\delta) \}$. 
  \item $(\beta_+, \beta_-)$ is a minimal element of $\{ \underline{\gamma} \in \mathcal{P}^{\preceq^-}(d_i \delta) \mid \underline{\gamma} > (d_i\delta) \}$. 
  \end{enumerate}
\end{lemma}

\begin{proof}
  It is immediate from the form of $\preceq^+$ and $\preceq^-$.   
\end{proof}

\begin{proposition}[{cf. \cite[Lemma 13.3]{MR3694676}}] \label{prop:imaginarycuspidal}
  Assume that $(X_N^{(r)}, i) \neq (A_{2l}^{(2)}, l)$. 
  Then, 
  \begin{enumerate}
  \item $(L^w(\beta_-), L^w(\beta_+))$ is unmixing, and
  \item the head of $L^w(\beta_-) \circ L^w(\beta_+)$ is a self-dual simple $R^w(d_i \delta)$-module, not just an $R(d_i\delta)$-module. 
  \end{enumerate}\end{proposition}

\begin{proof}
(1) 
Note that $\alpha'_+ = \beta_+$ and $\alpha'_- = \beta_-$. 
Since $\preceq^+$ is of coarse type $w$, we have $L^w(\beta_+) \simeq L^{\preceq^+}(\beta_+)$ and $L^w(\beta_-) \simeq L^{\preceq^+}(\beta_-)$. 
By Lemma \ref{lem:chambercoweightmodule}, $L^w(\beta_-)$ and $L^w(\beta_+)$ are $\preceq^-$-cuspidal. 
Since $\beta_- \succ^- \beta_+$, we deduce that $(L^w(\beta_-), L^w(\beta_+))$ is unmixing by Lemma \ref{lem:cuspidalseq}. 

(2) By (1) and Lemma \ref{lem:unmixing}, the head of $L^w(\beta_-) \circ L^w(\beta_+)$ is self-dual simple. 
It remains to prove that it is an $R^w(d_i\delta)$-module. 
$L^w(\beta_-)$ and $L^w(\beta_+)$ are also $\preceq^+$-cuspidal and $\beta_- \prec^+ \beta_+$, so $(L^w(\beta_+), L^w(\beta_-))$ is unmixing. 
Hence, Lemma \ref{lem:unmixingr} shows 
\[ 
\Lambda(L^w(\beta_-), L^w(\beta_+)) + \Lambda(L^w(\beta_+), L^w(\beta_-)) = -2(\beta_-, \beta_+) = 2(\alpha_i,\alpha_i) \neq 0. 
\]
Proposition \ref{prop:commuting} (1) shows that $L^w(\beta_+) \circ L^w(\beta_-)$ is not simple. 
On the other hand, Proposition \ref{prop:properstandardcomposition} (2) and Lemma \ref{lem:minimalpair} (1) show that 
every composition factor of $L^w(\beta_+) \circ L^w(\beta_-)$ other than $L^w(\beta_+) \nabla L^w(\beta_-)$ is a $\preceq^+$-cuspidal $R(d_i \delta)$-module. 
Hence, $L^w(\beta_-) \nabla L^w(\beta_+) \simeq L^w(\beta_+) \Delta L^w(\beta_-)$ (up to degree shift) is a $\preceq^+$-cuspidal module. 
Since $\preceq^+$ is of coarse type $w$, (2) follows. 
\end{proof}

\begin{definition} \label{def:minimalimaginary}
When $(X_N^{(r)},i) \neq (A_{2l}^{(2)},l)$, we define 
\[
L_i^w(d_i \delta) = L^w(\beta_-) \nabla L^w(\beta_+).
\] \index{$L_i^w(d_i\delta)$}
\end{definition}

Next, we address the case $(X_N^{(r)},i) = (A_{2l}^{(2)},l)$. 

\begin{lemma} \label{lem:independence}
  Let $\epsilon \in \{+,-\}$. 
  Assume that $(X_N^{(r)},i) = (A_{2l}^{(2)},l)$. 
  For any cuspidal orders $\preceq, \preceq'$ of coarse type $w$, we have $L^{\preceq}(\alpha'_{\epsilon}) \simeq L^{\preceq'}(\alpha'_{\epsilon})$. 
\end{lemma}

\begin{proof}
  The proof is almost identical to that of Proposition \ref{prop:chambercoweight}. 
  We prove the case $w\alpha_i \in \mathring{\Phi}_+$: the other case is similar.  
  Since $\alpha'_+ = \beta_+$, we have $L^{\preceq}(\alpha'_+) = L^w(\alpha'_+) = L^{\preceq'}(\alpha'_+)$ by Proposition \ref{prop:chambercoweight}. 
  It remains to prove the assertion for $\alpha'_-$. 
  Take $\gamma \in \W^*(L^{\preceq}(\alpha'_-)) \cap \Phi_+$. 
  Then, we have $\gamma \succeq \alpha'_-$ and $\alpha'_-  - \gamma \in \spn_{\mathbb{Z}_{\geq 0}} \Phi_{+, \preceq \alpha'_-}^{\text{min}}$. 
  It follows that $p(\alpha'_- - \gamma) \in w \mathring{Q}_-$. 
  
 If $p(\alpha'_- - \gamma) =0$, then $\alpha'_- - \gamma \in \mathbb{Z}_{\geq 0}\delta$. 
 Since $\alpha'_- -\gamma \in \spn_{\mathbb{Z}_{\geq 0}} \Phi_{+, \preceq \alpha'_-}^{\mathrm{min}}$ and $\alpha'_- \prec \delta$, we must have $\alpha'_- -\gamma = 0$ by the convexity of $\preceq$. 

 Suppose $p(\gamma) = 0$, then $\gamma \in \mathbb{Z}_{> 0}\delta$. 
 Since $\alpha'_- - \delta = -2 \alpha'_+ \in Q_-$, it contradicts $\alpha'_- -\gamma \in \spn_{\mathbb{Z}_{\geq 0}} \Phi_{+, \preceq \alpha'_-}^{\mathrm{min}} \subset Q_+$. 
 
  Assume $p(\alpha'_- - \gamma), p(\gamma) \neq 0$ and suppose $p(\gamma) \in w \mathring{Q}_-$. 
  Since $-2w\alpha_i = p(\gamma) + p(\alpha'_- -\gamma)$, the only possible option is that $p(\gamma) = p(\alpha'_- -\gamma) = -w\alpha_i$. 
  Hence, we have $\gamma, \alpha'_- -\gamma \in \widetilde{-w\alpha_i} + \mathbb{Z}_{\geq 0}\delta$. 
  It follows that $\widetilde{-2w\alpha_i} = \alpha'_- = \gamma + (\alpha'_- - \gamma) \in 2(\widetilde{-w\alpha_i}) + \mathbb{Z}_{\geq 0}\delta$. 
  However, $\widetilde{-2w\alpha_i} = -2w\alpha_i + \delta, \widetilde{-w\alpha_i} = -w\alpha_i + \delta$, so it is impossible. 
  Therefore, we deduce that $p(\gamma) \in w \mathring{Q}_+$ and $\gamma \succeq' \delta \succeq' \alpha'_-$. 

We have shown that $\W^*(L^{\preceq}(\alpha'_-)) \cap \Phi_+ \subset \Phi_{+, \succeq' \alpha'_-}$ and the Lemma follows from Corollary \ref{cor:cuspidalcriterion}. 
\end{proof}

Let $L^w(\alpha'_{\epsilon}) \ (\epsilon \in \{+,-\})$ denote the module in this lemma.  \index{$L^w(\alpha'_+), L^w(\alpha'_-)$}
We define $L^w(n\alpha'_{\epsilon})\index{$L^w(n\alpha'_+), L^w(n\alpha'_-)$} = q_l^{n(n-1)/2} L^w(\alpha'_{\epsilon})^{\circ n}$, which is self-dual simple by Proposition \ref{prop:realcuspidal}. 
By Lemma \ref{lem:chambercoweightmodule}, we have 
\[
L^w(\alpha'_+) \simeq L^{\preceq^+}(\alpha'_+) \simeq L^{\preceq^-}(\alpha'_+),\  L^w(\alpha'_-) \simeq L^{\preceq^+}(\alpha'_-) \simeq L^{\preceq^-}(\alpha'_-). 
\]

\begin{proposition} \label{prop:imaginarycuspidalA22}
  Assume that $(X_N^{(r)},i) = (A_{2l}^{(2)},l)$. 
  If $w\alpha_l \in \mathring{\Phi}_+$, we have 
  \begin{enumerate}
    \item $(L^w(\alpha'_-), L^w(\alpha'_+))$ is unmixing and $L^w(\alpha'_-) \nabla L^w(\alpha'_+) \simeq L^w(\beta_-)$. 
    Moreover, $L^w(\alpha'_+) ~\circ ~L^w(\alpha'_-)$ is of length two with head $L^w(\alpha'_+)\nabla L^w(\alpha'_-)$ and socle $q_l^4 L^w(\beta_-)$.   
    \item $\Lambda (L^w(\beta_-), L^w(\beta_+)) = 2(\alpha_l, \alpha_l)$ and $L^w(\beta_-) \nabla L^w(\beta_+)$ is an $R^w(\delta)$-module. 
    \item $(L^w(\alpha'_-), L^w(2\alpha'_+))$ is unmixing and $L^w(\alpha'_-) \nabla L^w(2\alpha'_+)$ is isomorphic to $q_l L^w(\beta_-) \nabla L^w(\beta_+)$. 
  \end{enumerate}
 If $w\alpha_l \in \mathring{\Phi}_-$, we have 
 \begin{enumerate}
  \item $(L^w(\alpha'_-), L^w(\alpha'_+))$ is unmixing and $L^w(\alpha'_-) \nabla L^w(\alpha'_+) \simeq L^w(\beta_+)$. 
  Moreover, $L^w(\alpha'_+) ~\circ ~L^w(\alpha'_-)$ is of length two with head $L^w(\alpha'_+)\nabla L^w(\alpha'_-)$ and socle $q_l^4 L^w(\beta_+)$. 
  \item $\Lambda (L^w(\beta_-), L^w(\beta_+)) = 2(\alpha_l, \alpha_l)$ and $L^w(\beta_-) \nabla L^w(\beta_+)$ is an $R^w(\delta)$-module. 
  \item $(L^w(2\alpha'_-), L^w(\alpha'_+))$ is unmixing and $L^w(2\alpha'_-) \nabla L^w(\alpha'_+)$ is isomorphic to $q_l L^w(\beta_-) \nabla L^w(\beta_+)$. 
\end{enumerate}
\end{proposition}

\begin{proof}
  Since both cases are proved in the same manner, we only consider the case $w\alpha_l \in \mathring{\Phi}_+$.
  Note that $\alpha'_+ = \beta_+, \alpha'_+ + \alpha'_- = \beta_-, 2\alpha'_+ + \alpha'_- = \delta$. 
  
  (1) 
  The proof is almost identical to that of Proposition \ref{prop:imaginarycuspidal}.  
  By Lemma \ref{lem:chambercoweightmodule}, both $(L^w(\alpha'_+), L^w(\alpha'_-))$ and $(L^w(\alpha'_-), L^w(\alpha'_+))$ are unmixing, and Lemma \ref{lem:unmixingr} shows
  \begin{align*}
  &\Lambda (L^w(\alpha'_-), L^w(\alpha'_+)) = \Lambda (L^w(\alpha'_+), L^w(\alpha'_-)) = -(\alpha'_-, \alpha'_+) = 2(\alpha_l, \alpha_l), \\
  &\Lambda (L^w(\alpha'_-), L^w(\alpha'_+))+ \Lambda(L^w(\alpha'_+), L^w(\alpha'_-)) = 4(\alpha_l, \alpha_l) \neq 0. 
  \end{align*}
  Note that $(\alpha'_-, \alpha'_+)$ is a minimal element of $\{ \underline{\gamma} \in \mathcal{P}^{\preceq^+}(\beta_-) \mid \underline{\gamma} > (\beta_-)\}$. 
  Therefore, $L^w(\alpha'_+) \circ L^w(\alpha'_-)$ is not simple by Proposition \ref{prop:commuting}, and every composition factor other than $L^w(\alpha'_+) \nabla L^w(\alpha'_-)$ is $\preceq^+$-cuspidal by Proposition \ref{prop:properstandardcomposition}.
  It follows that 
  \begin{align*}
  L^w(\alpha'_-) \nabla L^w(\alpha'_+) &\simeq \Image (r_{L^w(\alpha'_-), L^w(\alpha'_+)}\colon L^w(\alpha'_-) \circ L^w(\alpha'_+) \\
  &\to q_l^{-4} L^w(\alpha'_+) \circ L^w(\alpha'_-)) \\
  &\simeq q_l^{-4} L^w(\alpha'_+) \Delta L^w(\alpha'_-) 
  \end{align*}
  is a $\preceq^+$-cuspidal $R(\beta_-)$-module. 
  By Lemma \ref{lem:unmixing} (3), it is self-dual. 
  Since there is only one self-dual $\preceq^+$-cuspidal $R(\beta_-)$-module $L^w(\beta_-)$, we have $L^w(\alpha'_-) \nabla L^w(\alpha'_+) \simeq L^w(\beta_-)$. 
  Moreover, Lemma \ref{lem:unmixing} shows $[L^w(\alpha'_+)\circ L^w(\alpha'_-) : L^w(\alpha'_+) \Delta L^w(\alpha'_-)]_q = 1$, so $L^w(\alpha'_+) \circ L^w(\alpha'_-)$ must be of length two. 

  (2) 
  We compute 
  \begin{align*}
  \Lambda (L^w(\beta_-), L^w(\beta_+)) &= \Lambda (L^w(\alpha'_-) \nabla L^w(\beta_+), L^w(\beta_+)) \quad \text{by (1)} \\
  &= \Lambda (L^w(\alpha'_-), L^w(\beta_+)) \quad \text{by Theorem \ref{thm:rmatrix} (7)} \\
  &= 2(\alpha_l, \alpha_l) \quad \text{by the proof of (1)}. 
  \end{align*}
  Since $(L^w(\beta_+), L^w(\beta_-))$ is unmixing by $\preceq^+$-cuspidality and Lemma \ref{lem:cuspidalseq}, we have 
  \[
  \Lambda (L^w(\beta_+), L^w(\beta_-)) = -(\beta_+, \beta_-) = (\alpha_l, \alpha_l) \ \text{(Lemma \ref{lem:unmixingr})}. 
  \]
  Hence, $\Lambda (L^w(\beta_-), L^w(\beta_+)) + \Lambda (L^w(\beta_+), L^w(\beta_-)) = 3(\alpha_l, \alpha_l) \neq 0$, which implies that $L^w(\beta_-) \circ L^w(\beta_+)$ is not simple by Proposition \ref{prop:commuting}.
  The same argument as in Proposition \ref{prop:imaginarycuspidal} proves the assertion. 

  (3) By $\preceq^-$-cuspidality and Lemma \ref{lem:cuspidalseq}, $(L^w(\alpha'_-), L^w(2\alpha'_+))$ is unmixing, and the head of $L^w(\alpha'_-) \circ L^w(2\alpha'_+)$ is self-dual simple by Lemma \ref{lem:unmixing}. 
  On the other hand, we have a surjection 
  \begin{align*}
  L^w(\alpha'_-) \circ L^w(2\alpha'_+) &\simeq q_l L^w(\alpha'_-) \circ L^w(\alpha'_+) \circ L^w(\alpha'_+) \quad \text{by Proposition \ref{prop:realcuspidal}} \\
  &\twoheadrightarrow q_l L^w(\beta_-) \circ L^w(\beta_+) \quad \text{by (1)} \\
  &\twoheadrightarrow q_l L^w(\beta_-) \nabla L^w(\beta_+). 
  \end{align*}
  It proves the assertion. 
\end{proof}

\begin{definition} \label{def:minimalimaginaryA22}
  When $(X_N^{(r)},i) = (A_{2l}^{(2)}, l)$, we define 
  \[
  L_l^w(\delta) = q_l L^w(\beta_-) \nabla L^w(\beta_+).
  \] \index{$L_i^w(d_i\delta)$}
\end{definition}

\begin{proposition} \label{prop:minimalres}
  \begin{enumerate}
  \item If $(X_N^{(r)},i) \neq (A_{2l}^{(2)}, l)$, then $\Res_{\beta_-, \beta_+} L_i^w(d_i\delta) = L^w(\beta_-) \otimes L^w(\beta_+)$. 
  \item If $(X_N^{(r)},i)= (A_{2l}^{(2)}, l)$, then $\Res_{\beta-, \beta_+} L_l^w(\delta)$ is of length two with head $q_l^{-1} L^w(\beta_-)\otimes L^w(\beta_+)$ and socle $q_l L^w(\beta_-)\otimes L^w(\beta_+)$. 
  \item If $L$ is a self-dual simple $R^w(d_i \delta)$-module that is not isomorphic to $L_i^w(d_i \delta)$, 
  then $\Res_{\beta_-, \beta_+} L = 0$. 
  \end{enumerate}
\end{proposition}

\begin{proof}
  We claim that, for any simple $R^w(d_i\delta)$-module $L$, $\Res_{\beta_-, \beta_+} L$ is an $(R^{\preceq^+}(\beta_-) \otimes R^{\preceq^+}(\beta_+))$-module.
  Suppose to the contrary that $\Res_{\beta_-, \beta_+}L$ has a composition factor $L' \otimes L''$ with $L'$ not being $\preceq^+$-cuspidal. 
  Then, there exists $\alpha \in \Phi_{+, \succ^+ \beta_-} \cap \W(L')$ by Corollary \ref{cor:cuspidalcriterion}. 
  Note that $\beta_- - \alpha \in Q_+$. 
  Since $\W(L') \subset \W(L)$ and $L$ is $\preceq^+$-cuspidal, $\alpha \preceq^+ \delta$. 
  We verify that $\alpha \succ^+ \beta_-$ contradicts $\beta_- - \alpha \in Q_+$ using Figure \ref{fig:positiveroots}.  
  If $(X_N^{(r)}, i) \neq (A_{2l}^{(2)}, l)$, we have either $\alpha \in \beta_- + \mathbb{Z}_{>0} d_i \delta$ or $\alpha \in \mathbb{Z}_{>0} \delta$. 
  If $(X_N^{(r)}, i) = (A_{2l}^{(2)}, l)$ and $w\alpha_l \in \mathring{\Phi}_+$, 
  then we have $\alpha \in \beta_- + \mathbb{Z}_{> 0}\delta$ or $\alpha \in \alpha'_- + 2\mathbb{Z}_{> 0}\delta$ or $\alpha \in \mathbb{Z}_{>0}\delta$. 
  If $(X_N^{(r)}, i) = (A_{2l}^{(2)}, l)$ and $w\alpha_l \in \mathring{\Phi}_-$,
  then we have $\alpha \in \beta_- + \mathbb{Z}_{> 0}\delta$ or $\alpha \in \widetilde{-2w\alpha_l} + 2\mathbb{Z}_{\geq 0}\delta$ or $\alpha \in \mathbb{Z}_{>0}\delta$. 
  In any case, it contradicts $\beta_- - \alpha \in Q_+$. 
  When $L''$ is not cuspidal, by considering $\W^*(L'')$ we obtain a contradiction in the same manner. 
  The claim is proved. 

  (1) follows from $(L^w(\beta_-),L^w(\beta_+))$ being unmixing (Proposition \ref{prop:imaginarycuspidal}). 

  (2) 
  Assume that $w\alpha_l \in \mathring{\Phi}_+$:
  the other case is similar. 

  First, we compute $[\Res_{\beta_-, \beta_+}L_l^w(\delta)]$. 
  By the discussion above, it is $a [L^w(\beta_-) \otimes L^w(\beta_+)]$ for some $a \in \mathbb{Z}[q,q^{-1}]$. 
  We apply $\Res_{\alpha'_-, \beta_+,\beta_+}$ to determine $a$. 
  We have 
  \begin{align*}
  &[\Res_{\alpha'_-, \beta_+, \beta_+} L_l^w(\delta)] \\
  &= [L^w(\alpha'_-) \otimes \Res_{\beta_+, \beta_+} L^w(2\beta_+)] \quad \text{by Proposition \ref{prop:imaginarycuspidalA22} (3)} \\
  &= [2]_l [L^w(\alpha'_-) \otimes L^w(\beta_+)\otimes L^w(\beta_+)] \\
  &\text{by the Mackey filtration and $\preceq^+$-cuspidality (cf. \cite[Lemma 2.11]{MR3205728})}, 
  \end{align*}
  while
  \begin{align*}
  &[(\Res_{\alpha'_-, \beta_+} L^w(\beta_-)) \otimes L^w(\beta_+)] \\ 
  &= [L^w(\alpha'_-) \otimes L^w(\beta_+) \otimes L^w(\beta_+)] \quad \text{by Proposition \ref{prop:imaginarycuspidalA22} (1)}. 
  \end{align*}
  We deduce $a = [2]_l$, that is, $[\Res_{\beta_-,\beta_+} L_l^w(\delta)] = [2]_l [L^w(\beta_-) \otimes L^w(\beta_+)]$. 

  On the other hand, the induction-restriction adjunction shows
  \begin{align*}
  &\Hom_{R(\beta_-) \otimes R(\beta_+)} (L^w(\beta_-) \otimes L^w(\beta_+), \Res_{\beta_-, \beta_+} L_l^w(\delta)) \\
  &\simeq \Hom_{R(\delta)} (L^w(\beta_-) \circ L^w(\beta_+), L_l^w(\delta)), 
  \end{align*}
 which is isomorphic to $q_l \mathbf{k}$ by Proposition \ref{prop:imaginarycuspidalA22}. 
 Therefore, the assertion follows. 

  (3) 
  We ignore degree shifts. 
  Suppose that $\Res_{\beta_-, \beta_+} L \neq 0$.
  It is a nonzero $(R^{\preceq^+}(\beta_-) \otimes R^{\preceq^+}(\beta_+))$-module by the claim, so it has a simple submodule isomorphic to $L^w(\beta_-) \otimes L^w(\beta_+)$. 
  By the induction-restriction adjunction, we have a nonzero homomorphism $L^w(\beta_-) \circ L^w(\beta_+) \to L$. 
  However, the head of $L^w(\beta_-) \circ L^w(\beta_+)$ is $L_i^w(d_i\delta)$, so $L \simeq L_i^w(d_i \delta)$.
  It contradicts the assumption. 
\end{proof}

Let $\Delta_i^w(d_i \delta)$\index{$\Delta_i^w(d_i\delta)$} be the projective cover of $L_i^w(d_i \delta)$ in $\gMod{R^w(d_i \delta)}$. 

\begin{theorem}[{cf. \cite[Theorem 17.1]{MR3694676}}] \label{thm:psimodule}

 If $(X_N^{(r)}, i) \neq (A_{2l}^{(2)}, l)$, then there exists a short exact sequence
 \[
 0 \to q_i^2 \widehat{L}^w(\beta_+) \circ \widehat{L}^w(\beta_-) \to \widehat{L}^w (\beta_-) \circ \widehat{L}^w(\beta_+) \to \Delta_i^w(d_i\delta) \to 0, 
 \]
 where the injection is the renormalized R-matrix. 
 If $(X_N^{(r)}, i) = (A_{2l}^{(2)}, l)$, then there exist short exact sequences
 \begin{align*}
 0 \to q_l^2 \widehat{L}^w(\beta_+) \circ \widehat{L}^w(\beta_-) \to \widehat{L}^w (\beta_-) \circ \widehat{L}^w(\beta_+) \to q_l \Delta_l^w(\delta) \oplus q_l^{-1} \Delta_l^w(\delta) \to 0, \\
 0 \to q_l^4 \widehat{L}^w(\alpha'_+) \circ \widehat{L}^w(\alpha'_-) \to \widehat{L}^w(\alpha'_-) \circ \widehat{L}^w(\alpha'_+) \to \widehat{L}^w(\alpha'_+ + \alpha'_-) \to 0. 
 \end{align*}
 where the injections are the renormalized R-matrices. 
\end{theorem}

\begin{proof}
 We provide a proof to the first and the second short exact sequence.
 The third one is proved in the same manner. 
 Since $(L^w(\beta_+), L^w(\beta_-))$ is unmixing by the $\preceq^+$-cuspidality, we have $\Lambda (L^w(\beta_+), L^w(\beta_-)) = -(\beta_+, \beta_-) = (\alpha_i, \alpha_i)$ by Lemma \ref{lem:unmixingr}. 
 Hence, $\renormalizedR{\widehat{L}^w(\beta_+)}{\widehat{L}^w(\beta_-)}$ is of degree $(\alpha_i, \alpha_i)$. 
 It is injective by Proposition \ref{prop:reninj}. 

 Let $C = \Cok (\renormalizedR{\widehat{L}^w(\beta_+)}{\widehat{L}^w(\beta_-)})$.
 We prove that $C$ is a projective object of $\gMod{R^w(d_i\delta)}$. 
 First, we verify that $C$ is an $R^w(d_i \delta)$-module. 
 We compute in $K(\gmod{R})_{\mathbb{Z}((q))}$, 
 \begin{align*}
 [C]_q &= [\widehat{L}^w(\beta_-) \circ \widehat{L}^w(\beta_+)]_q - q_i^2 [\widehat{L}^w(\beta_+) \circ \widehat{L}^w(\beta_-)]_q \\
  & = (1-q_i^2)^{-2} ([L^w(\beta_-)] [L^w(\beta_+)] - q_i^2 [L^w(\beta_+)][L^w(\beta_-)]). 
\end{align*}
On the other hand, as explained in the proof of Proposition \ref{prop:imaginarycuspidal} and Proposition \ref{prop:imaginarycuspidalA22}, we have 
\begin{align*}
[L^w(\beta_+)][L^w(\beta_-)] &= [L^w(\beta_+)\nabla L^w(\beta_-)] + \text{($R^w(d_i\delta)$-modules)}, \\
[L^w(\beta_-)][L^w(\beta_+)] &= q_i^2[L^w(\beta_+)\nabla L^w(\beta_-)] + \text{($R^w(d_i \delta)$-modules)}. 
\end{align*}
We conclude that every composition factor of $C$ is an $R^w(d_i \delta)$-module.

Next, we prove that $C$ is a projective $R^w(d_i \delta)$-module.
It suffices to  show that $\Ext_{R(d_i \delta)}^1 (C, L) = 0$ for any simple $R^w(d_i \delta)$-module $L$.
Applying $\Ext_{R(d_i \delta)}^{\bullet} (\cdot, L)$ to the short exact sequence 
\[
0 \to q_i^2 \widehat{L}^w (\beta_+) \circ \widehat{L}^w (\beta_-) \to \widehat{L}^w (\beta_-) \circ \widehat{L}^w (\beta_+) \to C \to 0, 
\]
we obtain a long exact sequence. 
We have $\Res_{\beta_+, \beta_-} L = 0$ since $L \in \gMod{R^w(d_i\delta)}$.  
By the induction-restriction adjunction, $\Ext_{R(d_i \delta)}^{\bullet} (\widehat{L}^w(\beta_+)\circ \widehat{L}^w(\beta_-), L) = 0$. 
Hence, we have an isomorphism 
\[
\Ext_{R(d_i\delta)}^{\bullet} (C, L) \simeq \Ext_{R(d_i\delta)}^{\bullet} (\widehat{L}^w(\beta_-) \circ \widehat{L}^w(\beta_+), L), 
\]
which is isomorphic to $\Ext_{R(d_i\delta)}^{\bullet} (\widehat{L}^w(\beta_-)\otimes \widehat{L}^w(\beta_+), \Res_{\beta_-, \beta_+}L)$ by the induction-restriction adjunction. 
By Proposition \ref{prop:minimalres}, the restriction is an $R^{\preceq^+}(\beta_-) \otimes R^{\preceq^+}(\beta_+)$-module. 
Hence, the assertion follows from Theorem \ref{thm:ext1}.

It remains to determine $\hd C$.  
All we have to do is take an arbitrary simple $R^w(d_i\delta)$-module $L$ and compute $\Hom_{R(d_i\delta)}(C,L)$ as in the previous paragraph using Proposition \ref{prop:minimalres}. 
\end{proof}

\begin{corollary} \label{cor:imaginarymodule}
 For $i \in \mathring{I}$, $\Psi_1(\Delta_i^w(d_i\delta))$ coincides with the imaginary root vector $S_i^{w, (1)}$ introduced in \cite{MR3874704}. 
\end{corollary}

\begin{proof}
By Theorem \ref{thm:psimodule} and Section \ref{sub:realrootmodules},
it is immediate from the definitions in \cite{MR3874704, MR2066942}. 
\end{proof}

\appendix

\section{Proof of Proposition \ref{prop:dethead}} \label{app:correction}

In the proof of \cite[Proposition 4.6]{MR3771147}, the following claim is incorrect: 
if $w \geq v \geq x \geq y \in W$, then $(M(w\Lambda,v\Lambda), M(v\Lambda,x\Lambda), M(x\Lambda,y\Lambda))$ is unmixing. 
We provide a counterexample below. 
Let $C$ be the Cartan matrix of type $A_2$, $I = \{1,2\}$, so that $W = \Sym_3$. 
Let $\Lambda = \Lambda_1 + \Lambda_2$ and $w = s_1s_2s_1, v=s_2s_1, x = s_1, y = e$.  
Then, $M(w\Lambda,v\Lambda) \simeq L(1), M(v\Lambda,x\Lambda) \simeq L(2^2), M(x\Lambda,y\Lambda) \simeq L(1)$. 
We have 
\begin{align*}
&e(1,2,2,1)(L(1) \circ L(2^2) \circ L(1)) \\ 
&= (L(1) \otimes L(2^2) \otimes L(1)) \oplus \tau_{(1,4)} (L(1) \otimes L(2^2) \otimes L(1)). 
\end{align*}
Hence, $(L(1), L(2^2), L(1))$ is not unmixing. 

We note that alternative proofs of this proposition for symmetric cases are given in \cite[Theorem 10.3.1]{MR3758148} and \cite[Theorem 11.2]{mcnamara2021clustermonomialsdualcanonical}. 
McNamara's argument appears to extend to nonsymmetric cases as well. 
Below, we provide a brief, corrected proof. 

\begin{proof}[Proof of Proposition \ref{prop:dethead}]
By Lemma \ref{lem:detsupport}, we have $\W^*(M(w\Lambda,v\Lambda)) \subset Q_+ \cap vQ_+, \W(M(v\Lambda,x\Lambda)) \subset Q_+ \cap v Q_-$, so $(M(w\Lambda,v\Lambda), M(v\Lambda,x\Lambda))$ is unmixing. 
By Lemma \ref{lem:unmixing} (3), the head $M(w\Lambda,v\Lambda) \nabla M(v\Lambda,x\Lambda)$ is self-dual simple. 

First, assume that $x = e$. 
Set $\beta = v\lambda-w\lambda, \gamma = \Lambda-v\Lambda$. 
By the description of $M(w\Lambda,v\Lambda)$ given in Lemma \ref{lem:determinantial}, we have 
\[
M(w\Lambda,v\Lambda) \simeq E_{j_1}^{*(b_1)} \cdots E_{j_n}^{*(b_n)}M(w\Lambda,\Lambda) \subset \Res_{\beta,\gamma}M(w\Lambda,\Lambda)
\]
as an $R(\beta)$-module. 
Hence, $X = \Hom_{R(\beta)} (M(w\Lambda,v\Lambda), \Res_{\beta,\gamma}M(w\Lambda,\Lambda))$ is a nonzero $R(\gamma)$-module.
Moreover, since $M(w\Lambda,\Lambda)$ is an $R^{\Lambda}(w\Lambda)$-module, $X$ is an $R^{\Lambda}(v\Lambda)$-module. 
By Theorem \ref{thm:sl2categorification}, $R^{\Lambda}(v\Lambda)$ has a unique simple module $M(v\Lambda,\Lambda)$ up to degree shift. 
Hence, there exists an injective $R^{\Lambda}(v\Lambda)$-homomorphism $q^d M(v\Lambda,\Lambda) \to X$ for some $d \in \mathbb{Z}$. 
By the $\otimes$-$\Hom$ adjunction and the induction-restriction adjunction, it induces a nonzero homomorphism $q^d M(w\Lambda, v\Lambda) \circ M(v\Lambda,\Lambda) \to M(w\Lambda,\Lambda)$, hence $M(w\Lambda,v\Lambda)\nabla M(v\Lambda,\Lambda) \simeq q^{-d} M(w\Lambda, \Lambda)$. 
Since both $M(w\Lambda,v\Lambda)\nabla M(v\Lambda,\Lambda)$ and $M(w\lambda,\Lambda)$ are self-dual, $d = 0$. 
Therefore, the proof is complete in this case.

In general, the reduction argument in \cite[Proposition 4.6]{MR3771147} to the case $x=e$ is valid. 
\end{proof}

\section{Comparison with other references} \label{app:B}

Different references use various conventions for quantum groups. 
We summarize in a table below how to translate other references into our setup. 
To translate one reference into another, simply replace each symbol appearing in the table with the corresponding symbol in the same row. 

\begin{table}[h]
\centering
\begin{tabular}{ccccc} \toprule
This paper & \cite{MR2066942, MR3874704} & \cite{MR3090232} & \cite{MR2914878} & \cite{MR2759715} \\ \midrule 
$q$ & $q^{-1}$ & $q^{-1}$ & $q$ & $v^{-1}$ \\
$e_i$ & $F_i$ & $f_i$ & $e_i$ & $F_i$ \\
$f_i$ & $E_i$ & $e_i$ & $f_i$ & $E_i$ \\
$q^h$ & $q^h$ & $q^{-h}$ & $q^h$ & $K_h$ \\
$\quantum{}$ & $\quantum{}$ & $\quantum{}^{\text{op}}$ & $\quantum{}$ & $\mathbf{U}$ \\
$T_i$ & $T_i$ & $T_i$ & $T_i^{-1} = T'_{i,-1} $ & $T''_{i,1}$\\
$T_i^{-1}$ & $T_i^{-1}$ & $T_i^{-1}$ & $T_i = T''_{i,1}$ & $T'_{i,-1}$ \\
$f_{\beta_k}$ && $E(\beta_k)$ & $F_{-1}(\beta_k)$ & \\ 
$f_{\lambda}$ && $E(\mathbf{a})$ & $F_{-1}(\mathbf{c}, \tilde{w})$ & \\ 
convex order $\preceq$ & $\preceq^{\mathrm{op}}$ &&& \\ \bottomrule
\end{tabular}
\alt{A table that indicates how to translate otherreferences into our setup}
\caption{Comparison with other references}
\end{table}

\begingroup
\let\clearpage\relax
\printindex
\endgroup

\bibliographystyle{amsalpha}
\bibliography{library.bib}

\end{document}